\documentclass{article}
\usepackage[%
journal=AIHPC,    %% Replace XXX by one of the above journal codes.
lang=british,   %% Change to `american' if you use American English.
]{ems-journal}

\usepackage[utf8]{inputenc}
\usepackage{mathtools}
\usepackage{imakeidx}
\usepackage{needspace}

\newtheorem{theorem}{Theorem}[section]
\newtheorem{lemma}[theorem]{Lemma}
\newtheorem{proposition}[theorem]{Proposition}
\newtheorem{corollary}[theorem]{Corollary}
\newtheorem{definition}[theorem]{Definition}
\newtheorem{notation}[theorem]{Notation}
\newtheorem{remark}[theorem]{Remark}

\DeclareMathOperator{\spec}{spec}
\DeclarePairedDelimiter\ceil{\lceil}{\rceil}
\DeclarePairedDelimiter\floor{\lfloor}{\rfloor}
\DeclareMathOperator{\dom}{Dom}
\let\Im\relax
\DeclareMathOperator{\Im}{Im}
\let\Re\relax
\DeclareMathOperator{\Re}{Re}
\usepackage{xparse, etoolbox}
\DeclarePairedDelimiterX\innerp[2]{\langle}{\rangle}{#1,#2}
\DeclareMathOperator*{\fsum}{\sum^{\mathrm{finite}}}
\DeclareMathOperator{\kadm}{\mathit{k}-adm}

\numberwithin{equation}{section}

\makeatletter
\newcommand{\subalign}[1]{%
  \vcenter{%
    \Let@ \restore@math@cr \default@tag
    \baselineskip\fontdimen10 \scriptfont\tw@
    \advance\baselineskip\fontdimen12 \scriptfont\tw@
    \lineskip\thr@@\fontdimen8 \scriptfont\thr@@
    \lineskiplimit\lineskip
    \ialign{\hfil$\m@th\scriptstyle##$&$\m@th\scriptstyle{}##$\hfil\crcr
      #1\crcr
    }%
  }%
}
\makeatother

\makeindex[columns=2, title=Index of Notations]

\begin{document}

%------
% Insert the title of your paper and (if necessary)
% a short title for the running head.
%------
\title{Construction of blow-up solutions for the focusing energy-critical nonlinear wave equation in $\mathbb R^4$ and $\mathbb R^5$}
\titlemark{Construction of blow-up solutions to (NLW)}

%------

%%%% Pls fill in all fields for each author
%%%% Label the authors by their position in the authors' list using {}
%%%% If you published any math paper ever, you have an MR Author ID.
%  Please look it up in three easy (and free) steps:
% 1. copy the bibliographic data of any published paper (co-)authored by you in the search field at https://mathscinet.ams.org/mathscinet/freetools/mref
% 2. Hit your name in the search result
% 3. Find your MR Author ID in the first row, copy it in the \mrid{} field
%%%% If you have not created your ORCID yet, you may like to do it now, pls copy it in the field \orcid{}
%%%% Abbreviate first names for the running head

\emsauthor{1}{
	\givenname{Dylan}
	\surname{Samuelian}
        \mrid{}
	\orcid{}}{D.~Samuelian}
%%%% Repeat the same fields for each numbered author

%%%% Please provide detailed address info for each author
%%%% Use the same numbering as for \emsauthor above
%%%% Please look up the ROR ID of your institute here: https://ror.org
\Emsaffil{1}{
	\department{Institute of Mathematics – EPFL SB MATH}
	\organisation{Ecole Polytechnique Fédérale de Lausanne (EPFL)}
	\rorid{02s376052}
	\address{MA A2 383 (Bâtiment MA), Station 8}
	\zip{CH-1015}
	\city{Lausanne}
	\country{Switzerland}
	\affemail{dylan.samuelian@epfl.ch}}

%------
% Add MSC 2020 codes according to https://zbmath.org/classification/.
% A unique primary MSC code (in curly brackets) is mandatory,
% while secondary MSC codes (in square brackets) are optional.
%------
\classification[35B44]{35C08}

%------
% Add a list of keywords.
%------

%------
% Insert your abstract.
%------
\begin{abstract}
We construct solutions $u(x,t)$ to the focusing, energy-critical, nonlinear wave equation \begin{equation} \partial_{tt}u - \Delta u - |u|^{p-1}u = 0, \quad t \geq 0, \ x \in \mathbb{R}^d, \ d \geq 3, \ p = (d+2)/(d-2) \end{equation} in dimension $d \in \{4,5\}$, exhibiting finite-time Type II blow-up precisely at $x = t = 0$ with a prescribed polynomial blow-up rate of $t^{-1-\nu}$, where $\nu > 1$ for $d = 4$ and $\nu > 3$ for $d = 5$. Such solutions have been constructed by Krieger-Schlag-Tataru for $d = 3$ and by Jendrej for $d = 5$. The work of Jendrej includes the extremal case $\nu = 3$, which our method does not address, and the regime $\nu > 8$. The major difference between dimensions $4$ and $5$ consists in the renormalization procedure. In $d = 4$, we essentially follow the Krieger-Schlag-Tataru scheme developed for the 3-dimensional equation. This scheme has been applied with success for other equations such as the 3D-critical NLS, Schrödinger maps or wave maps. In all of these cases, the polynomial structure of the nonlinearity permits the use of simple algebraic manipulations to control error terms. By contrast, the case $d = 5$ requires a modified setup due to the lower regularity of the nonlinearity, which complicates the treatment of nonlinear error terms.
\end{abstract}

\maketitle

%------
% INSERT THE BODY OF THE PAPER HERE (except
% acknowledgments, funding info and bibliography)
%------

\tableofcontents

\section{Introduction}

We study the focusing, energy-critical, nonlinear wave equation
\begin{equation}
    \partial_{tt}u - \Delta u - |u|^{p-1}u = 0, \quad t \geq 0, \ x \in \mathbb R^d, \ d \geq 3, \ p = (d+2)/(d-2) \quad \text{(NLW)}.
\label{wave eq}
\index{p@$p = (d+2)/(d-2)$, the critical exponent}
\end{equation}
The equation is called energy-critical because it is invariant under the scaling $u_{\lambda}(x,t) = \lambda^{(d-2)/2}u(\lambda x, \lambda t)$. Moreover, the $\dot{H}^1 \times L^2$ norm of the rescaled data $(u_{\lambda}, \partial_t u_{\lambda})$ at $t = 0$ is independent of $\lambda$. The equation is locally well-posed (in the sense that a solution given by the Duhamel formula exists) for initial data in $\dot{H}^1 \times L^2$. If a solution fails to be global, then its $L_{t,x}^{\frac{2(d+1)}{(d-2)}}$-norm blows up (\cite{Kenig_Merle_local_wellposed_low_dim}, \cite{Critical_local}). Changing the sign of the nonlinearity leads to the defocusing case, which is known to be globally well-posed (\cite{Critical_defocusing_global}).

A classical example of blow-up is due to Levine (\cite[Section 12.5.1]{evans10}). Utilizing the conserved energy functional
\begin{equation}
E(u(t),\partial_tu(t)) = \int_{\mathbb R^d} \underbrace{\frac{1}{2} |\nabla_{t,x}u(t)|^2}_{kinetic \ en.} \underbrace{- \frac{1}{p+1}|u(t)|^{p+1}}_{potential \ en.}dx
\index{E(u(t),\partial_tu(t))@$E(u(t),\partial_tu(t))$, the conserved energy}
\end{equation}
of the solution, Levine showed that if the initial data is smooth and compactly supported with negative energy $E(u(0),\partial_tu(0)) < 0$, then the solution cannot exist globally in time. Denoting by $T_{\max} \in [0,+\infty]$ the maximal forward time of existence of the solution, $$(u(t),\partial_tu(t)) \in C([0,T_{\max}), \dot{H}^1 \times L^2),$$ 
 two distinct types of solutions are distinguished:
\begin{enumerate}
    \item Type I: $||(u(t),\partial_tu(t))||_{L^{\infty}([0,T_{\max}), \dot{H}^1 \times L^2)} = +\infty,$
    \item Type II: $||(u(t),\partial_tu(t))||_{L^{\infty}([0,T_{\max}), \dot{H}^1 \times L^2)} < +\infty.$
\end{enumerate}
Examples of Type I blow-up solutions can be obtained by considering the solution
$$
u_T(x,t) = c_p(T-t)^{-\frac{2}{p-1}}, \quad c_p = \left(\frac{2(p+1)}{(p-1)^2}\right)^{\frac{1}{p-1}}
$$
to the wave equation. Selecting initial data $(u_T(x,0), \partial_t u_T(x,0))$ with an appropriate spatial cutoff produces such a blow-up at a finite time $T_{\max} \in (0,T]$ (\cite[Section 6]{Subcritical_local}).

In the following, we are interested in constructing radially symmetric Type II blow-ups. The study of Type II radial solutions is closely related to the stationary radial solution
\begin{equation}
    W(x) = \left( 1 + \frac{|x|^2}{d(d-2)} \right)^{-\frac{d-2}{2}}
    \index{W(x)@$W(x)$, the ground state}
\end{equation}
called the ground state (or ``bubble'' of energy), which is also an extremizer for the Sobolev embedding $\dot{H}^1 \hookrightarrow L^{p + 1}$. Roughly speaking, the soliton resolution conjecture (recently solved in \cite{soliton_odd}, \cite{soliton_jendrej}) asserts that any radial Type II solution behaves asymptotically as a superposition of dynamically scaled bubbles plus a radiation term, which is given by a free wave if the solution is global and a stationary element in $\dot{H}^1$ otherwise. 

The goal of this paper is to construct explicit solutions $u(x,t)$ of (\ref{wave eq}) in dimension $d \in \{4,5\}$, exhibiting finite-time Type II blow-up precisely at the origin $x = t = 0$, with a prescribed polynomial blow-up rate $t^{-1-\nu}$, where $\nu > 1$ for $d = 4$ and $\nu > 3$ for $d = 5$. Such solutions were previously constructed by Krieger-Schlag-Tataru for $d = 3$ (\cite{Krieger_2012}) with $\nu > 0$ and by Jendrej for $d = 5$ (\cite{Jendrej_blowup5d}) with $\nu \in \{3\} \cup (8,+\infty)$.

As in \cite{Krieger_2012}, where they prove that the previously known range of exponents $\nu > 1/2$ for $d = 3$ can be relaxed to $\nu > 0$, it is likely that our assumption $\nu > 1$ in dimension $4$ can be relaxed to $\nu > 0$. This restriction is a technicality arising from a nonlinear Sobolev estimate. In Proposition \ref{prop:local lipschitz}, we exploit the embedding of the algebra $H^{1 + \frac{(6-d)\nu}{2}}(\mathbb R^d)$ into $C^0_b(\mathbb R^d)$ to establish a local Lipschitz property on a nonlinear operator, which is essential to apply the Banach Fixed Point theorem. A similar limitation arises in the work of Jendrej \cite[Lemma 4.6, Proposition 4.7]{Jendrej_blowup5d}, where it results in the condition $\nu > 8$ for $d = 5$. Our analysis improves the permissible range to all $\nu > 3$ in the fifth dimension. However, in contrast to the fourth dimension, the restriction $\nu > 3$ is more than technical: while it also occurs in the nonlinear estimates from Proposition \ref{prop:local lipschitz}, the restriction is crucial in ensuring the positivity of the approximation $u_2$ (see (\ref{H(z) definition}) and (\ref{v_2 dominant component})), which is needed to handle the absolute value in the nonlinearity $F(x) = |x|^{p-1}x$ without losing regularity. 
    
To our knowledge, the $d = 4$ case has not been previously addressed. We also remark that in dimension $d = 6$, infinite-time superposition of two bubbles have been constructed (\cite{jendrej2022soliton}). Although the soliton resolution conjecture is now proven, explicit constructions of finite- or infinite-time solutions exhibiting a dynamically scaled bubble profile have only been achieved in low dimensions $3 \leq d \leq 6$. Whether such constructions can be extended to higher dimensions, where the nonlinearity lacks twice differentiability, remains an open question. 

The main result of this paper is the following theorem:

\begin{notation}
    We write $u(x) \in H^{s-}(\mathbb R^d)$ or $|u(x)| \leq |1-x|^{s-}$ if the property holds with exponent $s-\delta$ for all sufficiently small $\delta > 0$ instead of $s$. A similar meaning applies to expressions such as $u(x) \in H^{s+}(\mathbb R^d)$ or $|u(x)| \leq |1-x|^{s+}$. 
    \index{Hs-@$H^{s\pm}(\mathbb R^d)$}
    \index{u(x) < 1-x@$\lvert u(x) \rvert \leq \lvert 1-x \rvert^{s\pm}$}
\end{notation}

\begin{theorem}\label{thm:blow-up, main thm}
    Let $d \in \{4,5\}$, $\nu > \nu_0(d)$ where $\nu_0(4) = 1$ and $\nu_0(5) = 3$, $\delta > 0$, $N_0 \gg 1+\nu$ be fixed. There exists a radial solution $u(x,t)$ of (\ref{wave eq}) on $\mathbb R^d \times [0,t_0] $, $t_0 \ll 1$, which has the form:
\begin{enumerate}
    \item $u(x,t) = \lambda(t)^{\frac{d-2}{2}}W(\lambda(t)x) + \eta(x,t)$ inside the backward light cone $$C = \{(x,t): 0 \leq |x| \leq t, 0 < t \leq t_0\}.$$ 
    Moreover, if $d = 5$, $u(x,t) > 0$ on $C$.
    \item $\lambda(t) = t^{-1-\nu}$ and the solution blows up at $r = t = 0$.
    \item $\eta(x,t)$ can be decomposed as
$\eta(x,t) = u^e(x,t) + \varepsilon(x,t)$, where $u^e \in C^{\frac{1}{2}+\frac{6-d}{2}\nu-}(C)$ and
\begin{align*}
\sup_{0 < t < t_0} t^{-\frac{6-d}{2}\nu-1}||u^e||_{H^{1+\frac{6-d}{2}\nu-}(\mathbb R^d)} +  t^{-\frac{6-d}{2}\nu} ||\partial_t u^e||_{H^{\frac{6-d}{2}\nu-}(\mathbb R^d)} &< +\infty \\
\sup_{0 < t < t_0} t^{-N_0}||\varepsilon||_{H^{1+\frac{6-d}{2}\nu-}(\mathbb R^d)} + t^{-N_0+1}||\partial_t \varepsilon||_{H^{\frac{6-d}{2}\nu-}(\mathbb R^d)} &< +\infty
\end{align*}
    \item The local energy 
$$E_{loc} = \int_{|x| < t} \eta_t^2 + |\nabla \eta|^2  + |\eta|^{p+1} dx$$
of $\eta(x,t)$ in the light cone $|x| \leq t$ vanishes as $t \to 0$
    \item Outside the light cone, the energy of the solution can be controlled
$$\int_{|x| \geq t} u_t^2 + |\nabla u|^2  + |u|^{p+1} dx \leq \delta$$
for all sufficiently small $t > 0$.
\end{enumerate}
\index{N0@$N_0$, the smallness of the approximation error $e_k$ for large $k$, as well as the $\varepsilon$ part from Theorem \ref{thm:blow-up, main thm}}
\end{theorem}

Our proof proceeds in two main steps. In sections \ref{section:renormalization step, basic idea} to \ref{section:renormalization step, next iterates}, starting from the bubble $u_0 = \lambda^{\frac{d-2}{2}}W(\lambda x)$, we linearize and simplify (NLW) to iteratively construct a sequence of approximate solutions $u_k = u_0 + v_1 + ... + v_k$ to (NLW) on a cone $0 < r < t < t_0 \ll 1$. The correction terms $v_k$ are smooth on $0 \leq r < t < t_0$, except for a logarithmic-power singularity of the form $(1-a)^{\frac{1}{2}+\frac{1}{2}\nu}\log(1-a)^j$, $a = r/t$ at the boundary $r = t$. Moreover, the pointwise error $|F(u_k)-\square u_k|$ decreases at each iteration. The term $u^e(x,t)$ from Theorem \ref{thm:blow-up, main thm} is precisely the difference $u_k(x,t)-u_0(x,t)$ for sufficiently large $k$, extended from the cone to all of $\mathbb R^d$ while keeping the same size and regularity.

 In sections \ref{section:spectral theory of the linearized operator} to \ref{section:exact solution fourier method}, we find an exact solution $\lambda^{\frac{d-2}{2}}W + u^e + \varepsilon$ within cone by solving a fixed-point problem in a generalized Fourier space $L^2(\mathbb R, d\rho(\xi))$. This space arises naturally when analyzing the spectral properties of the perturbed Schrödinger operator $$
 \mathcal{L} = -\partial_{RR} - pW(R)^{p-1} + \frac{1}{R^2} \cdot \left( \frac{(d-3)(d-1)}{4} \right), \quad R = \lambda |x|, 
 $$ which emerges from the linearization of (NLW) around the ground state $u_0$. Finally, we show in section \ref{section:end of proof} using finite propagation of speed and well-posedness theory that this solution extends outside the cone.

The principal difference between the cases $d = 4$ and $d = 5$ in this paper lies precisely in this renormalization procedure. For $d = 4$, discussed separately in Appendix \ref{section:appendix, dimension 4}, our method closely follows the approach of Krieger-Schlag-Tataru developed originally the 3-dimensional case. Their scheme has been successfully applied to various other equations, including the 3D-critical NLS, Schrödinger maps, wave maps (see, e.g.,  \cite{ortoleva2012nondispersive}).  In all of these cases, the polynomial structure of the nonlinearity permits the use of simple algebraic manipulations to control error terms. By contrast, the case $d = 5$ requires a modified setup due to the lower regularity of the nonlinearity, which complicates the treatment of nonlinear error terms. We address this difficulty by carefully distinguishing three distinct spatial regions:
$$
R \lesssim (t\lambda)^{\frac{2}{3}}, \quad (t\lambda)^{\frac{2}{3}} \lesssim R \lesssim (t\lambda)^{\frac{2}{3}+\varepsilon}, \quad R \gtrsim (t\lambda)^{\frac{2}{3}+\varepsilon}
$$
using cutoff functions. This allows for the use of convergent multinomial expansions on each region to treat the nonlinear errors. Yet, the renormalization step encounters serious challenges in higher dimensions $d \geq 6$, the main difficulty being that solving equation (\ref{even v}) introduces singularities at the tip of the cone $r = t$ unless $\nu > 0$ is very small (Remark \ref{rmk:loss of regularity high dimension}). In that scenario, the resulting solutions exhibit low regularity, which would necessitate modifying the Banach spaces employed in the fixed-point argument. Additionally, positivity of the approximations $u_k$ must be carefully verified to remove the absolute value in the nonlinearity during the approximation step, which might further restrict the range of admissible values for $\nu$. Indeed, this explicitly occurs in dimension $d = 5$ where one must impose the condition $\nu > 3$ to ensure positivity. Finally, for $d \geq 6$, the generalized eigenfunction $\phi(R,\xi)$, $R,\xi \geq 0$, of the perturbed Schrödinger operator has a singularity at $\xi \sim 0$, $R^2\xi \sim  1$, which gets worse as the dimension increases (Proposition \ref{prop:spectral density properties}, Corollary \ref{phi upper bound R^2z > 1}). Such singularities could yield less favorable estimates in the transference identity (Section \ref{section:transferrance identity}), when passing back and forth from physical to generalized Fourier space.

% In particular, the energy of $u(x,t)$ can be made arbitrarily close to the energy of the stationary solution $W(x)$. This is because we can make it arbitrarily small outside the cone and inside, we observe that since

% $$\partial_t ( \lambda(t)^{(d-2)/2}W(\lambda(t)x) ) \lesssim \frac{(r^2 + t^{2 \nu + 2})t^{(-\frac{1}{2}d(\nu + 1) - \nu - 2)}}{[(r^2 + t^{2\nu + 2})t^{-2\nu - 2}]^{\frac{1}{2}d}},$$

% square-integrating this in the light cone $|x| < t$ and using 
% $$[(r^2 + t^{2\nu + 2})t^{-2\nu - 2}]^{d} \gtrsim r^d t^{-d(\nu + 1)}$$

% gives us that the contribution of this term tends to zero as $t \to 0$, as well as the contribution of $\eta$. It remains only the contribution of 

% $$\int_{\mathbb R^d} \frac{1}{2}\left(|\nabla W_{\lambda(t)}|^2  \right) - \frac{|W_{\lambda(t)}|^{p+1}}{p+1} dx$$

% For this, we use the fact that $W(x)$ is an extremizer of the Sobolev embedding $\dot{H}^1(\mathbb R^d) \hookrightarrow L^{p+1}(\mathbb R^d)$. Hence, this contribution is independent of $\lambda$.

\section{Renormalization Step: Basic idea} \label{section:renormalization step, basic idea}
This section outlines the strategy for constructing an approximate solution to the nonlinear wave equation (\ref{wave eq}) within the cone $0 \leq |x| \leq t$ for small times $0 < t \leq t_0$. The core of our method is an iterative process starting from the scaled ground state $u_0$, and successively adding correction terms $v_k$. At each step, we produce an improved approximation $u_{k+1} = u_k + v_k$ where the corresponding approximation error $e_k = F(u_k) - \square u_k$, $\square = \partial_{tt} - \Delta$, $F(x) = |x|^{p-1}x$,  has decreased. \index{square@$\square = \partial_{tt} - \Delta$, d'Alembert operator} \index{F(x)@$F(x) = \lvert x \rvert^{p-1}x$, the nonlinearity} These corrections are determined by solving a pair of linearized equations: an elliptic-type equation to cancel the dominant error near the origin, and a wave-type equation to cancel the dominant error near the boundary of the backward light cone. 

Let $R = \lambda(t)r$, $u_0(R) = \lambda^{(d-2)/2}(t)W(R)$, $d \in \{4,5\}$.
\index{R@$R = \lambda(t)\lvert x \rvert$, a radial variable}
\index{lambda t@$\lambda(t) = t^{-1-\nu}$, polynomial rate of blow-up}
\index{u0@$u_0(R) = \lambda^{(d-2)/2}(t)W(R)$} From the current approximation $u_k$, we set $u_{k+1} = u_k + v_k$ where $v_k$ is a correction term and at each step, we compute the error $e_k = F(u_k) - \square u_k$. If $u$ were an exact solution to $\square u = F(u)$, then the difference $\varepsilon = u - u_{k-1}$ would satisfy
\begin{align*}
    -\square \varepsilon = - \square u + \square u_{k-1} = -F(u) - e_{k-1} + F(u_{k-1}),
\end{align*}
i.e.,
\begin{align} \label{space-variable equation epsilon}
    -\square \varepsilon = -F(u_{k-1}+\varepsilon) - e_{k-1} + F(u_{k-1}).
\end{align}
Linearizing around $\varepsilon = 0$, we obtain the approximation
\begin{equation*}
    -\square \varepsilon + F'(u_{k-1}) \varepsilon + e_{k-1} \approx 0.
\end{equation*}
Further approximating $u_{k-1}$ by $u_0$, we simplify this to
\begin{equation}
    -\square \varepsilon + pu_0^{p-1} \varepsilon + e_{k-1} \approx 0.
\end{equation}
Thus, the correction terms $v_k$ are constructed, roughly, as follows:
\begin{align}
    \Delta v_{1} + pu_0^{p-1} v_{1} + e_{0} &= 0, \quad e_0 = u_0^p - \square u_0
    \label{first v} \\
    - \square v_{2k} + e^0_{2k-1} &= 0, \quad k \geq 1
    \label{even v}
\end{align}
and
\begin{equation}
    \Delta v_{2k+1} + pu_0^{p-1} v_{2k+1} + e^0_{2k} = 0, \quad k \geq 1
    \label{odd v}
\end{equation}
in radial coordinates with zero Cauchy data at the origin. From this point onwards, we focus exclusively on the case $d = 5$ and readers are referred to Appendix \ref{section:appendix, dimension 4} for $d = 4$. We note that equation (\ref{first v}), which can already be solved explicitly using $e_0$ (see Section \ref{section:renormalization step, first iterate}), is treated separately from equations (\ref{even v}) and (\ref{odd v}), for which a careful analysis of the nonlinearity is necessary to isolate a suitable forcing term $e^0_k$. The overall strategy is the same in both dimensions $d \in \{4,5\}$, but the definition of the forcing terms $e_k^0$ differs slightly since no cutoff is used in dimension $d = 4$.

Equation (\ref{even v}) will be solved in the self-similar variable $a = r/t$, $a \in (0,1)$. This allows us to improve the approximation error near the tip of the cone. The forcing term $e^0_{2k-1}$ for (\ref{even v}) is extracted from $e_{2k-1}$ by keeping only the non-negligible component near the tip of the cone. The remainder 
$$
t^2e^1_{2k-1}:= t^2[e_{2k-1}-e^0_{2k-1}]
$$
is then negligible near the cone tip, and we subsequently improve upon it near the origin in the next iteration. Thus, the updated error is given by
\begin{align*}
    t^2e_{2k} &= t^2[F(u_{2k}) - \square u_{2k}] = t^2[F(v_{2k} + u_{2k-1}) - \square (v_{2k} + u_{2k-1})] \\
    &= t^2[e_{2k-1} - \square v_{2k}] + t^2[F(v_{2k} + u_{2k-1}) - F(u_{2k-1})] \\
    &= t^2[e_{2k-1}-e^0_{2k-1}] + t^2[F(v_{2k} + u_{2k-1}) - F(u_{2k-1})] \\
    &= t^2e^1_{2k-1} + t^2[F(v_{2k} + u_{2k-1}) - F(u_{2k-1})],
\end{align*}
where we used $F(u_{2k-1}) - \square u_{2k-1} = e_{2k-1}$.
\index{a@$a = \lvert x\rvert/t$, a self-similar variable}
The nonlinear part $F(v_{2k} + u_{2k-1}) - F(u_{2k-1})$, which is supported near the tip of the cone, is smaller in magnitude compared to $e_{2k-1}$. It is included within $e^1_{2k}$ and will be further improved upon when constructing the subsequent correction $v_{2k+2}$.

Equation (\ref{odd v}) is solved in the variables $(R,t) = (r\lambda(t),t)$, treating $t$ as a parameter. It allows improving our current error near the origin. The forcing term $e^0_{2k}$ for (\ref{odd v}) is precisely given by the non-negligible component of $e^1_{2k-1}$. The resulting remainder 
$$
t^2e^1_{2k}:= t^2[e_{2k}-e^0_{2k}] = t^2[e^1_{2k-1}-e^0_{2k}] + t^2[F(v_{2k} + u_{2k-1}) - F(u_{2k-1})]
$$
has better smallness properties throughout the entire cone, thus requiring no immediate further improvement. The new error becomes
\begin{align*}
    t^2e_{2k+1} &= t^2[F(u_{2k+1}) - \square u_{2k+1}] \\
    &= t^2[F(v_{2k+1} + u_{2k}) - F(u_{2k}) + F(u_{2k}) - \square (v_{2k+1} + u_{2k})] \\
    &= t^2e^1_{2k} - t^2\partial_t^2 v_{2k+1} + t^2[F(v_{2k+1} + u_{2k}) - F(u_{2k}) - F'(u_0)v_{2k+1}],
\end{align*}
where we have used $F(u_{2k}) - \square u_{2k} = e_{2k}$ and $-\square v_{2k+1} = -\partial_t^2v_{2k+1} - e^0_{2k} - pu_0^{p-1}v_{2k+1}$.At this stage, the error $e_{2k+1}$ has better smallness on the whole cone compared to $e_{2k}$. 

This iterative process is then repeated finitely many times. Specifically, if $1/3 > \varepsilon > 0$, $N_0  \gg \nu > 3$ are fixed, then performing $K_0 = K_0(N_0,\varepsilon) \in \mathbb N$ iterations, where $$2 + \left( \frac{2}{3}-2\varepsilon \right)(K_0-1) \geq N_0$$ leads to an approximate solution of (\ref{wave eq}) with an error of order $\lambda^{\frac{3}{2}}(t\lambda)^{-N_0}$.
\index{N0@$N_0$, the smallness of the approximation error $e_k$ for large $k$, as well as the $\varepsilon$ part from Theorem \ref{thm:blow-up, main thm}}

\section{Renormalization Step: The First Iterate} \label{section:renormalization step, first iterate}

As previously mentioned, (\ref{first v}) can be explicitly solved using a power-series Ansatz (also known as Frobenius method). In this section, we explicitly compute the first correction $v_1$ by solving the elliptic-type equation (\ref{first v}). We then carefully analyze its analytic properties and asymptotic behavior of this first correction, since this first correction forms the foundation for all subsequent steps in the renormalization procedure.

Let us set $u_0(R,t) = \lambda^{\frac{3}{2}}W(R)$, where $W$ is the ground state solution. We define the constants
\begin{equation}
    C_1(\nu) = \frac{105}{128}\pi \nu (1+\nu), \quad C_2(\nu) = \frac{1}{4}(\nu-3)(\nu-5)C_1(\nu),
\index{C1@$C_1(\nu)$, $C_2(\nu)$, some constants coming from $v_1$} \label{eq:constant c_1,c_2}
\end{equation}
which will appear later. First, observe that both $u_0$, $t^2 e_0 \in \lambda^{\frac{3}{2}}C^{\omega}([0,+\infty])$, meaning that they are real-analytic, with an even expansion at $R = 0$ and a regular expansion at $R = +\infty$ with dominant term of order $R^{-3}$. Explicitly,
\begin{align*}
    t^2 e_0(R,t) &= -\lambda^{\frac{3}{2}} \cdot \frac{45 \sqrt{15} (\nu +1) \left(225 (3 \nu +5)+(3 \nu +1) R^4-210 (\nu +1) R^2\right)}{4 \left(R^2+15\right)^{7/2}} \\
    &= \lambda^{\frac{3}{2}} \cdot E_0(R).
\end{align*}
In radial coordinates, (\ref{odd v}) is expressed as
$$t^2 \mathcal{L}_r v_1(r,t) = t^2 e_0(r,t), \quad r \geq 0, \quad
\mathcal{L}_r = -\partial_r^2 - \frac{4}{r} \partial_r - pW(r)^{p-1},$$ where $t$ is treated as a parameter. We seek a solution in the variables $(R,t) = (r\lambda,t)$ variables and rewrite the equation as
\begin{equation}
    (t \lambda)^2 \mathcal{L} v_1(R,t) = t^2 e_0(R,t), \quad R \geq 0, \quad \mathcal{L} = -\partial_R^2 - \frac{4}{R} \partial_R - pW(R)^{p-1}. \label{odd v, change of variables}
\end{equation}
Since $t$ is a parameter and the variables of the forcing term $t^2e_0(R,t)$ are separated, we can expect to find a solution $v_1(R,t)$ in the form
$$v_1(R,t) = \frac{\lambda^{\frac{3}{2}}}{(t\lambda)^2} V_1(R),$$
where $V_1(R)$ is analytic on $[0,+\infty)$ with an even expansion at $R = 0$ starting from order $2$. At $R = +\infty$, an expansion involving powers and logarithms is also expected, consistent with the Frobenius method applied to find solutions to ordinary differential equations with regular singular points.

To facilitate our analysis, instead of restricting ourselves to the positive real line, we shall adopt a complex-analytic framework. Setting $R = z$, we note that the operator $\mathcal{L}$ has a regular singularity at $z = 0$. Consequently, $V_1(z)$ can be found near the origin by looking for a power series solution as in Theorem \ref{thm:inhomogeneous fuchs ode} ($r_1 = 0, r_2 = -3, \beta = 2$). Since 
 $u_0$ and $e_0$ both extend holomorphically to the half-plane
\begin{align*}
    \{z \in \mathbb C: |z| < \sqrt{15} \} \cup \{z \in \mathbb C: \Re(z) > 0\},
\end{align*}
\index{z@$z$, a complex variable which replaces $R$ or $a$ depending on the context}
standard ODE theory ensures that $V_1$ admits a holomorphic extension to the same domain. To study the behaviour of $V_1$ at infinity on the right half-plane, we consider the change of variables $V(z) = V_1(z^{-1})$. Then $V(z)$ solves
\begin{equation}\label{step 1 at infinity}
    V''(z) - \frac{2}{z}V'(z) + z^{-4}pW(z^{-1})^{p-1}V(z) = -z^{-4}E_0(z^{-1}), \quad |z| < \frac{1}{\sqrt{15}}, \ \Re(z) > 0.
\end{equation}
This is again a regular singular equation at $z = 0$, because
$$(z^{-2}+15)^{-1/2} = \frac{z}{(1+15z^2)^{1/2}}$$
on this side of the half-plane. Therefore, $V(z)$ coincides with the particular solution
$$z \sum_{n=0}^{+\infty}v_nz^n + v_{-1}u_1(z)\log(z)$$
given by Theorem \ref{thm:inhomogeneous fuchs ode} ($r_1=3, r_2 = 0, \beta = 1$), modulo some linear combination of the fundamental system $\{u_1(z),u_2(z)\}$ (see (\ref{fuchs fundamental system}) from Appendix \ref{section:appendix, ode}) which introduces a dominant term of order $z^0$. On the real-line, one explicitly has
\begin{equation} \label{v_1 formula}
    V_1 (R) = \frac{\begin{aligned}
                   &\sqrt{15} (\nu +1) \Big[-360 \left(R^2-15\right) R^5-100800 \nu  \left(R^2-15\right) R^3 \text{arccoth} \left(\frac{30}{R^2}+1\right) \\
                   &-75 \nu  \left(13 R^6-1397 R^4+6195 R^2+4725\right) R \\
                   &+7 \sqrt{15} \nu  \left(R^8+300 R^6-20250 R^4+67500 R^2+50625\right) \arctan \left(\frac{R}{\sqrt{15}}\right)\Big]
                \end{aligned}}{64 R^3 \left(R^2+15\right)^{5/2}}.
\end{equation}
It follows that $V_1(R)$ is real-analytic on $[0,+\infty)$ with an even Taylor expansion at zero starting from order $2$ with a positive coefficient. At infinity, it has an asymptotic expansion of the form:
$$\underbrace{\frac{105}{128} \pi \nu (1+\nu)}_{=:C_1(\nu)} R^0 - \frac{45}{8}\sqrt{15}(1+\nu)(1+3\nu)R^{-1} + \frac{55125}{256}\pi \nu(1+\nu)R^{-2} + \mathcal{O}(\log(R)R^{-3}),$$
which is positive as well. However, $V_1(R)$ may take negative values on a fixed compact interval in $(0,+\infty)$; this does not pose a problem because if $t_0$ is small enough, then $u_0(R,t) + v_1(R,t) > 0$ on $[0,+\infty) \times (0,t_0]$.

Functions such as $u_0$ and $v_1$, which have an even power series expansion at $R = 0$ and a series of power and logarithms (with bounded logarithmic exponent) at infinity, are regrouped into the following function space $S^{2n}(R^I,\log(R)^J)$. This function space is used to describe correction and error terms near the origin of the light cone.

\begin{definition}[Space $S^{2n}(R^I,\log(R)^J)$] \label{space s^m(r^k)}
Let $R_0 = 4\sqrt{15}$ be fixed. \index{R0@$R_0 = 4\sqrt{15}$, the threshold distinguishing small and large $R$} \index{S-space@$S^{2n}(R^I,\log(R)^J)$, a vector space of smooth functions with respect to the radial variable $R$} For $I,J,n \in \mathbb N_{\geq 0}$, define $S^{2n}(R^I,\log(R)^J)$ as the vector space of smooth functions $w(R)$ on $(0,+\infty)$ which satisfy:
\begin{enumerate}
    \item $w$ admits a holomorphic extension to an open neighbourhood of the set
$$\{z \in \mathbb C: |z| \leq \sqrt{15}/2 \} \cup \{z \in \mathbb C: \Re(z) > 0, |z| \geq R_0\}.$$
    \item $w$ has a zero of order $2n$ at $z = 0$ with an even Taylor expansion.
    \item For $\Re(z) > 0$ and $|z| \geq R_0$, $w(z)$ has the following expansion:
    $$w(z) =  z^I \sum_{j=0}^{J}w_{j}(z^{-1})
 \log(z)^j,$$
    where the functions $w_{j}(y)$ are holomorphic on a neighbourhood of $\{y \in \mathbb C: |y| \leq R_0^{-1}\}$.
\end{enumerate}
For $I < 0$, the space $S^{2n}(R^I,\log(R)^J)$ is defined as the subspace of $S^{2n}(R^0,\log(R)^J)$ for which each function $w_{j}$ in the expansion at infinity has a zero of order at least $|I|$. 
\end{definition}

\begin{remark} \label{rmk: s^2n properties}
 We start with a few remarks regarding the spaces $S^{2n}(R^I, \log(R)^J)$.
\begin{enumerate} 
\item For any $l \in \mathbb N$, $(R\partial_R)^l S^{2n}(R^I,\log(R)^J) \subset S^{2n}(R^I,\log(R)^J)$. 
    \item An element $w \in S^{2n}(R^m,\log(R)^J)$ is not necessarily holomorphic on the whole half-plane $\Re(z) > 0$, allowing room for localized cutoff.
    \item The subspaces $S^{2n}(R^I,\log(R)^J)$ with $I < 0$ are only relevant for the functions $u_0$, $t^2e_0$ and $u_0^p$. In these cases, one has $u_0, t^2e_0 \in \lambda^{\frac{3}{2}} S^0(R^{-3}, \log(R)^0)$, $u_0^p \in \lambda^{\frac{3p}{2}} S^0(R^{-7},\log(R)^0)$ and $v_1 \in \lambda^{\frac{3}{2}} (t\lambda)^{-2} S^2(R^{0},\log(R))$, where no logarithm occurs in the dominant term at infinity. More precisely, $w_0(y) = \mathcal{O}(1)$ is the dominant component, while $w_{j}(y) = \mathcal{O}(y^3)$, $j \neq 0$.
    \item One has
    \begin{align*}
        t^2 u_0^p \left( \frac{v_1}{u_0} \right)^{n} &\in \dfrac{\lambda^{\frac{3}{2}}}{(t\lambda)^{2(n-1)}} S^{2n}(R^{-7+3n},\log(R)^{n} ).
    \end{align*}
\end{enumerate} 
\end{remark}

\section{Renormalization Step: The Hypergeometric Function and the second iterate}

Having established the appropriate function spaces required for solving (\ref{odd v}), it is instructive, before proceeding with the full iterative scheme, to motivate and introduce the function spaces $\mathcal{Q}$ that will occur when solving equation (\ref{even v}) near the tip of the cone. The simplest case for equation (\ref{even v}), when rewritten in radial coordinates, corresponds to a forcing term that is a pure power of $t$:
$$
t^2\left(-\partial_t^2 + \partial_r^2 + \frac{d-1}{r}\partial_r\right) v =t^s.
$$
This scenario will appear as the dominant part, near the tip of the cone, of the error $e_1$ that remains after the first iteration. Equivalently, if we look for a solution $v(r,t) = t^s w(r,t)$, then: 
$$
t^2 \left( - \left( \partial_t + \frac{s}{t} \right)^2 + \partial_r^2 + \frac{d-1}{r}\partial_r \right)w  = 1.
$$
Introducing the self-similar variable $a = r/t$, this equation becomes $L_{s} w(a) = 1$, $0 < a < 1$, where
$$L_{s} = (1-a^2)\partial_{aa} + ((d-1)a^{-1} + 2as - 2a)\partial_a + (s - s^2).$$
% $$L_{x} = (1-a^2)\partial_{aa} + ((d-1)a^{-1} + 2a\beta - 2a)\partial_a + (\beta - \beta^2)$$
Seeking a solution of the form $w(a) = W(a^2)$, we naturally reduce to an hypergeometric equation for $W(z)$:
$$z(1-z)W''(z) + \left( \frac{d}{2} + z \left(s - \frac{3}{2} \right) \right) W'(z) + \frac{s - s^2}{4} = 1, \quad 0 < z < 1$$
whose parameters are
$$
\tilde{\alpha} = -\frac{s}{2}, \tilde{\beta} = -\frac{s}{2} + \frac{1}{2}, \tilde{\gamma} = \frac{d}{2},
$$
and for which a particular solution is explicitly given by
$$
\frac{1}{\alpha \beta}[F(\tilde{\alpha},\tilde{\beta};\tilde{\gamma},z) - 1],
$$
where $F$ is the Gauss hypergeometric function defined as follows.
\begin{definition}\label{gauss hypergeometric definition}
Let $\alpha, \beta, \gamma \in \mathbb R$, $\gamma \notin \mathbb Z_{\leq 0}$. The Gauss hypergeometric function $F$ is defined by the series
\begin{equation*} 
F(\alpha,\beta;\gamma;z) = \sum_{n=0}^{+\infty} \frac{(\alpha)_n(\beta)_n}{(\gamma)_nn!}z^n, \quad |z| < 1, \quad (x)_n = x(x+1)...(x+n-1), \quad (x)_0 = 1
\end{equation*}
\index{F(alpha)@$F(\alpha,\beta;\gamma;z)$, the Gauss Hypergeometric function}
which converges absolutely for $|z| < 1$, as well as for $|z| \leq 1$ if $\gamma-\alpha-\beta > 0$.
\end{definition}

The hypergeometric function satisfies the hypergeometric equation
\begin{equation}\label{hypergeometric ode}
    z(1-z)w''(z) + (\gamma - (\alpha + \beta + 1)z)w'(z) - \alpha \beta w(z) = 0, \quad 0 < z <1.
\end{equation}
This equation has regular singular points at $z = 0$ and $z = 1$. The roots of the indicial equation at these points are respectively $\{0,1-\gamma\}$ and $\{0,\gamma-\alpha-\beta\}$. Fundamental systems can be found as described in Appendix \ref{section:appendix, ode} (see (\ref{fuchs fundamental system})).

In our application, we focus on the case where $\gamma > 0$, $\alpha \beta > 0$, which ensures monotonicity and positivity of $F(\alpha,\beta;\gamma,z) - 1$ on $[0,1)$, as shown in Corollary \ref{monotonicity hypergeom function}. This property is essential to ensure positivity of the second correction term $v_2$. Additionally, we will also impose $\gamma - \alpha - \beta > 0$, which allows us to determine which indicial root is the largest at $z = 1$ and guarantees continuity of $F(\alpha,\beta;\gamma,z)$ up to the boundary $|z| = 1$.

\begin{lemma}[A monotonicity and positivity result]\label{monotonicity hypergeom eq}
Let $\alpha \beta > 0$ and $g(z) \in C^0([0,1))$ be non-negative. Suppose $w(z)$ is a $C^1([0,1)) \cap C^2((0,1))$-solution to the inhomogeneous hypergeometric equation
\begin{equation*}
    z(1-z)w''(z) + (\gamma - (\alpha + \beta + 1)z)w'(z) - \alpha \beta w(z) = g(z), \quad 0 \leq z < 1.
\end{equation*}
with initial conditions $w(0) \geq  0, w'(0) > 0$. Then $w(z)$ is strictly increasing on $[0,1)$.

%The same result holds if $g(z)$ is non-negative, $g(0) = 0$ and $g(z)$ is nowhere locally constant.
\end{lemma}

\begin{proof}
At $z = 0$, we have that $w'(0) > 0$, which implies that $w(z)$ is strictly increasing near $z = 0$. Let $I = [0,z_0)$ be the maximal interval on which $w'(z) > 0$. Suppose for a contradiction that $z_0 < 1$. Then $w'(z_0) = 0$ and $w(z_0) > w(0) \geq 0$ by continuity and strict monotonicity. Evaluating the differential equation at $z_0$, we obtain:
$$w''(z_0) = \frac{\alpha \beta}{z_0(1-z_0)}w(z_0) + \frac{g(z_0)}{z_0(1-z_0)} > 0,$$
which implies that $z_0$ is a local strict minimum of $w(z)$ and $w(z)$ is strictly decreasing on the left of $z_0$: a contradiction.
% Finally, if $g(0) = 0$ but $g(z) > 0$ near $z > 0$, then $w(z)$ is the pointwise limit on $(0,1)$ of some sequence of increasing functions $(w_n)$ solving the same equation but with forcing term $g_n(z) = g(z) + \frac{1}{n}$, which proves that $w(z)$ is increasing on $(0,1)$. If the monotonicity is not strict, we can find some open interval $I \subset (0,1)$ on which $w(z)$ is constant. Hence $w'(z) = w''(z) = 0$ on $I$, from which the ode implies that $g(z)$ is constant on $I$ as well. 
\end{proof}

\begin{corollary}\label{monotonicity hypergeom function}
If $\gamma > 0$ and $\alpha \beta > 0$, then $z \mapsto F(\alpha,\beta;\gamma,z) - 1$ is strictly increasing on $[0,1)$.
\end{corollary}

\begin{proof}
    The proof follows from Lemma \ref{monotonicity hypergeom eq} since $w(z) = F(\alpha,\beta;\gamma,z) - 1$ solves
    $$z(1-z)w''(z) + (\gamma - (\alpha + \beta + 1)z)w'(z) - \alpha \beta w(z) = \alpha \beta, \quad 0 \leq z < 1$$
    with $w(0) = 0$ and $w'(0) = \alpha \beta \gamma^{-1} > 0$.
\end{proof}
% \begin{corollary}[A comparison result]
% Let $\gamma > 0$, $\alpha \beta > 0$, $\gamma$, $\gamma - \alpha - \beta \in \mathbb R \setminus \mathbb Z$. Let $g_i(z) \in C^0([0,1))$, $i = 1,2$, be two forcing terms for which $g_2(z) \geq g_1(z)$ and either $g_2(0) > g_1(0)$ or $g_2(z) - g_1(z)$ is nowhere locally constant. Denote $w_i(z) \in C^1([0,1)) \cap C^2((0,1))$, $w_i(0) = 0$, the associated solutions to the hypergeometric equation. 

% Then $w_2(z) - w_1(z)$ is strictly increasing. In particular, $w_2(z) > w_1(z)$ on $(0,1)$.
% \end{corollary}

We now define a specific hypergeometric-type function
    \begin{align} \label{H(z) definition}
    H(z) &= 4 C_1(\nu) (F(\tilde{\alpha},\tilde{\beta};\tilde{\gamma},z)-1 ), \quad 0 \leq z < 1 \\
    s &=  \frac{3}{2}(-1-\nu) + 2\nu, \tilde{\alpha} = -\frac{s}{2}, \tilde{\beta} = -\frac{s}{2} + \frac{1}{2}, \tilde{\gamma} = \frac{5}{2} \notag
    \end{align} \index{H(z)@$H(z)$, an hypergeometric-like function}
     with $C_1$ defined as in (\ref{eq:constant c_1,c_2}). It turns out that the term
     $$
      \frac{\lambda^{\frac{2}{3}}}{(t\lambda)^2}H(a^2) 
     $$
    will be the dominant component of the second correction term $v_2$. The parameters satisfy $\tilde{\alpha} \tilde{\beta} > 0$ if $\nu > 3$, $\tilde{\gamma} > 0$ and $\tilde{\gamma}-\tilde{\alpha}-\tilde{\beta} = \frac{1}{2} + \frac{1}{2}\nu > 0$. Therefore, $H(z)$ is strictly increasing on $[0,1)$ (Corollary \ref{monotonicity hypergeom function}), holomorphic on $|z| < 1$ and continuous on $|z| \leq 1$. Using a fundamental system for the hypergeometric equation (see (\ref{fuchs fundamental system})), we obtain the singular expansion at $z = 1$:
    \begin{align} \label{expansion of H(z) near z = 1}
    F(\tilde{\alpha},\tilde{\beta};\tilde{\gamma},z)-1 = q_0(1-z)+q_1(1-z)(1-z)^{\frac{1}{2}+\frac{1}{2}\nu}+q_2(1-z)(1-z)^{\frac{1}{2}+\frac{1}{2}\nu}\log(1-z),
    \end{align}
    where $q_i(y)$ is holomorphic on $|y| < 1$. Moreover, since $H(1) > 0$, one has $q_0(0) > 0$ as well. 
    
    As with the space $S^{2n}(R^I, \log(R)^J)$, we now define a family of functions $\mathcal{Q}$ in the self-similar variable $a = r/t$ having a series of power and logarithms at $a = 0$ and $a = 1$. These functions will be used to describe correction and error terms near the tip of the cone. Their structure and regularity are motivated by the Frobenius method applied to the hypergeometric equation. However, extra care must be taken with regard to the domain on which our functions are holomorphic. Although $H(z)$ is holomorphic on $|z| < 1$, the nonlinear term
    $$
    \left( u_0 + v_1 + \frac{\lambda^{\frac{2}{3}}}{(t\lambda)^2}H(a^2) \right)^p
    $$
    has, near $r \sim t$, $a = 1$, a dominant component of the form
    $$
    \frac{\lambda^{\frac{2}{3}}}{(t\lambda)^2}\left( H(a^2) + C_1(\nu) \right)^p
    $$
    with the constant term coming from $v_1$. This expression is holomorphic only on the subset of the unit disk where $H(z) + C_1(\nu) \neq 0$, which may be smaller than $|z| < 1$. Thus, to include this function in the family $\mathcal{Q}$, we cannot demand holomorphy on the entire unit disk. Fix $U$ to be an open, simply-connected neighbourhood $U \subset B(0,1) \subset \mathbb C$ such that
    \begin{enumerate}
        \item $(0,1) \subset U$,
        \item $U$ contains $\{a \in B(0,1): |a-0| \leq a_0 \text{ or } |a-1| \leq a_0\}$ for some $0 < a_0 \ll 1$,
        \item $H(a)$ and $H(a) + C_1(\nu)$ (defined in (\ref{H(z) definition})) are non-zero everywhere on $U \setminus \{0\}$,
        \item If $q_0(y)$ is the analytic function appearing in the expansion of $H(z)$ near $z = 1$ (see (\ref{expansion of H(z) near z = 1})), then $q_0(y)$ and $q_0(y) + C_1(\nu)$ are non-zero everywhere on $|y| \leq a_0$.
    \end{enumerate}
    
\begin{definition}[Space $\tilde{\mathcal{Q}}$ and $\mathcal{Q}$]\label{definition of tilde Q}
\index{Q@$\tilde{\mathcal{Q}}$, $\mathcal{Q}$, families of holomorphic functions with respect to the self-similar variable $a$}
Let $U$ be as described above. Let $\tilde{\mathcal{Q}}_{\beta}$, $\beta \in \mathbb R$, be the vector space of real-analytic functions $q(a): (0,1) \rightarrow \mathbb R$ satisfying:
\begin{enumerate}
\item $q(a)$ extends holomorphically to $U$.
    \item Near $a = 0$, one has the finite expansion:
    $$q(a) = q_0(a) + \sum_{j=1}^{L}q_{j}(a) \log(a)^j,$$
    where $0 \leq L < +\infty$ and each $q_{j}$ is holomorphic on a neighbourhood of $|a-0| \leq a_0$.
    \item Near $a = 1$, one has the expansion:
    $$q(a) = q_{0,0}(1-a) + \sum_{i=1}^{+\infty}(1-a)^{\beta(i)} \sum_{j=0}^{L_i}q_{i,j}(1-a)\log(1-a)^j,$$
    where $L_i \geq 0$, $\sup_{i\geq 1}L_i/i < +\infty$ , the functions $q_{i,j}(y)$ are holomorphic on a neighbourhood of $|y| \leq a_0$, $\beta(i) \geq \beta$ and either the expansion is finite, i.e.,
    $$|\{(i,j) \in \mathbb N^2: q_{i,j}(y) \not \equiv  0\}| < +\infty, $$
    or the growth of $i \mapsto \beta(i)$ is at least linear, i.e., 
    $$\beta(i+1) > \beta(i), \quad \lim \limits_{i \to +\infty}\beta(i) = +\infty, \quad \inf_{i \geq 2}\left(\frac{\beta(i)-\beta}{i}\right) > 0.$$
    Moreover, there should exist $\varepsilon > 0$ and $C > 0$ for which
    $$|\beta(i)| + ||q_{i,j}||_{A(|y| \leq a_0+\varepsilon)} \leq C^i \quad \forall 0 \leq j \leq L_i, \forall i \in \mathbb N_{\geq 1}.$$
    where $A\left(|y| \leq R \right)$ is the Wiener algebra (see Definition \ref{wiener space}). In other words, the growth of the sequences $\beta(i)$ and $q_{i,j}$ must be at most exponential exponential.
%    $$\forall i,j \left( q_{i,j}(a) = \sum_{n=0}^{\infty}q_n(1-a)^n, \ a_0 < a < 1 \Longrightarrow |q_n| \leq C^{i}a_0^{-n} \right)$$
\end{enumerate}
Finally, we define $\mathcal{Q}_{\beta}$ as the vector space of real-analytic functions on $(0,1)$ of the form $a \mapsto \tilde{q}(a^2)$ for some $\tilde{q} \in \tilde{\mathcal{Q}}_{\beta}$.
\end{definition}

\begin{remark}
We start with some remarks concerning the spaces $\mathcal{Q}$ and $\tilde{\mathcal{Q}}$.
\begin{enumerate}
    \item We will primarily work with the spaces $\tilde{\mathcal{Q}}_{\frac{1}{2} + \frac{1}{2}\nu}$ and $\mathcal{Q}_{\frac{1}{2} + \frac{1}{2}\nu}$, but it is convenient, for theoretical results, to allow any $\beta \in \mathbb R$.
    \item As established in (\ref{H(z) definition}) and (\ref{expansion of H(z) near z = 1}), $H(a) \in \tilde{\mathcal{Q}}_{\frac{1}{2} + \frac{1}{2}\nu}$ with a finite expansion at $a = 1$. Moreover, for any exponent $e \in \mathbb R$, $(H(a) + C_1(\nu))^e \in \tilde{\mathcal{Q}}_{\frac{1}{2} + \frac{1}{2}\nu}$ as well. The expansion near $a = 1$ is obtained via a binomial expansion:
    \begin{align*}
        (H(z) + C_1(\nu))^e = \sum_{n=0}^{\infty} \binom{e}{n} &\left( q_0(1-z) + C_1(\nu) \right)^{e-n} \cdot \\
        &(1-z)^{\left( \frac{1}{2} + \frac{1}{2}\nu\right) n} \left( q_1(1-z) + q_2(1-z)\log(1-z) \right)^n.
    \end{align*}
    \item We do not require that the expansion of a $\tilde{Q}$-element near $a = 1$ converges everywhere on $|a-1| < a_0$. However, the coefficient functions $q_{i,j}$ must all be defined and holomorphic around the disk $|a-1| \leq a_0$. The growth condition on $q_{i,j}$ ensures uniform and absolute convergence of the expansion on a smaller ball (depending only on $a_0$) around $a = 1$.
    \item The linear growth estimate on $\beta(i)$ is essential. If we solve the hypergeometric equation near $z = 1$ with a forcing term of the form $(1-z)^{\beta(i)}q_{i,j}(1-z)\log(1-z)^j$, then the estimate (\ref{exponential upper bound on w_k}) on the solution from Theorem \ref{thm:inhomogeneous fuchs ode} takes the form:
    \begin{equation*}
          ||w(1-z)||_{A(|1-z| < a_0 + \varepsilon)}  \leq C^i \cdot ||q_{i,j}||_{L^{\infty}(|1-z| < a_0 + \varepsilon}.
   \end{equation*}
   Thus, if the growth of $q_{i,j}$ is at most exponential in $i$, then so is the growth of the solution. Accordingly, if we solve (\ref{even v}) with a forcing term from a family $\mathcal{Q}_{\beta}$, we expect the solution to remain within $\mathcal{Q}_{\tilde{\beta}}$.
\end{enumerate}
\end{remark}

\begin{proposition}[Product rules] \label{q algebra calculation rules}
    The following algebraic rules hold: 
    \begin{enumerate}
        \item[A.] Differentiation: $\partial_a \left( a^{\delta} \mathcal{Q}_{\beta} \right)  \subset \delta a^{\delta-1} \mathcal{Q}_{\beta} + a^{\delta-1} \mathcal{Q}_{\beta-1}$ for any $\beta, \delta \in \mathbb R$. 
        \item[B.] Summation: $\mathcal{Q}_{\beta_1} + \mathcal{Q}_{\beta_2} \subset \mathcal{Q}_{\min\{\beta_1,\beta_2\}}$ for any $\beta_1,\beta_2 \in \mathbb R$. 
        \item[C.] Product: $\mathcal{Q}_{\beta_1} \cdot \mathcal{Q}_{\beta_2} \subset  \mathcal{Q}_{\min\{\beta_1,\beta_2,\beta_1+\beta_2\}}$ for any $\beta_1,\beta_2 \in \mathbb R$. In particular, if $\beta \geq 0$, then $\mathcal{Q}_{\beta}$ is an algebra.
    \end{enumerate}
\end{proposition}

\begin{proof}
    We prove only the differentiation rule. Suppose $q(a) = a^{\delta}Q(a^2)$, where $Q \in \tilde{\mathcal{Q}}_{\beta}$. Then, it holds that
    $$
    \partial_a q(a) = \delta a^{\delta-1}Q(a^2) + 2a^{\delta+1}Q'(a^2) = a^{\delta-1} \left( \delta Q(a^2) + 2a^2Q'(a^2) \right).
    $$
    It suffices to show that $a\partial_a Q(a) \in \tilde{\mathcal{Q}}_{\beta-1}$. Since $Q(a)$ is holomorphic on $U$, so is $a\partial_aQ$. Near $a = 0$, one has
    \begin{align*}
        Q(a) &= q_0(a) + \sum_{j=1}^{L}q_{j}(a) \log(a)^j, \\
        aQ'(a) &=aq_0'(a) + \sum_{j=1}^{L}aq_{j}'(a) \log(a)^j + \sum_{j=1}^{L} j q_{j}(a) \log(a)^{j-1},
    \end{align*}
    where $0 \leq L < +\infty$ and each $q_{j}, a\partial_a q_j$ is holomorphic on a neighbourhood of $|a-0| \leq a_0$. Near $a = 1$,  
   \begin{align*}
       Q(a) &= q_{0,0}(1-a) + \sum_{i=1}^{+\infty}(1-a)^{\beta(i)} \sum_{j=0}^{L_i}q_{i,j}(1-a)\log(1-a)^j,
   \end{align*}
    where $0 \leq L_i \leq L(i+1)$ for some $L > 0$, each $q_{i,j}(y)$ is holomorphic on $|y| < a_0 + \varepsilon$, $\beta(i) \geq \beta$ and we assume without loss of generality that we have an infinite series with 
    $$\beta(i+1) > \beta(i), \quad \lim \limits_{i \to +\infty}\beta(i) = +\infty, \quad \inf_{i \geq 2}\left(\frac{\beta(i)-\beta}{i}\right) > 0,$$
    as well as
    $$|\beta(i)| + ||q_{i,j}||_{A(|y| \leq a_0+\varepsilon)} \leq C^i \quad \forall 0 \leq j \leq L_i, \forall i \in \mathbb N_{\geq 1},$$
    for some constants $C, \varepsilon > 0$. Formally consider the sum of derivatives:
    \begin{align*}
           S(a) = q_{0,0}'(1-a) &- \sum_{i=1}^{+\infty}(1-a)^{\beta(i)-1} \sum_{j=0}^{L_i}\beta(i)q_{i,j}(1-a)\log(1-a)^j \\
           &-\sum_{i=1}^{+\infty}(1-a)^{\beta(i)} \sum_{j=0}^{L_i}q_{i,j}'(1-a)\log(1-a)^j \\
           &-\sum_{i=1}^{+\infty}(1-a)^{\beta(i)-1} \sum_{j=0}^{L_i}jq_{i,j}(1-a)\log(1-a)^{j-1}.
    \end{align*}
    For all $0 \leq j \leq L_i, \forall i \in \mathbb N_{\geq 1}$, one has
    $$
    ||\beta(i)q_{i,j}||_{A(|y| \leq a_0+\varepsilon)} +  ||j q_{i,j}||_{A(|y| \leq a_0+\varepsilon)}  \leq C^iC^i + L(i+1)C^i \leq \tilde{C}^i,
    $$
    as well as
    $$
||q_{i,j}'||_{A(|y| \leq a_0+\varepsilon/2)} \lesssim_{a_0,\varepsilon} ||q_{i,j}||_{A(|y| \leq a_0+\varepsilon)} \leq \tilde{C}^i.
    $$
    Hence, the sum of derivatives converges normally on some ball $|a-1| \lesssim \min\{\tilde{C}^{-1},a_0\}$. Integrating, we recover $Q(a)$ modulo some additive constant. Thus, $S(a) = Q'(a)$ by the Identity Theorem, and $Q'(a)$ has the desired expansion near $a = 1$.
\end{proof}

\section{Renormalization Step: Preliminaries for the next iterates} \label{section:renormalization step, preliminaries next iterates}

This section establishes the technical framework required for all subsequent renormalization steps. We first partition the light cone into three distinct regions to analyze nonlinear error terms such as:
$$
t^2e_1(R,t) = t^2[F(u_0 + v_1) - F(u_0) - F'(u_0)v_1] - t^2\partial_{tt} (v_1(r\lambda,t))
$$
through a multinomial expansion within each region. The idea is that it will be easier to solve (\ref{even v}) and (\ref{odd v}) for each term in the expansion and sum everything back. The multinomial expansion depends on whether $u_0 \geq v_1$, $u_0 \sim v_1$ or $u_0 \leq v_1$, which corresponds to the three regions $R \lesssim (t\lambda)^{\frac{2}{3}}, R \sim (t\lambda)^{\frac{2}{3}}, R \gtrsim (t\lambda)^{\frac{2}{3}}$ of the light cone. 

The first step is to fix an appropriate constant $m \ll 1$ and define the region $R \leq m(t\lambda)^{\frac{2}{3}}$. This constant is chosen so that $(u_0 + v_1)^p$ can be reliably expanded around $u_0$ in that region $R \leq m(t\lambda)^{\frac{2}{3}}$ using a binomial expansion. Moreover, $m$ should be sufficiently small so that one can keep only the terms from the expansion with a bounded logarithmic exponent. This restriction is important, as solving equation (\ref{odd v, change of variables}) at $R = +\infty$ with a logarithmic forcing term $R^{I}\log(R)^J$, $I \leq 0$, introduces factors of order $J!$ in the solution. Even if $|I| \gg J$, the bound (\ref{exponential upper bound on w_k}) from Theorem \ref{thm:inhomogeneous fuchs ode} is not available. This makes controlling convergence of solutions more difficult if one allows for a sequence of logarithmic terms $\log(R)^J$, $J \to +\infty$, in the binomial expansion.

Once the constant $m$ is defined, we will divide the cone into three main regions, define the correction and error function spaces and prove multiple computation rules that will facilitate the iteration scheme.

\begin{definition}[Constant $m$]
    Let $t_0 \ll 1$ be sufficiently small so that
    \begin{enumerate}
        \item $|u_0(z,t)| \geq 2|v_1(z,t)|$ on $$\left( \{z \in \mathbb C: |z| \leq \sqrt{15}/2\} \cup \{z \in \mathbb R: z \in [0,2\sqrt{15}]\} \right) \times (0,t_0].$$
        \item $u_0(R,t) + v_1(R,t) > 0$ on $[0,+\infty) \times (0,t_0]$.
    \end{enumerate}
    For $\Re(z) > 0, |z| > \sqrt{15}$, consider the expansion:
     $$\frac{v_1(z,t)}{u_0(z,t)} = \frac{z^{3}}{(t\lambda)^{2}} \left(w_0(z^{-1}) + z^{-3}w_1(z^{-1})\log(z) \right),$$
 where $w_j(y)$ is holomorphic on $|y| < (\sqrt{15})^{-1}$, $w_j(0) \neq 0$. Fix $m \ll 1$ any constant for which
    \begin{equation}
        (2m)^3 \max_{j \in \{0,1\}}\left[1+||w_j||_{A\left(|y| \leq 2/(3\sqrt{15}) \right)}\right] \leq \frac{1}{4}. \label{m definition}
    \end{equation}
    \index{m@$m$, the constant in $R \leq m(t\lambda)^{\frac{2}{3}}$} \sloppy In particular, this choice ensures:
    $|u_0(z,t)| \geq 2|v_1(z,t)|$
 on $$\{(z,t) \in \mathbb C \times (0,1): \Re(z) > 0, 2\sqrt{15} \leq |z| \leq m(t\lambda)^{\frac{2}{3}}, 0 < t \leq t_0\},$$
    which is a non-empty set for $t_0$ small enough. 
\end{definition}

Given this constant $m$, we can decompose the cone into four regions.

\begin{definition}[$k$-admissible pairs]
Define the overlapping regions:
\begin{align*}
    C_{\mathrm{ori}} &= \{(R,t): 0 < t \leq t_0, 0 \leq R \leq m (t\lambda)^{\frac{2}{3}}\}, \\
    C_{\mathrm{mid}} &= \{(R,t): 0 < t \leq t_0, \frac{m}{2}(t\lambda)^{\frac{2}{3}} \leq R \leq 2(t\lambda)^{\frac{2}{3}+\varepsilon}\}, \\
    C_{\frac{2}{3}+\varepsilon} &= \{(R,t): 0 < t \leq t_0,(t\lambda)^{\frac{2}{3}+\varepsilon} \leq R \leq 2(t\lambda)^{\frac{2}{3}+\varepsilon}\}, \\
    C_{\mathrm{tip}} &= \{(R,t): 0 < t \leq t_0, (t\lambda)^{\frac{2}{3}+\varepsilon} \leq R \leq (t\lambda)\}. 
\end{align*}
\index{C1@$C_{\mathrm{ori}}$, the subset $0 \leq R \leq m (t\lambda)^{\frac{2}{3}}$ of the cone}
\index{C2@$C_{\mathrm{mid}}$, the subset $\frac{m}{2}(t\lambda)^{\frac{2}{3}} \leq R \leq 2(t\lambda)^{\frac{2}{3}+\varepsilon}$ of the cone}
\index{C3@$C_{\frac{2}{3}+\varepsilon}$, the subset $(t\lambda)^{\frac{2}{3}+\varepsilon} \leq R \leq 2(t\lambda)^{\frac{2}{3}+\varepsilon}$ of the cone}
\index{C4@$C_{\mathrm{tip}}$, the subset $(t\lambda)^{\frac{2}{3}+\varepsilon} \leq R \leq (t\lambda)$ of the cone}
A pair of indices $(\alpha,i) \in \mathbb R \times \mathbb Z$ will be called $k$-admissible for $k > 1$ on a region $C_x \neq C_{\mathrm{ori}}$ if 
$$\forall (R,t) \in C_{x}: \frac{|R|^i}{(t\lambda)^{\alpha}} \lesssim \frac{1}{(t\lambda)^{2+\left( \frac{2}{3}-2\varepsilon \right) \cdot (k-1)}}.$$
\index{k-admi@$k$-admissible pairs}
% \begin{alignat*}{3}
%     i &\geq 0, \quad \alpha - i &&\geq 2+\left( \frac{2}{3}-2\varepsilon \right) \cdot (k-1) \quad \text{ or } \\
%     i &\leq 0, \quad \alpha - \frac{2}{3}i &&\geq 2+\left( \frac{2}{3}-2\varepsilon \right) \cdot (k-1) 
% \end{alignat*}
When $k = 1$, the pair $(\alpha,i)$ is called 1-admissible if $(\alpha,i) = (2,0)$ or if it is 2-admissible.

On $C_{\mathrm{ori}}$, the pair will be called $1$-admissible if $(\alpha,i) = (2,0)$ or 
$$\forall (R,t) \in C_{\mathrm{ori}}: \frac{|R|^i}{(t\lambda)^{\alpha}} \lesssim \frac{1}{(t\lambda)^{2+\frac{2}{3}}}$$
and $k$-admissible for $k > 1$ if $(\alpha-2,i-2)$ is $(k-1)$-admissible. In particular, this implies
$$\forall (R,t) \in C_{\mathrm{ori}}: \frac{|R|^i}{(t\lambda)^{\alpha}} \lesssim \frac{1}{(t\lambda)^{2+ \frac{2}{3} \cdot (k-1)}}.$$
We will usually omit the region $C_x$ as it will be clear from the context.
\end{definition}

The notion of admissible pairs allows to easily quantify the smallness of our correction and error terms. We also observe that if $(\alpha,i)$ is $k$-admissible on $C_{\mathrm{ori}}$, then $i \geq 0$ and $(\alpha,i)$ is also $k$-admissible on $C_{\mathrm{mid}}$. Similarly, if $(\alpha,i)$ is $k$-admissible on $C_{\mathrm{mid}}$ or $C_{\mathrm{tip}}$, then $(\alpha,i)$ is also $k$-admissible on $C_{\frac{2}{3}+\varepsilon}$.

Next, we endow the space $w \in S^{2n}(R^I, \log(R)^J)$ with a norm giving control on $w$ on $R \leq m(t\lambda)^{\frac{2}{3}}$. This norm is useful to consider series of $S^{2n}(R^I, \log(R)^J)$ elements and ensure their convergence.

\begin{definition}[Norm on the space $S^{2n}(R^I,\log(R)^J)$] \label{norm of space s^m(r^k)}
 For $I,J,k,n \in \mathbb N_{\geq 0}$, consider $w \in S^{2n}(R^I,\log(R)^J)$ and its expansion:
 $$w(z) =  z^I \sum_{j=0}^{J}w_{j}(z^{-1})\log(z)^j$$
 on $R \geq R_0$. Define the following semi-norms:
\begin{align*}
||w||_{S, \mathrm{ori}} &= ||w(z)||_{A(|z| \leq \sqrt{15}/2)} + ||w(z)||_{L^{\infty}(z \in [\sqrt{15}/2, R_0] \subset \mathbb R)}, \\
||w||_{S,I,J,\infty} &=  m^I \max_{0 \leq j \leq J} ||w_{j}(y)||_{A(|y| \leq R_0^{-1})},
\end{align*}
\index{S-space norms@$\lvert \lvert w \rvert \rvert_{S, \mathrm{ori}}, \lvert \lvert w \rvert \rvert_{S,I,J,\infty}$, semi-norms on $S^{2n}(R^I,\log(R)^J)$}
the sum of which creates a norm.
\end{definition}

\begin{remark}
 We start with a few observations regarding the semi-norms on $S^{2n}(R^I, \log(R)^J)$.
\begin{enumerate} 
    \item One has a product structure: if $v \in S^{2n_1}(R^{I_1},\log(R)^{J_1})$ and $w \in S^{2n_2}(R^{I_2},\log(R)^{J_2})$ for $n_1,n_2,I_1,I_2,J_1,J_2 \geq 0$, then $v \cdot w \in S^{2(n_1+n_2)}(R^{I_1+I_2},\log(R)^{J_1+J_2})$ and
    \begin{align*}
||v \cdot w||_{S, \mathrm{ori}} &\leq ||v||_{S, \mathrm{ori}} \cdot ||w||_{S, \mathrm{ori}}, \\
||v \cdot w||_{S, I_1 + I_2, J_1 + J_2, \infty} &\leq  ||v||_{S,I_1,J_1, \infty} \cdot ||w||_{S,I_2,J_2, \infty}.
\end{align*}
    \item We will not define, nor use, a norm $|| \cdot ||_{S,I,J,\infty}$ with $I < 0$.
    \item If $v \in S^{0}(R^{-I_1},\log(R)^{J_1})$ is fixed and $w \in S^{2n}(R^{I_2},\log(R)^{J_2})$ for $n,I_1,I_2,J_1,J_2 \geq 0$, then $v \cdot w \in S^{2n}(R^{-I_1+I_2},\log(R)^{J_1+J_2})$ and
    \begin{align*}
||v \cdot w||_{S, \mathrm{ori}} &\leq C(I_1, m, R_0, v) \cdot ||w||_{S, \mathrm{ori}}, \\
||v \cdot w||_{S, -i + I_2, J_1 + J_2, \infty} &\leq C(I_1, m, R_0, v) \cdot ||w||_{S,I_2,J_2, \infty} \quad \forall 0 \leq i \leq \min\{I_1,I_2\},
\end{align*}
i.e., the product map is continuous.
\end{enumerate} 
\end{remark}

\begin{notation}
Throughout the whole paper, $\chi_{[a,+\infty)}: \mathbb R \rightarrow [0,1]$ denotes a fixed smooth transition function which satisfies $\chi(x) = 0$ in a neighbourhood of $x = a$, $\chi(x) = 1$ in an open neighbourhood of $[2a,+\infty)$ and $\chi(x) > 0$ on $\mathrm{int}(\mathrm{supp}(\chi_{[a,+\infty)}))$. In particular, such a transition function is supported on $[a,+\infty)$. Explicitly, up to an affine transformation (depending only on $a \in \mathbb R$), one can choose
$$
\chi(x) = \frac{f(x)}{f(x) + f(1-x)}, \quad f(x) = \begin{cases}
    e^{-\frac{1}{x}}, \quad &x > 0 \\
    0, \quad &x \leq 0
\end{cases}
$$
\index{chi@$\chi_{[a,+\infty)}$, a fixed transition function}
\end{notation}

On our three main regions of the light cone, observe that:
 \begin{enumerate}
     \item $C_{\mathrm{ori}}$: $F(u_0+v_1) = (u_0 + v_1)^p$ can be expanded around $u_0$.
     \item $C_{\mathrm{mid}}$:  $(u_0+v_1)^p$ can be expanded around the dominant component $\lambda^{\frac{3}{2}}(R^{-3} (15)^{\frac{3}{2}} + (t\lambda)^{-2} C_1(\nu))$ of $u_0 + v_1$.
     \item $C_{\mathrm{tip}}$:  $(u_0+v_1)^p$ can be expanded around the dominant component $\lambda^{\frac{3}{2}}(t\lambda)^{-2}C_1(\nu)$ of $v_1$.
 \end{enumerate}
 On $C_{\mathrm{mid}}$ and $C_{\mathrm{tip}}$, we are able to terminate the expansions of $(u_0 + v_1)^p$ as each term $(v_1/u_0)^n$ exhibits improved decay in either $R$ or $(t\lambda)$. Near the origin, thanks to $m$ being chosen small enough, we will be able to discard any logarithmic power above some threshold. Indeed, consider the binomial expansion:
 $$
 t^2(u_0 + v_1)^p = t^2u_0^p\left( 1 + \frac{v_1}{u_0} \right) = t^2 u_0^p \sum_{n=0}^{+\infty} \binom{p}{n} \left( \frac{v_1}{u_0} \right)^n,
 $$
 where we recall that $|v_1/u_0| \leq 1/2$ on $C_{\mathrm{ori}}$. We aim to replace $(v_1/u_0)^n$ by a suitable approximation where the logarithmic exponents are bounded.  When $\Re(z) > 0$ and $|z| \geq 2\sqrt{15}$, we have:
 \begin{align*}
     \left( \frac{v_1}{u_0} \right)^n &= \frac{z^{3n}}{(t\lambda)^{2n}} \left(w_0(z^{-1}) + z^{-3}w_1(z^{-1})\log(z)  \right)^n \\
     &= \frac{z^{3n}}{(t\lambda)^{2n}} \sum_{\substack{i+j = n \\ i,j \geq 0}} \binom{n}{i,j} w_0(z^{-1})^i \left[ z^{-3}w_1(z^{-1})\log(z) \right]^j.
 \end{align*}
 Thanks to $m$ being chosen small enough in (\ref{m definition}), the error part 
 $$E_{n} = \frac{z^{3n}}{(t\lambda)^{2n}} \sum_{\substack{i+j = n \\ i,j \geq 0 \\ j \geq \frac{3}{2}N_0}} \binom{n}{i,j} w_0(z^{-1})^i \left[ z^{-3}w_1(z^{-1})\log(z) \right]^j  $$
 can be estimated by
 $$
 \sum_{\substack{i+j = n \\ i,j \geq 0 \\ j \geq \frac{3}{2}N_0}} \binom{n}{i,j} \frac{1}{4^n} \left[ z^{-3} \log(z) \right]^j = \mathcal{O}(2^{-n} z^{-3N_0}) = \mathcal{O}(2^{-n}(t\lambda)^{-N_0})
 $$
 on the region $2m(t\lambda)^{\frac{1}{3}} \leq |z| \leq m(t\lambda)^{\frac{2}{3}}, 0 < t \leq t_0$ when $n \geq N_0$. Similarly, on the region $3\sqrt{15}/2 \leq |z| \leq 2m(t\lambda)^{\frac{1}{3}}, 0 < t \leq t_0$,  we have the estimate
  $$
(t\lambda)^{-n} \sum_{\substack{i+j = n \\ i,j \geq 0 \\ j \geq \frac{3}{2}N_0}} \binom{n}{i,j} \frac{1}{4^n} \left[ z^{-3} \log(z) \right]^j = \mathcal{O}(2^{-n}(t\lambda)^{-N_0}), \quad n \geq N_0,
 $$
Using Cauchy's Integral Formula, for any multi-index $(l_1,l_2) \in \mathbb N^2$, we obtain
  \begin{equation}
      (t\partial_t)^{l_1}\partial_z^{l_2} E_n = \mathcal{O}(2^{-n}(t\lambda)^{-N_0}), \quad n \geq N_0, \label{eq: E_n uniform bound}
  \end{equation}
 on the smaller region $2\sqrt{15} \leq |z| \leq m(t\lambda)^{\frac{1}{3}}, 0 < t \leq t_0$.  Hence, we define the ``truncation'' operator:
 \begin{align}
     T\left[ \left( \frac{v_1}{u_0}\right)^n \right] &= 
 \begin{cases}  \left( \frac{v_1}{u_0}\right)^n &\text{ if } 0 \leq n < N_0 \\
     \left( \frac{v_1}{u_0}\right)^n - \chi_{[2\sqrt{15},+\infty)}(|z|) E_{n} &\text{ otherwise}
 \end{cases} \label{truncation}
 \index{truncation@$T$, the truncation operator}
 \end{align}
 On $4\sqrt{15} = R_0 \leq |z| \leq m(t\lambda)^{\frac{2}{3}}$, the error part $E_n$ is completely removed and the loagrithmic powers are capped to $\log(R)^{3N_0}$. In particular, we have:
    \begin{align*}
        t^2 u_0^p \left( \frac{v_1}{u_0} \right)^{n} &\in \dfrac{\lambda^{\frac{3}{2}}}{(t\lambda)^{2(n-1)}} S^{2n}(R^{-7+3n},\log(R)^{n} ), \\
        t^2 u_0^p  T \left[ \left( \frac{v_1}{u_0} \right)^{n} \right] &\in \dfrac{\lambda^{\frac{3}{2}}}{(t\lambda)^{2(n-1)}} S^{2n}(R^{-7+3n},\log(R)^{3N_0} ). 
    \end{align*}
    As the function $t^2u_0^p (v_1/u_0)^n$ is holomorphic around $|z| \leq \sqrt{15}/2$ and $z \in [\sqrt{15}/2,R_0]$, there exists $C(l_1,l_2) > 1$ for which
    \begin{align*}
        \left[ \dfrac{\lambda^{\frac{3}{2}}}{(t\lambda)^{2(n-1)}} \right]^{-1} \cdot \left| \left| (t\partial_t)^{l_1}( z  \partial_z)^{l_2} t^2 u_0^p \left( \frac{v_1}{u_0} \right)^{n} \right| \right|_{S, \mathrm{ori}} \leq C^n 
    \end{align*}
    using the product rule when $l_2 = 0$ and the complex-differentiability when $l_2 > 0$. Combining this together with the estimate (\ref{eq: E_n uniform bound}), we conclude
    \begin{align*}
        \left[ \dfrac{\lambda^{\frac{3}{2}}}{(t\lambda)^{2(n-1)}} \right]^{-1} \cdot \left| \left| (t\partial_t)^{l_1}( z  \partial_z)^{l_2} t^2 u_0^p  T \left[ \left( \frac{v_1}{u_0} \right)^{n} \right] \right| \right|_{S, \mathrm{ori}} \leq C^n 
    \end{align*}
    since differentiating the cutoff $\chi_{[2\sqrt{15},+\infty)}$ introduces no issue. Near the origin, we allow an exponential growth in $n$, because the prefactor $(t\lambda)^{-2(n-1)}$ cancels this growth on a sufficiently small cone. When $3n - 7  > 0$, we also obtain
\begin{align*}
    \left[ \dfrac{\lambda^{\frac{3}{2}}}{(t\lambda)^{2(n-1)}} \right]^{-1} \cdot \left| \left| (t\partial_t)^{l_1}( z \partial_z)^{l_2}   t^2 u_0^p T \left[ \left( \frac{v_1}{u_0} \right)^{n} \right] \right| \right|_{S,-7+3n,3N_0,\infty} \\
    = \left[ \dfrac{\lambda^{\frac{3}{2}}}{(t\lambda)^{2(n-1)}} \right]^{-1} \cdot \left| \left| (t\partial_t)^{l_1}( z  \partial_z)^{l_2}  t^2 u_0^p \left( \frac{v_1}{u_0} \right)^{n} \right| \right|_{S,-7+3n,3N_0,\infty}  \leq c^n
\end{align*}
    for some $0 < c(l_1,l_2) < 1$. For $l_2 = 0$, this is a consequence of how small $m$ was chosen in (\ref{m definition}) and the product rule. For $l_2 > 0$, this follows from the complex-differentiability. These bounds ensure the convergence of the series
 $$
 t^2(u_0 + v_1)^p \approx t^2 u_0^p \sum_{n=0}^{+\infty} \binom{p}{n} T \left[ \left( \frac{v_1}{u_0} \right)^n \right]
 $$
 and its derivatives. The difference between $t^2(u_0+v_1)^p$ and its truncation, i.e.,
    $$
     t^2(u_0 + v_1)^p - t^2 u_0^p \sum_{n=0}^{+\infty} \binom{p}{n} T \left[ \left( \frac{v_1}{u_0} \right)^n \right]
    $$
    is negligible in the following sense, which makes the truncation a suitable approximation for use in the iteration scheme.
    
\begin{definition}[Negligible terms]\label{negligible simeq definition}
    We say that $f(R,a,t) \in \mathcal{E}_{N_0,\nu}$ if the function $f\left(R,\frac{R}{(t\lambda)},t\right)$ is smooth on $C = \{(R,t): 0 < t < t_0, 0 < R < (t\lambda)\}$ and if for any indices $i,j$,
    $$
    \left|(t \partial_t)^i (\langle R \rangle \partial_R)^j f\left(R,\frac{R}{(t\lambda)},t\right) \right| \lesssim \lambda^{\frac{3}{2}}(t\lambda)^{-N_0} \left[1 + \left(1-\frac{R}{(t\lambda)} \right)^{\frac{1}{2}+\frac{1}{2}\nu-i-j-} \right]
    $$
    on the entire cone. In other words, $f$ has the desired smallness.
    \index{E(N0,nu)@$\mathcal{E}_{N_0,\nu}$, a space of negligible terms}
\end{definition}
% \begin{definition}[Negligible terms]\label{negligible simeq definition}
%     We say that $f(R,a,t) \in \mathcal{E}_{N,\beta}$ for $N,\beta \in \mathbb R$ if $f\left(R,\frac{R}{(t\lambda)},t\right)$ is smooth on $C = \{(R,t): 0 < t < t_0, 0 < R < (t\lambda)\}$ and if for any indices $i,j$,
%     $$
%     \left|(\langle R \rangle^i\partial_R^i)(t^j\partial_t^j) f\left(R,\frac{R}{(t\lambda)},t\right) \right| \lesssim \lambda^{\frac{3}{2}}(t\lambda)^{-N} \left[1 + \left(1-\frac{R}{(t\lambda)} \right)^{\beta-i-j-} \right]
%     $$
%     on the whole cone. We also write $\mathcal{E}_{N} = \cap_{\beta \in \mathbb R} \mathcal{E}_{N,\beta}$. Theses spaces will keep track of the smallness of correction and error terms. If $f \in \mathcal{E}_{N_0,\frac{1}{2}+\frac{1}{2}\nu}$, $f$ has the desired smallness and is called ''negligible``.
%     \index{E(N,beta)@$\mathcal{E}_{N,\beta}$, a space measuring the smallness of correction and error terms}
% \end{definition}

As noted earlier,
$$
     t^2(u_0 + v_1)^p - t^2 u_0^p \sum_{n=0}^{+\infty} \binom{p}{n} T \left[ \left( \frac{v_1}{u_0} \right)^n \right] = t^2 u_0^p [1-\chi_{[2\sqrt{15},+\infty)}(|z|)] \sum_{n=0}^{+\infty} \cdot E_{n} \in \mathcal{E}_{N_0,\nu}
$$
due to (\ref{eq: E_n uniform bound}), with no singularity at $R = (t\lambda)$. The use of three different binomial expansions for $t^2(u_0+v_1)^p$  across different regions of the cone motivates the following definitions of correction and error spaces.

\begin{definition}[Correction space $V_k$]
\index{Vk@$V_{2k-1}$, $V_{2k}$, the vector spaces of correction terms}
    Let $V_{2k-1}$, $k \geq 1$, be the vector space spanned by the set of smooth functions $v(R,t) = v_{\mathrm{ori}}(R,t) + v_{\mathrm{mid}}(R,t) + v_{\mathrm{tip}}(R,t) + \eta(R,t)$ on the cone $0 \leq R \leq (t\lambda)$, $0 < t \leq t_0$, where each component admits a decomposition as specified below:
    \begin{enumerate}
    \item $0 \leq R \leq m(t\lambda)^{\frac{2}{3}}$: The component $v_{\mathrm{ori}}$ is given by a convergent sum of the form: $$v_{\mathrm{ori}}(R,t) = \left( \sum_{n = 0}^{+\infty} \ \frac{\lambda^{\frac{3}{2}}}{(t\lambda)^{\alpha+2n}} w_{n}(R) \right) \chi_{[1/m,+\infty)} 
 \left(\frac{(t\lambda)^{\frac{2}{3}}}{R} \right)^{1+j},$$
     where $j \in \mathbb N_{\geq 0}$,  $w_{n} \in S^{2(k-1)}(R^{I+3n}, \log(R)^{J})$ for some common $I,J \in \mathbb N_{\geq 0}$ and $(\alpha,I)$ is $k$-admissible on $C_{\mathrm{ori}}$. The convergence is understood as follows: for any fixed derivative, there exists constants $0 < c(l) < 1 < C(l)$ and $n_0(l) > 0$ for which
    \begin{align*}
    \left| \left| (R\partial_R)^l w_{n} \right| \right|_{S, \mathrm{ori}}  \leq C^n, \quad \left| \left| (R\partial_R)^l w_{n} \right| \right|_{S, I+3n, J, \infty }  \leq c^n
    \end{align*}
    for all $n \geq n_0(l)$.
    \item $\frac{m}{2}(t\lambda)^{\frac{2}{3}} \leq R \leq 2(t\lambda)^{\frac{2}{3}+\varepsilon}$: The component $v_{\mathrm{mid}}$ is a single term of the form:
    $$v_{\mathrm{mid}}(R,t) = \frac{\lambda^{\frac{3}{2}}}{(t\lambda)^{\alpha}}R^i \log(R)^{j_1} \log(t\lambda)^{j_2} g \left( \frac{(t\lambda)^{\frac{2}{3}}}{R} \right) (1-\chi_{[1,+\infty)})\left( \frac{R}{(t\lambda)^{\frac{2}{3}+\varepsilon}} \right)^{1+j_3},$$
where $(\alpha,i)$ is $k$-admissible on $C_{\mathrm{mid}}$, $j_l \in \mathbb N_{\geq 0}$, $g(y)$ is any smooth function on $(0,+\infty)$ which is zero when $y \geq 2/m$ and expands as a finite sum of holomorphic functions and logarithms near $y = 0$.\footnote{More precisely, there should exist some $y_0 > 0$ and finitely many holomorphic functions $g_0,g_1,...$ around $|y| \leq y_0$ for which we can write $g(y) = \sum_{j=0}^{J}g_j(y) \log(y)^j$ whenever $|y| \leq y_0$}  

We note that we can always build a basis of our vector space by omitting the powers of $\log(t\lambda)$ in $v_{\mathrm{mid}}$ as they can always be rewritten as $\frac{3}{2}\left( \log(y) - \log(R) \right)$, $y = (t\lambda)^{\frac{2}{3}}/R$, and the $\log(y)$ part can be included in the $g(y)$-type functions.
\item $(t\lambda)^{\frac{2}{3}+\varepsilon} \leq R \leq (t\lambda)$: The component $v_{\mathrm{tip}}$ is a single term of the form:
    $$v_{\mathrm{tip}}(R,t) = \frac{\lambda^{\frac{3}{2}}}{(t\lambda)^{\alpha}}R^i \log(R)^{j_1} \log(t\lambda)^{j_2} h \left( \frac{R}{(t\lambda)^{\frac{2}{3}+\varepsilon}} \right),$$
where $(\alpha,i)$ is $k$-admissible on $C_{\mathrm{tip}}$, $j_l \in \mathbb N_{\geq 0}$ and $h(y)$ is any smooth function which is constant outside $[1,2]$ and zero when $y$ is in a neighbourhood of $1$.

\item $\eta \in \mathcal{E}_{N_0,\nu}$ and $\eta$ has no singularity at the boundary $R = (t\lambda)$. 
\end{enumerate}

Similarly, let $V_{2k}$, $k \geq 1$, be the vector space generated by smooth functions $v(R,t)$ on the cone $0 \leq R < (t\lambda)$, $0 < t \leq t_0$, whose derivatives up to order $\floor{\left(\frac{1}{2}+\frac{1}{2}\nu\right)-}$ are continuous at the boundary $R = (t\lambda)$, and which are given by a finite sum of the form:
    $$v(R,t) = \frac{\lambda^{\frac{3}{2}}}{(t\lambda)^{\alpha}}R^i \log(R)^{j_1} \log(t\lambda)^{j_2} q \left( \frac{R}{(t\lambda)} \right)  \chi_{[1,+\infty)}\left(  \frac{R}{(t\lambda)^{\frac{2}{3}+\varepsilon}} \right)^{1+j_3},$$
where $(\alpha,i)$ is $k$-admissible on $C_{\mathrm{tip}}$, $j_l \in \mathbb N_{\geq 0}$ and $q(a) \in a^2\mathcal{Q}_{k,\frac{1}{2}+\frac{1}{2}\nu}$. We note that for these functions, the $\log(t\lambda)$ powers can always be rewritten so that they do not appear in the finite sum.
\end{definition}

As a consistency check, we verify that the first correction term $v_1(R,t)$ belongs to $V_1$. If $V_1(R)$ is defined as in (\ref{v_1 formula}), then
\begin{equation}
     v_1(R,t):= \frac{\lambda^{\frac{3}{2}}}{(t\lambda)^{2}} V_1(R) \in V_1  \label{eq: v_1 belongs to V_1}
\end{equation}
as a consequence of the following proposition applied for $k = 1$, $(\alpha,I) = (2,0)$.

\begin{proposition}[Examples of $V_{2k-1}$ functions]\label{inclusion S^2n into V_2k-1}
    Let $(\alpha,I)$ be $k$-admissible on $C_{\mathrm{mid}}$ and $C_{\mathrm{tip}}$, $k \geq 1$. For $J \in \mathbb N_{\geq 0}$, 
    $$
    \frac{\lambda^{\frac{3}{2}}}{(t\lambda)^{\alpha}}S^0(R^I, \log(R)^J) \cdot (1- \chi_{[1/m,+\infty)}) 
 \left(\frac{(t\lambda)^{\frac{2}{3}}}{R} \right) \subset V_{2k-1}.
    $$
    Furthermore, if $(\alpha,I)$ is also $k$-admissible on $C_{\mathrm{ori}}$, then
    $$
\frac{\lambda^{\frac{3}{2}}}{(t\lambda)^{\alpha}}S^0(R^I, \log(R)^J) \subset  V_{2k-1}.
    $$
\end{proposition}

\begin{proof}
    The second statement follows from the first one by writing
    \begin{align*}
            \frac{\lambda^{\frac{3}{2}}}{(t\lambda)^{\alpha}}S^0(R^I, \log(R)^J) &= \frac{\lambda^{\frac{3}{2}}}{(t\lambda)^{\alpha}}S^0(R^I, \log(R)^J)  \cdot \chi_{[1/m,+\infty)} \\
            &\phantom{=}+\frac{\lambda^{\frac{3}{2}}}{(t\lambda)^{\alpha}}S^0(R^I, \log(R)^J)  \cdot (1- \chi_{[1/m,+\infty)}). 
    \end{align*}
    Let $V(R) \in \frac{\lambda^{\frac{3}{2}}}{(t\lambda)^{\alpha}} S^0(R^I, \log(R)^J)$. On $R \geq R_0$, we can write
    $$
    V(R) = R^I \sum_{j=0}^{J}w_{j}(R^{-1})\log(R)^j,
    $$
    where $w_{j}(y)$ is holomorphic on $|y| < R_0^{-1} + \delta$ for some $0 < \delta \ll 1$. Now, rewrite $V(R)$ as
    $$
    V(R) = \sum_{i=0}^{|I|+N_0} \sum_{j=0}^J c_{k,j}R^{I-k} \log(R)^j +  \sum_{j=0}^{J}\tilde{w}_{j}(R^{-1})\log(R)^j,
    $$
    where each $\tilde{w}_j(y)$ is holomorphic with a zero of order at least $y^{N_0+1}$ at $y = 0$. Consider a smooth cutoff function $\chi(y)$ such that $\chi(y) = 1$ for $|y| < R_0^{-1} + \delta/2$ and $\chi(y) = 0$ for $|y| \geq R_0^{-1} + \delta$. Then, we have
    $$
    V(R) = \sum_{i=0}^{|I|+N_0} \sum_{j=0}^J c_{k,j} R^{I-k} \log(R)^j +  \chi(R^{-1}) \sum_{j=0}^{J}\tilde{w}_{j}(R^{-1})\log(R)^j, \quad R \gtrsim (t\lambda)^{\frac{2}{3}} \gg R_0.
    $$
    Therefore, on $R \gtrsim (t\lambda)^{\frac{2}{3}}$, we obtain
    $$
    V(R) = \sum_{i=0}^{|I|+N_0} \sum_{j=0}^J c_{k,j} R^{I-k} \log(R)^j +  S^0(R^{-N_0-1}, \log(R)^J) =: \tilde{V}(R) +  S^0(R^{-N_0-1}, \log(R)^J).
    $$
    Consider the expression:
    $$
\frac{\lambda^{\frac{3}{2}}}{(t\lambda)^{\alpha}} S^0(R^{-N_0-1}, \log(R)^J)  \cdot (1- \chi_{[1/m,+\infty)}) 
 \left(\frac{(t\lambda)^{\frac{2}{3}}}{R} \right).
    $$
    This expression is in $\mathcal{E}_{N_0,\nu}$ with no singularity at $a = 1$ since no $\mathcal{Q}$ functions is involved. Indeed, this remainder is sufficiently small in a pointwise sense. Moreover, for all indices $l_1,l_2 \geq 0$, we have:
\begin{align}
    (t\partial_t)^{l_1} (R\partial_R)^{l_2} \frac{\lambda^{\frac{3}{2}}}{(t\lambda)^{\alpha}}R^i \log(R)^{j_1} \log(t\lambda)^{j_2} &= c_{l_1,l_2,\alpha,i,j_2,j_2} \frac{\lambda^{\frac{3}{2}}}{(t\lambda)^{\alpha}}R^i \log(R)^{j_1} \log(t\lambda)^{j_2} \notag\\
    (t\partial_t)^{l_1} (R\partial_R)^{l_2} S^{2n}(R^I,\log(R)^J)&\subset S^{2n}(R^I,\log(R)^J) \notag \\
    (t\partial_t)^{l_1} (R\partial_R)^{l_2} g \left( \frac{(t\lambda)^{\frac{2}{3}}}{R} \right) &= (-1)^{l_1+l_2}\left[\frac{2 \nu}{3} \right]^{l_1} \left[ (y\partial_y)^{l_1+l_2}g \right] \left( \frac{(t\lambda)^{\frac{2}{3}}}{R} \right) \notag \\
     (t\partial_t)^{l_1} (R\partial_R)^{l_2} h \left(  \frac{R}{(t\lambda)^{\frac{2}{3}+\varepsilon}} \right) &= \left[\nu \left(\frac{2}{3}+\varepsilon \right) \right]^{l_1} \left[(y\partial_y)^{l_1+l_2}h \right] \left(  \frac{R}{(t\lambda)^{\frac{2}{3}+\varepsilon}} \right), \label{eq: derivative t R in terms of a, y}
\end{align}
so that the smallness of the expression is preserved under differentiation with $(t\partial_t)^{l_1} (R\partial_R)^{l_2}$. 

Finally, we consider the finite sum $\tilde{V}(R)$. We extract a $v_{\mathrm{tip}}$ and $v_{\mathrm{mid}}$ component by decomposing:
\begin{align*}
   \frac{\lambda^{\frac{3}{2}}}{(t\lambda)^{\alpha}}  \tilde{V}(R)  \cdot (1- \chi_{[1/m,+\infty)}) 
 \left(\frac{(t\lambda)^{\frac{2}{3}}}{R} \right)  =\frac{\lambda^{\frac{3}{2}}}{(t\lambda)^{\alpha}}  \tilde{V}(R)  \cdot \chi_{[1,+\infty)}\left( \frac{R}{(t\lambda)^{\frac{2}{3}+\varepsilon}} \right) \\
+\frac{\lambda^{\frac{3}{2}}}{(t\lambda)^{\alpha}}  \tilde{V}(R)  \cdot (1- \chi_{[1/m,+\infty)}) 
 \left(\frac{(t\lambda)^{\frac{2}{3}}}{R} \right) \cdot (1-\chi_{[1,+\infty)})\left( \frac{R}{(t\lambda)^{\frac{2}{3}+\varepsilon}} \right).
\end{align*}
\end{proof}

We now define the error spaces, whose structure is similar to that of the correction spaces.
\begin{definition}[Error space $E_k$]
\index{Ek@$E_{\mathrm{ori},k}$, $E_{\mathrm{tip},k}$, the vector spaces of error terms}
    Let $E_{\mathrm{ori},k}$, $k \geq 1$, be the vector space spanned by the set of smooth functions $e(R,t) = e_{\mathrm{ori}}(R,t) + e_{\mathrm{mid}}(R,t) + e_{\mathrm{tip}}(R,t) + \eta(R,t)$ on the cone $0 \leq R \leq (t\lambda)$, $0 < t \leq t_0$, where each component admits a decomposition as specified below:
    \begin{enumerate}
    \item $0 \leq R \leq m(t\lambda)^{\frac{2}{3}}$: The component $e_{\mathrm{ori}}$ is given by a convergent sum of the form: $$e_{\mathrm{ori}}(R,t) = \left( \sum_{n = 0}^{+\infty} \ \frac{\lambda^{\frac{3}{2}}}{(t\lambda)^{\alpha+2n}} w_{n}(R) \right) \chi_{[1/m,+\infty)} 
 \left(\frac{(t\lambda)^{\frac{2}{3}}}{R} \right)^{1+j},$$
     where $j \in \mathbb N_{\geq 0}$,  $w_{n} \in S^{2(k-1)}(R^{I+3n}, \log(R)^{J})$ for some common $I,J \in \mathbb N_{\geq 0}$ and $(\alpha,I)$ is $k$-admissible on $C_{\mathrm{ori}}$. Moreover, for any fixed derivative, there exists constants $0 < c(l) < 1 < C(l)$ and $n_0(l) > 0$ for which
    \begin{align*}
    \left| \left| (R\partial_R)^l w_{n} \right| \right|_{S, \mathrm{ori}}  \leq C^n, \quad \left| \left| (R\partial_R)^l w_{n} \right| \right|_{S, I+3n, J, \infty }  \leq c^n
    \end{align*}
    for all $n \geq n_0(l)$. 
    \item $\frac{m}{2}(t\lambda)^{\frac{2}{3}} \leq R \leq 2(t\lambda)^{\frac{2}{3}+\varepsilon}$: The component $e_{\mathrm{mid}}$ is given by a single term of the form: 
    $$e_{\mathrm{mid}}(R,t) = \frac{\lambda^{\frac{3}{2}}}{(t\lambda)^{\alpha}}R^i \log(R)^{j_1} \log(t\lambda)^{j_2} g \left( \frac{(t\lambda)^{\frac{2}{3}}}{R} \right) (1-\chi_{[1,+\infty)})\left( \frac{R}{(t\lambda)^{\frac{2}{3}+\varepsilon}} \right)^{1+j_3},$$
where $(\alpha,i)$ is $k$-admissible on $C_{\mathrm{mid}}$, $j_l \in \mathbb N_{\geq 0}$, $g(y)$ is any smooth function on $(0,+\infty)$ which is zero when $y \geq 2/m$ and expands as a finite sum of holomorphic functions and logarithms near $y = 0$.\footnote{More precisely, there should exist some $y_0 > 0$ and a finite number of holomorphic functions $g_1,g_2,...$ around $|y| \leq y_0$ for which we can write $g(y) = \sum_{j=0}^{J}g_j(y) \log(y)^j$ whenever $|y| \leq y_0$}  

We note that we can always build a basis of our vector space by omitting the powers of $\log(t\lambda)$ in $v_{\mathrm{mid}}$ as they can always be rewritten as $\frac{3}{2}\left( \log(y) - \log(R) \right)$, $y = (t\lambda)^{\frac{2}{3}}/R$, and the $\log(y)$ part can be included in the $g(y)$-type functions.
\item $(t\lambda)^{\frac{2}{3}+\varepsilon} \leq R \leq (t\lambda)$: The component $e_{\mathrm{tip}}$ is given by a single term of the form: 
    $$e_{\mathrm{tip}}(R,t) = \frac{\lambda^{\frac{3}{2}}}{(t\lambda)^{\alpha}}R^i \log(R)^{j_1} \log(t\lambda)^{j_2} h \left( \frac{R}{(t\lambda)^{\frac{2}{3}+\varepsilon}} \right),$$
where $(\alpha,i)$ is $k$-admissible on $C_{\frac{2}{3}+\varepsilon}$, $j_l \in \mathbb N_{\geq 0}$ and $h(y)$ is any smooth function with compact support in $(1,2)$. We note that the $\log(t\lambda)$ powers can always be rewritten so that they do not appear in the finite sum for $e_{\mathrm{tip}}$.

\item $\eta \in \mathcal{E}_{N_0,\nu}$ and $\eta$ has no singularity at the boundary $R = (t\lambda)$. 
\end{enumerate}

The error space $E_{\mathrm{tip},k}$ is defined as the vector space of functions $E_{\mathrm{tip},k} = \frac{1}{a^2} V_{2k} + \mathcal{E}_{N_0,\nu}$. 
\end{definition}

 \begin{notation}
The symbol
$$
\fsum_{\alpha \: \ ...}
$$
indicates summation over a finite set of indices $\alpha$ satisfying a certain relation. If the exact set of indices does not matter, such a notation will be used. It will mostly be used when considering a sum over a finite set of $k$-admissible pairs $(\alpha,i)$, which are not explicitly defined.
\index{finite sum@$\fsum$, any kind of finite sum}
\end{notation}

\begin{remark} We make a few remarks regarding the definitions of the correction and error spaces.
    \begin{enumerate}
        \item To simplify the definitions of $V_{2k-1}$, $V_{2k}$, $E_{\mathrm{ori},k}$ and $E_{\mathrm{tip},k}$, we gave the definition in terms of a basis. That means that the correction terms and error terms on each region are always given by a finite sum of such elements. For example, a term $v_{\mathrm{mid}}$ has the form:
        $$\fsum_{\substack{(\alpha,i) \kadm \\ j_1,j_2,j_3 \geq 0}}\frac{\lambda^{\frac{3}{2}}}{(t\lambda)^{\alpha}}R^i \log(R)^{j_1} \log(t\lambda)^{j_2} g_{\alpha,i,j_1,j_2,j_3} \left( \frac{(t\lambda)^{\frac{2}{3}}}{R} \right) (1-\chi_{[1,+\infty)})\left( \frac{R}{(t\lambda)^{\frac{2}{3}+\varepsilon}} \right)^{1+j_3}.$$
        The cutoff functions are fixed, but the functions $w_n(R) = w_{n}^{\alpha,I,J,j}(R)$, \sloppy $g(y) = g_{\alpha,i,j_1,j_2,j_3}(y)$, $h(y) sloppy= h_{\alpha,i,j_1,j_2}(y)$ and $q(a) = q_{\alpha,i,j_1,j_2}(a)$ can be anything with the desired properties.
        \item The only singularities that can happen are at $a = 1$, when considering functions from $V_{2k}$ and $E_{\mathrm{tip},k}$. In particular, restricting such a function near the origin or the middle region of the cone removes the singularity.
        \item The main difference between $V_{2k-1}$ and $E_{\mathrm{ori},k}$ lies near the tip of the cone $(t\lambda)^{\frac{2}{3}+\varepsilon} \leq R \leq (t\lambda)$. Up to a negligible part, error terms are supported on $(t\lambda)^{\frac{2}{3}+\varepsilon} \leq R \leq 2(t\lambda)^{\frac{2}{3}+\varepsilon}$, while correction terms are supported on the whole cone. 
        \item For the error terms, we write $e_{k} \simeq \tilde{e}_k$ if $e_k - \tilde{e}_k \in \mathcal{E}_{N_0,\nu}$. In other words, the equality $e_k(R,a,t) = \tilde{e}_k(R,a,t)$ holds up to negligible terms, which do not affect the subsequent analysis. This negligible difference is carried over to all subsequent error terms $(e_{k+1},e_{k+2})$, but it does not affect the algorithm, so we will omit them when describing the next error terms. It is important to note that in the renormalization step, we never differentiate an error term. Hence, the singularity of an $E_{\mathrm{tip},k}$ element at $a = 1$ cannot get worse. We do differentiate even correction terms $v_{2k}$, but in that case, the support is always restricted to the middle or the origin part of the cone, which removes the singularity.
    \end{enumerate}
\end{remark}

We conclude this section by proving some useful computation rules concerning these correction and error spaces. 

\begin{proposition}[Stability under differentiation]\label{prop:stability under differentiation}
For any $l_1,l_2,k \in \mathbb N$, 
$$(t \partial_t)^{l_1}(R\partial_t)^{l_2} V_{2k-1} \subset V_{2k-1}, \quad (t \partial_t)^{l_1}(R\partial_t)^{l_2} E_{\mathrm{ori},k} \subset E_{\mathrm{ori},k}.$$
\end{proposition}

\begin{proof}
    We prove the inclusion $(t \partial_t)^{l_1}(R\partial_t)^{l_2} V_{2k-1} \subset V_{2k-1}$. Recall first the identities (\ref{eq: derivative t R in terms of a, y}) which hold for any exponents and any smooth functions $g$ and $h$. 
\item[\textbf{Tip part:}] Suppose 
$$v_{\mathrm{tip}}(R,t) = \frac{\lambda^{\frac{3}{2}}}{(t\lambda)^{\alpha}}R^i \log(R)^{j_1} \log(t\lambda)^{j_2} h \left( \frac{R}{(t\lambda)^{\frac{2}{3}+\varepsilon}} \right).$$
Then, applying the derivatives yields:
$$
(t \partial_t)^{l_1}(R\partial_t)^{l_2} v_{\mathrm{tip}} = \fsum_{k_1,k_2 \geq 0} c_{k_1,k_2}\frac{\lambda^{\frac{3}{2}}}{(t\lambda)^{\alpha}}R^i \log(R)^{j_1} \log(t\lambda)^{j_2} \left[ (y\partial y)^{k_1+k_2} h \right] \left( \frac{R}{(t\lambda)^{\frac{2}{3}+\varepsilon}} \right),
$$
which is a finite sum of elements $\tilde{v}_{\mathrm{tip},k_1,k_2}$ having a suitable form. 

\item[\textbf{Middle part:}] Similarly, suppose 
$$v_{\mathrm{mid}}(R,t) = \frac{\lambda^{\frac{3}{2}}}{(t\lambda)^{\alpha}}R^i \log(R)^{j_1} \log(t\lambda)^{j_2} g \left( \frac{(t\lambda)^{\frac{2}{3}}}{R} \right) (1-\chi_{[1,+\infty)})\left( \frac{R}{(t\lambda)^{\frac{2}{3}+\varepsilon}} \right)^{1+j_3}.$$
Then: 
\begin{align*}
    (t \partial_t)^{l_1}(R\partial_t)^{l_2}v_{\mathrm{mid}}(R,t) = (1-\chi_{[1,+\infty)})^{1+j_3} (t \partial_t)^{l_1}(R\partial_t)^{l_2} \frac{\lambda^{\frac{3}{2}}}{(t\lambda)^{\alpha}}R^i \log(R)^{j_1} \log(t\lambda)^{j_2} g + \\ 
\phantom{=} \sum_{l = 1}^{l_1+l_2} \left[ (y \partial_y)^l (1-\chi_{[1,+\infty)})^{1+j_3} \right] \cdot \fsum_{k_1,k_2 \geq 0} c_{l,k_1,k_2} (t \partial_t)^{k_1}(R\partial_t)^{k_2} \frac{\lambda^{\frac{3}{2}}}{(t\lambda)^{\alpha}}R^i \log(R)^{j_1} \log(t\lambda)^{j_2} g.
\end{align*}
The first part is treated as for $v_{\mathrm{tip}}$. It produces a finite sum of $\tilde{v}_{\mathrm{mid}}$ elements. As for the second part, it yields a finite sum of elements of the form:
$$
\tilde{v}_{\mathrm{mid}} \cdot h \left( \frac{R}{(t\lambda)^{\frac{2}{3}+\varepsilon}} \right) := \frac{\lambda^{\frac{3}{2}}}{(t\lambda)^{\alpha}}R^i \log(R)^{j_1} \log(t\lambda)^{j_2} g \left( \frac{(t\lambda)^{\frac{2}{3}}}{R} \right) h \left( \frac{R}{(t\lambda)^{\frac{2}{3}+\varepsilon}} \right), 
$$
where $h$ is smooth with compact support in $(1,2)$. Then $R \sim (t\lambda)^{\frac{2}{3}+\varepsilon}$ and $(t\lambda)^{\frac{2}{3}}R^{-1} \sim (t\lambda)^{-\varepsilon} \sim 0$. Near $y = 0$, one can expand
$$
g(y) = g_0(y) + \sum_{j=1}^{L}g_{j}(y) \log(y)^j
$$
for some functions $g_j$ holomorphic around zero. Consider the $M$-th order Taylor polynomial $P_j(y)$ of $g_j(y)$ near $y = 0$ with $M = \ceil{N_0\varepsilon^{-1}}$, meaning that the remainder $\eta_j(y) = g_j(y) - P_j(y)$ is a holomorphic function around zero with $(y\partial_y)^l [g_j(y) - P_j(y)] = \mathcal{O}(y^{M})$ for any fixed $l \geq 0$. Then, we obtain
$$
h \cdot \tilde{v}_{\mathrm{mid}}(R,t) = 
   \frac{\lambda^{\frac{3}{2}}}{(t\lambda)^{\alpha}}R^i \log(R)^{j_1} \log(t\lambda)^{j_2} \left(P_0(y) + \sum_{j=1}^{L}P_{j}(y) \log(y)^j \right) h \left(  \frac{R}{(t\lambda)^{\frac{2}{3}+\varepsilon}} \right)+ \eta 
$$
for $y = (t\lambda)^{\frac{2}{3}}R^{-1}$. The remainder
\begin{align}
    \eta =  \frac{\lambda^{\frac{3}{2}}}{(t\lambda)^{\alpha}}R^i \log(R)^{j_1} \log(t\lambda)^{j_2} \left(\eta_0(y) + \sum_{j=1}^{L}\eta_{j}(y) \log(y)^j \right) h \left(  \frac{R}{(t\lambda)^{\frac{2}{3}+\varepsilon}} \right) \label{eq:remainder on mid/tip is negligible}
\end{align}
is supported on $R \sim (t\lambda)^{\frac{2}{3}+\varepsilon}$, so it has no singularity at the tip of the cone.  We verify that $\eta \in \mathcal{E}_{N_0,\nu}$. The $\log(y)$ rewrites as $-\log(R) + \frac{2}{3}\log(t\lambda)$, so assume without loss of generality that we have only one term:
$$
\eta =  \frac{\lambda^{\frac{3}{2}}}{(t\lambda)^{\alpha}}R^i \log(R)^{j_1} \log(t\lambda)^{j_2} \eta_0(y)h \left(  \frac{R}{(t\lambda)^{\frac{2}{3}+\varepsilon}} \right).
$$
Given that  $(\alpha,i)$ is $k$-admissible and $y^M \sim (t\lambda)^{-N_0}$ on $R \sim (t\lambda)^{\frac{2}{3}+\varepsilon}$, the remainder have the desired smallness. Furthermore, the smallness is left unchanged under $(t \partial_t)^{l_1}(R\partial_t)^{l_2}$, as a consequence of the equalities (\ref{eq: derivative t R in terms of a, y}).

\item[\textbf{Origin part:}] Finally, consider a term of the form:
$$
v_{\mathrm{ori}}(R,t) = \left( \sum_{n = 0}^{+\infty} \ \frac{\lambda^{\frac{3}{2}}}{(t\lambda)^{\alpha+2n}} w_{n}(R) \right) f
 \left(\frac{(t\lambda)^{\frac{2}{3}}}{R} \right),
 $$
     where $f = \chi_{[1/m,+\infty)}^{1+j}$, $w_{n} \in S^{2(k-1)}(R^{I+3n}, \log(R)^{J})$ for some common $I,J \in \mathbb N_{\geq 0}$ with $(\alpha,I)$ being $k$-admissible on $C_{\mathrm{ori}}$. One has
         \begin{align*}
     (R \partial_R) v_{\mathrm{ori}}
 &=  f
 \left(\frac{(t\lambda)^{\frac{2}{3}}}{R} \right) \cdot    \sum_{n = 0}^{+\infty}  \frac{\lambda^{\frac{3}{2}}}{(t\lambda)^{\alpha+2n}} (R\partial_R)w_n(R) \\
 &\phantom{=}+  \sum_{n = 0}^{+\infty} \frac{\lambda^{\frac{3}{2}}}{(t\lambda)^{\alpha+2n}} w_{n}(R) \cdot f'
 \left(\frac{(t\lambda)^{\frac{2}{3}}}{R} \right) \cdot \left(\frac{(t\lambda)^{\frac{2}{3}}}{R} \right) \\
 &= f
 \left(\frac{(t\lambda)^{\frac{2}{3}}}{R} \right) \cdot    \sum_{n = 0}^{+\infty}  \frac{\lambda^{\frac{3}{2}}}{(t\lambda)^{\alpha+2n}} (R\partial_R)w_n(R) \\
 &\phantom{=}+  \tilde{f}  \left(\frac{(t\lambda)^{\frac{2}{3}}}{R} \right) \sum_{n = 0}^{+\infty} \frac{\lambda^{\frac{3}{2}}}{(t\lambda)^{\alpha+2n}} w_{n}(R),
    \end{align*}
    where $\tilde{w}_{n} \in S^{2(k-1)}(R^{I+3n}, \log(R)^{J})$  and $(\alpha,I)$ remains $k$-admissible. A similar equality holds for $(t\partial_t) v_{\mathrm{ori}}$ and, by induction, 
   \begin{align*}
       (t \partial_t)^{l_1}(R \partial_R)^{l_2} v_{\mathrm{ori}} &= 
      f  \left(\frac{(t\lambda)^{\frac{2}{3}}}{R} \right) \sum_{n = 0}^{+\infty} (t \partial_t)^{l_1}\frac{\lambda^{\frac{3}{2}}}{(t\lambda)^{\alpha+2n}} (R \partial_R)^{l_2} w_{n}(R) 
       \\ 
&\phantom{=}+\sum_{j=1}^M \tilde{f}_j  \left(\frac{(t\lambda)^{\frac{2}{3}}}{R} \right) \sum_{n = 0}^{+\infty} \frac{\lambda^{\frac{3}{2}}}{(t\lambda)^{\alpha+2n}} w_{n,j}(R),
   \end{align*}
    where $M = M(l_1,l_2) < +\infty$ and each $\tilde{f}_j$ is compactly supported on $(1/m,2/m)$. To conclude, it suffices to show that a term of the form: 
    \begin{align}
\tilde{f} \left(\frac{(t\lambda)^{\frac{2}{3}}}{R} \right) \sum_{n = 0}^{+\infty} \frac{\lambda^{\frac{3}{2}}}{(t\lambda)^{\alpha+2n}} w_{n}(R) \label{eq:approximation of restricted ori term}
    \end{align}
can be approximated by a finite sum if $\tilde{f}$ is compactly supported, so that they yield a finite number of $\tilde{v}_{\mathrm{mid}}$ components. Such a term is supported on $\frac{m}{2}(t\lambda)^{\frac{2}{3}} \leq R \leq m(t\lambda)^{\frac{2}{3}}$. On this region, 
     $$
     w_n(z) =  z^{I+3n} \sum_{j=0}^{J}w_{n,j}(z^{-1})
 \log(z)^j
     $$
     and $w_{n}(z)$ is well-approximated by a Taylor polynomial of degree $3N_0 + 1$,
$$
\tilde{w}_{n,j}(z) =   \sum_{l=0}^{3N_0 + 1} w_{n,j,l}z^{l}.$$
We check that summing all the errors that we do with this approximation is negligible in the sense that it belongs to $\mathcal{E}_{N_0,\nu}$ (and it has no singularity at the tip of the cone as the error is supported on $R \sim (t\lambda)^{\frac{2}{3}}$). Recall that:
$||(z\partial_z)^{l}w_{n,j}||_{A(|y| \leq R_0^{-1})} \leq m^{-3n}c_l^n$
for some $0 < c_l < 1$ and all $n \geq n_l$. As $(z\partial_z)^l$ is a linear combination of $z\partial_z$, $z^2\partial_z^2$, ..., $z^l \partial_z^l$, it follows by induction that $$
R_0^{-l} ||\partial_z^l w_{n,j}||_{A(|y| \leq R_0^{-1})}  = ||z^l\partial_z^{l}w_{n,j}||_{A(|y| \leq R_0^{-1})}  \leq m^{-3n}\tilde{c}_l^n
$$
for some $0 < \tilde{c}_l < 1$ and all $n \geq \tilde{n}_l$. Working with $W_{n,j,k} = (z\partial_z)^k w_{n,j}$, one similarly deduces
$$
||\partial_z^l W_{n,j,k}||_{A(|y| \leq R_0^{-1})} = ||\partial_z^l (z\partial_z)^k w_{n,j}||_{A(|y| \leq R_0^{-1})} \leq R_0^{l}m^{-3n}\tilde{c}_{l,k}^n
$$ 
for some $0 < \tilde{c}_{l,k} < 1$ and for all $n \geq \tilde{n}_{l,k}$. Then $(z\partial_z)^{l}\tilde{w}_{n,j}$ is the Taylor polynomial of degree $3N_0 + 1$ of $(z\partial_z)^{l}w_{n,j}$ and
$$
(z\partial_z)^{l}(w_{n,j} - \tilde{w}_{n,j}) = \int_{[0,z]} \frac{\partial_y^{(3N_0+2)} (y\partial_y)^l w_{n,j}(y)}{(3N_0+2)!}(z-y)^{3N_0+2}dy = \mathcal{O}(m^{-3n}c_l^{n}z^{3N_0+2})
$$
for some $0 < c_l < 1$ and all $n$ large enough. On $\frac{m}{2}(t\lambda)^{\frac{2}{3}} \leq R = z^{-1} \leq m(t\lambda)^{\frac{2}{3}}$,
$$
(z\partial_z)^{l}(w_{n,j} - \tilde{w}_{n,j}) = \mathcal{O}\left( c_l^n (t\lambda)^{-N_0-2/3} \right)
$$
The final error in approximating the expression
    $$
\tilde{f} \left(\frac{(t\lambda)^{\frac{2}{3}}}{R} \right) \sum_{n = 0}^{+\infty} \frac{\lambda^{\frac{3}{2}}}{(t\lambda)^{\alpha+2n}} w_{n}(R) 
$$ 
is given by
    $$
\tilde{f} \left(\frac{(t\lambda)^{\frac{2}{3}}}{R} \right) \sum_{n = 0}^{+\infty} \frac{\lambda^{\frac{3}{2}}}{(t\lambda)^{\alpha+2n}} z^{i+3n} \sum_{j=0}^{J}\left( w_{n,j}(z^{-1}) - \tilde{w}_{n,j}(z^{-1}) \right) \log(z)^j.
$$ 
Given the compact support of $\tilde{f}$, the estimate on $(z\partial_z)^{l}(w_{n,j} - \tilde{w}_{n,j})$ and the admissibility of the pair $(\alpha,I)$, the desired smallness is achieved.
\end{proof}

For $V_{2k}$ and $E_{\mathrm{ori},k}$, such stability under differentiation does not hold in general, because if one applies the operator $(z\partial_z)^l$, then one gets an element of similar form but the $q(a)$ coefficient belongs to $\mathcal{Q}_{\frac{1}{2}+\frac{1}{2}\nu-l}$ instead of $\mathcal{Q}_{\frac{1}{2}+\frac{1}{2}\nu}$. However, smallness is preserved in the following sense:

\begin{proposition}[Smallness is preserved under differentiation]\label{prop:stability of smallness under differenatiation}
Let $v(R,t) \in V_{2k} \cup E_{\mathrm{tip},k}$. For any $l_1,l_2,k \in \mathbb N$, it holds that
$$\left|(t\partial_t)^{l_1}(R\partial_R)^{l_2}v \right| \lesssim \frac{\lambda^{\frac{3}{2}}}{(t\lambda)^{2+\left( \frac{2}{3}-2\varepsilon \right) \cdot (k-1)}}\left[1 + \left(1-\frac{R}{(t\lambda)} \right)^{\frac{1}{2}+\frac{1}{2}\nu-l_1-l_2-} \right].
    $$
\end{proposition}

\begin{proof}
    The proof is similar to that of Proposition \ref{prop:stability under differentiation}. The singularity comes from differentiating the $\mathcal{Q}_{\frac{1}{2}+\frac{1}{2}\nu}$ elements.
\end{proof}

\begin{proposition}[Product Rules] \label{product rules for V, E}
   Let $k,k_1,k_2 \geq 1$, $g(y)$ be any smooth function on $(0,+\infty)$ which is zero when $y \geq 2/m$ and expands as a finite sum of holomorphic functions and logarithms near $y = 0$, and let $h(y)$ be any smooth function with support in $(-\infty,2)$. The following product rules hold:
    \begin{enumerate}
        % \item $$\frac{1}{(t\lambda)^{\alpha}} V_{2k-1} \subset V_{2k+1}, \quad \frac{1}{(t\lambda)^{\alpha}} V_{2k} \subset V_{2k+2}$$ for any $\alpha \geq 2/3$.
        \item 
        $$V_{2k} \subset E_{\mathrm{tip},k}, \quad V_{2k-1} \subset E_{\mathrm{ori},k} + E_{\mathrm{tip},k}.$$ \label{prop:product rule 1}
        \item 
        $$g \left( \frac{(t\lambda)^{\frac{2}{3}}}{R}  \right) (V_{2k-1} \cup E_{\mathrm{ori},k}) \subset E_{\mathrm{ori},k}, \quad h \left(  \frac{R}{(t\lambda)^{\frac{2}{3}+\varepsilon}} \right) (V_{2k} \cup E_{\mathrm{tip},k}) \subset E_{\mathrm{ori},k}.$$ \label{prop:product rule 2}
        \item
        \begin{align*}
         \left[\frac{\lambda^{\frac{3}{2}}}{(t\lambda)^{2}} \right]^{-1} ( V_{2k_1-1} \cup E_{\mathrm{ori},k_1} ) \cdot  \left[\frac{\lambda^{\frac{3}{2}}}{(t\lambda)^{2}} \right]^{-1} E_{\mathrm{ori},k_2} &\subset \left[\frac{\lambda^{\frac{3}{2}}}{(t\lambda)^{2}} \right]^{-1} E_{\mathrm{ori},k_1+k_2},\\
            \left[\frac{\lambda^{\frac{3}{2}}}{(t\lambda)^{2}} \right]^{-1} ( V_{2k_1} \cup E_{\mathrm{tip},k_1} ) \cdot \left[\frac{\lambda^{\frac{3}{2}}}{(t\lambda)^{2}} \right]^{-1} E_{\mathrm{tip},k_2} &\subset \left[\frac{\lambda^{\frac{3}{2}}}{(t\lambda)^{2}} \right]^{-1} E_{\mathrm{tip},k_1+k_2}, \\
             \left[\frac{\lambda^{\frac{3}{2}}}{(t\lambda)^{2}} \right]^{-1} ( V_{2k_1-1} \cup E_{\mathrm{ori},k_1} ) \cdot \left[\frac{\lambda^{\frac{3}{2}}}{(t\lambda)^{2}} \right]^{-1} E_{\mathrm{tip},k_2} &\subset \left[\frac{\lambda^{\frac{3}{2}}}{(t\lambda)^{2}} \right]^{-1} E_{\mathrm{ori},k_1+k_2}.
        \end{align*} \label{prop:product rule 3}
        %         \begin{align*}
        %  E_{\mathrm{ori},k_1} \cdot  E_{\mathrm{ori},k_2} &\subset \left[\frac{\lambda^{\frac{3}{2}}}{(t\lambda)^{2}} \right] E_{\mathrm{ori},k_1+k_2}\\
        %      E_{\mathrm{tip},k_1}  \cdot E_{\mathrm{tip},k_2} &\subset \left[\frac{\lambda^{\frac{3}{2}}}{(t\lambda)^{2}} \right] E_{\mathrm{tip},k_1+k_2} \\
        %       E_{\mathrm{ori},k_1}  \cdot  E_{\mathrm{tip},k_2} &\subset \left[\frac{\lambda^{\frac{3}{2}}}{(t\lambda)^{2}} \right] E_{\mathrm{ori},k_1+k_2} 
        % \end{align*}
    \end{enumerate}
\end{proposition}

\begin{remark}
    From (\ref{prop:product rule 1}) and (\ref{prop:product rule 3}), one also deduces the inclusion
    $$
    \left[\frac{\lambda^{\frac{3}{2}}}{(t\lambda)^{2}} \right]^{-1}  V_{2k_1-1}  \cdot  \left[\frac{\lambda^{\frac{3}{2}}}{(t\lambda)^{2}} \right]^{-1} V_{2k_2} \subset \left[\frac{\lambda^{\frac{3}{2}}}{(t\lambda)^{2}} \right]^{-1} (E_{\mathrm{ori},k_1+k_2} +  E_{\mathrm{tip},k_1+k_2})
    $$
    and the same conclusion holds for the other pairs $(V_{2k_1-1}, V_{2k_2-1})$ or $(V_{2k_1}, V_{2k_2})$.
\end{remark}

\begin{proof}
\item[\ref{prop:product rule 1}.] 
    The inclusion $V_{2k} \subset E_{\mathrm{tip},k}$ is straightforward. If $v \in V_{2k-1}$ has parts $v_{\mathrm{ori}}, v_{\mathrm{mid}}, v_{\mathrm{tip}}$ where
     $$v_{\mathrm{tip}}(R,t) = \fsum_{\substack{(\alpha,i) \kadm \\ j_1,j_2 \geq 0}} \frac{\lambda^{\frac{3}{2}}}{(t\lambda)^{\alpha}}R^i \log(R)^{j_1} \log(t\lambda)^{j_2} h_{\alpha,i,j_1,j_2} \left( \frac{R}{(t\lambda)^{\frac{2}{3}+\varepsilon}} \right),$$
     the other inclusion $V_{2k-1} \subset E_{\mathrm{ori},k} + E_{\mathrm{tip},k}$ is obtained by writing
     \begin{align*}
         e_{\mathrm{ori}} = v_{\mathrm{ori}} &+ v_{\mathrm{mid}}   \\
         &+ \chi_{[1,+\infty)}(x)  \fsum_{\substack{(\alpha,i) \kadm \\ j_1,j_2 \geq 0}} \frac{\lambda^{\frac{3}{2}}}{(t\lambda)^{\alpha}}R^i \log(R)^{j_1} \log(t\lambda)^{j_2} \left[ h_{\alpha,i,j_1,j_2} (x) - h_{\alpha,i,j_1,j_2}(2) \right]
     \end{align*}
     for the variable
     $$
x = \frac{R}{(t\lambda)^{\frac{2}{3}+\varepsilon}}
     $$
     and then $e_{\mathrm{tip}} = v - e_{\mathrm{ori}}$. This proves the first assertion.

   \item[\ref{prop:product rule 2}.] To prove the inclusion $$h \left(  \frac{R}{(t\lambda)^{\frac{2}{3}+\varepsilon}} \right) (V_{2k} \cup E_{\mathrm{tip},k}) \subset E_{\mathrm{ori},k},$$
   consider an element
   $$
   v(R,t) = 
   \frac{\lambda^{\frac{3}{2}}}{(t\lambda)^{\alpha}}R^i \log(R)^{j_1} \log(t\lambda)^{j_2} q \left( \frac{R}{(t\lambda)} \right)  \chi_{[1,+\infty)}\left(  \frac{R}{(t\lambda)^{\frac{2}{3}+\varepsilon}} \right)^{1+j_3},
   $$
where $(\alpha,i)$ is $k$-admissible on $C_{\mathrm{tip}}$, $j_l \in \mathbb N_{\geq 0}$ and $\tilde{h}(y)$ is any smooth function which is constant outside $[1,2]$ and zero when $y$ is in a neighbourhood of $y = 1$. The proof is identical for an error term. The product $\tilde{h}(y) = h \cdot  \chi_{[1,+\infty)}^{1+j_3}$ is compactly supported in $(1,2)$. Hence, $R \sim (t\lambda)^{\frac{2}{3}+\varepsilon}$ and $R(t\lambda)^{-1} \sim (t\lambda)^{-\frac{1}{3}+\varepsilon} \sim 0$. One can expand
$$
q(a) = q_0(a) + \sum_{j=1}^{L}q_{j}(a) \log(a)^j
$$
for some functions $q_j$ holomorphic around $a = 0$. Consider the $M$-th order Taylor polynomial $P_j(a)$ of $q_j(a)$ near $a = 0$ with $M = \ceil{N_0(\frac{1}{3}-\varepsilon)^{-1}}$, meaning that the remainder $\eta_j(a) = q_j(a) - P_j(a)$ is a holomorphic function around zero with $(a\partial_a)^l [q_j(a) - P_j(a)] = \mathcal{O}(a^{M})$ for any fixed $l \geq 0$. Therefore,
$$
h \cdot v(R,t) = 
   \frac{\lambda^{\frac{3}{2}}}{(t\lambda)^{\alpha}}R^i \log(R)^{j_1} \log(t\lambda)^{j_2} \left(P_0(a) + \sum_{j=1}^{L}P_{j}(a) \log(a)^j \right) \tilde{h} \left(  \frac{R}{(t\lambda)^{\frac{2}{3}+\varepsilon}} \right)+ \eta 
$$
for $a = R(t\lambda)^{-1}$ and $\eta \in \mathcal{E}_{N_0,\nu}$ element (by a reasoning analogous to the proof for (\ref{eq:remainder on mid/tip is negligible})).

For the other inclusion
   $$g \left(  \frac{(t\lambda)^{\frac{2}{3}}}{R}  \right) (V_{2k-1} \cup E_{\mathrm{ori},k}) \subset E_{\mathrm{ori},k},$$
    the only issue is to show that, for the origin component $v_{\mathrm{ori}}$ of an element in $V_{2k-1}$ (resp. $E_{\mathrm{ori},k}$), then
    $$g \left( \frac{(t\lambda)^{\frac{2}{3}}}{R} \right) v_{\mathrm{ori}}$$
    can be approximated by a finite sum with the desired smallness on $C_{\mathrm{mid}}$. This follows from the proof for (\ref{eq:approximation of restricted ori term}).
   \item[\ref{prop:product rule 3}.]  For the first identity, write
    $$v_1 = v_{1,\mathrm{ori}} + v_{1,\mathrm{mid}} + v_{1,\mathrm{tip}} + \eta_1 \in V_{2k_1-1},$$
    and similarly with $e_2 \in E_{\mathrm{ori},k_2}$. Define the decomposition:
    \begin{align*}
        e_{\mathrm{ori}} &= v_{1,\mathrm{ori}} \cdot e_{2,\mathrm{ori}}, \\
        e_{\mathrm{mid}} &= v_{1,\mathrm{mid}} \cdot e_{2,\mathrm{mid}} + \tilde{v}_{1,\mathrm{ori}} \cdot e_{2,\mathrm{mid}} + v_{1,\mathrm{mid}} \cdot \tilde{e}_{2,\mathrm{ori}}, \\
        e_{\mathrm{tip}} &= v_{1,\mathrm{tip}} \cdot e_{2,\mathrm{tip}} + v_{1,\mathrm{mid}} \cdot e_{2,\mathrm{tip}} + v_{1,\mathrm{tip}} \cdot e_{2,\mathrm{mid}}, \\
        \eta &= v_1 \cdot \eta_2 + \eta_1 \cdot e_2 + (v_{1,\mathrm{ori}} - \tilde{v}_{1,\mathrm{ori}}) \cdot e_{2,\mathrm{mid}} + v_{1,\mathrm{mid}} \cdot (e_{2,\mathrm{ori}} - \tilde{e}_{2,\mathrm{ori}}).  
    \end{align*}
     Then, we claim that:
     \begin{align*}
         \left[\frac{\lambda^{\frac{3}{2}}}{(t\lambda)^{2}} \right]^{-1} ( e_{\mathrm{ori}} + e_{\mathrm{mid}} + e_{\mathrm{tip}} + \eta)  \in E_{\mathrm{ori},k_1+k_2}. 
     \end{align*}
     We only treat the mixed term $v_{1,\mathrm{ori}} \cdot e_{2,\mathrm{mid}}$. The other terms are handled similarly and do not introduce additional difficulties because for each component of $v$ and $e$ on $C_{\mathrm{ori}}, C_{\mathrm{mid}}, C_{\mathrm{tip}}$, there is a natural product structure coming from the definition. 
     
     The mixed term $v_{1,\mathrm{ori}} \cdot e_{2,\mathrm{mid}}$ is supported on $\frac{m}{2}(t\lambda)^{\frac{2}{3}} \leq R \leq m(t\lambda)^{\frac{2}{3}}$. If $g(y)$ is a smooth function coming from $e_{2,\mathrm{mid}}$, then
     $$v_{1,\mathrm{ori}} \cdot g\left( \frac{(t\lambda)^{\frac{2}{3}}}{R} \right)$$
     can be approximated by a finite sum as described in (\ref{eq:approximation of restricted ori term}). Then, it is only a matter of multiplying a finite number of $k_1$-admissible terms on $C_{\mathrm{ori}}$ (hence $k_1$-admissible on $C_{\mathrm{mid}}$) coming from $v_{1,\mathrm{ori}}$ with a finite number of $k_2$-admissible terms on $C_{\mathrm{mid}}$ coming from $e_{2,\mathrm{mid}}$.

     The second identity follows by multiplying both finite sums and using the fact that $\mathcal{Q}_{\frac{1}{2}+\frac{1}{2}\nu}$ is an algebra. The third identity follows by multiplying both finite sums and approximating the $q(a) \in \mathcal{Q}_{\frac{1}{2}+\frac{1}{2}\nu}$ functions by finite sums as in (\ref{prop:product rule 2}).
\end{proof}

\begin{corollary}[Nonlinear rules] \label{nonlinear rules for V,E}
    Let $u_0 = \lambda^{\frac{3}{2}}W(R)$,
    \begin{align*}
        w_1 &= \frac{\lambda^{\frac{3}{2}}}{(t\lambda)^2} C_1(\nu)  + V_3\\
        w_2 &= \frac{\lambda^{\frac{3}{2}}}{(t\lambda)^2} H(a^2) \cdot \chi_{[1,+\infty)} \left( \frac{R}{(t\lambda)^{\frac{2}{3}+\varepsilon}} \right) + V_4, \quad a = R/(t\lambda),
    \end{align*}
    where $H(z)$ is defined in (\ref{H(z) definition}), and $w_{k} \in V_{2k-1} \cup V_{2k}$, $k \geq 2$. Then, it holds that
    $$
    t^2[F(u_0 + w_1 + w_2 + w_k) - F(u_0 + w_1 + w_2)]  \cdot (1-\chi_{[1/m,+\infty)}) \left( \frac{(t\lambda)^{\frac{2}{3}}}{R} \right) \in E_{\mathrm{ori},k+1} + E_{\mathrm{tip},k+1},
    $$
    as well as
\begin{align*}
    t^2F(u_0 + w_1 + w_2) \cdot (1-\chi_{[1/m,+\infty)}) &\in E_{\mathrm{ori},2} + E_{\mathrm{tip},2}, \\ 
    t^2F(u_0 + w_1) \cdot (1-\chi_{[1/m,+\infty)}) &\in E_{\mathrm{ori},2} + E_{\mathrm{tip},2}.
\end{align*}
\end{corollary}

\begin{proof}
    We distinguish the two regions $\frac{m}{2}(t\lambda)^{\frac{2}{3}} \leq R \leq 2(t\lambda)^{\frac{2}{3}+\varepsilon}$ and $R \gtrsim 2(t\lambda)^{\frac{2}{3}+\varepsilon}$ using a cutoff $\chi_{[1,+\infty)}$. On the first region, the nonlinearity contributes to the middle part of an $E_{\mathrm{ori},k+1}$ element, and on the second region, we get the tip part of an $E_{\mathrm{ori},k+1}$ element plus some $E_{\mathrm{tip},k+1}$ element. More precisely, we proceed as follows:
     \item[\textbf{1.} $\frac{m}{2}(t\lambda)^{\frac{2}{3}} \leq R \leq 2(t\lambda)^{\frac{2}{3}+\varepsilon}$:] We perform a multinomial expansion around the dominant component 
        $$ \frac{\lambda^{\frac{3}{2}}}{(t\lambda)^2} g\left( \frac{(t\lambda)^{\frac{2}{3}}}{R} \right), \quad g(y) = y^3  (15)^{\frac{3}{2}} + C_1(\nu),$$
        of $u_0 + w_1 + w_2 + w_k$ and $u_0 + w_1 + w_2$. Define the cutoff:
        \begin{align*}
         \chi(R,t) &:=(1-\chi_{[1,+\infty)})\left( \frac{R}{(t\lambda)^{\frac{2}{3}+\varepsilon}} \right) \cdot (1-\chi_{[1/m,+\infty)} ) \left( \frac{(t\lambda)^{\frac{2}{3}}}{R} \right),\\
         \chi(R,t) &= \left[ \chi(R,t)^{\frac{1}{n}} \right]^{n}, \quad n \geq 0.
       \end{align*}
       As $\chi_{[a,+\infty)}(R,t)$ was chosen so that $\chi > 0$ on $\mathrm{int}(\mathrm{supp}(\chi_{[a,+\infty)}))$, the $n$-th root of $\chi$ remains a product of two smooth transition functions.  Observe that:
        \begin{align}
            \left| \left( u_0 + w_1 + w_2 - \frac{\lambda^{\frac{3}{2}}}{(t\lambda)^2} g \right) \right| &\lesssim \left( \frac{\lambda^{\frac{3}{2}}}{R^4} + \frac{\lambda^{\frac{3}{2}}}{(t\lambda)^{2 + \left( \frac{2}{3} - 2\varepsilon \right)}} + a^2 \frac{\lambda^{\frac{3}{2}}}{(t\lambda)^2} \right) \notag \\
            &\lesssim \frac{\lambda^{\frac{3}{2}}}{(t\lambda)^{2 + \left( \frac{2}{3} - 2\varepsilon \right)}}, \notag \\
            |w_k | &\lesssim \frac{\lambda^{\frac{3}{2}}}{(t\lambda)^{2 + \left( \frac{2}{3} - 2\varepsilon \right) \cdot (k-1) }}, \quad \frac{m}{2}(t\lambda)^{\frac{2}{3}} \leq R \leq 2(t\lambda)^{\frac{2}{3}+\varepsilon}, \label{eq:smallness of E_1, E_2}
        \end{align}
        and the smallness is preserved under $(t \partial_t)^{l_1}(R \partial_R)^{l_2}$ as in Proposition \ref{prop:stability under differentiation} and \ref{prop:stability of smallness under differenatiation} (note that $a = 1$ is not included in this region, hence there is no singularity of type $(1-a)^{\beta}$). Define: 
        $$
        E_1:=\left(\frac{ u_0 + w_1 + w_2 - \frac{\lambda^{\frac{3}{2}}}{(t\lambda)^2} g  }{\frac{\lambda^{\frac{3}{2}}}{(t\lambda)^2} g} \right), \quad E_2:=  \left( \frac{w_{k}}{\frac{\lambda^{\frac{3}{2}}}{(t\lambda)^2} g} \right).
        $$
        Using the product rules (\ref{prop:product rule 2}) and (\ref{prop:product rule 3}) from Proposition \ref{product rules for V, E},
        $$
       E_1^i \cdot \chi^{\frac{i}{i+j}} =  \left[ \frac{\lambda^{\frac{3}{2}}}{(t\lambda)^2}  \right]^{-i} \left( u_0 + w_1 + w_2 - \frac{\lambda^{\frac{3}{2}}}{(t\lambda)^2} g \right)^i \cdot \chi^{\frac{i}{i+j}} \in \left[ \frac{\lambda^{\frac{3}{2}}}{(t\lambda)^2}  \right]^{-1} E_{\mathrm{ori},i}, \ i \geq 1,
        $$
        as well as
        $$
        E_2^j \cdot \chi^{\frac{j}{i+j}} = \left[ \frac{\lambda^{\frac{3}{2}}}{(t\lambda)^2}  \right]^{-j} w_k^j \cdot \chi^{\frac{j}{i+j}} \in \left[ \frac{\lambda^{\frac{3}{2}}}{(t\lambda)^2}  \right]^{-1} E_{\mathrm{ori},jk}, \quad j \geq 1,
         $$
         since $w_k \in V_{2k-1} \cup V_{2k}$.
       Write
       \begin{align*}
        N(R,t) &:= t^2[F(u_0 + w_1 + w_2 + w_k) - F(u_0 + w_1 + w_2)].
       \end{align*}
       We find that
       \begin{align*}
           N \cdot \chi &\simeq  \frac{\lambda^{\frac{3}{2}}}{(t\lambda)^\frac{8}{3}} \sum_{\substack{N_0 \gtrsim i \geq 0 \\ N_0 \gtrsim j \geq 1} } \binom{p}{i,j} g^p \left(\frac{ u_0 + w_1 + w_2 - \frac{\lambda^{\frac{3}{2}}}{(t\lambda)^2} g  }{\frac{\lambda^{\frac{3}{2}}}{(t\lambda)^2} g} \right)^i \chi^{\frac{i}{i+j}} \left( \frac{w_{k}}{\frac{\lambda^{\frac{3}{2}}}{(t\lambda)^2} g} \right)^j \chi^{\frac{j}{i+j}}  \\
           &\simeq   \frac{\lambda^{\frac{3}{2}}}{(t\lambda)^\frac{8}{3}}  \sum_{\substack{N_0 \gtrsim i \geq 0 \\ N_0 \gtrsim j \geq 1} } \left[ \frac{\lambda^{\frac{3}{2}}}{(t\lambda)^2}  \right]^{-1} E_{\mathrm{ori},i}  \left[ \frac{\lambda^{\frac{3}{2}}}{(t\lambda)^2}  \right]^{-1} E_{\mathrm{ori},jk}  \in E_{\mathrm{ori},k+1}
       \end{align*}
       using the product rule (\ref{prop:product rule 3}) from Proposition \ref{product rules for V, E}. The remainder term
       $$
        \eta \cdot \chi = \frac{\lambda^{\frac{3}{2}}}{(t\lambda)^\frac{8}{3}} \sum_{\substack{i \gtrsim N_0 \\ j \gtrsim N_0} } \binom{p}{i,j} g^p E_1^i E_2^j \chi
        $$
       belongs to $\mathcal{E}_{N_0,\nu}$ with no singularity since the support is restricted away from the tip of the cone. Indeed, 
      \begin{align*}
          (t\partial_t)^{l_1}(R\partial_R)^{l_2} \left[ E_1^i \right] \lesssim_{l_1,l_2} \frac{i^{l_1+l_2}}{(t\lambda)^{\left( \frac{2}{3} - 2\varepsilon \right) \cdot i}}, \quad \frac{m}{2}(t\lambda)^{\frac{2}{3}} \leq R \leq 2(t\lambda)^{\frac{2}{3}+\varepsilon} 
      \end{align*}
       using Faa di Bruno's formula, as smallness is preserved under differentiation. A similar estimate holds for $E_2^j$. Derivatives falling on $g^p$ or $\chi$ cause no loss of smallness, as was shown in (\ref{eq: derivative t R in terms of a, y}). Hence, 
       \begin{align*}
           (t\partial_t)^{l_1}(R\partial_R)^{l_2}[\eta  \cdot \chi] &\lesssim_{l_1,l_2,p} \frac{\lambda^{\frac{3}{2}}}{(t\lambda)^\frac{8}{3}} \sum_{\substack{i \gtrsim N_0 \\ j \gtrsim N_0} } \binom{p}{i,j} i^{l_1+j_2} j^{l_1+j_2}  \left( \frac{1}{(t\lambda)^{\left( \frac{2}{3} - 2\varepsilon \right) }} \right)^{(i+j)} \\
           &\lesssim_{l_1,l_2,p} \frac{\lambda^{\frac{3}{2}}}{(t\lambda)^\frac{8}{3}} \sum_{\substack{i \gtrsim N_0 \\ j \gtrsim N_0} }\binom{p}{i,j}  \left( \frac{2}{(t\lambda)^{\left( \frac{2}{3} - 2\varepsilon \right) }} \right)^{(i+j)}  \\
           &\lesssim_{l_1,l_2,p} \frac{\lambda^{\frac{3}{2}}}{(t\lambda)^{\frac{8}{3}+N_0}} \sum_{\substack{i \gtrsim N_0 \\ j \gtrsim N_0} }\binom{p}{i,j}  \left( \frac{2}{(t\lambda)^{\left( \frac{2}{3} - 2\varepsilon -\delta \right) }} \right)^{(i+j)} \\
           &\lesssim_{l_1,l_2,p} \frac{\lambda^{\frac{3}{2}}}{(t\lambda)^{\frac{8}{3}+N_0}} \left( 1 + \frac{4}{(t\lambda)^{\left( \frac{2}{3} - 2\varepsilon -\delta \right) }} \right)^{p} 
       \end{align*}
       given $\delta > 0$ small enough so that $\frac{2}{3}-2\varepsilon-\delta > 0$ and $i,j \geq \delta^{-1}N_0$. Similarly, the elements
      \begin{align*}
       t^2 F(u_0 + w_1 + w_2) \cdot \chi &\simeq  \frac{\lambda^{\frac{3}{2}}}{(t\lambda)^\frac{8}{3}} \sum_{N_0 \gtrsim i \geq 0} \binom{p}{i} g^p \left(\frac{ u_0 + w_1 + w_2 - \frac{\lambda^{\frac{3}{2}}}{(t\lambda)^2} g  }{\frac{\lambda^{\frac{3}{2}}}{(t\lambda)^2} g} \right)^i  \cdot \chi, \\
       t^2 F(u_0 + w_1) \cdot \chi &\simeq  \frac{\lambda^{\frac{3}{2}}}{(t\lambda)^\frac{8}{3}} \sum_{N_0 \gtrsim i \geq 0} \binom{p}{i} g^p \left(\frac{ u_0 + w_1 - \frac{\lambda^{\frac{3}{2}}}{(t\lambda)^2} g  }{\frac{\lambda^{\frac{3}{2}}}{(t\lambda)^2} g} \right)^i  \cdot \chi 
      \end{align*}
       are in $E_{\mathrm{ori},2}$ due to the presence of the $\dfrac{\lambda^{\frac{3}{2}}}{(t\lambda)^\frac{8}{3}}$ factor.
       
       \item[\textbf{2.} $2(t\lambda)^{\frac{2}{3}+\varepsilon} \leq R \leq (t\lambda)$:] We perform a multinomial expansion around the dominant component
        $$ \frac{\lambda^{\frac{3}{2}}}{(t\lambda)^2} q \left( \frac{R}{(t\lambda)} \right), \quad q(a) = C_1(\nu) + H(a^2)$$
        of $u_0 + w_1 + w_2 + w_k$ and $u_0 + w_1 + w_2$. Observe that
        \begin{align*}
            \left| \left( u_0 + w_1 + w_2 - \frac{\lambda^{\frac{3}{2}}}{(t\lambda)^2} q \right) \right| &\lesssim \left( \frac{\lambda^{\frac{3}{2}}}{R^3} + \frac{\lambda^{\frac{3}{2}}}{(t\lambda)^{2 + \left( \frac{2}{3} - 2\varepsilon \right)}} + (1-\chi_{[1,+\infty)}) a^2 \frac{\lambda^{\frac{3}{2}}}{(t\lambda)^2} \right) \\
            &\lesssim \frac{\lambda^{\frac{3}{2}}}{(t\lambda)^{2 + \left( \frac{2}{3} - 2\varepsilon \right)}}, \quad 2(t\lambda)^{\frac{2}{3}+\varepsilon} \leq R \leq (t\lambda)
        \end{align*}
        and the smallness is preserved under $(t \partial_t)^{l_1}(R \partial_R)^{l_2}$ as in Proposition \ref{prop:stability under differentiation} and \ref{prop:stability of smallness under differenatiation} (up to increasing the singularity at $a = 1$ as in Proposition \ref{prop:stability of smallness under differenatiation}). We remark that
      \begin{align*}
        \left( w_2 - \dfrac{\lambda^{\frac{3}{2}}}{(t\lambda)^2} q \right) \chi_{[1,+\infty)}^{\frac{1}{i+j}} &= \frac{\lambda^{\frac{3}{2}}}{(t\lambda)^2} H(a^2) \cdot (1- \chi_{[1,+\infty)}) \left( \frac{R}{(t\lambda)^{\frac{2}{3}+\varepsilon}} \right)  \chi_{[1,+\infty)} \left( \frac{R}{(t\lambda)^{\frac{2}{3}+\varepsilon}} \right)^{\frac{1}{i+j}}  \\
        &\phantom{=}+ V_4 \in E_{\mathrm{ori},1} +  E_{\mathrm{tip},1}. 
      \end{align*}
      Similarly, it holds that
      $$
(u_0+w_1) \chi_{[1,+\infty)}^{\frac{1}{i+j}} \in V_1 \subset E_{\mathrm{ori},1} +  E_{\mathrm{tip},1}.
      $$
        Then, applying the product rules from Proposition \ref{product rules for V, E}:
        $$
            \left[ \frac{\lambda^{\frac{3}{2}}}{(t\lambda)^2}  \right]^{-i} \left( u_0 + w_1 + w_2 - \frac{\lambda^{\frac{3}{2}}}{(t\lambda)^2} q \right)^i \cdot \chi_{[1,+\infty)}^{\frac{i}{i+j}} \in \left[ \frac{\lambda^{\frac{3}{2}}}{(t\lambda)^2}  \right]^{-1} (E_{\mathrm{ori},i} + E_{\mathrm{tip},i}), \quad i \geq 1.
        $$
 We conclude as in the first part that
        $$  t^2[F(u_0 + w_1 + w_2 + w_k) - F(u_0 + w_1 + w_2)] \cdot \chi_{[1,+\infty)} \in E_{\mathrm{ori},k+1} + E_{\mathrm{tip},k+1}.$$
        Similarly, it holds that
      \begin{align*}
       t^2 F(u_0 + w_1 + w_2) \cdot \chi_{[1,+\infty)} &\simeq  \frac{\lambda^{\frac{3}{2}}}{(t\lambda)^\frac{8}{3}} \sum_{N_0 \gtrsim i \geq 0} \binom{p}{i} t^2q^p \left(\frac{ u_0 + w_1 + w_2 - \frac{\lambda^{\frac{3}{2}}}{(t\lambda)^2} q  }{\frac{\lambda^{\frac{3}{2}}}{(t\lambda)^2} q} \right)^i  \cdot \chi_{[1,+\infty)} \\
       &\in E_{\mathrm{ori},2} + E_{\mathrm{tip},2}, 
      \end{align*}
      as well as
      \begin{align*}
             t^2 F(u_0 + w_1) \cdot \chi_{[1,+\infty)}  &\simeq  \frac{\lambda^{\frac{3}{2}}}{(t\lambda)^\frac{8}{3}} \sum_{N_0 \gtrsim i \geq 0} \binom{p}{i} t^2C_1(\nu)^p \left(\frac{ u_0 + w_1 - \frac{\lambda^{\frac{3}{2}}}{(t\lambda)^2} C_1(\nu)  }{\frac{\lambda^{\frac{3}{2}}}{(t\lambda)^2} C_1(\nu)} \right)^i  \cdot \chi_{[1,+\infty)} \\
       &\in E_{\mathrm{tip},2}.
       \end{align*}
\end{proof}

In the next section, we prove the following theorem:

\begin{theorem}[Construction of an approximate solution] \label{Thm:Approximate Solution}
        Assume $d = 5$. The successive errors and correction terms satisfy the following properties when $k \geq 1$:
    \begin{enumerate}
    \item  $t^2e_{2k-1} = t^2e_{2k-1}^0 + t^2e_{2k-1}^1\in E_{\mathrm{tip},k} + E_{\mathrm{ori},k}$.
    
    \item $v_{2k} \in V_{2k}$.
    
    Moreover, the function $v_2$ is non-negative everywhere on $0 \leq R  \leq (t\lambda)$, $0 < t \leq t_0 \ll 1$ and takes the form:
    \begin{align*}
        v_2(R,t) &= \frac{\lambda^{\frac{3}{2}}}{(t\lambda)^{2}}H(a^2)\chi_{[1,+\infty)}\left(  \frac{R}{(t\lambda)^{\frac{2}{3}+\varepsilon}} \right) \\
        &\phantom{=}+ \fsum_{\substack{(\alpha,i) \mathrm{2-adm} \\ j \geq 0}}\frac{\lambda^{\frac{3}{2}}}{(t\lambda)^{\alpha}}R^i \log(R)^{j} Q_{\alpha,i,j} \left(a \right)  \chi_{[1,+\infty)}\left(  \frac{R}{(t\lambda)^{\frac{2}{3}+\varepsilon}} \right),
    \end{align*}
    where $a = R/(t\lambda)$, $H(z)$ is defined as in (\ref{H(z) definition}) and is a positive function on $(0,1)$.
        
    \item $t^2e_{2k} = t^2e_{2k}^0 + t^2e_{2k}^1 \in E_{\mathrm{ori},k} + E_{\mathrm{tip},k+1}$.
    
    \item $v_{2k+1} \in V_{2k+1}$.
    \end{enumerate}
    
    In particular, the approximate solution
    $$u_{k} = u_0 + v_1 + v_2 + \sum_{i=3}^k v_i, \quad k \geq 3$$
    is positive everywhere on $0 \leq R < (t\lambda)$, $0 < t \leq t_0$, and has asymptotics
    \begin{align*}
    | (\langle R \rangle^i \partial_R^i)(t^j\partial_t^j) u_k| &\lesssim \begin{dcases}
        \lambda^{\frac{3}{2}} \quad &0 \leq R \lesssim 1 \\
        \frac{\lambda^{\frac{3}{2}}}{1+R^3} \quad &1 \lesssim R \lesssim (t\lambda)^{\frac{2}{3}} \\
        \frac{\lambda^{\frac{3}{2}}}{(t\lambda)^2}\left[1 + \left(1-\frac{R}{(t\lambda)} \right)^{\frac{1}{2}+\frac{1}{2}\nu-i-j-} \right] \quad &(t\lambda)^{\frac{2}{3}} \lesssim R < (t\lambda)
    \end{dcases} \\
    |(R^i \partial_R^i) (t^j\partial_t^j)(u_k-u_0)| &\lesssim \begin{dcases}
        \frac{\lambda^{\frac{3}{2}}}{(t\lambda)^2} R^{i+\max\{2-i,0\}} \quad &0 \leq R \lesssim 1 \\
        \frac{\lambda^{\frac{3}{2}}}{(t\lambda)^2}\left[1 + \left(1-\frac{R}{(t\lambda)} \right)^{\frac{1}{2}+\frac{1}{2}\nu-i-j-} \right] \quad &1 \lesssim R < (t\lambda)
    \end{dcases}
    \end{align*}
    where $\lesssim$ can be replaced by $\asymp$ when $i = j = 0$.
\end{theorem}

The analogous theorem in dimension 4 is stated in Theorem \ref{Thm:Approximate Solution d = 4}.

\begin{remark}
    we observe that 
     $$
     \left(  \int_{|x| < t} (u_k-u_0)^2 dx \right)^{\frac{1}{2}} \lesssim \frac{\lambda^{\frac{3}{2}}}{(t\lambda)^2} \left( \int_{0}^t r^4 dr \right)^{\frac{1}{2}} \lesssim t^{\frac{1}{2}\nu + 1}
     $$
    with a similar estimate for the $R$ derivatives and $t$ derivatives (we lose one power of $t$ when differentiating with respect to $t$). 

    Extending $u^e = u_k-u_0$ to all of $\mathbb R^d \times [0,t_0]$ as a function of the same size and regularity, supported on $0 < |x| < 2t$ (as described in Remark \ref{extension of v_k, e_k}), one obtains the energy decay for $u^e$ claimed in Theorem \ref{thm:blow-up, main thm}.
\end{remark}

\section{Renormalization Step: Next Iterates} \label{section:renormalization step, next iterates}

We now perform the main inductive argument of the renormalization procedure in dimension $d = 5$, explaining how to construct the even correction terms $v_{2k}$ from the error $e_{2k-1}$ by solving a wave-like equation in self-similar coordinates, and the odd correction terms $v_{2k+1}$ from the error $e_{2k}$ using an elliptic-like equation. We prove that at each step, there is a systematic decrease in the error, thereby completing the proof of Theorem \ref{Thm:Approximate Solution}.

\subsection{Proof of Theorem \ref{Thm:Approximate Solution}}
The proof is done by induction. Assuming the claimed decomposition of $t^2e_{2k-1}$ (which will be proven to be true for $k = 1$), we show that $v_{2k}$, $t^2e_{2k}$, $v_{2k+1}$, $t^2e_{2k+1}$ all have the desired form.

\subsubsection{Construction of $e_1$ from $v_1$}

The error $t^2e_1(R,t)$ is given exactly by 
\begin{align*}
    t^2e_1(R,t) &= \underbrace{t^2[F(u_0 + v_1) - F(u_0) - F'(u_0)v_1]}_{=: N(e_1)} - t^2\partial_{tt} (v_1(r\lambda,t)), 
\end{align*}
where $v_1(R,t) = \lambda^{\frac{3}{2}}(t\lambda)^{-2} V_1(R)$, $V_1(R)$ is as in (\ref{v_1 formula}), $R = rt^{-1-\nu}$, and
\begin{align*}
    t^2\partial_{tt} (v_1(r\lambda,t)) &= (t^2\partial_{tt}v_1)(R,t) + \left(-\frac{3}{2}-\frac{3}{2}\nu \right)\left(-\frac{3}{2}-\frac{3}{2}\nu -1 \right) (R\partial_Rv_1)(R,t) \\
    &\phantom{=}+ 2\left(-\frac{3}{2}-\frac{3}{2}\nu \right) (R\partial_R t\partial_t v_1)(R,t) + \left(-\frac{3}{2}-\frac{3}{2}\nu \right)^2 (R^2\partial_{RR} v_1)(R,t) \\
    &\in \frac{ \lambda^{\frac{3}{2}} }{ (t \lambda)^2 } S^2(R^0,\log(R)) \\
    t^2F(u_0) &\in \frac{ \lambda^{\frac{3}{2}} }{ (t \lambda)^{-2} } S^0(R^{-7}), \quad t^2F'(u_0)v_1 \in \lambda^{\frac{3}{2}}  S^2(R^{-4},\log(R)).
\end{align*}
Moreover, we note that $t^2\partial_{tt} (v_1(r\lambda,t))$ has a constant dominant term 
$$
\frac{1}{4}(\nu-3)(\nu-5) C_1(\nu)R^0 =: C_2(\nu)R^0
$$
at $R \to +\infty$ and that $t^2\partial_{tt} (v_1(r\lambda,t))$ is the dominant component of the error near the tip of the cone. Using Corollary \ref{nonlinear rules for V,E}, we conclude that
$$
N(e_1) \cdot (1-\chi_{[1/m,+\infty)}) \left(\frac{(t\lambda)^{\frac{2}{3}}}{R} \right) \in E_{\mathrm{ori},2} + E_{\mathrm{tip},2},
$$
and one checks separately that
\begin{align*}
 t^2\partial_{tt} (v_1(r\lambda,t)) \cdot (1-\chi_{[1/m,+\infty)}) &\in E_{\mathrm{ori},1} + E_{\mathrm{tip},1} \\
 t^2F'(u_0)v_1 \cdot (1-\chi_{[1/m,+\infty)}) &\in E_{\mathrm{ori},2} + E_{\mathrm{tip},2} \\
  t^2F(u_0) \cdot (1-\chi_{[1/m,+\infty)}) &\in E_{\mathrm{ori},2} + E_{\mathrm{tip},2}
\end{align*}
by applying Proposition \ref{inclusion S^2n into V_2k-1} and the inclusion $V_{2k-1} \subset E_{\mathrm{ori},k} + E_{\mathrm{tip},k}$. It remains to analyze the term:
$$
\left( N(e_1) - t^2 \partial_{tt}v_1 \right) \cdot \chi_{[1/m,+\infty)} \left(\frac{(t\lambda)^{\frac{2}{3}}}{R} \right),
$$
which contributes to the origin part of an $E_{\mathrm{ori},1}$ element. To see this, we perform a binomial expansion around $u_0$:
\begin{align*}
   N(e_1) \cdot \chi_{[1/m,+\infty)} \simeq  \chi_{[1/m,+\infty)}  \cdot \sum_{n=2}^{\infty} \binom{p}{n} t^2 u_0^p T \left[ \left( \frac{v_1}{u_0} \right)^{n} \right],
\end{align*}
where we recall that $T$ is the “truncation” operator defined in equation (\ref{truncation}), and
\begin{align*}
     \sum_{n=2}^{\infty} \binom{p}{n} t^2 u_0^p T \left[ \left( \frac{v_1}{u_0} \right)^{n} \right] - t^2 \partial_{tt}v_1  &\in  \sum_{n=2}^{\infty}  \dfrac{\lambda^{\frac{3}{2}}}{(t\lambda)^{2(n-1)}} S^{2n}(R^{-7+3n},\log(R)^{3N_0}) \\
    &\phantom{\in}-  \dfrac{ \lambda^{\frac{3}{2}} }{ (t \lambda)^2 } S^2(R^0,\log(R)).
\end{align*}
Hence, the sum has an appropriate form for a $\tilde{e}_{\mathrm{ori}}$ component.

\subsubsection{Construction of $v_{2k}$ from $t^2e_{2k-1}$}

For each term coming from the finite sum of $t^2e_{2k-1}^0$ on $R \geq 2(t\lambda)^{\frac{2}{3}+\varepsilon}$, i.e. each term of the form:
\begin{equation} \label{t^2 tilde(e)_(2k-1)^0}
    t^2\tilde{e}_{2k-1}^0(R,a,t):=  \frac{\lambda^{\frac{3}{2}}}{(t\lambda)^{\alpha}}q_{\alpha,i,j_1,j_2}(a)R^i \log(R)^{j_1}, \quad R > 0, a \in (0,1), 0 < t \leq t_0,
\end{equation}
where $(\alpha,i)$ is $k$-admissible, $j_1, j_2 \geq 0$ and $q_{\alpha,i,j_1,j_2}(a) \in \mathcal{Q}_{\frac{1}{2}+\frac{1}{2}\nu}$, we solve
$$t^2(-\partial_t^2 + \partial_r^2 + \frac{4}{r}\partial_r) \tilde{v}_{2k} = -t^2\tilde{e}_{2k-1}^0$$
and then apply back the omitted cutoff 
$$\chi_{[1,+\infty)}\left( \frac{R}{(t\lambda)^{\frac{2}{3}+\varepsilon}}\right)^{1+j_2}$$
to the solution $\tilde{v}_{2k}$. Summing all these solutions, we obtain the correction $v_{2k}$. As shown below in Theorem \ref{hypergeometric ode with error forcing term}, the solution takes the form:
$$
v_{2k} = \fsum_{\substack{(\alpha,i) \kadm \\ j_1,j_2 \geq 0}} \frac{\lambda^{\frac{3}{2}}}{(t\lambda)^{\alpha}} R^i \left( \sum_{0 \leq l \leq j_1} Q_{\alpha,i,j_1,j_2,l}(a) \log(R)^{l} \right) \chi_{[1,+\infty)}\left( \frac{R}{(t\lambda)^{\frac{2}{3}+\varepsilon}}\right)^{j_2},$$
where $Q_{\alpha,i,j_1,j_2,l}(a) \in a^2\mathcal{Q}_{\frac{1}{2}+\frac{1}{2}\nu}$. The correction term $v_{2k}$ has comparable size to $v_{2k-1}$ near the tip of the cone $a \sim 1$. However, we will obtain a smaller error near the tip of the cone. Using an appropriate change of variables, we reduce the problem to solving an hypergeometric equation, which we first study in the following lemma.

% Hence, at $z = 0$, our hypergeometric equation has fundamental system
% \begin{align*}
% f_{1}(z)&= F_1(z) = F({\alpha,\beta;\gamma};z) \\
% f_{2}(z)&= z^{1-\gamma}F_2(z) + c_1 \log(z) f_1(z)
% \end{align*}
% for some $c_1 \in \mathbb R$ and $F_i(z)$ analytic on $|z| < 1$, $F_i(0) = 1$. 

% Similarly, at $z = 1$, we have a fundamental system
% \begin{align*}
% g_{1}(z)&=(1-z)^{\gamma-\alpha-\beta}G_1(1-z) \\
% g_{2}(z)&=G_2(1-z) + c_2 \log(1-z) g_1(z)
% \end{align*}
% for some $c_2 \in \mathbb R$ and $G_i(z)$ analytic on $|z| < 1$, $G_i(0) = 1$.

\begin{lemma}[Hypergeometric ODE with $\tilde{\mathcal{Q}}$ forcing term] \label{Q inhomogeneous hypergeom}
   Write $z = a \in \mathbb C$. Let $\alpha, \beta, \gamma \in \mathbb R$ such that $\gamma \notin \mathbb Z_{\leq 0}$ and $\gamma - \alpha - \beta > 0$. Let $b,r \in \mathbb R$ and $q(z) \in z^r\tilde{\mathcal{Q}}_{b}$. Then the inhomogeneous hypergeometric equation
    \begin{equation}
    z(1-z)w''(z) + (\gamma - (\alpha + \beta + 1)z)w'(z) - \alpha \beta w(z) = q(z), \quad 0 < z < 1 \label{eq:hypergoemetric equation with Q tilde forcing term}
\end{equation} 
has a particular solution $w(z) \in z^{r+1}\tilde{\mathcal{Q}}_{\min\{b+1, \gamma-\alpha-\beta\}}$. 

Moreover, if $-r-1 \notin \mathbb N_{\geq 0}$ and $-\gamma - r \notin \mathbb N_{\geq 0}$ (e.g. if $r = 0$) and the worst logarithmic singularity of $q(z)$ near $z = 0$ is bounded by $\log(z)^J$, $J \in \mathbb N_{\geq 0}$, then so is the worst logarithmic singularity of the solution $w(z)$.
\end{lemma}

\begin{proof}
Around $|z| \leq a_0$, we expand:
$$q(z) = z^r \left( \sum_{j=0}^{L}q_{j}(z) \log(z)^j \right)$$
Equation (\ref{eq:hypergoemetric equation with Q tilde forcing term}) near zero becomes
$$
w''(z) + \frac{(\gamma - (\alpha + \beta + 1)z)}{z(1-z)}w'(z) - \frac{\alpha \beta}{z(1-z)} w(z) = z^{-1}(1-z)^{-1}q(z).
$$
Hence, we must solve a finite number of hypergeometric equations of the form:
$$
w''(z) + \frac{(\gamma - (\alpha + \beta + 1)z)}{z(1-z)}w'(z) - \frac{\alpha \beta}{z(1-z)} w(z) = z^{(r+1)-2} \tilde{q}_j(z) \log(z)^j,
$$
where $\tilde{q}_j(z) = (1-z)^{-1}q_j(z)$ is holomorphic around $|z| \leq a_0$. The indicial roots at zero are $\{0,1-\gamma\}$. Using Theorem \ref{thm:inhomogeneous fuchs ode}, there exists a particular solution of the form 
$$w(z) = z^{r+1}\left( \sum_{j=0}^{L} \sum_{k=0}^{j+2} Q_{j,k}(z) \log(z)^k \right)$$
If $0 - (r+1) = -r-1 \notin \mathbb N_{\geq 0}$ and $(1-\gamma) - (r+1) = -\gamma-r \notin \mathbb N_{\geq 0}$, Theorem \ref{thm:inhomogeneous fuchs ode} also ensures that the sum over the indices $k$ only goes from $k = 0$ to $j$, i.e., $Q_{j,k} = 0$ for $k \in \{j+1,j+2\}$ and the logarithmic singularity cannot increase.

Since $q(z)$ is holomorphic on $(0,1) \subset U \subset B(0,1)$, we can use regular ODE theory to get a holomorphic extension of $w(z)$ on $U$ solving the equation. On $|z-1| < a_0$, $z^r$ is analytic, so we can write
$$q(z) = q_0(1-z) + \sum_{i=1}^{+\infty}(1-z)^{\beta(i)} \sum_{j=0}^{iL}q_{i,j}(1-z)\log(1-z)^j,$$
where either the sum is finite or $\beta(i) \geq c(i-1) + b$ for $c > 0$ small enough and the growth condition
$$||q_{i,j}||_{L^{\infty}(|z-1| < z_0)} \leq C^i$$
holds for some $C > 0$. Using Theorem \ref{thm:inhomogeneous fuchs ode} once again (the indicial roots at $z = 1$ are $\{0,\gamma-\alpha-\beta\}$), a particular solution is given by
\begin{align*}
    Q_{0}(1-z) &+ \sum_{i=1}^{+\infty}(1-z)^{\beta(i)+1} \sum_{j=0}^{L_i} \sum_{k=0}^{j+1}Q_{i,j,k}(1-z)\log(1-z)^k \\
    = Q_{0}(1-z) &+ \sum_{i=1}^{+\infty}(1-z)^{\beta(i)+1} \sum_{j=0}^{L_i+1}\tilde{Q}_{i,j}(1-z) \log(1-z)^j
\end{align*}
meaning that $w(z)$ must match this particular solution modulo some linear combination of the fundamental system (see (\ref{fuchs fundamental system}) from Appendix \ref{section:appendix, ode}) of the ODE. The fundamental system introduces a $(1-z)^{\gamma-\alpha-\beta}\log(1-z)$ singularity in the solution. We note that the hypergeometric equation near $z = 1$ with analytic forcing term $q_0(1-z)$ can always yield a logarithm-free analytic solution $Q_0(1-z)$ thanks to Remark \ref{remark:analytic sol inhomogeneous fuchs equation}.

If the expansion for $q(z)$ is finite, then so is the resulting sum for the solution. If it is infinite, then one must verify that the boundedness condition is still verified by the $\tilde{Q}_{i,j}$ when $i$ is sufficiently large. This holds due to the estimates from Theorem \ref{thm:inhomogeneous fuchs ode}, as because the growth of the $\beta(i)$ exponents is at least linear in $i$, while the logarithmic exponents growth is at most linear. In other words, when $i \in \mathbb N$ is large enough, one has 
\begin{enumerate}
    \item $\beta(i) - \max\{|r_1|,|r_2|\} > (c/2)i > 1$, where $\{r_1, r_2\}$ are the indicial roots of the equation at $z = 1$.
    \item and an exponential upper bound $$\left( \frac{j}{\beta(i)-r_k} \right)^j \lesssim C(c,L)^i, \quad k \in \{1,2\},$$
because $0 \leq j \leq iL$, $\beta(i) - \max\{|r_1|,|r_2|\} \geq (c/2)i$.
\end{enumerate}
\end{proof}

\begin{theorem}[Particular solution to (\ref{even v})]\label{hypergeometric ode with error forcing term}
Let $d \geq 1$. Let $e(R,a,t) = t^{s}q(a)R^i\log(R)^k$ where $s,i \in \mathbb R$, $s - \nu i > -(d-1)/2$, $k \in \mathbb N$, $q \in a^{\delta}\mathcal{Q}_{\beta}$, $\beta, \delta \in \mathbb R$. Then one can find a solution to
$$t^2\left(-\partial_t^2 + \partial_r^2 + \frac{d-1}{r}\partial_r\right) v = e(r\lambda, r/t, t)$$
of the form $$v(R,a,t) = t^sR^i\sum_{0 \leq l \leq k} Q_l(a)\log(R)^l$$ where $Q_l \in a^{\delta+2}\mathcal{Q}_{\min\{\beta+1,s-\nu i + \frac{d-1}{2}\}}$, $R = r \lambda$, $a = r/t$. Moreover, if $\delta - i = 0$ and $q(a)$ has no logarithmic singularity at  $a= 0$, then so does $a^{-2}Q_l(a)$ for any $l$.
\end{theorem}

\begin{proof}
Writing $R^i = a^i(t\lambda)^{i}$, we can assume without loss of generality that $i = 0$ and $s > -(d-1)/2$. Plugging $v$ in the equation and matching powers of $\log(R)$, we find the following system of recursive equations for $Q_k$, $Q_{k-1}$, ..., $Q_0$:
\begin{align}
    t^2(-\partial_t^2 + \partial_r^2 + \frac{d-1}{r}\partial_r) \left[ t^{s} Q_k \right]  &= t^{s} q(a),  \notag \\
    t^2(-\partial_t^2 + \partial_r^2 + \frac{d-1}{r}\partial_r)\left[ t^{s} Q_{k-1} \right]  &= (d-2)kt^{s}Q_k(a)a^{-2} +(2s-1)(\nu+1)kt^{s}Q_k(a), \notag  \\
    t^2(-\partial_t^2 + \partial_r^2 + \frac{d-1}{r}\partial_r)\left[ t^{s} Q_{l-2} \right] &= (d-2)(l-1)t^{s}Q_{l-1}(a)a^{-2}  \notag \\
    &\phantom{=}+(2s-1)(\nu+1)(l-1)t^{s}Q_{l-1}(a)  \notag \\
    &\phantom{=}- l(l-1)t^{s}Q_{l}(a)\left[(\nu + 1)^2 - a^{-2}\right] \notag \\
    &\phantom{=}-2(l-1) t^{s} Q_{l-1}'(a)\left[ (\nu + 1) a - a^{-1} \right], \label{eq: system of equation near tip of the cone}
\end{align}
where $a = r/t$, $0 < r < t < t_0$. Hence, we must solve equations of the form:
$$t^2\left(-\partial_t^2 + \partial_r^2 + \frac{d-1}{r}\partial_r\right) \left[ t^{s} w \right]  = t^{s} f \in t^{s} a^{\delta} \mathcal{Q}_{\beta},$$
which is equivalent to
$$t^2 \left( - \left( \partial_t + \frac{s}{t} \right)^2 + \partial_r^2 + \frac{d-1}{r}\partial_r \right)w(a)  = f(a) \in a^{\delta} \mathcal{Q}_{\beta},$$
or $L_{s} w(a) = f(a)$, $0 < a < 1$, where
$$L_{s} = (1-a^2)\partial_{aa} + ((d-1)a^{-1} + 2as - 2a)\partial_a + (s - s^2).$$
% $$L_{x} = (1-a^2)\partial_{aa} + ((d-1)a^{-1} + 2a\beta - 2a)\partial_a + (\beta - \beta^2)$$
Finally, writing $f(a) = a^{\delta}F(a^2)$ and looking for a solution of the form $w(a) = W(a^2)$, we reduce to an hypergeometric equation for $W(z)$:
$$z(1-z)W''(z) + \left( \frac{d}{2} + z \left(s - \frac{3}{2} \right) \right) W'(z) + \frac{s - s^2}{4} = z^{\frac{\delta}{2}}F(z), \quad 0 < z < 1,$$
\index{z@$z$, a complex variable which replaces $R$ or $a$ depending on the context}
whose parameters are
$$
\tilde{\alpha} = -\frac{s}{2}, \tilde{\beta} = -\frac{s}{2} + \frac{1}{2}, \tilde{\gamma} = \frac{d}{2}.
$$
Since $d \geq 1$ and $s > -(d-1)/2$, we can use Lemma \ref{Q inhomogeneous hypergeom} to obtain a solution $W(z) \in a^{\frac{\delta}{2}+1}\tilde{\mathcal{Q}}_{\min\{\beta+1,s+\frac{d-1}{2}\}}$. This yields a solution $w(a) \in a^{\delta+2}\mathcal{Q}_{\min\{\beta+1,s+\frac{d-1}{2}\}}$. If $\delta = 0$ and $f(z)$ has no logarithmic singularity at $z = 0$, then the Lemma also implies that the $\mathcal{Q}$ part of $w(a)$ has no logarithmic singularity at $z = 0$ either.

 Since $q(a) \in a^{\delta}\mathcal{Q}_{\beta}$, we find $Q_k \in a^{\delta+2}\mathcal{Q}_{\min\{\beta+1,s + \frac{d-1}{2}\}}$. Solving for $Q_{k-1}$ using $Q_k$ leads to $$Q_{k-1} \in a^{\delta+2}\mathcal{Q}_{\min\{ \min\{\beta+1,s + \frac{d-1}{2}\} + 1, s+ \frac{d-1}{2}\}} = a^{\delta+2}\mathcal{Q}_{\min\{\beta+2,s + \frac{d-1}{2}\}} $$
Furthermore, observe that
$$
Q_{l-1}(a), a^{-2}Q_{l-1}(a), aQ'_{l-1}(a), a^{-1}Q'_{l-1}(a) = a^{-2} a \partial_a Q_{l-1}(a) \in a^{\delta} \mathcal{Q}_{\min\{\beta+2,s + \frac{d-1}{2}\}-1},
$$
and solving for the other $Q_{l-2}$'s leads to 
$$
Q_{l-2} \in a^{\delta+2}\mathcal{Q}_{\min\{ \min\{\beta+2,s+ \frac{d-1}{2}\}, s + \frac{d-1}{2}\}} = a^{\delta+2}\mathcal{Q}_{\min\{\beta+2,s+ \frac{d-1}{2}\}}.
$$
Finally, if $\delta = 0$ and $q(a)$ has no logarithmic singularity at $a = 0$, then so does $a^{-2}Q_k(a)$ and this property propagates to $Q_{k-1}$ and all $Q_{l-2}$ by induction.
\end{proof}

\begin{corollary}  \label{hypergeometric ode with error forcing term, dim 5}
In dimension $d = 5$, for a forcing term of the form: $$\frac{\lambda^{\frac{3}{2}}}{(t\lambda)^{\alpha}}q(a)R^i \log(R)^j,$$ 
with $\alpha \geq i + 2$ and $q(a) \in \mathcal{Q}_{\frac{1}{2}+\frac{1}{2}\nu}$, one can apply the theorem and obtain coefficients $Q_l \in a^2\mathcal{Q}_{\frac{1}{2} + \frac{1}{2}\nu}$ in the system (\ref{eq: system of equation near tip of the cone}).
\end{corollary}

\begin{proof}
    This is a direct application of Theorem \ref{eq:hypergoemetric equation with Q tilde forcing term} with 
\begin{align*}
        \beta &= \frac{1}{2}\nu + \frac{1}{2} \\
        s - \nu i &= \frac{3}{2}(-1-\nu) + \alpha \nu - \nu i = \nu \left( \alpha - i - \frac{3}{2} \right) - \frac{3}{2} \geq \frac{1}{2}\nu - \frac{3}{2} \\
        s-\nu i + \frac{(d-1)}{2} &\geq \frac{1}{2}\nu + \frac{1}{2}.
\end{align*}
\end{proof}

\begin{remark}[Loss of regularity in higher dimensions]\label{rmk:loss of regularity high dimension}
One of the main difficulties in generalizing this method of constructing blow-ups in higher dimensions is that solving the ODE with a forcing term $$\frac{\lambda^{\frac{d-2}{2}}}{(t\lambda)^2}$$
introduces a $(1-a)^{\frac{1}{2}+ \nu \left( \frac{6-d}{2} \right)}$ singularity, meaning that for $d > 6$, the obtained correction term is not even continuous at $a = 1$ unless $\nu > 0$ is small. 
\end{remark}

When $k = 1$, the dominant component of $-t^2\tilde{e}_1^0$ on the interval $2(t\lambda)^{\frac{2}{3}+\varepsilon} \leq R \leq (t\lambda)$ is of the form:
$$
\frac{\lambda^{\frac{3}{2}}}{(t\lambda)^{2}}C_2(\nu),
$$ 
and arises from $-t^2\partial_{tt}(v_1(r\lambda,t)) \in  \frac{\lambda^{\frac{3}{2}}}{(t\lambda)^{2}}S^{2}(R^0,\log(R))$. The remainder of the error is negligible compared to this term.

One gets a correction term $v_2$ whose dominant component is of the form:
\begin{align}
    \frac{\lambda^{\frac{3}{2}}}{(t\lambda)^{2}}H(a^2) \cdot \chi_{[1,+\infty)}\left( \frac{R}{(t\lambda)^{\frac{2}{3}+\varepsilon}}\right), \label{v_2 dominant component}
\end{align}
 where $H(z)$, $H(0) = 0$, solves
$$
z(1-z)H''(z) + \left( \frac{5}{2} + z \left(s - \frac{3}{2} \right) \right) H'(z) + \frac{s - s^2}{4} = C_2(\nu), \quad 0 < z < 1,
$$
with $s = \frac{3}{2}(-1-\nu) + 2\nu$ and $\nu > 3$. Explicitly, $H(z)$ is given by
\begin{equation*} 
    H(z) = \frac{C_2(\nu)}{\tilde{\alpha}\tilde{\beta}}\left( F(\tilde{\alpha},\tilde{\beta};\tilde{\gamma},z)-1 \right) = 4C_1(\nu)\left( F(\tilde{\alpha},\tilde{\beta};\tilde{\gamma},z)-1 \right), \quad 0 \leq z < 1,
\end{equation*}
where
$$\tilde{\alpha} = -\frac{s}{2}, \tilde{\beta} = -\frac{s}{2} + \frac{1}{2}, \tilde{\gamma} = \frac{5}{2}.$$
This is exactly the function from $\tilde{\mathcal{Q}}_{\frac{1}{2} + \frac{1}{2}\nu}$ defined earlier in (\ref{H(z) definition}) and for which $(H(z^2) + C_1(\nu))^e \in \tilde{\mathcal{Q}}_{\frac{1}{2} + \frac{1}{2}\nu}$ for any exponent $e \in \mathbb R$. 

Finally, observe that $u_0 + v_1 + v_2$, which is equal to $u_0+v_1$ on $0 \leq R \leq (t\lambda)^{\frac{2}{3}+\varepsilon}$ because of the cutoff, is also positive on $(t\lambda)^{\frac{2}{3}+\varepsilon} \leq R \leq (t\lambda)$, $0 < t \leq t_0 \ll 1$, by positivity of $H(z)$.

\subsubsection{Computation of $t^2e_{2k}$ from $v_{2k}$}

The error $t^2e_{2k}$ is given by
$$t^2e_{2k} \simeq E^t(v_{2k}) + t^2e_{2k-1}^1 + t^2[F(v_{2k} + u_{2k-1}) - F(u_{2k-1})],$$
where $E^t(v_{2k})$ denotes the components of $t^2 \square v_{2k}$ where at least one derivative falls in a cutoff $\chi_{[1,+\infty)}$. By construction, $t^2e_{2k-1}^1 \in E_{\mathrm{ori},k}$. Moreover, we prove that $E^t(v_{2k}) \in E_{\mathrm{ori},k}$ as well and 
$$
N(R,t) := t^2[F(v_{2k} + u_{2k-1}) - F(u_{2k-1})] \in E_{\mathrm{ori},k+1} + E_{\mathrm{tip},k+1}.
$$
% We split $t^2e_{2k} $ into $t^2e_{2k}^0 + t^2e_{2k}^1 \in E_{\mathrm{ori},k} + E_{\mathrm{tip},k+1}$ as follows
%     \begin{align*}
%             t^2e_{2k}^0 &\simeq E^t(v_{2k}) + t^2e_{2k-1}^1 \\
%             &+ t^2[F(v_{2k} + u_{2k-1}) - F(u_{2k-1})] \cdot \left[1-\chi_{[1,+\infty)} \left( \frac{R}{(t\lambda)^{\frac{2}{3}+\varepsilon}} \right) \right]\\
%             &+ t^2[F(v_{2k} + u_{2k-1}) - F(\tilde{v}_{2k} + u_{2k-1})] \cdot \chi_{[1,+\infty)} \left( \frac{R}{(t\lambda)^{\frac{2}{3}+\varepsilon}} \right) \\
%             t^2e_{2k}^1 &\simeq t^2[F(\tilde{v}_{2k} + u_{2k-1}) - F(u_{2k-1})] \cdot \chi_{[1,+\infty)} \left( \frac{R}{(t\lambda)^{\frac{2}{3}+\varepsilon}} \right)
%         \end{align*}
% where $\tilde{v}_{2k}$ is obtained from $v_{2k}$ by removing all the cutoff functions.
Assume for simplicity and by linearity that we only have one term:
$$v_{2k} = \frac{\lambda^{\frac{3}{2}}}{(t\lambda)^{\alpha}} R^i \left( \sum_{0 \leq l \leq j} Q_{\alpha,i,j,l}(a) \log(R)^{l} \right) \chi_{[1,+\infty)}\left( \frac{R}{(t\lambda)^{\frac{2}{3}+\varepsilon}}\right) = \tilde{v}_{2k} \cdot \chi_{[1,+\infty)} \in V_{2k} \subset E_{\mathrm{tip},k}.$$
To prove that $E^t(v_{2k}) \in E_{\mathrm{ori},k}$, observe that
\begin{align} \label{Et(v_2k)}
    E^t(v_{2k}) &= \nu\left(\frac{2}{3}+\varepsilon\right) \chi_{[1,+\infty)}'\left( \frac{R}{(t\lambda)^{\frac{2}{3}+\varepsilon}}\right) \frac{R}{(t\lambda)^{\frac{2}{3}+\varepsilon}} t\partial_t(\tilde{v}_{2k}(r\lambda,r/t,t)) \\
    &\phantom{=}+ \nu\left(\frac{2}{3}+\varepsilon\right)\left[\nu \left(\frac{2}{3}+\varepsilon\right)-1\right] \chi_{[1,+\infty)}''\left( \frac{R}{(t\lambda)^{\frac{2}{3}+\varepsilon}}\right) \left(\frac{R}{(t\lambda)^{\frac{2}{3}+\varepsilon}} \right)^2 \tilde{v}_{2k}(r\lambda,r/t,t) \notag \\
    &\phantom{=}- a^{-2} \frac{R} {(t\lambda)^{\frac{2}{3}+\varepsilon}} \chi_{[1,+\infty)}'\left( \frac{R}{(t\lambda)^{\frac{2}{3}+\varepsilon}}\right) \tilde{v}_{2k}(r\lambda,r/t,t) \notag \\
    &\phantom{=}- a^{-2} \frac{R} {(t\lambda)^{\frac{2}{3}+\varepsilon}} \chi_{[1,+\infty)}'\left( \frac{R}{(t\lambda)^{\frac{2}{3}+\varepsilon}}\right) r\partial_r (\tilde{v}_{2k}(r\lambda,r/t,t)) \notag \\
    &\phantom{=}- a^{-2} \left( \frac{R} {(t\lambda)^{\frac{2}{3}+\varepsilon}} \right)^2 \chi_{[1,+\infty)}''\left( \frac{R}{(t\lambda)^{\frac{2}{3}+\varepsilon}}\right) \tilde{v}_{2k}(r\lambda,r/t,t). \notag
\end{align}
is an element of $E_{\mathrm{ori},k}$ supported on the region $(t\lambda)^{\frac{2}{3}+\varepsilon} \leq R \leq 2(t\lambda)^{\frac{2}{3}+\varepsilon}$. Approximating  each $Q_{\alpha,i,j}(a) \in a^2\mathcal{Q}_{\frac{1}{2}+\frac{1}{2}\nu}$ by a finite sum as in the proof of (\ref{prop:product rule 2}) from Proposition \ref{product rules for V, E}, we obtain a finite sum of $\tilde{e}_{\mathrm{mid}}$ components, together with some approximation error belonging to $\mathcal{E}_{N_0,\nu}$. For this term $E^t(v_{2k})$, there is no apparent gain of smallness compared to $t^2e_{2k-1}$ or $v_{2k}$, but the support is now near the origin.

Finally, we deal with the nonlinear part $N(R,t)$ of $t^2e_{2k}$. This part is supported on $(t\lambda)^{\frac{2}{3}+\varepsilon} \leq R \leq 2(t\lambda)^{\frac{2}{3}+\varepsilon}$. Therefore, we can introduce a harmless cutoff $(1-\chi_{[1/m,+\infty)})$, i.e.,
$$
N(R,t) = N(R,t) \cdot (1-\chi_{[1/m,+\infty)}) \left( \frac{(t\lambda)^{\frac{2}{3}}}{R}\right),
$$
and apply Corollary \ref{nonlinear rules for V,E} with $$w_1 = v_1 + \sum_{i=2}^{k}v_{2i-1}, \quad w_2 = v_2 + \sum_{i=2}^{k-1}v_{2i}, \quad w_k = v_{2k}$$ if $k > 1$ to conclude that $N(R,t) \in E_{\mathrm{ori},k+1} + E_{\mathrm{tip},k+1}$. When $k = 1$, we apply the second part of Corollary \ref{nonlinear rules for V,E} with $w_1 = v_1, w_2 = v_2$ to obtain separately $t^2 F(u_0 + v_1 + v_2)$, $t^2F(u_0+v_1) \in E_{\mathrm{ori},2} + E_{\mathrm{tip},2}$ by choosing $w_1 = v_1, w_2 = v_2$.

\subsubsection{Construction of $v_{2k+1}$ from $t^2e_{2k}$}\label{subsubsec: construction v_2k+1}

We solve (\ref{odd v}) again. As in Secion \ref{section:renormalization step, first iterate}, we are led to solve
\begin{equation}
    (t\lambda)^2\mathcal{L}v_{2k+1}(R,t) = t^2e_{2k}^0(R,t), \quad R \geq 0, \quad \mathcal{L} = -\partial_R^2 - \frac{4}{R} \partial_R - pW(R)^{p-1}, \label{eq:v_2k+1 L operator def}
\end{equation}
where $t$ is treated as a parameter and $t^2e_{2k}^0(R,t) \in E_{\mathrm{ori},k}$ is supported on $0 \leq R \leq 2(t\lambda)^{\frac{2}{3}+\varepsilon}$. We solve the equation separately for each part $t^2e_{2k,\mathrm{ori}}^0$, $t^2e_{2k,\mathrm{mid}}^0$ and $t^2e_{2k,\mathrm{tip}}^0$ and, as stated earlier, the negligible part $\eta$ can be ignored and absorbed into $t^2e_{2k}^1$ and subsequent error terms. On the domain $0 \leq R \leq m(t\lambda)^{\frac{2}{3}}$, we can solve the equation for each element 
$$\frac{\lambda^{\frac{3}{2}}}{(t\lambda)^{\alpha+2n}} w_{n}^{\alpha,I,j}(R), n \in \mathbb N,$$
where $w_{n}^{\alpha,I,j} \in S^{2(k-1)}(R^{I+3n}, \log(R)^{J})$ for some common $J \in \mathbb N_{\geq 0}$, as given in the decomposition of $t^2e_{2k}^0(R,t)$. After applying back the cutoffs $\chi_{[1/m,+\infty)}^{1+j}$ to the solution, the resulting solution is
$$v_{2k+1,\mathrm{ori}}(R,t) = \fsum_{ \substack{(\alpha,I) \kadm \\ j \geq 0} }  \sum_{n = 0}^{+\infty} \ \frac{\lambda^{\frac{3}{2}}}{(t\lambda)^{\alpha+2+2n}} W_{n}^{\alpha,I,j}(R)\chi_{[1/m,+\infty)}\left(\frac{(t\lambda)^{\frac{2}{3}}}{R} \right)^{1+j},$$
where $W_n^{\alpha,I,j} \in S^{2k}(R^{I+2+3n}, \log(R)^{J+2})$. In particular, if $(\alpha,I)$ was $k$-admissible on $C_{\mathrm{ori}}$, then the new pair $(\alpha + 2, i + 2)$ is $(k+1)$-admissible on $C_{\mathrm{ori}}$.

\begin{theorem}  \label{usual step 1}
Let $w(R,t) \in \dfrac{\lambda^{\frac{3}{2}}}{(t\lambda)^{\alpha}}S^{2n}(R^{I}, \log(R)^J)$ for some $\alpha \in \mathbb R$, $I,J,n \in \mathbb N_{\geq 0}$. Let $\mathcal{L}$ be as in (\ref{eq:v_2k+1 L operator def}). Then the correction term $v(R,t)$ defined by the equation
$$(t\lambda)^2 \mathcal{L}v(R,t) = w(R,t), \quad v(0,t) = v'(0,t) = 0,$$
 belongs to $\dfrac{\lambda^{\frac{3}{2}}}{(t\lambda)^{\alpha+2}}S^{2n+2}(R^{I+2} \log(R)^{J+2})$, and satisfies the estimates
\begin{align*}
    ||v||_{S, \mathrm{ori}} &\leq C(\mathcal{L},R_0)||w||_{S, \mathrm{ori}}, \\
    ||v||_{S,\infty} &\leq C(\mathcal{L},R_0,J,m)||w||_{S,\infty},
\end{align*}
for some constant $C$ which is independent of $n$, $I$ and $\alpha$.
\end{theorem}

\begin{proof}
Let $w(z,t) = \lambda^{\frac{3}{2}} (t\lambda)^{-\alpha} w(z)$ and $v(z,t) =  \lambda^{\frac{3}{2}} (t\lambda)^{-\alpha-2} v(z)$. On a neighbourhood of $|z| \leq \sqrt{15}/2$, $w(z)$ is holomorphic with a zero of order $2n$. Using a logarithmic-free fundamental system $\{u_1, u_2\}$ found as in (\ref{fuchs fundamental system}) ($r_1 = 0, r_2 = -3$, $u_1(z) = o(1)$, $u_2(z) = o(z^{-3})$, $W(u_1,u_2)(z)^{-1} = o(z^4)$) on this neighbourhood, extended smoothly to $\mathbb R$ using regular ODE theory, one obtains the solution
\begin{align*}
    v(z) &= \int_{[0,z]} [u_2(z)u_1(y)-u_1(z)u_2(y)]W(u_1,u_2)(y)^{-1} w(y)dy, \\
     v'(z) &= \int_{[0,z]} [u_2'(z)u_1(y)-u_1'(z)u_2(y)]W(u_1,u_2)(y)^{-1} w(y)dy,
\end{align*}
which has the desired regularity and a zero of order $2n+2$ at $z = 0$. Moreover, 
\begin{align*}
    ||v(z)||_{L^{\infty}(F)}  + ||v'(z)||_{L^{\infty}(F)} \leq C(\mathcal{L},R_0) \cdot ||w(z)||_{L^{\infty}(F)} 
\end{align*}
where $F = \{z \in \mathbb C: |z| \leq \sqrt{15}/2 \} \cup \{z \in \mathbb R:  \sqrt{15}/2 \leq z \leq R_0 \}$, because the singularity of the integrand at the origin (which comes from $u_1,u_2$) is removable.

On a neighbourhood of $\Re(z) > 0, |z| \geq R_0$, there is an expansion
$$w(z) = z^I \sum_{j=0}^J w_{j} \log(z)^j$$ 
and we can use Theorem \ref{thm:inhomogeneous fuchs ode} ($r_1 = 3, r_2 = 0, \beta = -I-2$, see also (\ref{step 1 at infinity}) for the ODE at infinity) to find a particular solution of the form 
$$\tilde{v}(z) = z^{I+2}\sum_{j=0}^{J}\sum_{k=0}^{j+2} v_{j,k}(z^{-1})  \log(z)^k = z^{I+2} \sum_{k=0}^{J+2} \underbrace{\left( \sum_{j = \max\{0,k-2\}}^{J} v_{j,k}(z^{-1}) \right)}_{=: v_k} \log(z)^k$$
with the coefficient estimate
$$||v_{j,k}||_{A(|y| \leq R_0^{-1})} \leq C(\mathcal{L},R_0,J) ||w_{j}||_{L^{\infty}(|y| \leq R_0^{-1})}.$$
This implies that
\begin{align*}
    ||\tilde{v}||_{S,\infty} &= m^{I+2} \max_{0 \leq k \leq J+2} ||v_{k}||_{L^{\infty}(|y| \leq R_0^{-1})} \\ 
    &\leq m^{-2} J C(\mathcal{L},R_0,J) m^I \max_{0 \leq j \leq J} ||w_{j}||_{L^{\infty}(|y| \leq R_0^{-1})} \\
&\leq m^{-2} J C(\mathcal{L},J) ||w||_{S,\infty}. 
\end{align*}
Hence, $v(z)$ must match this particular solution modulo some linear combination 
$$c_1 U_1(z^{-1}) + c_2U_2(z^{-1}) = z^{I} \left( c_1 z^{-I}U_1(z^{-1}) + c_2z^{-I}U_2(z^{-1}) \right) $$
of the fundamental system $\{U_1(y), U_2(y) = \tilde{U}_2(y) + c \cdot U_1(y) \log(y) \}$ at infinity. It remains to check that 
$$||c_1 y^{I} U_1(y)||_{A^{\infty}(|y| \leq R_0^{-1})} +||c_2y^{I} \tilde{U}_2(y)||_{A^{\infty}(|y| \leq R_0^{-1})}  \leq C(\mathcal{L},R_0,J,m)||w||_{S,\infty}$$
to conclude the proof. It is sufficient to prove that
$$
|c_1| + |c_2| \leq R_0^I \cdot C(\mathcal{L},R_0,J,m) ||w||_{S,\infty}.
$$
To this end, we first observe that
$$|v(R_0)| + |v'(R_0)| \leq C(\mathcal{L},R_0) ||w||_{S, \mathrm{ori}}$$
and
$$|\tilde{v}(R_0)| + |\tilde{v}'(R_0)| \leq R_0^{I} \cdot C(\mathcal{L},R_0,J) ||w||_{S,\infty}$$
using our estimates on $v$, $v'$,$\tilde{v}$, $\tilde{v}'$. Since $c_1$, $c_2$ solves the linear system:
\begin{align*}
    c_1U_1(R_0^{-1}) + c_2U_2(R_0^{-1}) + \tilde{v}(R_0) &= v(R_0), \\
    c_1U_1'(R_0^{-1}) + c_2U_2'(R_0^{-1}) + \tilde{v}'(R_0) &= v'(R_0), 
\end{align*}
one finds that
$$
|c_1| + |c_2| \leq R_0^{I} \cdot C(\mathcal{L},R_0,J,m).
$$
This finishes the proof.
\end{proof}

 Applying back the cutoffs $$\chi_{[1/m,+\infty)}\left( \frac{(t\lambda)^{\frac{2}{3}}}{R} \right)^{1+j}$$ creates an additional error $E^t(v_{2k+1,\mathrm{ori}}) \in E_{\mathrm{ori},k+1}$, supported on $\frac{m}{2} (t\lambda)^{\frac{2}{3}} \leq R \leq m(t\lambda)^{\frac{2}{3}}$, made of those terms in $\mathcal{L}v_{2k+1,\mathrm{ori}}$ where at least one derivative hits the cutoff. Indeed, assuming for simplicity and by linearity that we have only one sum
$$v_{2k+1,\mathrm{ori}}(R,t) = \chi_{[1/m,+\infty)} \left(\frac{(t\lambda)^{\frac{2}{3}}}{R} \right) \underbrace{\sum_{n = 0}^{+\infty} \ \frac{\lambda^{\frac{3}{2}}}{(t\lambda)^{\alpha+2+2n}} W_{n}^{\alpha,I}(R)}_{=: \tilde{v}_{2k+1,\mathrm{ori}}},$$
this leads to an error term:
\begin{align*}
    E^t(v_{2k+1,\mathrm{ori}}) &=  +\frac{1}{R^2} \left( \frac{(t\lambda)^{\frac{2}{3}}}{R} \right)^2 \chi''_{[1/m,2/m]}\left( \frac{(t\lambda)^{\frac{2}{3}}}{R} \right) \tilde{v}_{2k+1,\mathrm{ori}} \\
    &-  \frac{1}{R^2} \left( \frac{(t\lambda)^{\frac{2}{3}}}{R} \right) \chi'_{[1/m,2/m]}\left( \frac{(t\lambda)^{\frac{2}{3}}}{R} \right) R\partial_R \tilde{v}_{2k+1,\mathrm{ori}} \\
    &- \frac{2}{R^2}  \left( \frac{(t\lambda)^{\frac{2}{3}}}{R} \right) \chi'_{[1/m,2/m]}\left( \frac{(t\lambda)^{\frac{2}{3}}}{R} \right) \tilde{v}_{2k+1,\mathrm{ori}}.
\end{align*}
It follows from Proposition \ref{prop:stability under differentiation} and point (\ref{prop:product rule 2}) from Proposition \ref{product rules for V, E}, as well as the gain of smallness of order $R^{-2} \sim (t\lambda)^{-\frac{4}{3}}$, that $E^t(v_{2k+1,\mathrm{ori}}) \in E_{\mathrm{ori},k+1}$. Hence, this error is absorbed into the next error $t^2e_{2k+1}$.

Next, we solve (\ref{odd v}) with each term 
 $$w(R,t) = \frac{\lambda^{\frac{3}{2}}}{(t\lambda)^{\alpha}}R^i \log(R)^{j_1} g_{\alpha,i,j_1,j_2} \left( \frac{(t\lambda)^\frac{2}{3}}{R} \right) (1-\chi_{[1,+\infty)}) \left( \frac{R}{(t\lambda)^{\frac{2}{3}+\varepsilon}} \right)^{1+j_2}$$
 coming from $t^2e_{2k,\mathrm{mid}}^0$. As before, we first solve the ODE without the cutoff, then apply the cutoff afterward. The error $E^t(v_{2k+1,\mathrm{mid}})$ caused by this simplification, which is supported on $(t\lambda)^{\frac{2}{3}+\varepsilon}\leq R \leq 2(t\lambda)^{\frac{2}{3}+\varepsilon}$, can be included in $t^2e_{2k+1}$ similarly to $E^t(v_{2k+1,\mathrm{ori}})$. In order to show this, we start with a lemma concerning primitives for our $g_{\alpha,i,j_1,j_2}(y)$ functions.

\begin{lemma} \label{lemma:primitive of z^ig(z)log(z)^j}
    Assume that $g(z)$ is smooth on $(0,+\infty)$, zero on $(2/m,+\infty)$ and expands as a finite sum of holomorphic functions and logarithms on $|z| < z_0$,
    $$g(z) = \sum_{j=0}^J g_j(z) \log(z)^j, \quad |z| < z_0, z \notin \mathbb R_{\leq 0}.$$
    
    If $I \in \mathbb Z$, there exists a primitive of $z^I g(z)$ of the form $z^{\min\{I+1,0\}}G(z)$, where $G(z)$ is smooth on $(0,+\infty)$, zero on $(2/m,+\infty)$ and expands as a finite sum of holomorphic functions and logarithms on $|z| < z_0$. Explicitly, 
    $$z^{\min\{I+1,0\}}G(z) = \int_{z}^{2/m} y^I g(y) dy$$
    when $z \in (0,+\infty)$.
\end{lemma}

\begin{proof}
Without loss of generality, we assume that $g(z) = g_{J}(z)\log(z)^J$ near $z = 0$,
$$g_J(z) = \sum_{n=0}^{\infty}g_nz^n, \quad |z| < z_0$$
instead of having a sum with different logarithmic exponents. The primitive
    $$\tilde{G}(z) = \int_{z}^{2/m} y^I g(y) dy, \quad z \in (0,+\infty)$$
     of $z^I g(z)$ is smooth on $(0,+\infty)$, zero on $(2/m,+\infty)$ and extends holomorphically on $|z| < z_0$, $z \notin \mathbb R_{\leq 0}$, via:
     $$G(z) = \int_{[z_0/2,z]} y^I g_J(y) \log(y)^J dy + \underbrace{\int_{z_0/2}^{2/m} y^I g(y) dy}_{\text{cst}}, \quad |z| < z_0, z \notin \mathbb R_{\leq 0}.$$
     We compute
     \begin{align*}
\int_{[z_0/2,z]} y^I g_J(y)\log(y)^J dy &= \int_{[z_0/2,z]} \sum_{n=0}^{\infty} g_n y^{n+I} \log(y)^J dy \\
    &= \sum_{n=0}^{\infty} \int_{[z_0/2,z]} g_n y^{n+I} \log(y)^J dy, \quad |z| < z_0, z \notin \mathbb R_{\leq 0}.
\end{align*}
If $\delta \in \mathbb R \setminus \{-1\}$, $J \in \mathbb N_{\geq 0}$, then a primitive of $y^{\delta} \log(y)^J$ is given by 
$$z^{\delta+1} \sum_{k=0}^J \frac{(-1)^k J!}{(J-k)!} \cdot \frac{\log(z)^{J-k}}{(\delta+1)^{k+1}},$$
and if $\delta = -1$, a primitive is given by $(J+1)^{-1}\log(y)^{J+1}$. For $k \in \{0,...,J\}$, let
\begin{align*}
    A_{J,k}(z) =  \frac{(-1)^k J!}{(J-k)!} \sum_{\substack{n=0 \\ n + I \neq -1}}^{+\infty} \frac{g_{n}}{(n+I+1)^{k+1} }z^n,
\end{align*}
which is holomorphic on $|z| < |z_0|$. Then, for $|z| < z_0, z \notin \mathbb R_{\leq 0}$, we obtain
$$\tilde{G}(z) = z^{I+1} \left( \sum_{k=0}^J A_{J,k}(z) \log(z)^{J-k} + z^{-I-1}g_{-1-I}(J+1)^{-1}\log(y)^{J+1}\right) + C,$$
     with the convention that $g_{-1-I} = 0$ if $I \notin \mathbb Z_{\leq -1}$. If $I + 1 \in \mathbb N_{\geq 0}$, we set $G(z) = \tilde{G}(z)$. Otherwise, we set $G(z) = z^{-I-1}\tilde{G}(z)$.
\end{proof}

\begin{theorem}  \label{second step 1}
Let $w(R,t) = \dfrac{\lambda^{\frac{3}{2}}}{(t\lambda)^{\alpha}}R^I \log(R)^J g\left(\dfrac{(t\lambda)^{\frac{2}{3}}}{R}\right)$ for some $k$-admissible pair $(\alpha, I)$ on $C_{\mathrm{mid}}$, $J \in \mathbb N_{\geq 0}$ and a smooth $g(y)$ on $(0,+\infty)$ which is zero on $(2/m,+\infty)$ and extends as a finite sum of holomorphic functions and logarithms near $0$. Let $\mathcal{L}$ be as in (\ref{eq:v_2k+1 L operator def}). Then the correction term $v(R,t)$ obtained by solving
$$(t\lambda)^2 \mathcal{L}v(R,t) = w(R,t), \quad v(0,t) = v'(0,t) = 0$$
is of the form:
$$v(R,t) = \fsum_{\substack{(\tilde{\alpha},i) \ \mathrm{(k+1)}\text{-adm} \\ 0 \leq j \leq J+2}} \frac{\lambda^{\frac{3}{2}}}{(t\lambda)^{\tilde{\alpha}}}R^{i} \log(R)^j G_{\tilde{\alpha},i,j} \left( \frac{(t\lambda)^{\frac{2}{3}} }{R}\right) + \eta,$$
where $(\tilde{\alpha},i)$ is $(k+1)$-admissible on $C_{\mathrm{mid}}$, $\eta \in \mathcal{E}_{N_0,\nu}$, $G_{\tilde{\alpha},i,j}(y)$ is smooth on $(0,+\infty)$, zero on $(2/m,+\infty)$ and extends as a finite sum of holomorphic functions and logarithms near $y = 0$.
\end{theorem}

\begin{remark}
    The specific initial condition for the ODE is not critical, as we will multiply the solution by the cutoffs we ignored, which will make it zero at $R = 0$.
\end{remark}

\begin{proof} 
 The solution is explicitly given by
\begin{align*}
    v(R,t) = \frac{1}{(t\lambda)^2} R^{-2}\int_{\frac{m}{2}(t\lambda)^{\frac{2}{3}}}^{R}\left[ \theta(R)\phi(s) - \theta(s)\phi(R) \right]s^2w(s,t)ds,
\end{align*}
where $\{\phi, \theta\}$ is the fundamental system from Section 4 (see (\ref{phi}) and (\ref{theta})). When $\Re(z) > 0$, $|z| \geq R_0$, we have:
\begin{align*}
    \phi(z) &= \underbrace{z^{-1}\sum_{n=0}^{3N_0+1}a_iz^{-i}}_{\tilde{\phi}(z)} + z^{-1}\sum_{n=3N_0+2}^{+\infty}a_iz^{-i} \\
        \theta(z) &= \underbrace{c \log(z)\tilde{\phi}(z) + z^{2}\sum_{n=0}^{3N_0+3}b_iz^{-i}}_{\tilde{\theta}(z)} + z^{-1}\sum_{n=3N_0+4}^{+\infty}b_iz^{-i} + c \log(z)(\phi(z)-\tilde{\phi}(z)).
\end{align*}
Plugging into the formula for $v$, we obtain
\begin{align*}
    v(R,t) &= \frac{1}{(t\lambda)^2} R^{-2}\int_{\frac{m}{2}(t\lambda)^{\frac{2}{3}}}^{R}\left[ \tilde{\theta}(R)\tilde{\phi}(s) - \tilde{\theta}(s)\tilde{\phi}(R) \right]s^2w(s,t)ds + \eta,
\end{align*}
where $\eta \in \mathcal{E}_{N_0,\nu}$. Expanding the integrand, we are led to analyze a finite number of terms of the form:
$$\frac{1}{(t\lambda)^{2+\alpha}} \log(R)^{\delta} R^{l} \int_{\frac{m}{2}(t\lambda)^{\frac{2}{3}}}^{R}s^{i} \log(s)^{j} g \left( \frac{(t\lambda)^{\frac{2}{3}}}{s} \right) ds,$$
where $\delta \in \{0,1\}$, $(\alpha,I)$ is $k$-admissible on $C_{\mathrm{mid}}$ by hypothesis, $0 \leq j \leq J + 1$ and either ($i \leq  I + 4$, $l \leq -3$) or ($i \leq  I + 1$, $l \leq 0$) depending on whether we deal with $\tilde{\theta}(R)\tilde{\phi}(s)$ or $\tilde{\theta}(s)\tilde{\phi}(R)$. We do a change of variables $s = (t\lambda)^{\frac{2}{3}}y^{-1}$ in the integral and obtain
$$
(t\lambda)^{-2-\alpha+\frac{2}{3}(i+1)} \log(R)^{\delta} R^{l} \int_{\frac{(t\lambda)^{\frac{2}{3}}}{R}}^{m/2}y^{-i-2} \left( - \log(y) + \log((t\lambda)^{\frac{2}{3}}) \right)^{j} g\left( y \right) dy.
$$
We are thus left to analyze a finite number of integrals of the form:
$$
(I) = (t\lambda)^{-2-\alpha+\frac{2}{3}(i+1)} \log((t\lambda)^{\frac{2}{3}})^{j_1}  \log(R)^{\delta} R^{l} \int_{\frac{(t\lambda)^{\frac{2}{3}}}{R}}^{m/2} y^{-i-2}  \log(y)^{j_2} g \left( y \right) dy
$$
with $j = j_1 + j_2$.
Let
$$y^{\min\{-i-1,0\}}G_{i,j_2}(y) = \int_{y}^{m/2} s^{-i-2}  \log(s)^{j_2} g \left( s \right) ds$$
be the primitive from Lemma \ref{lemma:primitive of z^ig(z)log(z)^j}, where $G_{i,j_2}(y)$ a smooth function on $(0,+\infty)$ being zero when $y > 2/m$ and which expands as a finite sum of holomorphic functions and logarithms near $y = 0$. Plugging the primitive into (I), we finally get
$$
(I) = (t\lambda)^{-2-\alpha+\frac{2}{3}(i+1)} \left[ 
 \log \left( \frac{(t\lambda)^{\frac{2}{3}}}{R} \right) + \log(R) \right]^{j_1} \log(R)^{\delta} R^{l} \left( \frac{(t\lambda)^{\frac{2}{3}}}{R} \right)^{\min\{-i-1,0\}} G_{i,j_2} \left( \frac{(t\lambda)^{\frac{2}{3}}}{R} \right).
 $$
Finally, we check the smallness. If $-i-1 \geq 0$, then the smallness is determined by
$$(t\lambda)^{-2-\alpha+\frac{2}{3}(i+1)} R^l$$
 where $l \leq 0$ and $i+l+1 \leq I+2$. In the middle region $m(t\lambda)^{\frac{2}{3}} \leq R \leq (t\lambda)^{\frac{2}{3}+\varepsilon}$, this is bounded by
\begin{align*}
    (t\lambda)^{-2-\alpha+\frac{2}{3}(i+1)+\frac{2}{3}l} &= \frac{(t\lambda)^{\frac{2}{3}(i+1+l)}}{(t\lambda)^{\alpha+2}} =  \frac{(t\lambda)^{\frac{2}{3}(I+2)}}{(t\lambda)^{\alpha+2}} \cdot (t\lambda)^{\frac{2}{3}(i+l+1-I-2)} \\
    &\lesssim \frac{(t\lambda)^{\frac{2}{3}(I+2)}}{(t\lambda)^{\alpha+2}} \lesssim \frac{|R|^I}{(t\lambda)^{\alpha}} \cdot \frac{1}{(t\lambda)^{\frac{2}{3}}}
\end{align*}
 meaning that $(\alpha + 2 - \frac{2}{3}(i+1), l)$ is $(k+1)$-admissible on $C_{\mathrm{mid}}$ if $(\alpha,I)$ is $k$-admissible. Otherwise, $-i-1<0$ and the smallness is determined by
$$(t\lambda)^{-2-\alpha} R^{l+i+1} = \frac{R^{I+2}}{(t\lambda)^{\alpha+2}} \cdot R^{l+ i-I-1}$$
 where $l+i-I-1 \leq 0$. In that case, there is also a gain of smoothness in the middle region and  $(\alpha + 2, l+i+1)$ is $(k+1)$-admissible.
\end{proof}

Finally, we handle the contribution from  $t^2e_{2k,\mathrm{tip}}^0$, with a result structurally similar to the one for $t^2e_{2k,\mathrm{mid}}^0$. However, there is a subtle difference regarding the admissibility of the resulting exponent pairs, which are admissible on $C_{\mathrm{tip}}$ while the initial pairs were only admissible on $C_{\frac{2}{3}+\varepsilon}$. Specifically, the new $R^i$-factors for the solution all have negative exponents $i \leq 0$. For such a term, admissibility on $C_{\frac{2}{3}+\varepsilon}$ and $C_{\mathrm{tip}}$ are equivalent.

 \begin{theorem} \label{third step 1}
Let $w(R,t) = \dfrac{\lambda^{\frac{3}{2}}}{(t\lambda)^{\alpha}}R^I \log(R)^J h\left(\dfrac{R}{(t\lambda)^{\frac{2}{3}+\varepsilon}}\right)$ for some $k$-admissible pair $(\alpha,I)$ on $C_{\frac{2}{3}+\varepsilon}$ and a smooth $h(y)$ which has compact support in $(1,2)$. Let $\mathcal{L}$ be as in (\ref{eq:v_2k+1 L operator def}). Then the correction term $v(R,t)$ obtained by solving
$$(t\lambda)^2 \mathcal{L}v(R,t) = w(R,t), \quad v(0,t) = v'(0,t) = 0$$
is of the form:
$$v(R,t) = \fsum_{\substack{(\tilde{\alpha},i) \ \mathrm{(k+1)}\text{-adm} \\ 0 \leq j_1 \leq 1 \\
0 \leq j_2 \leq J}} \frac{\lambda^{\frac{3}{2}}}{(t\lambda)^{\tilde{\alpha}}}R^{i} \log(R)^{j_1} \log(t\lambda)^{j_2} H_{\tilde{\alpha},i,j} \left( \frac{R}{(t\lambda)^{\frac{2}{3}+\varepsilon}} \right) + \eta,$$
where $i \leq 0$, $(\tilde{\alpha},i)$ is $(k+1)$-admissible on $C_{\mathrm{tip}}$, $\eta \in \mathcal{E}_{N_0,\nu}$ has no singularity at the boundary $R = (t\lambda)$ and $H_{\tilde{\alpha},i,j}(y)$ is smooth, constant outside $[1,2]$ and vanishes in a neighbourhood of $y = 1$.
 \end{theorem}

\begin{proof}
    As in the proof Theorem \ref{second step 1}, the solution $v(R,t)$ is composed of a negligible term and a finite sum of elements of the form:
    $$\frac{1}{(t\lambda)^{2+\alpha}} \log(R)^{\delta} R^{l} \int_{(t\lambda)^{\frac{2}{3}+\varepsilon}}^{R}s^{i} \log(s)^{j} h \left( \frac{s}{(t\lambda)^{\frac{2}{3}+\varepsilon}} \right) ds,$$
where $\delta \in \{0,1\}$, $(\alpha,I)$ is $k$-admissible on $C_{\mathrm{mid}}$ by hypothesis, $0 \leq j \leq J + 1$ and either ($i \leq  I + 4$, $l \leq -3$) or ($i \leq  I + 1$, $l \leq 0$) depending on whether we deal with $\tilde{\theta}(R)\tilde{\phi}(s)$ or $\tilde{\theta}(s)\tilde{\phi}(R)$.  We do a change of variables $s = (t\lambda)^{\frac{2}{3}+\varepsilon}y$ in the integral, which leads to a finite number of integrals of the form:
$$
(I) = (t\lambda)^{-2-\alpha+\left( \frac{2}{3} + \varepsilon \right) (i+1)} \log((t\lambda)^{\frac{2}{3}+\varepsilon})^{j_1}  \log(R)^{\delta} R^{l} \int_{1}^{\frac{R}{(t\lambda)^{\frac{2}{3}+\varepsilon}}} y^{i}  \log(y)^{j_2} h \left( y \right) dy
$$
with $j = j_1 + j_2$.
Let
$$H_{i,j_2}(y) = \int_{1}^{y} s^{i}  \log(s)^{j_2} h \left( s \right) ds.$$
Since $y^j \log(y)^{j_2} h(y)$ is smooth with compact support in $(1,2)$, the primitive $H_{i,j_2}(y)$ is a smooth function on $(0,+\infty)$, vanishing when $y < 1$ and constant when $y > 2$. We finally get the desired form:
$$
(I) =(t\lambda)^{-2-\alpha+\left( \frac{2}{3} + \varepsilon \right) (i+1)}  \log((t\lambda)^{\frac{2}{3}+\varepsilon})^{j_1}   \log(R)^{\delta} R^{l} H_{i,j_2} \left( \frac{R}{(t\lambda)^{\frac{2}{3}+\varepsilon}} \right).
 $$
Finally, we check the smallness, which is determined by
$$(t\lambda)^{-2-\alpha+\left( \frac{2}{3} + \varepsilon \right) (i+1)}R^l,$$
 where $l \leq 0$ and $i+l+1 \leq I+2$. On $R \sim (t\lambda)^{\frac{2}{3}+\varepsilon}$ or on the tip of the cone $(t\lambda)^{\frac{2}{3}+\varepsilon} \leq R \leq (t\lambda)$, this is bounded by
\begin{align*}
    (t\lambda)^{-2-\alpha+\left( \frac{2}{3} + \varepsilon \right) (i+1+l)} &= \frac{(t\lambda)^{\left( \frac{2}{3} + \varepsilon \right) (i+1+l)}}{(t\lambda)^{\alpha+2}} =  \frac{(t\lambda)^{\left( \frac{2}{3} + \varepsilon \right) (I+2)}}{(t\lambda)^{\alpha+2}} \cdot (t\lambda)^{\left( \frac{2}{3} + \varepsilon \right) (i+l+1-I-2)} \\
    &\lesssim \frac{(t\lambda)^{\left( \frac{2}{3} + \varepsilon \right) (I+2)}}{(t\lambda)^{\alpha+2}} \lesssim \frac{|R|^I}{(t\lambda)^{\alpha}} \cdot \frac{1}{(t\lambda)^{\frac{2}{3}-2\varepsilon}},
\end{align*}
 meaning that $\left[ \alpha + 2 - \left(\frac{2}{3}+\varepsilon \right) (i+1), l\right]$ is $(k+1)$-admissible on $C_{\frac{2}{3}+\varepsilon}$ if $(\alpha,I)$ is $k$-admissible. Since $l \leq 0$, admissibility on $C_{\frac{2}{3}+\varepsilon}$ and on $C_{\mathrm{tip}}$ are equivalent.
\end{proof}

All of this leads to the desired correction term $v_{2k+1} \in V_{2k+1}$.

\subsubsection{Computation of $t^2e_{2k+1}$ from $v_{2k+1}$}

The error term $t^2e_{2k+1} = t^2e_{2k+1}^0 + t^2e_{2k+1}^1\in E_{\mathrm{tip},k+1} + E_{\mathrm{ori},k+1}$ is given by
\begin{align*}
     t^2e_{2k+1}^0 &\simeq
     t^2[F(v_{2k+1} + u_{2k}) - F(u_{2k}) - F'(u_0)v_{2k+1}] \\
     &+ t^2\partial_{tt} v_{2k+1} + t^2e_{2k}^1 + E^t(v_{2k+1}),
\end{align*}
    where $E^t(v_{2k+1})$ consists of those components in $\mathcal{L} v_{2k+1,\mathrm{ori}}$ and $\mathcal{L} v_{2k+1,\mathrm{mid}}$ where at least one derivative hits the cutoff. It has already been established in Section \ref{subsubsec: construction v_2k+1} that $E^t(v_{2k+1}) \in E_{\mathrm{ori},k+1}$. A straightforward computation also shows that $t^2\partial_{tt} v_{2k+1} \in V_{2k+1}$ and we know, by Proposition \ref{product rules for V, E}, that
    $$V_{2k+1} \subset E_{\mathrm{ori},k+1} + E_{\mathrm{tip},k+1}.$$ 
    Hence, it remains to verify that
\begin{align*}
    N(e_{2k+1}) &= t^2[F(v_{2k+1} + u_{2k}) - F(u_{2k}) -F'(u_0)v_{2k+1}  ] \\
&\in E_{\mathrm{ori},k+1} + E_{\mathrm{tip},k+1}.
\end{align*}
Using Corollary \ref{nonlinear rules for V,E} with 
 $$w_1 = v_1 + \sum_{i=2}^{k}v_{2i-1}, \quad w_2 = v_2 + \sum_{i=2}^{k}v_{2i}, \quad w_k = v_{2k+1},$$ 
 we only need to handle
 $$
 N(e_{2k+1})_{\mathrm{ori}}:=  \chi_{[1/m,+\infty)} \left(\frac{(t\lambda)^{\frac{2}{3}}}{R} \right)\cdot N(e_{2k+1}) \in E_{\mathrm{ori},k+1}.
 $$
On the region $0 \leq R \leq m(t\lambda)^{\frac{2}{3}}$, write
$$N(e_{2k+1}) = t^2[F(v_{2k+1} + u_{2k}) - F(u_{2k}) - F'(u_{2k})v_{2k+1} + (F'(u_{2k})-F'(u_0))v_{2k+1}  ], $$
so that
\begin{align}
N(e_{2k+1})_{\mathrm{ori}} &\simeq  \chi_{[1/m,+\infty)} \cdot \sum_{ \substack{ i \geq 0 \\ 3N_0+1 \geq j \geq 0 \\
3N_0+1 \geq l \geq 2}} \binom{p}{i,j,l} t^2 u_0^p T \left[ \left( \frac{v_1}{u_0} \right)^{i} \right] \left( \frac{u_{2k}-u_0-v_1}{u_0} \right)^{j} \left( \frac{v_{2k+1}}{u_0} \right)^{l} \notag \\
&+ \chi_{[1/m,+\infty)}  \cdot \frac{v_{2k+1}}{u_0} \sum_{ \substack{i,j \geq 0 \\ i+j \geq 1 \\ 3N_0+1 \geq j \geq 0 }} \binom{p-1}{i,j} t^2 u_0^{p}  T \left[ \left( \frac{v_1}{u_0} \right)^{i} \right]  \left( \frac{u_{2k}-u_0-v_1}{u_0} \right)^{j}, \label{N(e_2k+1)^1_ori}
\end{align}
where $T$ is the truncation operator from (\ref{truncation}). Observe the following additional facts:
\begin{enumerate}
    \item \begin{align*}
    \chi_{[1/m,+\infty)} \cdot \sum_{i \geq 2} c_i  t^2 u_0^p T \left[ \left( \frac{v_1}{u_0} \right)^{i} \right] &\in V_1
\end{align*}
for any $(c_i)_{i\geq 0} \in \ell^{\infty}$.
\item For any $k \geq 1$, 
$$\chi_{[1/m,+\infty)}  u_0^{-1} V_{2k-1}  \cdot \subset \chi_{[1/m,+\infty)} \cdot \left[ \frac{ \lambda^{\frac{3}{2}} } {(t\lambda)^2} \right]^{-1} V_{2k-1} \subset V_{2k-1},$$
and
\begin{align*}
    t^2 u_0^p V_{2k-1} \subset \lambda^{\frac{3}{2}} (t\lambda)^2 V_{2k-1}, \quad
    t^2 u_0^p \left( \frac{v_1}{u_0} \right) V_{2k-1} \subset \lambda^{\frac{3}{2}} V_{2k-1}.
\end{align*}
\end{enumerate}

Examining first sum in (\ref{N(e_2k+1)^1_ori}), we must distinguish three cases : $i = 0, i = 1$ and $i \geq 2$ cases. Using the product rules from Proposition \ref{product rules for V, E}, we find that we have a finite sum of the form:
\begin{align*}
     \sum_{ \substack{ 3N_0+1 \geq j \geq 0 \\
3N_0+1 \geq l \geq 2}}  \left[ t^2u_0^p + t^2 u_0^p \left( \frac{v_1}{u_0} \right) + V_{1} \right] \cdot V_{2 \cdot 2 - 1}^j \cdot V_{2(k+1)-1}^l \cdot  \left[ \frac{ \lambda^{\frac{3}{2}} } {(t\lambda)^2} \right]^{-j-l}\\
\subset \left( 1 + (t\lambda)^2 + (t\lambda)^4 \right) E_{\mathrm{ori},k+3} \subset  E_{\mathrm{ori},k+1}.
\end{align*}
The other sum is treated analogously.

\section{The spectral theory of the linearized operator} \label{section:spectral theory of the linearized operator}
At this point, we pivot from constructing the approximate solution to developing the analytical tools required to find the final correction term needed to obtain an exact solution of (NLW) on the light cone. When solving equation (\ref{space-variable equation epsilon}) for $\varepsilon$ in Section \ref{section:exact solution fourier method}, the following Sturm-Liouville operator arises
\begin{align*}
    \mathcal{L} &= -\partial_{RR} - pW(R)^{p-1} + \frac{1}{R^2} \cdot \left( \frac{(d-3)(d-1)}{4} \right) \\
    &= -\partial_{RR} - \frac{d^2(d+2)(d-2)}{(R^2 + d(d-2))^2} + \frac{1}{R^2} \cdot \left( \frac{(d-3)(d-1)}{4} \right) \\
    &= -\partial_{RR} + V(R)
\end{align*}
\index{L@$\mathcal{L}, \mathcal{L}_R$, the perturbed Schrödinger operator}
on $(0,+\infty)$, which is self-adjoint on 
$$\dom(\mathcal{L}) = \{ f \in L^2((0,+\infty)): f,f' \in AC_{loc}((0,+\infty)), \ \mathcal{L}f \in L^2((0,+\infty))\}.$$
This section is dedicated to studying this perturbed Schrödinger operator. We will characterize its spectrum, construct its generalized eigenfunctions $\phi(R,\xi)$ as well as a spectral measure $d\rho(\xi)$ which will be key tools to construct a generalized Fourier transform. When expressing the equation for $\varepsilon$ in the generalized Fourier space, $\mathcal{L}$ is transformed into a multiplication by $\xi$, which will make the equation easier to solve using a fixed point argument.

We will study this operator for general $d \in \mathbb N_{\geq 4}$ and will later restrict to $d \in \{4,5\}$. We aim to obtain precise asymptotic estimates on the spectral measure, the eigenfunctions and the Jost solution. A fundamental system $\{\phi(R), \theta(R)\}$ of $\mathcal{L}f = 0$ with $W(\theta,\phi) = 1$ is given by
\begin{equation}
    \phi(R) = R^{\frac{d-1}{2}} \frac{d}{d \lambda}\Bigr|_{\lambda = 1} \left( \lambda^{\frac{d-2}{2}}W(\lambda R) \right) =  \frac{R^{\frac{d-1}{2}}(R^2-d(d-2))}{(R^2+d(d-2))^{\frac{d}{2}}}
    \label{phi}
\end{equation}
\begin{equation}
    \theta(R) = \frac{R^{\frac{d-1}{2}}(R^2-d(d-2))}{(R^2+d(d-2))^{\frac{d}{2}}} \cdot \int^R \frac{1}{\phi(s)^2}ds, 
    \label{theta}
\end{equation}
\index{phi@$\phi(R)$, an explicit element of the fundamental system for $\mathcal{L}u = 0$}
\index{theta@$\theta(R)$, an explicit element of the fundamental system for $\mathcal{L}u = 0$}
where 
$$
\int^R \frac{1}{\phi(s)^2}ds = \begin{dcases}
a \log(R) + \frac{b}{R^2-d(d-2)} + \sum_{\subalign{i&=-(d-2) \\ i &= 0 \mod 2}}^{d-2} c_i R^i &\text{ if } d = 0 \mod 2 \\
\phantom{a \log(R) +} \frac{bR}{R^2-d(d-2)} + \sum_{\subalign{i&=-(d-2) \\ i &= 1 \mod 2}}^{d-2} c_i R^i &\text{ if } d = 1 \mod 2,
\end{dcases}
$$
and the following asymptotics hold for $\phi$ and $\theta$:
\begin{equation}
    \phi(R) \asymp \begin{cases}
    R^{\frac{d-1}{2}}, \quad R \to 0 \\
    R^{\frac{3-d}{2}}, \quad R \to +\infty \\
    \end{cases} \quad 
    \theta(R) \asymp \begin{cases}
    R^{\frac{3-d}{2}}, \quad R \to 0 \\
    R^{\frac{d-1}{2}}, \quad R \to +\infty \\
    \end{cases} 
    \label{phi theta asympotics}
\end{equation}
with symbol-type behaviour of the derivatives.

From the behaviour of $\mathcal{L}$ at the endpoints and the number of zeros of $\phi(R)$ on $(0,+\infty)$, one deduces the following spectral properties for $\mathcal{L}$:

\begin{proposition}[Properties of the Sturm-Liouville operator]
For $d \in \mathbb N_{\geq 4}$, $\mathcal{L}$ is limit point at zero and limit point at infinity. Moreover, the spectrum of $\mathcal{L}$ decomposes as follows:
\begin{enumerate}
    \item Essential spectrum: $\spec_{ess}(\mathcal L)= [0,+\infty)$.
    \item Absolutely continuous spectrum: $\spec_{ac}(\mathcal L) = [0,+\infty)$.
    \item Singularly continuous spectrum: $\spec_{sc}(\mathcal L) = \emptyset$.
    \item Pure point spectrum: $\spec_{pp}(\mathcal L) = \{\xi_4\}$ if $d = 4$ and $\spec_{pp}(\mathcal L) = \{\xi_d,0\}$ otherwise for some $\xi_d < 0$.
\end{enumerate}
\end{proposition}

\begin{proof}
All of these properties follow only from the behaviour of $\phi(R)$ and the potential $V(R)$. One can look at \cite{Fulton1} for the limit-point, limit-circle dichotomy and at \cite[Lemma 9.35]{Teschl}, \cite[Theorem 15.3]{Weidmann}, \cite[Chapter XIII.7, Theorem 40, Theorem 55 and Corollary 56]
{Dunford} for the remainder of the statement.
% For the limit point, limit circle dichotomy, one can look at \cite{Fulton1} where many criterions and references are available. The first claim follows from \cite{Teschl}, Lemma 9.35. Then Theorem 15.3 from \cite{Weidmann} proves that $\mathcal{L}$ has purely absolutely continuous spectrum on $(0,+\infty)$. The point $0$ is also part of the absolutely continuous spectrum since it is a closed set (e.g. \cite{Teschl}, Lemma 3.15). 
% Moreover, the spectrum is bounded from below with only one negative point $\xi_d < 0$, which is an eigenvalue (\cite{Dunford}, Chapter XIII.7, Theorem 40, Theorem 55 and Corollary 56). Hence, there is no singularly continuous spectrum and since $\spec_{pp}(\mathcal{L})$ is the closure of the eigenvalues, the proposition follow.
\end{proof}

More generally, one can find a fundamental system $\{\phi(R,z), \theta(R,z)\}$ for $\mathcal{L}u = zu$ of the following form:
\index{phi z@$\phi(R,z)$, part the fundamental system for $\mathcal{L}u = zu$ obtained in Proposition \ref{fundamental system Lu = zu}}
\index{theta z@$\theta(R,z)$, part the fundamental system for $\mathcal{L}u = zu$ obtained in Proposition \ref{fundamental system Lu = zu}}

\begin{proposition}[Expansion for $\phi(R,z)$]\label{fundamental system Lu = zu}
For $z \in \mathbb C$, there exists a fundamental system $\{\phi(R,z), \theta(R,z)\}$, $W(\theta,\phi) = 1$, for $\mathcal{L}u = zu$, real-valued whenever $z \in \mathbb R$ and such that $\phi(R,z)$ is given by the following absolutely convergent series:
\begin{equation}
    \phi(R,z) = \phi(R) + R^{\frac{d-1}{2}} \sum_{j=1}^{\infty}(R^2z)^j \phi_j(R^2),
    \label{phi(R,z)}
\end{equation}
where $\phi_j$ is holomorphic on $U = \{z \in \mathbb C: \Re(z) > -\frac{1}{2}d(d-2)\}$,
\begin{align*}
|\phi_j(u)| &\leq \frac{C^j}{(j-1)!} (1 + |u|)^{-1}  \cdot \log(1 + |u|)^{\delta_{d = 4}}, \quad \forall u \in U, \\
|\phi_1(u)| &\gtrsim (1 + |u|)^{-1}  \cdot \log(1 + |u|)^{\delta_{d = 4}}, \quad \forall u \in U, |u| \gtrsim 1,
\end{align*}
for some constant $C > 0$ independent of $j$ and $u$ (the logarithm appears only in dimension $d = 4$). Moreover, $\theta(R,z)$ is entire with respect to $z$, $R^{\frac{d-3}{2}}\theta(R,z) \in C^0([0,+\infty) \times \mathbb C)$ and it is a Frobenius type solution in the following sense:
$$\lim \limits_{R \to 0} R^{-(l+1)} \frac{d^{n_l + 1}}{dz^{n_l + 1}}\theta(R,z) = 0, \quad l = \frac{d-3}{2}, n_l = \floor{l + 1/2}.$$
\end{proposition}

\begin{proof}
The exact form of $\theta(R,z)$ does not matter. We only need its existent, which is proved in \cite{Kostenko2}. As for $\phi(R,z)$, we try to look for a solution of the form: $$\phi(R,z) = R^{-\frac{d-1}{2}}\sum_{j=0}^{\infty}z^jf_j(R), \quad \phi(0,z) = \phi'(0,z) = 0,$$
where $f_0(R) = R^{\frac{d-1}{2}}\phi(R)$. Then, $f_j$ must solve
$$\mathcal{L}(R^{-\frac{d-1}{2}}f_j) = R^{-\frac{d-1}{2}}f_{j-1}, \quad f_j(0) = f_j'(0) = 0.$$
By induction, we find such solutions $f_j$ and prove that $f_j$ has a zero of order $R^{(d-1)+2j}$ at $R = 0$, $R^{-(d-1)}f_j(R) = g_j(R^2)$ for some $g_j$ analytic on $U$ and for $z \in U$, $j \geq 1$:
$$|f_j(z)| \leq \frac{C^j}{(j-1)!} (1+|z|)^{(d-1)+2(j-1)}(\cdot \log(1 + |z|) \text{ if } d = 4).$$
Using the variation of parameters, one gets
\begin{align*}
    f_j(R) &= -\int_0^R \frac{ R^{\frac{d-1}{2}} }{ s^{\frac{d-1}{2}} } \left[\phi(R)\theta(s) - \phi(s)\theta(R) \right] f_{j-1}(s)ds.
\end{align*}
Let $c(R) = \int^s \frac{1}{\phi(s)^2} ds$ as in (\ref{theta}) and
$$F_{j-1}(R) = \int_{0}^R s^{-\frac{d-1}{2}}\phi(s)f_{j-1}(s)ds = \int_{0}^R \frac{s^2-d(d-2)}{(s^2+d(d-2))^{\frac{d}{2}}}f_{j-1}(s)ds.$$
Then $F_{j-1}(R)$ has a zero of order $R^{d + 2(j-1)}$ at $R = 0$ and $R^{-d}F_{j-1}(R) = G_{j-1}(R^2)$ where $G_{j-1}$ is analytic on $U$. Hence, for $R < d(d-2)$, we can use integration by parts and write 
\begin{align}
    f_j(R) &= -R^{\frac{d-1}{2}} \phi(R) \int_0^R s^{-\frac{d-1}{2}}\phi(s) \left[c(s)-c(R) \right] f_{j-1}(s)ds \notag \\
    &= R^{\frac{d-1}{2}} \phi(R) \int_0^R F_{j-1}(s)c'(s)ds \notag \\
    &= \frac{R^{d-1}(R^2-d(d-2))}{(R^2+d(d-2))^{\frac{d}{2}}} \int_0^R G_{j-1}(s^2)\frac{s(s^2 + d(d-2))^d}{(s^2-d(d-2))^2}ds.
    \label{f_j at zero}
\end{align}
From this last formula (\ref{f_j at zero}), we deduce that $f_j$ has a zero of order $R^{(d-1)+2j}$ at $R = 0$, $f_j$ extends as an holomorphic function on $U \cap B_{d(d-2)}(0)$ and $R^{-(d-1)}f_j$ has an even expansion around $R = 0$. In fact, $f_j$ extends holomorphically on $U$. If we let $c_1(R) = c(R)$ modulo the part which is singular at $R = d(d-2)$ and $c_2(R) = c(R) - c_1(R)$, then we can write
\begin{align}
    f_j(R) &= -R^{\frac{d-1}{2}} \phi(R) \int_0^R s^{-\frac{d-1}{2}}\phi(s) \left[c_2(s)-c_2(R) \right] f_{j-1}(s)ds \notag \\
    &\phantom{=}+ R^{\frac{d-1}{2}} \phi(R) \int_0^R F_{j-1}(s)c_1'(s)ds, \label{f_j}
\end{align}
and this extends holomorphically on $U$ since multiplication by $\phi(R)$ and $\phi(s)$ removes the singularity. Then one can bound $f_j(z)$ as follows:
\begin{align*}
    |f_j(z)| \leq C \left( (1+|z|) \int_0^{|z|} (1+|s|)^{1-d} |f_{j-1}(s)|ds +  \int_0^{|z|} (1+|s|)^{2-d} |f_{j-1}(s)|ds \right. \\
    \phantom{=} \left. + (1+|z|) \int_0^{|z|} |s|^{2d-3} |s^{-d}F_{j-1}(s)| ds\right) \quad \forall z \in U
\end{align*}
using (\ref{f_j}). Moreover, one has
\begin{align*}
    |f_0(z)| \leq C(1+|z|), \quad |z^{-d} F_0(z)| \leq C (1+|z|)^{-d}\cdot \log(1 + |z|)^{\delta_{d=4}},
\end{align*}
so we deduce
\begin{align*}
    |f_1(z)| &\leq 3C^2 (1+|z|)^{d-1}\cdot \log(1 + |z|)^{\delta_{d=4}}, \\
    |z^{-d}F_1(z)| &\leq 3C^2 (1+|z|)^{2-d}\cdot \log(1 + |z|)^{\delta_{d=4}}.
\end{align*}
It follows by induction that
\begin{align*}
    |f_j(z)| &\leq \frac{3^jC^{j+1}}{(j-1)!} (1+|z|)^{(d-1)+2(j-1)}\cdot \log(1 + |z|)^{\delta_{d=4}}, \\
    |z^{-d}F_j(z)| &\leq \frac{3^jC^{j+1}}{j!} (1+|z|)^{2j-d}\cdot \log(1 + |z|)^{\delta_{d=4}}.
\end{align*}
Finally, set $\phi_j(R^2) = R^{-(d-1)}R^{-2j}f_j(R)$. Note that in dimension $d = 4$, one explicitly has
$$F_0(z) = -\frac{\splitfrac{ z^6(7+\log(512))+24z^4(2+\log(512))+192z^2(1+\log(512))}{-3\left(z^2+8\right)^3\log\left(z^2+8\right)+1536\log(8)}}{6\left(z^2+8\right)^3},$$
so that we get the lower bound
\begin{align*}
|f_1(z)| &\gtrsim \left| z^{\frac{d-1}{2}} \phi(z) \int_1^z F_{j-1}(y)c_1'(y)dy \right| \gtrsim |z| \int_1^{|z|} |F_j(y)| \cdot y^{d-3} dy \\
&\gtrsim |z| \int_1^z \log(y)y^{d-3} \gtrsim |z|^{d-1} \log(|z|)  \quad \forall z \in U, |z| \gtrsim 1.
\end{align*}
In dimension $d > 4$,
$$F_0(z) = \int_0^z \phi(s)^2ds \geq \min_{y \in B_1(0)} \int_0^y \phi(s)^2ds > 0 \quad \forall z \in U, |z| \geq 1$$
and we deduce the lower bound
\begin{align*}
|f_1(z)| &\gtrsim \left| z^{\frac{d-1}{2}} \phi(z) \int_1^z F_{j-1}(y)c_1'(y)dy \right| \gtrsim |z| \int_1^{|z|} |F_j(y)| \cdot y^{d-3} dy \\
&\gtrsim |z|^{d-1}  \quad \forall z \in U, |z| \gtrsim 1.
\end{align*}
\end{proof}

\begin{remark}
    In particular, $\phi(R,z) \in C^{\infty}((0,+\infty) \times U)$, $U =  \{z \in \mathbb C:  \Re(z) > -\frac{1}{2}d(d-2)\}$. Indeed, fix $R_0 > 0, z_0 \in U$. Around $z = R_0^2$,
    $$
    |\phi_j(z)| \leq \frac{C^j}{(j-1)!}, \quad
 |\phi_j^{(n)}(z)| \lesssim_{n,R_0} \frac{C^j}{(j-1)!}
    $$
    using Cauchy's Integral Formula. Next, observe that
    \begin{align*}
        \partial_{R}^{k_1}\partial_z^{k_2} \left[(R^2z)^j \phi_j(R^2) \right] &= \frac{j!}{(j-k_2)!}z^{j-k_2} \sum_{l_1+l_2=k_1} \binom{k_1}{l_1} \frac{(2j)!}{(2j-l_1)!}R^{2j-l_1} \partial_R^{l_2} \left[ \phi_j(R^2)\right] \\
        &= \fsum_{\substack{l_1 \in \mathbb Z \\ 0 \leq l_2 \leq k_1}} c_{l_1,l_2,k_1,k_2} z^{-k_2} R^{-l_1} (R^2z)^{j-k_2} \phi_j^{(l_2)}(R^2),
    \end{align*}
    where the sum is zero if $k_2 > j$. Hence, the sum
    $$
\fsum_{\substack{l_1 \in \mathbb Z \\ 0 \leq l_2 \leq k_1}} \sum_{j=k_2+1}^{\infty} c_{l_1,l_2,k_1,k_2} z^{-k_2} R^{-l_1} (R^2z)^{j-k_2} \phi_j^{(l_2)}(R^2)
    $$
    converges uniformly around $(R,z) = (R_0,z_0)$ since
$$
|c_{l_1,l_2,k_1,k_2} z^{-k_2} R^{-l_1} (R^2z)^{j-k_2} \phi_j^{(l_2)}(R^2)| \lesssim_{l_1,l_2,k_1,k_2,R_0,z_0} \frac{C(R_0,z_0)^j}{(j-1)!}.
$$
Therefore, the sum of derivatives
    $$
    \sum_{j=k_2+1}^{\infty}  \partial_{R}^{k_1}\partial_z^{k_2} \left[(R^2z)^j \phi_j(R^2) \right]
    $$
    converges uniformly for any $k_1, k_2 \in \mathbb N$, which is sufficient to prove the smoothness of $\phi(R,z)$.
\end{remark}
\begin{remark}
There exists $0 < \xi_0 < \delta_0 < \delta_1 \ll 1$ (depending on the absolute constants from the Proposition \ref{fundamental system Lu = zu}) such that for all $0 < \xi < \xi_0$, for all $\delta \in [\delta_0, \delta_1]$, one has
\begin{equation}
    \phi(R,\xi) \gtrsim R^{\frac{d-5}{2}} \cdot \log(1+R^2)^{\delta_{d=4}}  \label{phi lower bound}
\end{equation}
when $R = \delta \xi^{-\frac{1}{2}}$.
\end{remark}

\begin{corollary} \label{phi upper bound R^2z < 1}
When $|R^2z| \lesssim 1$, we have the following pointwise estimate on $\phi(R,z)$ for any $k \geq 0$ and $l \geq 1$:
\begin{align*}
|(R^k\partial_R^k)\phi(R,z)| &\lesssim R^{\frac{d-1}{2}}\langle R \rangle^{-(d-2)} + (R^2z) R^{\frac{d-1}{2}} \langle R^2 \rangle^{-1}  \cdot \log(1+R^2)^{\delta_{d=4}}   \\
|(R^k\partial_R^k)(z^l\partial_z^l)\phi(R,z)| &\lesssim (R^2z)^l R^{\frac{d-1}{2}} \langle R^2 \rangle^{-1}\log(1+R^2)^{\delta_{d=4}}.
\end{align*}
In particular, 
$$|(R^k\partial_R^k)\phi(R,z)| + |(R^k\partial_R^k)(z^l\partial_z^l)\phi(R,z)| \lesssim R^{\frac{d-1}{2}} \langle R \rangle^{-2} \cdot \log(1+R^2)^{\delta_{d=4}}  \lesssim 1$$
is uniformly bounded on $|R^2z| \lesssim 1$ if $d \in \{4,5\}$.
\end{corollary}

\begin{proof}
It suffices to differentiate the series (\ref{phi(R,z)}) and then distinguish the cases $R \leq 1$, $R \geq 1$.
\end{proof}

% \begin{corollary}
%     When $|R^2z| \lesssim 1$, $k \geq 0$ and $l \geq 1$, one has:
%     \begin{align*}
%         |(R^k\partial_R^k)\partial_z^l\phi(R,z)| &\lesssim R^{\frac{d-1}{2} + 2l} \langle R^2 \rangle^{-1} ( \cdot \log(1+R^2) \text{ if } d = 4 ) \\
%         &\lesssim R^{\frac{d-1}{2} + 2(l-1)} ( \cdot \log(1+R^2) \text{ if } d = 4 ) \\
%         &\lesssim R^l |z|^{\frac{5-d}{4}-\frac{l}{2}} ( \cdot \log(1+|z|^{-\frac{1}{2}}) \text{ if } d = 4 )  \\
%          |(R^k\partial_R^k)\partial_z^l\phi(R,z)| &\lesssim R^{l} |z|^{\frac{1-d}{4} - \frac{l}{2}} \lesssim  |z|^{\frac{1-d}{4} - l} 
%     \end{align*}
%     Similarly, when $k \geq 0$, $(R^k\partial_R^k)\phi(R,z)$ has the same upper bound as $(R^k\partial_R^k)(z\partial_z)\phi(R,z)$ with an additional $R^{\frac{d-1}{2}}\langle R \rangle^{-(d-2)}$ term.
    
% \end{corollary}

% \begin{proof}
%     Use Corollary \ref{phi upper bound R^2z < 1}. The first series of two inequalities follow trivially. For the third one, write $R^{\frac{d-1}{2} + 2(l-1)} = R^l R^{\frac{d-3}{2} + (l-1)}$ and use $|R| \lesssim |z|^{-\frac{1}{2}}$. For the last one, use $R^{\frac{d-1}{2}} \lesssim |z|^{\frac{1-d}{4}}$ and $|(R^2z)^lz^{-l}| \lesssim |(R^2z)^{\frac{l}{2}}z^{-l}| = R^{l}|z|^{-\frac{l}{2}}$.
% \end{proof}

Now, we define the singular $m$-function and state the spectral theorem for the self-adjoint operator $\mathcal{L}$:

\begin{definition}[Singular $m$-function]
Let $m(z): \mathbb C \setminus \mathbb R \rightarrow \mathbb C$, $m(\overline{z}) = \overline{m(z)}$, be the singular Weyl-Titchmarsh $m$-function. It is the unique function for which $\theta(R,z) + m(z)\phi(R,z)$ belongs to $L^2([1,+\infty))$ and solves $\mathcal{L}u = zu$ on $(0,+\infty)$. 
\end{definition}

\begin{theorem}[Spectral theorem] \label{spectral theorem}
The singular $m$-function is a generalized Nevanlinna function which defines a non-negative spectral density $d\rho$ on $\mathbb R$ via
$$\frac{1}{2}(d \rho((a,b)) + d \rho([a,b]))= \lim \limits_{\varepsilon \to 0} \frac{1}{\pi} \int_a^b \Im m(t + i \varepsilon) dt$$
such that the generalized Fourier transform
\begin{align*}
\mathcal{F}: L^2((0,+\infty)) &\rightarrow L^2(\mathbb R, d \rho) \\
f &\mapsto \hat{f}(\xi):= \lim \limits_{r \to +\infty} \int_0^r \phi(s,\xi)f(s)ds 
\end{align*}
\index{F@$\mathcal{F}$, the generalized Fourier transform}
is a unitary operator with inverse
\begin{align*}
\mathcal{F}^{-1}: L^2(\mathbb R, d \rho) &\rightarrow  L^2((0,+\infty)) \\
F &\mapsto \check{F}(R):= \lim \limits_{r \to +\infty} \int_{-r}^r \phi(R,\xi)F(\xi)d\rho(\xi) .
\end{align*}
Here, the limits must be understood as limits of functions in their respective $L^2$-space.

Moreover, if $E$ is the unique spectral family associated to the self-adjoint operator $\mathcal{L}$ on $\dom(\mathcal{L})$ and, for $f \in L^2((0,+\infty))$, $d \mu_{f}$ is its spectral measure, then
$$d \mu_f = |\hat{f}|^2 d \rho$$
\end{theorem}

\begin{proof}
See \cite[Lemma 3.4]{Gesztesy}, \cite[Theorem 3.4, Corollary 3.5]{Kostenko1} and \cite[Theorem 4.5]{Kostenko2}. 
\end{proof}

\begin{remark}
     If $\xi^*$ is an eigenvalue, the inverse Fourier transform of $\delta_{\xi^*}$ is a multiple of the eigenfunction $\phi(R,\xi^*)$. In other words, the $L^2(\mathbb R, d\rho)$-limit
    $$\lim \limits_{r \to +\infty} \int_{0}^r \phi(R,\xi)\phi(R,\xi^*)dR$$
    is $0$ d$\rho$-almost everywhere on $\{\xi \in \mathbb R: \xi \neq \xi^*\}$.
\end{remark}

\begin{notation}\label{not: notation of inner product for L^2 limit}
    We will write
    $$\innerp[\Big]{ f(s) }{ \phi(s,\xi)}_{L^2((0,+\infty))}:= \lim\limits_{r \to +\infty} \int_0^r \phi(s,\xi)f(s)ds$$
   as a limit of functions in $L^2(\mathbb R, d\rho)$,  even though
    $$\int_0^{+\infty} |\phi(s,\xi)f(s)|ds$$
    need not to be finite.
\end{notation}

\begin{remark}[On the decomposition of $d\rho$]
Since $\spec(\mathcal L) = \{\xi_d, 0\} \cup [0,+\infty) = \spec_{pp}(\mathcal{L}) \cup \spec_{ac}(\mathcal{L})$, we can write
$$
d\rho(\xi) = \frac{1}{||\phi(R,\xi_d)||_{L^2}} \delta_{\xi_d}(\xi) + \frac{1}{||\phi(R)||_{L^2}} \delta_{0}(\xi) + \rho(\xi) \chi_{(0,+\infty)}(\xi) d\xi
$$ for some $\rho(\xi) \in L^1_{loc}([0,+\infty))$.
\index{rho@$\rho(\xi)d\xi$, the absolutely continuous part of the spectral measure $d\rho(\xi)$}
\end{remark}

Next, we introduce the Jost solution which will be useful in the computation of asymptotics for $\rho(\xi)$. Note that in the following definition, we use the principal branch of the complex logarithm in order to define roots.

\begin{definition}[Jost solution]
For $z \in \mathbb C \setminus \mathbb R_{\leq 0}$, $\Im z \geq 0$, let $\psi^{+}(R,z)$ denote the Jost solution to $\mathcal{L}u = zu$ at $R = +\infty$ normalized so that 
$$\psi^+(R,z) \sim z^{-\frac{1}{4}}e^{iR\sqrt{z}}, \quad R|\sqrt{z}| \to +\infty.$$
It is given by $z^{-\frac{1}{4}}f_+(R,z)$, where $f_+(R,z)$ is the unique fixed point of 
\begin{equation}
    f_+(R,z) = e^{iR\sqrt{z}} - \int_R^{+\infty}\frac{\sin(\sqrt{z}(R-R'))}{\sqrt{z}}V(R')f_+(R',z)dR', \quad R > 0, \Im z \geq 0, z \neq 0, \label{f_+ definition}
\end{equation}
where $\mathcal{L} = -\partial_{RR} + V(R)$. 
% This fixed point satisfies the following upper bound:
% \begin{align}
% |f_+(R,z)| &\leq Ce^{-R \Im \sqrt{z}} \exp \left(|z|C \int_R^{+\infty} |V(R')| \frac{R'}{1+R'|\sqrt{z}|} dR' \right) \label{f_+ upper bound} \\
% &\leq Ce^{-R \Im \sqrt{z}} \exp \left(\tilde{C}\frac{|z|}{R|\sqrt{z}|}\right), \quad R > 0, \Im z > 0 \notag
% \end{align}
For its construction, see \cite[Section 12.1.1]{Newton}.
\index{psi plus@$\psi^+(R,\xi)$, Jost solution for $\mathcal{L}u=zu$}
\end{definition}

% \begin{remark}
% If $0 < \xi \leq 1$ and $0 < R\xi^{\frac{1}{2}} \leq 1$ differentiating (\ref{f_+ definition}) with respect to $R$, we can deduce other bounds such as
% \begin{equation}
%     |\partial_R f_+(R, \xi)| \leq \xi^{\frac{1}{2}} + \frac{C}{R} \label{partial_R f_+ upper bound}
% \end{equation}
% \end{remark}

Next, we give an approximation for $\psi^+(R,\xi)$ which is useful when $R^2 \xi \gtrsim 1$:

\begin{proposition}\label{prop:psi^+ in terms of sigma}
For $\xi > 0$, $R^2 \xi \gtrsim 1$, $\psi^+(R,\xi)$ is of the form
$$\psi^+(R, \xi) = \xi^{-\frac{1}{4}}e^{iR\xi^{\frac{1}{2}}} \sigma(R\xi^{\frac{1}{2}},R)$$
where $\sigma(q,r)$ is well-approximated by the series
$$\sigma(q,r) \approx \sum_{j=0}^{+\infty}q^{-j}\psi_j^+(r)$$
for some zeroth-order symbol $\psi_j^+$ being analytic on $(0,+\infty]$, i.e.,
$$\sup_{r > 0}|(r\partial_r)^k \psi_j^+(r)| < +\infty \ \forall k \in \mathbb N_{\geq 0},$$
in the following sense:
$$\sup_{r > 0} \left| (r \partial_r)^{\alpha} (q \partial_q)^{\beta} \left[ \sigma(q,r) - \sum_{j=0}^{j_0}q^{-j}\psi_j^+(r) \right] \right| \leq c_{\alpha, \beta, j_0} q^{-j_0-1}, \quad \forall q \geq 1,$$
for any $\alpha, \beta \in \mathbb N_{\geq 0}$, for any $j_0$ large enough.
\end{proposition}

\begin{proof}
See \cite[Proposition 4.6]{Krieger_2007}.
\end{proof}

\begin{remark}\label{rmk: boundedness of sigma}
    In particular,
    $$\sup_{r > 0, q > 1}| (r \partial_r)^{\alpha} (q \partial_q)^{\beta}\sigma(q,r)| < +\infty \ \forall \alpha,\beta \in \mathbb N_{\geq 0}.$$
\end{remark}

\begin{corollary}\label{cor: psi^+ estimate}
For all $R, \xi > 0$, $R^2 \xi \gtrsim 1$, the following pointwise estimates hold for $\psi^+(R,\xi)$:
\begin{align}
    |(\xi^l\partial_{\xi}^l) (R^k \partial_R^k) \psi^+(R,\xi)| &\lesssim \xi^{-\frac{1}{4}} (R\xi^{\frac{1}{2}})^{l+k} \quad \forall k, l \geq 0. \label{psi_+ upper bound}
\end{align}
\end{corollary}

\begin{proof}
Observe that
\begin{align*}
     (\xi\partial_{\xi})F(R\xi^{\frac{1}{2}},R) &= \frac{1}{2} R\xi^{\frac{1}{2}}\partial_q F(R\xi^{\frac{1}{2}},R) = G(R\xi^{\frac{1}{2}},R), \quad G(q,r) = \frac{1}{2} q \partial_q F, \\
  (R\partial_{R})F(R\xi^{\frac{1}{2}},R) &= R\xi^{\frac{1}{2}}\partial_q F(R\xi^{\frac{1}{2}},R) + R \partial_R F(R\xi^{\frac{1}{2}},R) = H(R\xi^{\frac{1}{2}},R), \quad H = (q \partial_q  + r \partial_r) F,
\end{align*}
so that, by induction,
$$
 |(\xi\partial_{\xi})^l (R \partial_R)^k \sigma(R\xi^{\frac{1}{2}},R)| \lesssim  \sup_{\alpha \leq l, \beta \leq k} \sup_{r > 0, q > 1}| (r \partial_r)^{\alpha} (q \partial_q)^{\beta}\sigma(q,r)| < +\infty
$$
The same inequality holds with $(\xi^l \partial_{\xi}^l) (R^k \partial_R^k)$, as it is a linear combination of the differential operators $(\xi\partial_{\xi})^i (R \partial_R)^j$ for $i \leq l, j \leq k$. One also checks that
\begin{align*}
    (R^k\partial_R^k)(\xi^{-\frac{1}{4}}e^{iR\xi^{\frac{1}{2}}}) &= i^ke^{iR\xi^{\frac{1}{2}}} \xi^{-\frac{1}{4}}(R\xi^{\frac{1}{2}})^k, \\
    \left| (\xi^l \partial_{\xi}^l) (\xi^{\alpha} e^{iR\xi^{\frac{1}{2}}}) \right| &\lesssim \xi^{\alpha}(R\xi^{\frac{1}{2}})^l, \\
    \left|(\xi^l \partial_{\xi}^l)  [(R^k\partial_R^k)(\xi^{-\frac{1}{4}}e^{iR\xi^{\frac{1}{2}}})] \right| &\lesssim \xi^{-\frac{1}{4}}(R\xi^{\frac{1}{2}})^{k+l},
\end{align*}
which yields the result $\psi^+(R,\xi) = \xi^{-\frac{1}{4}}e^{iR\xi^{\frac{1}{2}}} \sigma(R\xi^{\frac{1}{2}},R)$ using the product rule.
\end{proof}

Now, we are ready to give growth estimates on the spectral density $\rho(\xi)$.

\begin{proposition}\label{prop:spectral density properties}
For $R > 0$ and $\xi > 0$, we have
\begin{equation*}
    \phi(R,\xi) = a(\xi)\psi^{+}(R,\xi) + \overline{a(\xi)\psi^{+}(R,\xi)},
\end{equation*}
where $a(\xi)$ is smooth, non-zero, has asymptotics
$$|a(\xi)| \asymp \begin{cases} \xi^{\frac{6-d}{4}} \cdot |\log(\xi)|^{\delta_{d=4}}, \ &\xi \ll 1 \\
\xi^{\frac{2-d}{4}}, \ &\xi \gg 1
\end{cases}$$
and symbol-type upper bounds
$$|(\xi \partial_{\xi})^k a(\xi)| \leq c_k|a(\xi)| \ \forall \xi > 0.$$
Moreover,
$$\rho(\xi) = \frac{1}{\pi |a(\xi)|^2}$$
and the corresponding asymptotics are
\begin{equation}
    |\rho(\xi)| \asymp \begin{cases} \xi^{\frac{d}{2} - 3} \cdot |\log(\xi)|^{-2 \delta_{d = 4}}, \ &\xi \ll 1 \\
 \xi^{\frac{d}{2} - 1}, \ &\xi \gg 1
\end{cases}
\end{equation}
with symbol-type upper bounds 
$$
|(\xi \partial_{\xi})^k \rho(\xi)| \leq \tilde{c}_k|\rho(\xi)| \ \forall \xi > 0.
$$
\index{rho@$\rho(\xi)d\xi$, the absolutely continuous part of the spectral measure $d\rho(\xi)$}
\end{proposition}

\begin{proof}
Following \cite[Proposition 4.7]{Krieger_2007}, we find that
\begin{align}
    a(\xi) &= -\frac{i}{2}W(\phi(\cdot, \xi), \psi^{+}(\cdot, \xi)), \label{a(xi) upper bound} \\
    |a(\xi)| &\geq \frac{|\partial_R \phi(R,\xi)|}{2|\partial_R \psi^{+}(R,\xi)|}, \label{a(xi) lower bound} \\
    \rho(\xi) &= \frac{1}{\pi |a(\xi)|^2}. \notag
\end{align}
The behaviour of $\rho(\xi)$ for large $\xi$ is well-known: see \cite[Theorem 2.1]{Kostenko_2013}.  

For small $\xi$, we proceed as in \cite[Proposition 4.7]{Krieger_2007}. We take $R = \delta \xi^{-\frac{1}{2}}$ as in (\ref{phi lower bound}) ($\delta$ is fixed and $\xi \to 0^+$) so that one gets
\begin{align*}
    |(R \partial_R)^i \phi(R, \xi)| &\sim \xi^{\frac{5-d}{4}} \cdot \log(1 + \xi^{-1})^{\delta_{d = 4}}, \\
    |(R \partial_R)^i \psi^+(R,\xi)| &\lesssim \xi^{-\frac{1}{4}},
\end{align*}
for $i = 0, 1$ using (\ref{phi(R,z)}), (\ref{phi lower bound}) and (\ref{psi_+ upper bound}). We conclude by applying these estimates with (\ref{a(xi) upper bound}) and (\ref{a(xi) lower bound}).
\end{proof}

\begin{corollary} \label{phi upper bound R^2z > 1}
When $R^2 \xi \gtrsim 1$, the following pointwise estimates hold for $\phi(R,\xi)$ for any $k, l \geq 0$:
\begin{align*}
    |\xi^l\partial_{\xi}^l R^k \partial_R^k \phi(R,\xi)| &\lesssim \xi^{\frac{5-d}{4}} (R\xi^{\frac{1}{2}})^{l+k} \langle \xi \rangle^{-1} \cdot \left( 1_{0 < \xi < 1/2}(\xi) |\log(\xi)| \right)^{\delta_{d = 4}}. 
%    &\lesssim \xi^{\frac{1-d}{4}} (R\xi^{\frac{1}{2}})^{l+k} \lesssim R^{\frac{d-1}{2}} (R\xi^{\frac{1}{2}})^{l+k} \notag
\end{align*}
The logarithm appears only in dimension $d = 4$ for small $\xi > 0$. In particular, in dimension $d \in \{4,5\}$, one has
$$|\phi(R,\xi)| \lesssim \langle \xi \rangle^{\frac{1-d}{4}}, \quad R^2 \xi \gtrsim 1.$$
\end{corollary}

\begin{proof}
    Write $\phi = 2 \Re( a(\xi) \psi^+(R,\xi))$ and use the estimates from Corollary \ref{cor: psi^+ estimate}, as well as 
    $$
    |(\xi \partial_{\xi})^k a(\xi)| \leq c_k|a(\xi)| \lesssim \xi^{\frac{6-d}{2}} \cdot \left( 1_{0 < \xi < 1/2}(\xi) |\log(\xi)| \right)^{\delta_{d = 4}}, \quad \xi > 0
    $$
    from Proposition \ref{prop:spectral density properties}.
\end{proof}

% \begin{corollary}
% When $R^2 \xi \gtrsim 1$ and $l \geq 0$, one has 
% \begin{align*}
%     |\partial_{\xi}^l\phi(R,\xi)| &\lesssim \xi^{\frac{5-d}{4}} (R\xi^{-\frac{1}{2}})^{l} \langle \xi \rangle^{-1} (\cdot |\log(\xi)| \text{ if } d = 4, 0 < \xi < 1/2) \\
%     &\lesssim \xi^{\frac{1-d}{4}} (R\xi^{-\frac{1}{2}})^{l} \\
%     |\partial_{\xi}^l\phi(R,\xi)| &\lesssim \xi^{\frac{3-d}{4}}R^{2(l-1)+1} (\cdot |\log(\xi)| \text{ if } d = 4, 0 < \xi < 1/2)  \\
%     &\lesssim R^{\frac{d-1}{2}+2(l-1)}  (\cdot  \log(1+R^2) \text{ if } d = 4, 0 < \xi < 1/2) \\
%     &\lesssim R^{\frac{d-1}{2}+2l}
% \end{align*}
% Similarly, when $k \geq 0$, $\partial_{\xi}^l (R^k \partial_R^k) \phi(R,\xi)$ has the same upper bound as $\partial_{\xi}^l\phi(R,\xi)$ multiplied by a factor $(R\xi^{\frac{1}{2}})^k$
% \end{corollary}

% \begin{proof}
%     The first two inequalities follow trivially from (\ref{phi upper bound R^2z > 1}). The next ones follow from (\ref{phi upper bound R^2z > 1}) by writing $(R\xi^{-\frac{1}{2}})^l = R\xi^{-\frac{1}{2}} R^{2(l-1)}(R\xi^{\frac{1}{2}})^{-(l-1)}$.
% \end{proof}

\begin{corollary} \label{phi upper bound}
Assume $d \in \{4,5\}$. Fix $0 < \xi_0 < 1$. For all $R, \xi > 0$, the following pointwise estimates hold for $\phi(R,\xi)$ when $l \geq 1$:
\begin{align*}
|\phi(R,\xi)| &\lesssim \langle \xi \rangle^{\frac{1-d}{4}}, \\
|R\partial_R \phi(R,\xi)| &\lesssim \min\{R\xi^{\frac{3-d}{2}},R^{\frac{d-1}{2}}\} \text{ if } \xi > 1,\\
|\partial_{\xi}^l \phi(R,\xi)| &\lesssim \min\{R^{\frac{d-1}{2}+2l},\xi^{\frac{1-d}{4}}(R\xi^{-\frac{1}{2}})^l\} \text{ if } \xi > \xi_0,  \\
    |\partial_{\xi}^l \phi(R,\xi)| &\lesssim \min\{ R^{\frac{d-1}{2}+2(l-1) }\log(1+R^2)^{\delta_{d = 4}},\xi^{\frac{5-d}{4}}(R\xi^{-\frac{1}{2}})^l |\log(\xi)|^{\delta_{d = 4}} \} \text{ if } \xi < \xi_0.
\end{align*}
\end{corollary}

\begin{proof}
    This is a combination of Corollary \ref{phi upper bound R^2z < 1} and Corollary \ref{phi upper bound R^2z > 1}.
\end{proof}

\begin{corollary} \label{F(xi,eta) estimates}
Assume $d \in \{4,5\}$. Let 
\begin{align*}
    W_0(R) &= [\mathcal{L},R\partial_R]-2\mathcal{L} = [V,R\partial_R]-2V(R) = -2V(R) - R\partial_RV(R) \\
    &= \frac{2d^2\left(d^2-4\right)\left((d-2)d-R^2\right)}{\left((d-2)d+R^2\right)^3},
\end{align*}
where $\mathcal{L} = -\partial_{RR} + V(R)$ (not to be confused with the ground state $W(x)$). The symmetric function
    $$|F(\xi,\eta)| = \innerp[\Big]{ W_0(R) \phi(R,\xi) }{ \phi(R,\eta) }_{L^2((0,+\infty))}$$
    is of class $C^1([0,+\infty)\times[0,+\infty)) \cap C^2((0,+\infty)\times(0,+\infty))$ and satisfies:
    \begin{align*}
        |F(\xi,\eta)| &\lesssim \begin{cases}
            \xi + \eta \quad &\text{ if } \xi + \eta \leq 1 \\
             (\xi+\eta)^{\frac{1-d}{2}}(1+|\xi^{\frac{1}{2}}-\eta^{\frac{1}{2}}|)^{-N} \quad &\text{ if } \xi + \eta \geq 1 
        \end{cases} \\
        |\partial_{\xi}F(\xi,\eta)| + |\partial_{\eta}F(\xi,\eta)|&\lesssim \begin{cases}
            1 \quad &\text{ if } \xi + \eta \leq 1 \\
             (\xi+\eta)^{\frac{-d}{2}}(1+|\xi^{\frac{1}{2}}-\eta^{\frac{1}{2}}|)^{-N} \quad &\text{ if } \xi + \eta \geq 1 
        \end{cases} \\
        |\partial_{\xi}\partial_{\eta}F(\xi,\eta)| &\lesssim \begin{cases}
            |\log(\xi + \eta)|^3 \quad &\text{ if } \xi + \eta \leq 1, \ d = 4 \\
           |\xi + \eta|^{-\frac{1}{2}}(1 + |\log(\xi/\eta)|) \quad &\text{ if } \xi + \eta \leq 1, \ d = 5 \\
             (\xi+\eta)^{\frac{-1-d}{2}}(1+|\xi^{\frac{1}{2}}-\eta^{\frac{1}{2}}|)^{-N} \quad &\text{ if } \xi + \eta \geq 1 
        \end{cases} 
    \end{align*} 
    for any fixed $N \in \mathbb N$.
    \index{F(xi,eta)@$F(\xi,\eta)$, a symmetric function which contributes to $\mathcal{K}$}
\end{corollary}

\begin{proof}
    The bounds from Corollary \ref{phi upper bound} imply
    \begin{align*}
        |F(\xi,\eta)| &\lesssim \langle \xi \rangle^{\frac{1-d}{4}} \langle \eta \rangle^{\frac{1-d}{4}}, \\
        |\partial_{\xi}F(\xi,\eta)| &\lesssim \langle \xi \rangle^{\frac{-1-d}{4}} \langle \eta \rangle^{\frac{1-d}{4}}, \\
        |\partial^2_{\xi \eta}F(\xi,\eta)| &\lesssim \xi^{\frac{-1-d}{4}}\eta^{\frac{-1-d}{4}} \text{ if } \xi > \frac{1}{2}, \eta > \frac{1}{2},
        % |\partial^2_{\xi}F(\xi,\eta)| &\lesssim \xi^{\frac{-3-d}{4}}\eta^{\frac{1-d}{4}} \text{ if } \xi > \frac{1}{2}, \eta > \frac{1}{2}
    \end{align*}
    which yield the desired bounds when $\xi \sim \eta$, $\xi + \eta \geq 1$. The same bounds from Corollary \ref{phi upper bound} combined with Dominated Convergence proves the $C^2$-regularity of $F(\xi,\eta)$, as well as the right-continuity at zero of $F(\xi,\eta)$, $\partial_{\xi}F$ and $\partial_{\eta}F$.

If $\xi + \eta \geq 1$ and $\xi, \eta$ are separated, then we proceed as in \cite[Theorem 5.1]{Krieger_2007}: Writing the integrals as limits, an integration by parts argument shows that 
\begin{equation*}
   \eta F(\xi,\eta)  = \innerp[\Big]{ W_0(R)\phi(R,\xi)}{\mathcal{L}\phi(R,\eta) }_{L^2_R} = \innerp[\Big]{ [\mathcal{L},W_0(R)]\phi(R,\xi)}{\phi(R,\eta) }_{L^2_R} + \xi F(\xi,\eta),
\end{equation*}
i.e.,
\begin{equation*}\label{integration by parts F(xi,eta)}
   (\eta-\xi) F(\xi,\eta)  = - \innerp[\Big]{ (2W_{0,R}(R)\partial_R - W_{0,RR}(R))\phi(R,\xi)}{\phi(R,\eta) }_{L^2_R}.
\end{equation*}
For fixed $\xi, \eta \geq 0$, the integration by parts is justified because $\phi(R,\eta)$, $\partial_R\phi(R,\eta)$ are bounded and $W_0(R)\phi(R,\xi)$, $\partial_{R}(W_0(R)\phi(R,\xi))$ vanish at zero and infinity.

By iteration, for arbitrary $k \in \mathbb N$, there exists rational functions $W_j^{\mathrm{odd}}(R)$, $W_j^{\mathrm{even}}(R) \in C^{\omega}([0,+\infty))$, $0 \leq j \leq k$, decaying as
$$\langle R \rangle |W_j^{\mathrm{odd}}(R)| + |W_j^{\mathrm{even}}(R)| \lesssim \langle R \rangle^{4-2k},$$
and respectively having odd/even expansions at $R = 0$, for which
\begin{equation}\label{iteration of integration by parts F(xi,eta)}
    (\eta-\xi)^{2k}F(\xi,\eta) = \innerp[\Big]{ \Bigl(\sum_{j=0}^{k-1}\xi^j W_j^{\mathrm{odd}}(R)\partial_R + \sum_{j=0}^{k}\xi^j W_j^{\mathrm{even}}(R) \Bigr) \phi(R,\xi) }{ \phi(R,\eta) }_{L^2_R}.
\end{equation}
This implies the desired bound for $F(\xi,\eta)$ and then differentiating with respect to $\xi$ and/or $\eta$ implies the other ones. If $\xi + \eta \lesssim 1$, one observes that $F(0,0) = 0$ because $$W_0(R) \phi(R,0) = \left([\mathcal{L},R\partial_R]-2\mathcal{L}\right)\phi(R,0) = -\mathcal{L}R\partial_R \phi(R,0),$$
    thus integration by parts yields
    $$F(0,0) = \innerp[\Big]{ -R\partial_R\phi(R,0) }{ \mathcal{L}\phi(R,0) }_{L^2_R} = 0.$$
    Combining this with the differentiability at zero, we obtain the bounds for $F(\xi,\eta)$, $\partial_{\xi}F$ and $\partial_{\eta}F$. It remains to prove the estimates for the second derivative. Since the case $d = 4$ has been treated in \cite[Theorem 5.1]{Krieger_2007}, we assume $d = 5$ but the strategy is the same. Using Corollary \ref{phi upper bound}, for $0 < \eta \leq \xi < 1/2$, one gets 
    \begin{align*}
        |\partial_{\xi}\partial_{\eta}F(\xi,\eta)| &\lesssim \int_{0}^{+\infty}\langle R \rangle^{-4} |\partial_{\xi} \phi(R,\xi)\partial_{\eta} \phi(R,\eta)| dR \\
        &\lesssim \int_{0}^{\xi^{-\frac{1}{2}}}\langle R \rangle^{-4} R^{4}dR +  \int_{\xi^{-\frac{1}{2}}}^{\eta^{-\frac{1}{2}}}\langle R \rangle^{-4} R^2 (R\xi^{-\frac{1}{2}}) dR  \\
        &\phantom{\lesssim}+ \int_{\eta^{-\frac{1}{2}}}^{+\infty} \langle R \rangle^{-4}(R\xi^{-\frac{1}{2}})  (R\eta^{-\frac{1}{2}})dR \\
        &\lesssim \xi^{-\frac{1}{2}} \cdot (1 +|\log(\xi/\eta)|) \lesssim (\xi+\eta)^{-\frac{1}{2}} \cdot (1 + |\log(\xi/\eta)|) 
    \end{align*}
    and we conclude by symmetry. Similarly, we can show that the derivatives $\partial_{\xi}^2F$ and $\partial_{\eta}^2F$ exist if $0 < \eta, \xi \leq 1/2$, but we do not need to estimate them.
\end{proof}

\section{The Transference Identity}\label{section:transferrance identity}

Assume $d \in \{4,5\}$ for the remainder of this paper. We are interested in studying the error one makes when passing from $\mathcal{F}(R \partial_R u)$ to $-2\xi\partial_{\xi}\mathcal{F}(u)$, as the operator $R \partial_R$ appears naturally when solving for $\varepsilon$ in Section \ref{section:exact solution fourier method}. This will allow us to translate the wave equation from physical to (generalized) Fourier space where the operator $\mathcal{L}$ is replaced by a multiplication by $\xi$ and one can ``interchange'' the operators $R \partial_R$ and $\mathcal{F}$ up to a controllable error. 
 
Let $L^2(\mathbb R, d\rho)$ denote the set of $d\rho$-measurable functions $f(\xi)$ that are square-integrable. These functions admit the following representation $d\rho$ almost-everywhere:
$$
f(\xi) = f(\xi_d) \delta_{\xi_d} + f(0) \delta_{0, d = 5} + f(\xi)1_{\xi > 0}(\xi),
$$
where $f_c(\xi):= f(\xi)1_{\xi > 0}(\xi) \in L^2((0,+\infty),\rho(\xi)d\xi)$. Let $C_c^{\infty}(\spec_{pp}(\mathcal{L}) \cup (0,+\infty))$ denote the subset of $L^2(\mathbb R, d\rho)$ for which $f_c(\xi) = f(\xi)1_{\xi > 0}(\xi) \in C_c^{\infty}((0,+\infty))$.  Our goal is to study the difference operator $\mathcal{K}$:
\begin{align*}
    \mathcal{K}:   C_c^{\infty}(\spec_{pp}(\mathcal{L}) \cup (0,+\infty)) &\mapsto L^2(\mathbb R,d\rho) \\
    f(\xi)  &\mapsto \mathcal{K}(f):=  \mathcal{F}(R \partial_R \mathcal{F}^{-1}f) + \mathcal{F}(\mathcal{F}^{-1} 2\xi \partial_{\xi} f),
\end{align*}
\index{K@$\mathcal{K}$, the transference operator}
where $\xi\partial_{\xi}$ acts as zero on the discrete component, and show that this is a well-defined bounded operator inbetween some weighted $L^2$-spaces.  As a first step, we are going to show that $\mathcal{K}$ is well-defined. 

\begin{proposition}
    The operator $\mathcal{K}$ is well-defined.
\end{proposition}

\begin{proof}
Denoting by $\phi_d(R)$ and $\phi_0(R)$ the normalized eigenfunctions, observe that
\begin{align*}
    \mathcal{F}^{-1} f &= f(\xi_d)  \mathcal{F}^{-1} \delta_{\xi_d} + f(0) \mathcal{F}^{-1} \delta_{0, d = 5} + \mathcal{F}^{-1} f_c  \\
    &= f(\xi_d) \phi_d(R) + f(0) \phi_0(R) \delta_{d = 5} + \mathcal{F}^{-1} f_c, \\
     \mathcal{F} (R \partial_R \mathcal{F}^{-1} f) &= f(\xi_d) \mathcal{F} (R \partial_R \phi_d) + f(0) \mathcal{F} (R \partial_R \phi_0)\delta_{d = 5} + \mathcal{F} (R \partial_R \mathcal{F}^{-1} f_c), \\
     \mathcal{F}(\mathcal{F}^{-1} 2\xi \partial_{\xi} f) &=  2\xi \partial_{\xi} f_c,
\end{align*}
where we recall that we write
    $$
    (\mathcal{F}f)(\xi) = \innerp[\Big]{ f(s) }{ \phi(s,\xi)}_{L^2((0,+\infty))}:= \lim\limits_{r \to +\infty} \int_0^r \phi(s,\xi)f(s)ds
    $$
   as a limit of functions in $L^2(\mathbb R, d\rho)$,  even though
    $$\int_0^{+\infty} |\phi(s,\xi)f(s)|ds$$
    need not to be finite.  In order to show that $\mathcal{K}$ is well-defined, it is necessary to show that $\mathcal{F}(R\partial_R \phi_d(R))$, $\mathcal{F}(R\partial_R \phi_0(R))$ and $\mathcal{F}(R\partial_R \mathcal{F}^{-1})$ are well-defined. To this end, it suffices to show that $R\partial_R \phi_d(R)$, $R\partial_R \phi_0(R)$ and $R\partial_R \mathcal{F}^{-1} f$ are $L^2((0,+\infty))$-functions in Lebesgue sense (Theorem \ref{spectral theorem}). Jost solution theory shows that $\phi_d(R)$ decays exponentially as $R \to +\infty$ and, by (\ref{phi theta asympotics}), $\phi_0(R)$ decays as $R^{-1}$ in dimension $d = 5$ where it appears with symbol-type behaviour of the derivatives. Hence, $R\partial_R \phi_d(R)$ and  $R\partial_R \phi_0(R)$  are $L^2$-functions. If $f \in C_c^{\infty}((0,+\infty))$, then our estimates on $R\partial_R\phi(R,\xi)$ from Corollary \ref{phi upper bound} shows that $R\partial_R\mathcal{F}^{-1}f$ is, a priori, only bounded. Yet, $\mathcal{F}^{-1}f$ decays like a Schwartz function: one can get rid of powers of $R$ using successive integration by parts as we show in Lemma \ref{lemma:F^-1 f is schwartz if f is smooth with compact support}.
\end{proof}

\begin{lemma}\label{lemma:F^-1 f is schwartz if f is smooth with compact support}
    Let $f \in C_c^{\infty}((0,+\infty))$. Then $\mathcal{F}^{-1}f(R)$ has an arbitrary fast polynomial decay at infinity. 
\end{lemma}

\begin{proof}
    Let $K \subset (0,+\infty)$ be the compact support of $f$ and
    $$
\mathcal{F}^{-1} f(R) = \int_{0}^{+\infty} \phi(R,\xi)f(\xi)d\rho(\xi) = \int_{K} \phi(R,\xi)f(\xi)\rho(\xi)d\xi,
    $$
    where $\rho(\xi)$ is smooth on $(0,+\infty)$ (Corollary \ref{prop:spectral density properties}). As $\phi(R,\xi)$ is smooth on $(0,+\infty) \times (0,+\infty)$, $K \subset (0,+\infty)$ is compact, one can interchange derivative and integral thanks to Dominated Convergence, i.e.,
    \begin{align*}
            (R\partial_R)^n\mathcal{F}^{-1} f(R) &=  \int_{K} (R\partial_R)^n\phi(R,\xi)f(\xi)\rho(\xi)d\xi \\
            &= 2 \Re \left(  \int_{K} a(\xi) (R\partial_R)^n\left[ \psi^+(R,\xi) \right] f(\xi)\rho(\xi)d\xi \right). 
    \end{align*}
    Observe that
 \begin{align*}
           (R\partial_{R})e^{iR\xi^{\frac{1}{2}}} F_0(R\xi^{\frac{1}{2}},R) &= e^{iR\xi^{\frac{1}{2}}} \left[ iR \xi^{\frac{1}{2}}F(R\xi^{\frac{1}{2}},R) + R\xi^{\frac{1}{2}}\partial_q F(R\xi^{\frac{1}{2}},R) + R \partial_R F(R\xi^{\frac{1}{2}},R) \right] \\
      &= e^{iR\xi^{\frac{1}{2}}} H(R\xi^{\frac{1}{2}},R), \quad H = (iq + q \partial_q  + r \partial_r) F_0.
 \end{align*}
If $F_0(q,r) = \sigma(q,r)$ is the function coming from $\psi^+$, then $(q \partial_q)^i(r \partial_r)^{j} \sigma$ is bounded on $(q,r) \in [1,+\infty) \times (0,+\infty)$ for all fixed $i,j \geq 0$ (Remark \ref{rmk: boundedness of sigma}). Hence,
 \begin{align*}
           (R\partial_{R})e^{iR\xi^{\frac{1}{2}}} \sigma(R\xi^{\frac{1}{2}},R) &= e^{iR\xi^{\frac{1}{2}}} H(R\xi^{\frac{1}{2}},R), \quad H = iqF_0 + F_1, \quad F_1 = (q \partial_q  + r \partial_r) F_0,
 \end{align*}
where $(q \partial_q)^i(r \partial_r)^{j} F_1$ is bounded on $(q,r) \in [1,+\infty) \times (0,+\infty)$. By induction, it holds that 
\begin{align*}
           (R\partial_{R})^ne^{iR\xi^{\frac{1}{2}}} \sigma(R\xi^{\frac{1}{2}},R) &= e^{iR\xi^{\frac{1}{2}}} H(R\xi^{\frac{1}{2}},R), \\
           H &= (iq)^nF_0 + (iq)^{n-1}F_1 + ... + (iq)^{n-1}F_{n-1} + F_n,
 \end{align*}
 where $(q \partial_q)^i(r \partial_r)^{j} F_k$ is bounded on $(q,r) \in [1,+\infty) \times (0,+\infty)$ for all $i,j \geq 0$ and all $0 \leq k \leq n$.  It is now sufficient to prove that for any smooth $F$ as above, there exists $C = C(F,f,a,\rho,K,n) > 0$ for which
 $$
\sup_{R > 0} \left| \int_{K} e^{iR\xi^{\frac{1}{2}}} (iR\xi^{\frac{1}{2}})^n F(R\xi^{\frac{1}{2}},R) a(\xi)f(\xi)\rho(\xi)d\xi \right| \leq C.
 $$
 Write $a(\xi) f(\xi) \rho(\xi) = \tilde{f}(\xi)$. First, observe that
 $$
(\xi \partial_{\xi})^n \left[ G(R\xi^{\frac{1}{2}}) \right] = \frac{1}{2^n} [(q\partial_q)^n G](R\xi^{\frac{1}{2}}).
 $$
 Applying this identity with $G(q) = e^{iq}$, one obtains
 $$
(\xi \partial_{\xi})^n \left[  e^{iR\xi^{\frac{1}{2}}} \right] = \frac{i^n}{2^n} (R\xi^{\frac{1}{2}})^n e^{iR\xi^{\frac{1}{2}}}. 
 $$
 Hence,
\begin{align*}
  \int_{K} e^{iR\xi^{\frac{1}{2}}} (iR\xi^{\frac{1}{2}})^{n+1} F(R\xi^{\frac{1}{2}},R) \tilde{f}(\xi) d\xi \\
  = c_n \xi \int_{0}^{+\infty} (\xi \partial_{\xi})^{n+1} \left[  e^{iR\xi^{\frac{1}{2}}} \right] F(R\xi^{\frac{1}{2}},R)  \tilde{f}(\xi) d\xi \\
  =  c_n \xi \int_{0}^{+\infty} e^{iR\xi^{\frac{1}{2}}} (-1-\xi\partial_{\xi})^{n+1} \left[ F(R\xi^{\frac{1}{2}},R)  \tilde{f}(\xi) \right] d\xi.
\end{align*}
As 
$$
(\xi\partial_{\xi})^i \left[ F(R\xi^{\frac{1}{2}},R)  \tilde{f}(\xi) \right] = \sum_{i_1 + i_2 = i} \binom{i}{i_1,i_2} (\xi\partial_{\xi})^{i_1} \tilde{f}(\xi) \frac{1}{2^{i_2}} [(q\partial_q)^{i_2} F](R\xi^{\frac{1}{2}},R)
$$
is bounded for $\xi \in K, R  >0$, this finishes the proof.
\end{proof}

  Our next goal is to prove boundedness of $\mathcal{K}$ on some appropriate weighted $L^2$-spaces by finding and analysing its kernel. Representing $L^2(\mathbb R, d\rho)$ as $\mathbb R^{d-3} \times L^2(\mathbb R, \rho(\xi)d\xi)$ using the natural map
$$
f(\xi) = f(\xi_d) \delta_{\xi_d} + f(0) \delta_{0, d = 5} + f_c(\xi) \mapsto (f(\xi_d),f(\xi_0),f_c(\xi)),
$$
observe that
\begin{align*}
     \mathcal{K}f &= f(\xi_d) \mathcal{F} (R \partial_R \phi_d) + f(0) \mathcal{F} (R \partial_R \phi_0)\delta_{d = 5} + \mathcal{F} (R \partial_R \mathcal{F}^{-1} f_c)  + 2\xi\partial_{\xi}f_c \\
     &= f(\xi_d) \mathcal{K}_d + f(0) \mathcal{K}_0 \delta_{d = 5} + \mathcal{K}_c(f_c) \\
     &= (\mathcal{K}_d, \mathcal{K}_0, \mathcal{K}_c) \cdot (f(\xi_d), f(0), f_c),
\end{align*}  
where ``d'' represents the negative discrete eigenvalue $\xi_d$, ``0'' stands for the 0-eigenvalue when $d = 5$ and ``c'' represents the continuous part of the spectrum. Extracting the discrete and continuous components from each Fourier transform (the discrete component is obtained by evaluating the Fourier transform at $\xi = \xi_d$ and $\xi = 0$), the operator $\mathcal{K}$ admits, respectively in $d = 4$ and $d = 5$, the following matrix representation
$$\begin{pmatrix} \mathcal{K}_{dd} & \mathcal{K}_{cd} \\
\mathcal{K}_{dc} & \mathcal{K}_{cc} \\
\end{pmatrix}, \quad \begin{pmatrix} \mathcal{K}_{dd} & \mathcal{K}_{0d} & \mathcal{K}_{cd} \\
\mathcal{K}_{d0} & \mathcal{K}_{00} & \mathcal{K}_{c0} \\
\mathcal{K}_{dc} & \mathcal{K}_{0c} & \mathcal{K}_{cc}
\end{pmatrix},$$
where for $x \in \{ ``d``,  ``c`` ,  ``0``\}$, $K_{xd}(\cdot)$ represents the evaluation of $\mathcal{K}_x(\cdot)(\xi)$ at $\xi = \xi_d$, $K_{x0}$ represents the value at $\xi = 0$ and $K_{xc}$ represents the continuous component of the transform. More precisely, 
\begin{alignat*}{2}
    \mathcal{K}_{dd} &= \Bigl< R\partial_R \phi_d(R),\phi_d(R) \Bigr>_{L^2_R}, \quad \mathcal{K}_{0d} = \Bigl< R\partial_R \phi_0(R),\phi_d(R) \Bigr>_{L^2_R}, \\
    \mathcal{K}_{cd} &= \Bigl< \int_0^{+\infty}f(\xi)R\partial_R\phi(R,\xi)\rho(\xi)d\xi ,\phi_d(R) \Bigr>_{L^2_R}, \\
     \mathcal{K}_{dc} &= \Bigl< R\partial_R \phi_d(R),\phi(R,\eta) \Bigr>_{L^2_R}, \quad \mathcal{K}_{0c} = \Bigl< R\partial_R \phi_0(R),\phi(R,\eta) \Bigr>_{L^2_R}, \\
    \mathcal{K}_{cc} &= \Bigl< \int_0^{+\infty}\left[ f(\xi)R\partial_R\phi(R,\xi) + 2 \xi\partial_\xi f(\xi) \phi(R,\xi)\right] \rho(\xi)d\xi ,\phi(R,\eta) \Bigr>_{L^2_R},
\end{alignat*}
and similarly for $\mathcal{K}_{d0}, \mathcal{K}_{00}, \mathcal{K}_{c0}$. We remark once again that some of these inner products only make sense as a limit of $L^2(\mathbb R, d\rho)$-functions (see Theorem \ref{spectral theorem} and Notation \ref{not: notation of inner product for L^2 limit}). We start by making $\mathcal{K}_{d}$ and $\mathcal{K}_0$ more explicit.

\begin{proposition}
    One has
    $$\mathcal{K}_{dd} = \mathcal{K}_{00} = -\frac{1}{2}, \quad \mathcal{K}_{0d} = \mathcal{K}_{d0} \in \mathbb R,$$
    as well as,
    $$
\mathcal{K}_{0c}(\eta) = \frac{F(0,\eta)}{\eta}, \quad \mathcal{K}_{dc}(\eta) = ||\phi(R,\xi_d)||_{L^2_R} \cdot \frac{F(\xi_d,\eta)}{\eta-\xi_d}, \quad \eta \geq 0.
    $$
    where $F$ is as in Corollary \ref{F(xi,eta) estimates}. Moreover, $\mathcal{K}_{0c}$, $\eta\partial_{\eta}\mathcal{K}_{0c}$, $\mathcal{K}_{dc}$ and $\eta\partial_{\eta}\mathcal{K}_{dc}$ are continuous on $[0,+\infty)$ and have an arbitrary fast polynomial decay at infinity.
\end{proposition}

    \begin{proof}

Jost solution theory shows that $\phi_d(R)$ decays exponentially as $R \to +\infty$ and $\phi_0(R)$ decays as $R^{-1}$ in dimension $d = 5$ by (\ref{phi theta asympotics}) . Hence,
$$\mathcal{K}_{dd} = \mathcal{K}_{00} = -\frac{1}{2}, \quad \mathcal{K}_{0d} = \mathcal{K}_{d0} \in \mathbb R.$$
 An integration by parts shows that
    $$\mathcal{K}_{0c}(\eta) = \frac{F(0,\eta)}{\eta}, \quad \eta \geq 0,$$
    since
     $$
     W(R) \phi_0(R) = \left([\mathcal{L},R\partial_R]-2\mathcal{L} \right)\phi_0(R) = \mathcal{L}R\partial_R \phi_0(R).
     $$
For $\eta \geq 0$ fixed, the integration by parts is justified because $\phi(R,\eta)$, $\partial_R\phi(R,\eta)$ are bounded and $R\partial_R \phi_0(R)$, $\partial_{R}(R\partial_R \phi_0(R))$ vanish at zero and infinity.  Similarly, $\mathcal{K}_{dc}(\eta)$ is given by
     $$\mathcal{K}_{dc}(\eta) = \frac{\innerp[\big]{ W(R) \phi_d(R) }{ \phi(R,\eta) }_{L^2_R}}{\eta-\xi_d} = ||\phi(R,\xi_d)||_{L^2_R} \cdot \frac{F(\xi_d,\eta)}{\eta-\xi_d}, \quad \eta \geq 0,$$
    because for $\eta \geq 0$,
    $$(\eta-\xi_d)\mathcal{K}_{dc}(\eta) = \Bigl< [\mathcal{L},R\partial_R] \phi_d(R), \phi(R,\eta) \Bigr>_{L^2_R}, \quad \Bigl< \phi_d(R), \phi(R,\eta) \Bigr>_{L^2_R} \propto \delta_{\xi_d}(\eta) = 0.$$
    Moreover, since formula (\ref{iteration of integration by parts F(xi,eta)}) also holds for $\xi = 0$ and $\xi = \xi_d$,  it follows that $\mathcal{K}_{0c}$, $\eta\partial_{\eta}\mathcal{K}_{0c}$, $\mathcal{K}_{dc}$ and $\eta\partial_{\eta}\mathcal{K}_{dc}$ are continuous on $[0,+\infty)$ and have an arbitrary fast polynomial decay at infinity.
\end{proof}
    
    It remains to study the components of $\mathcal{K}_c$, i.e.  $\mathcal{K}_{c0}$, $\mathcal{K}_{cd}$ and $\mathcal{K}_{cc}$. 
    
    \begin{proposition}
        If $f \in C_c^{\infty}((0,+\infty))$, then 
        \begin{align*}
            \mathcal{K}_{c0}f &= - \int_0^{+\infty}f(\xi) \rho(\xi) K_{0c}(\xi)d\xi, \\
            \mathcal{K}_{cd}f &= - \int_0^{+\infty}f(\xi) \rho(\xi) K_{dc}(\xi)d\xi.
        \end{align*}
    \end{proposition}
    \begin{proof}
    One has
    $$\mathcal{K}_{c0} = \lim \limits_{r \to +\infty}\int_0^r \int_0^{+\infty} f(\xi)R\partial_R\phi(R,\xi)\rho(\xi) \phi_0(R)dR d\xi.$$
    Since $f(\xi)$ has a compact support, we can interchange the order of integration and obtain
    $$\mathcal{K}_{c0} = \lim \limits_{r \to +\infty}  \int_0^{+\infty}f(\xi) \rho(\xi) \left( \int_0^r R\partial_R\phi(R,\xi) \phi_0(R) dR \right) d\xi. $$
    Integrating by parts in the $R$-integral, we get
    \begin{align*}
        \int_0^r R\partial_R\phi(R,\xi) \phi_0(R) dR &= r\phi(r,\xi)\phi_0(r) - \int_0^r \phi(R,\xi) R\partial_R\phi_0(R) dR \\
        &\phantom{=}- \int_0^r \phi(R,\xi)\phi_0(R) dR.
    \end{align*}
    The second and third term converges in $L^2(\mathbb R, d\rho)$. Hence,
    \begin{align*}
        \mathcal{K}_{c0}f &= \lim \limits_{r \to +\infty} r\phi_0(r) \int_0^{+\infty}f(\xi) \phi(r,\xi)\rho(\xi)d\xi - \int_0^{+\infty}f(\xi) \rho(\xi) K_{0c}(\xi)d\xi \\
        &\phantom{=}- \int_0^{+\infty}f(\xi) \rho(\xi) \delta_0(\xi) d\xi \\
        &= \left( \lim \limits_{r \to +\infty} r\phi_0(r) \right)  \left(\lim \limits_{r \to +\infty}  \mathcal{F}^{-1}(f)(r) \right) - \int_0^{+\infty}f(\xi) \rho(\xi) K_{0c}(\xi)d\xi \\
        &= - \int_0^{+\infty}f(\xi) \rho(\xi) K_{0c}(\xi)d\xi,
    \end{align*}
    because $\mathcal{F}^{-1}f(R)$ decays like a Schwartz function. Similarly, one computes
    $$\mathcal{K}_{cd}f = - \int_0^{+\infty}f(\xi) \rho(\xi) K_{dc}(\xi)d\xi.$$
    \end{proof}
        
    Finally, for $\mathcal{K}_{cc}$, we integrate by parts with respect to $\xi$ in the component
\begin{align*}
    \int_0^{+\infty} 2 \xi\partial_\xi f(\xi) \phi(R,\xi) \rho(\xi)d\xi &= - \int_0^{+\infty} 2 f(\xi) \phi(R,\xi) \xi\partial_\xi \rho(\xi)d\xi \\
    &\phantom{=}- \int_0^{+\infty} 2 f(\xi) \phi(R,\xi) \rho(\xi)d\xi \\
    &\phantom{=} - \int_0^{+\infty} 2 f(\xi)  \xi\partial_\xi\phi(R,\xi) \rho(\xi)d\xi \\
    &= -\int_0^{+\infty} 2 f(\xi) \left(1 + \frac{\xi\rho'(\xi)}{\rho(\xi)} \right) \phi(R,\xi)  \rho(\xi) d\xi \\
    &\phantom{=} - \int_0^{+\infty} 2 f(\xi)  \xi\partial_\xi\phi(R,\xi) \rho(\xi)d\xi,
\end{align*}
    yielding
    $$\mathcal{K}_{cc}f(\eta) = \Bigl< \int_0^{+\infty}f(\xi) \left[ R\partial_R - 2 \xi\partial_\xi \right]\phi(R,\xi) \rho(\xi)d\xi ,\phi(R,\eta) \Bigr>_{L^2_R} - 2 \left(1 + \frac{\eta\rho'(\eta)}{\rho(\eta)} \right)f(\eta)$$
    when $\eta \geq 0$. Again, the function
    $$u(R) = \int_0^{+\infty}f(\xi) \left[ R\partial_R - 2 \xi\partial_\xi \right]\phi(R,\xi) \rho(\xi)d\xi $$
    decays like a Schwartz function when $f \in C^{\infty}_c((0,+\infty))$ and the bounds one can get on $\sup_{R\geq 1}|R^nu^{(m)}(R)|$ will depend only on $n$, $m$, $||f||_{\infty}$ and $\text{supp}(f)$. Since $\phi(R,\eta)$ is uniformly bounded, it follows that $\mathcal{K}_{cc}$ is continuous if $C^{\infty}_c$ and $C^{\infty}$ are respectively endowed with the test function topology and the $L^{\infty}$ topology. Hence, the Schwartz kernel theorem (\cite[Chapter V]{Hormander}) shows that one can write
    $$\mathcal{K}_{cc}f(\eta) = \int_0^{+\infty}K(\eta,\xi)f(\xi)d\xi$$
    for some distribution-valued kernel $\eta \mapsto K(\eta,\xi)$ which is made more explicit in the following theorem.

    \begin{theorem}
        The operator $\mathcal{K}_{cc}$ admits the following representation
        $$\mathcal{K}_{cc} = -\left( \frac{3}{2} + \frac{\eta \rho'(\eta)}{\rho(\eta)} \right) \delta(\xi-\eta) + \mathcal{K}_{cc}^0$$
        where $\mathcal{K}_{cc}^0$ has kernel 
        $$K_{cc}^0(\eta,\xi) = \frac{\rho(\xi)}{\eta-\xi}F(\xi,\eta)$$
        and $F(\xi,\eta)$ is as in Corollary \ref{F(xi,eta) estimates}.
        \index{F(xi,eta)@$F(\xi,\eta)$, a symmetric function which contributes to $\mathcal{K}$}
    \end{theorem}

    \begin{proof}
        This follows from another integration by parts and a change of integration order. This is proved exactly as in \cite[Theorem 5.1]{Krieger_2007}.
    \end{proof}

    \begin{definition} \label{def:l^2,alpha_rho}
        Let $L^{2,\alpha}_{\rho}$, $\alpha \in \mathbb R$, be the set of d$\rho$-measurable functions $f(\xi)$ for which the following norm is finite:
        $$||f||_{L^{2,\alpha}_{\rho}}^2 = |f(\xi_d)|^2 + |f(0)|^2 \cdot \delta_{d = 5} + \int_0^{+\infty}|f(\xi)|^2 \langle \xi \rangle^{2\alpha}\rho(\xi)d\xi.$$
        This space can be represented as $\mathbb R^{d-3} \times L^2((0,+\infty), \langle \xi \rangle^{2\alpha}\rho(\xi)d\xi)$.
        \index{L 2 alpha rho@$L^{2,\alpha}_{\rho}$, a normed space on Fourier side}
    \end{definition}

    \begin{theorem} \label{thm:boundedness of K operator}
        For any $\alpha \in \mathbb R$, the operators $\mathcal{K}_{cc}^0$, $\mathcal{K}$, $[\mathcal{K},\xi\partial_{\xi}]$ maps
        $$\mathcal{K}_{cc}^0: L^{2,\alpha}_{\rho} \rightarrow L^{2,\alpha + 1/2}_{\rho}, \quad \mathcal{K}: L^{2,\alpha}_{\rho} \rightarrow L^{2,\alpha}_{\rho}, \quad [\mathcal{K},\xi\partial_{\xi}]: L^{2,\alpha}_{\rho} \rightarrow L^{2,\alpha}_{\rho}$$
        continuously, where $\xi \partial_{\xi}$ acts as zero on the discrete component.
    \end{theorem}

    \begin{proof}
    Recall that if $f \in \mathbb R^{2} \times C^{\infty}_c((0,+\infty))$ (in dimension 5, but the same can be said in dimension 4), then
    $$\mathcal{K}f(\eta) = \begin{pmatrix} k_{1,1}f(\xi_d) + k_{1,2}f(0) -\langle f \cdot \rho, K_{dc} \rangle_{L^2((0,+\infty))} \\
k_{2,1}f(\xi_d) + k_{2,2}f(0) -\langle f \cdot \rho, K_{0c} \rangle_{L^2((0,+\infty))}  \\
\mathcal{K}_{dc}(\eta)f(\xi_d) + \mathcal{K}_{0c}(\eta)f(0) + \mathcal{K}_{cc}f (\eta)
\end{pmatrix}.
$$
Similarly, we have that
\begin{align*}
    [\mathcal{K},\xi\partial_{\xi}]f(\eta) &= \begin{pmatrix} -\langle \xi\partial_{\xi} f, \rho \cdot K_{dc} \rangle_{L^2((0,+\infty))} \\
-\langle \xi\partial_{\xi}  f, \rho \cdot K_{0c} \rangle_{L^2((0,+\infty))}  \\
 \mathcal{K}_{cc}(\xi\partial_{\xi}f) (\eta) - \eta \partial_{\eta} \mathcal{K}_{cc}f (\eta) - \eta \partial_{\eta} \mathcal{K}_{dc}(\eta)f(\xi_d) - \eta \partial_{\eta} \mathcal{K}_{0c}(\eta)f(0)
\end{pmatrix} \\
&= \begin{pmatrix} \langle f , (\rho + \xi\partial_{\xi}\rho) \cdot K_{dc} + \rho \cdot \xi\partial_{\xi} K_{dc} \rangle_{L^2((0,+\infty))} \\
\langle f, (\rho + \xi\partial_{\xi} \rho) \cdot K_{0c} + \rho \cdot \xi\partial_{\xi} K_{0c} \rangle_{L^2((0,+\infty))}  \\
 \mathcal{K}_{cc}(\xi\partial_{\xi}f) (\eta) - \eta \partial_{\eta} \mathcal{K}_{cc}f (\eta) - \eta \partial_{\eta} \mathcal{K}_{dc}(\eta)f(\xi_d) - \eta \partial_{\eta} \mathcal{K}_{0c}(\eta)f(0)
\end{pmatrix},
\end{align*}

where $k_{i,j} \in \mathbb R$ and $\mathcal{K}_{dc}$, $\xi\partial_{\xi} \mathcal{K}_{dc}$, $\mathcal{K}_{0c}$, $\xi\partial_{\xi} \mathcal{K}_{0c}$ are continuous on $[0,+\infty)$ and fast-decaying functions. Hence, it suffices to study the mapping properties of $$\mathcal{K}_{cc} = -\left( \frac{3}{2} + \frac{\eta \rho'(\eta)}{\rho(\eta)} \right) \delta(\xi-\eta) + \mathcal{K}_{cc}^0, \quad [\mathcal{K}_{cc}, \xi\partial_{\xi}] = \eta\partial_{\eta}\left( \frac{\eta \rho'(\eta)}{\rho(\eta)} \right) \delta(\xi-\eta) +  [\mathcal{K}_{cc}^0, \xi\partial_{\xi}]$$
for which the dirac-delta contribution causes no issue.

\item[\textbf{Boundedness of $\mathcal{K}_{cc}^0$}:]

The boundedness of $\mathcal{K}_{cc}^0$ is equivalent to proving that the kernel 
$$\tilde{K}^0_{cc}(\eta,\xi) = \rho(\eta)^{\frac{1}{2}} \langle \rho \rangle^{\alpha + 1/2}K_{cc}^0(\eta,\xi)  \langle \eta \rangle^{-\alpha} \rho(\xi)^{-\frac{1}{2}}: L^2((0,+\infty)) \rightarrow L^2((0,+\infty))$$
acts, as a principal value integral, continuously. Write
\begin{align*}
\tilde{K}^0_{cc}(\eta,\xi) &= \frac{\sqrt{\rho(\eta)\rho(\xi)} \langle \eta \rangle^{\alpha + 1/2} \langle \xi \rangle^{-\alpha} F(\xi,\eta)}{\eta-\xi} = \frac{\tilde{F}(\xi,\eta)}{\eta-\xi}.
\end{align*}
First, we split the kernel into two regions: the diagonal
$$D =\{(\eta,\xi) \in (0,+\infty)^2: \frac{1}{4}\xi \leq \eta \leq 4\xi\},$$
where $\xi \sim \eta$, and its complementary $(0,+\infty)^2 \setminus D$, where one always has $|\xi-\eta| \geq \frac{5}{3}(\xi + \eta)$. Our estimates on $F(\xi,\eta)$ from Corollary \ref{F(xi,eta) estimates} show that
$$\left| \frac{\tilde{F}(\xi,\eta)}{\xi-\eta}  \right| \lesssim \begin{cases}
    \eta^{\frac{d-6}{4}} \xi^{\frac{d-6}{4}} (\cdot |\log(\xi)\log(\eta) |^{-1} \text{ if } d = 4) &\text{ if } (\xi,\eta) \notin D, \xi + \eta \lesssim 1 \\
    (1+\xi)^{-N}(1+\eta)^{-N} &\text{ if } (\xi,\eta) \notin D, \xi + \eta \gtrsim 1,
\end{cases}$$
meaning that $1_{(0,+\infty)^2 \setminus D} \cdot \tilde{K}^{0}_{cc}$ is a Hilbert-Schmidt kernel on $(0,+\infty)^2$. It remains to treat the diagonal part. In that region, one has the following estimates
    \begin{align*}
        |\tilde{F}(\xi,\eta)| &\lesssim \begin{cases}
           \eta^{\frac{d-4}{2}} (\cdot |\log(\eta)|^{-2} \text{ if } d = 4) \quad &\text{ if } \xi + \eta \leq 1 \\
            (1+|\xi^{\frac{1}{2}}-\eta^{\frac{1}{2}}|)^{-N} \quad &\text{ if } \xi + \eta \geq 1 
        \end{cases} \\
        |\partial_{\xi}\tilde{F}(\xi,\eta)| + |\partial_{\eta}\tilde{F}(\xi,\eta)|&\lesssim \begin{cases}
            \eta^{\frac{d-6}{2}} (\cdot |\log(\eta)|^{-2} \text{ if } d = 4) \quad &\text{ if } \xi + \eta \leq 1 \\
             \eta^{-\frac{1}{2}}(1+|\xi^{\frac{1}{2}}-\eta^{\frac{1}{2}}|)^{-N} \quad &\text{ if } \xi + \eta \geq 1 .
        \end{cases} 
    \end{align*} 
%where the worst term always come from the one where the derivative fall on $F(\xi,\eta)$.
We write
\begin{align*}
1_{D} \cdot \tilde{K}^0_{cc}(\eta,\xi) &= 1_{D} \cdot \left( \frac{\tilde{F}(\xi,\xi)}{\eta-\xi} + \frac{\tilde{F}(\xi,\eta)-\tilde{F}(\xi,\xi)}{\eta-\xi} \right) = (A) + (B)
\end{align*}
and the $L^2$-boundedness of $(A)$ follows from the boundedness of the Hilbert transform. 

If we further split the diagonal $D$ into $D \cap [0,1]^2$ and $D \setminus [0,1]^2$, then $(B) \cdot 1_{D \cap [0,1]^2}$ is Hilbert-Schmidt on $(0,+\infty)^2$. As for $(B) \cdot 1_{D \setminus [0,1]^2}$, we use Schur test: it suffices to prove that
$$\sup_{\xi \geq 0} \int_{\eta \geq 0} 1_{D \setminus [0,1]^2} \cdot \left| \frac{\tilde{F}(\xi,\eta)-\tilde{F}(\xi,\xi)}{\eta-\xi} \right| d\eta + \sup_{\eta \geq 0} \int_{\xi \geq 0} 1_{D \setminus [0,1]^2} \cdot \left| \frac{\tilde{F}(\xi,\eta)-\tilde{F}(\xi,\xi)}{\eta-\xi} \right| d\xi < +\infty$$
to get the $L^2((0,+\infty))$-boundedness. Given $\eta \sim \xi$ on $D$ and our bounds on $\nabla F$, it is enough to prove
$$\sup_{\xi \geq 1} \int_{\eta \geq 1} \eta^{-\frac{1}{2}}(1 + |\xi^{\frac{1}{2}}-\eta^{\frac{1}{2}}|)^{-N} d\eta < +\infty,$$
which is the case because
\begin{align*}
    \int_{\eta \geq 1} \eta^{-\frac{1}{2}}(1 + |\xi^{\frac{1}{2}}-\eta^{\frac{1}{2}}|)^{-N} d\eta &= \int_{1}^{\xi} \eta^{-\frac{1}{2}}(1 + \xi^{\frac{1}{2}}-\eta^{\frac{1}{2}})^{-N} d\eta + \int_{\xi}^{+\infty} \eta^{-\frac{1}{2}}(1 - \xi^{\frac{1}{2}}+\eta^{\frac{1}{2}})^{-N} d\eta \\
    &\lesssim \left| \left[ (1 + \xi^{\frac{1}{2}}-\eta^{\frac{1}{2}})^{-N+1} \right]^{\eta = \xi}_{1} \right| + \left| \left[ (1 - \xi^{\frac{1}{2}}+\eta^{\frac{1}{2}})^{-N+1} \right]^{\eta = +\infty}_{\xi} \right| \\
    &\lesssim  \left( \xi^{\frac{1}{2}} \right)^{-N+1} + 1 \\
    &\lesssim 1 \quad \forall \xi \geq 1.
\end{align*}
\begin{remark}
    In the above proof, we remark that one can also choose to write 
\begin{align*}
1_{D} \cdot \tilde{K}^0_{cc}(\eta,\xi) &= 1_{D} \cdot \left( \frac{\tilde{F}(\eta,\eta)}{\eta-\xi} + \frac{\tilde{F}(\xi,\eta)-\tilde{F}(\eta,\eta)}{\eta-\xi} \right)
\end{align*}
and use the same analysis. This is useful for the commutator estimate, because we can choose if we wish to differentiate $\tilde{F}$ with respect to $\xi$ or $\eta$. We do not need to have an estimate on the full gradient of $\tilde{F}$.
\end{remark}

% we use the T1-theorem we write
% \begin{align*}
% \tilde{K}^0_{cc}(\eta,\xi) &= \frac{\tilde{F}(\xi,\xi)}{\eta-\xi} + \frac{\tilde{F}(\xi,\eta)-\tilde{F}(\xi,\xi)}{\eta-\xi}
% \end{align*}
% Since 
% $$|\tilde{F}(\xi,\xi)| \lesssim \begin{cases}
%     \xi^{\frac{d}{2}-2} (\cdot |\log(\xi)|^{-2} \text{ if } d = 4) \quad & \xi < 1 \\
%     1 \quad & \xi \geq 1
% \end{cases}$$
% the $L^2$-boundedness of $\tilde{F}(\xi,\xi)(\eta-\xi)^{-1}$ follows from the boundedness of the Hilbert transform. To conclude, it suffices to prove that 
% $$|\tilde{F}(\xi,\eta)-\tilde{F}(\xi,\xi)| \leq C|\xi-\eta|$$
% for a constant C not depending on 

\item[\textbf{Boundedness of $[\mathcal{K}_{cc}^0,\xi\partial_{\xi}]$}:]

An integration by parts shows that the commutator $[\xi\partial_{\xi},\mathcal{K}_{cc}^0]$ has kernel
\begin{align*}
    K^{0,com}_{cc} &= (\eta\partial_{\eta} + \xi\partial_{\xi})K^0_{cc}(\eta,\xi) + K_{cc}^0(\eta,\xi) \\
    &= \frac{\rho(\xi)}{\eta-\xi}\left( \frac{\xi\rho'(\xi)}{\rho(\xi)}F(\xi,\eta) + (\eta\partial_{\eta} + \xi\partial_{\xi})F(\xi,\eta) \right).
\end{align*}
Then $\eta\partial_{\eta} F$, $\xi \partial_{\xi}F$ and  $\xi\rho'(\xi)\rho(\xi)^{-1}F(\xi,\eta)$ all satisfy the same estimates as $F(\xi,\eta)$, the only difference being on the diagonal, away from zero, where we lose a factor $\eta^{-\frac{1}{2}}$, but we gain it back by considering the $L^{2,\alpha}_{\rho} - L^{2,\alpha}_{\rho}$ boundedness instead of $L^{2,\alpha}_{\rho} - L^{2,\alpha+1/2}_{\rho}$. 

We remark that for $\eta \partial_{\eta} F$ (resp. $\xi \partial_{\xi} F$), we only need the estimate for the derivative with respect to $\xi$ (resp. $\eta$) to be the same as the one for $\nabla F$.
\end{proof}

\section{Exact solutions by means of Fourier method} \label{section:exact solution fourier method}
In this section, we construct the final piece of the solution $\varepsilon$, which corrects the approximate solution $u_k$ from Theorem \ref{Thm:Approximate Solution} and Theorem \ref{Thm:Approximate Solution d = 4} to an exact solution within the cone. We first transform the evolution equation for $\varepsilon$ into the generalized Fourier space and we formulate the equation as a fixed-point problem given by (\ref{contraction operator}). The core of the section is to prove, by using the properties of the transference operator, that this map is a contraction on a carefully chosen Banach space. 

Let $\varepsilon(r,t)$ be a solution of (\ref{space-variable equation epsilon}). Substituting
$$
\tilde{\varepsilon}(\tau,R) = R^{\frac{d-1}{2}}\varepsilon(t(\tau),r(\tau,R)), \quad \tau = \nu^{-1}t^{-\nu}, \ R = \lambda(t)r, \ \lambda(t) = t^{-1-\nu},
$$
\index{tau@$\tau = \nu^{-1}t^{-\nu}$}
we get
\begin{equation} \label{new space-variable equation epsilon}
    \mathcal{D}^2 \tilde{\varepsilon} - (d-2) \beta(\tau) \mathcal{D} \tilde{\varepsilon} - \frac{(d-1)(d-3+d\nu - \nu)}{4\nu} \dot{\beta}(\tau) \tilde{\varepsilon} + \mathcal{L} \tilde{\varepsilon} = \mathcal{N}\tilde{\varepsilon},
\end{equation}
where the following notations were used:
\begin{align*}
\lambda(\tau) &= \lambda(t(\tau)) = (\nu \tau)^{\frac{1+\nu}{\nu}}, \ \dot{\lambda}(\tau) = \partial_{\tau} \lambda(t(\tau)), \\
    \beta(\tau) &= \dot{\lambda}(\tau)\lambda(\tau)^{-1} = \frac{1+\nu}{\tau \nu}, \\
    \dot{\beta}(\tau) &= \partial_{\tau}\beta = -\frac{1+\nu}{\tau^2 \nu}, \\
    \mathcal{D} &= \partial_{\tau} + \beta(\tau)R \partial_R,  \\
    \mathcal{L} &= -\partial_{RR} - pW(R)^{p-1} + \frac{1}{R^2} \cdot \left( \frac{(d-3)(d-1)}{4} \right),  \\
    \mathcal{N} &= \lambda(\tau)^{-2}R^{\frac{d-1}{2}}\left[e_{k-1} + F(u_{k-1}+\chi\left(R \tau^{-1}\right)R^{-\frac{d-1}{2}}\tilde{\varepsilon}) \right.  \\
     &\phantom{=} \left.-F(u_{k-1}) - F'(u_0)\chi\left(R \tau^{-1}\right) R^{-\frac{d-1}{2}}\tilde{\varepsilon}\right].
\end{align*}
We assume until the end of the paper that $u_0, u_{k-1}$, $e_{k-1}$ are always extended on $0 \leq R < +\infty$, $\tau \geq \tau_0$, with the same size and regularity as well as being supported in $0 \leq R < 2\tau$. We remark that we have added a term $\chi(R\tau^{-1})$ in front of $R^{-\frac{d-1}{2}}\tilde{\varepsilon}$, where $0 \leq \chi = 1-\chi_{[1,+\infty)} \leq 1$ is a smooth transition function which is $1$ on $|x| \leq 1$ and $0$ on $|x| \geq 2$, in the definition of $\mathcal{N}$. All of this does not change the equation on the cone $0 \leq R < \tau$, $\tau \geq \tau_0$, of interest.
\index{chi2@$\chi(R\tau^{-1})$, another transition function}

\begin{remark}[On extending $u_k, e_k$ outside the cone] \label{extension of v_k, e_k}
In the following, we briefly describe how one can, for fixed $t$, extend $u_k, e_k$ on the whole $\mathbb R^d$. Multiplying this extension by $\chi(R\tau^{-1})$ restricts its support to the desired region $0 \leq R < 2\tau$.

Note that $u_0, v_1, e_1$ are already defined on the whole $0 \leq R < +\infty$, $0 < t \leq t_0$. For these terms, we only need to apply the cutoff. Otherwise, write $v(R,t)$ or $e(R,t)$ as a function $f(a,t)$, $a = R/(t\lambda) \in [0,1]$, and for fixed $t$, extend $f(a,t)$ on $[0,+\infty)$ while keeping the same Hölder regularity and a comparable Hölder constant, which gives the $(t\lambda)$ smallness. This can be done via Whitney's Extension Theorem (\cite[Chapter VI, Theorem 4]{Stein_Whitney_Extension}). However, since we are working on an interval, we can use the following simpler construction:
$$
\tilde{f}(a,t) =  \int_1^{+\infty} \phi(y)\psi(1-a(1-y))f(1-a(1-y),t)dy, \quad a > 1,
$$
where $\psi \in C^{\infty}(\mathbb R)$ is a smooth transition function which is $1$ on $(-\infty,1/2]$ and $0$ on $[3/4,+\infty)$ and $\phi \in C^{0}([1,+\infty))$ is a continuous function satisfying
   $$\int_{1}^{+\infty}\phi(s)ds = 1, \quad \int_{1}^{+\infty}s^n\phi(s)ds = 0 \ \forall n \geq 1, \quad \lim \limits_{s \to +\infty}s^n\phi(s) = 0, \quad \forall n \geq 0.$$
\end{remark}

Now, we translate equation (\ref{new space-variable equation epsilon}) to the Fourier side. Observe that
$$\mathcal{F} \left( \partial_{\tau} + \beta(\tau)R \partial_R \right) = \left(\partial_{\tau}-2\beta(\tau)\xi\partial_{\xi} \right)\mathcal{F} + \beta(\tau)\mathcal{K}\mathcal{F}$$
and
\begin{align*}
  \mathcal{F} \left( \partial_{\tau} + \beta(\tau)R \partial_R \right)^2 &= \left[\left(\partial_{\tau}-2\beta(\tau)\xi\partial_{\xi}\right)\mathcal{F} + \beta(\tau)\mathcal{K}\mathcal{F}\right] \left( \partial_{\tau} + \beta(\tau)R \partial_R \right) \\
  &= \left(\partial_{\tau}-2\beta(\tau)\xi\partial_{\xi} \right)^2\mathcal{F} + \left(\partial_{\tau}-2\beta(\tau)\xi\partial_{\xi} \right) \beta(\tau)\mathcal{K}\mathcal{F} \\
  &\phantom{=}+ \beta(\tau)\mathcal{K}\left(\partial_{\tau}-2\beta(\tau)\xi\partial_{\xi} \right)\mathcal{F} + \beta(\tau)^2\mathcal{K}^2\mathcal{F},
\end{align*}
i.e.,
\begin{align*}
\mathcal{D}_{\tau} &= \partial_{\tau}-2\beta(\tau)\xi\partial_{\xi} \\
\mathcal{F} \mathcal{D} &= \mathcal{D}_{\tau}\mathcal{F} + \beta(\tau)\mathcal{K}\mathcal{F} \\
\mathcal{F} \mathcal{D}^2 &= \mathcal{D}_{\tau}^2\mathcal{F} + \beta(\tau) \mathcal{K} D_{\tau}\mathcal{F} + D_{\tau} \mathcal{K}\beta(\tau)\mathcal{F} + \beta(\tau)^2\mathcal{K}^2\mathcal{F} \\
&= \mathcal{D}_{\tau}^2\mathcal{F} + \beta(\tau) \mathcal{K} D_{\tau}\mathcal{F} + \beta(\tau) D_{\tau} \mathcal{K} \mathcal{F} + \dot{\beta}(\tau)\mathcal{K} \mathcal{F} + \beta(\tau)^2\mathcal{K}^2\mathcal{F} \\
&= \mathcal{D}_{\tau}^2\mathcal{F} + 2\beta(\tau) \mathcal{K} D_{\tau}\mathcal{F} + \beta(\tau) [D_{\tau},\mathcal{K}]  \mathcal{F} + \dot{\beta}(\tau)\mathcal{K} \mathcal{F}+ \beta(\tau)^2\mathcal{K}^2\mathcal{F} \\
&= \mathcal{D}_{\tau}^2\mathcal{F} + 2\beta(\tau) \mathcal{K} D_{\tau}\mathcal{F} - 2 \beta(\tau)^2 [\xi\partial_{\xi},\mathcal{K}]  \mathcal{F} + \dot{\beta}(\tau)\mathcal{K} \mathcal{F}+ \beta(\tau)^2\mathcal{K}^2\mathcal{F}.
\end{align*}
Therefore, (\ref{new space-variable equation epsilon}) rewrites as
\begin{equation}\label{fourier-variable equation epsilon}
\begin{aligned}
    (\mathcal{D}^2_{\tau} + \beta(\tau) \mathcal{D}_{\tau} + \xi ) \mathcal{F}\tilde{\varepsilon} &= (d-1) \beta(\tau) \mathcal{D}_{\tau} \mathcal{F} \tilde{\varepsilon} + (d-2) \beta(\tau)^2\mathcal{K}\mathcal{F} \tilde{\varepsilon}\\
    &\phantom{=}-2\beta(\tau) \mathcal{K} D_{\tau}\mathcal{F} \tilde{\varepsilon} + 2\beta(\tau)^2[\xi\partial_{\xi},\mathcal{K}] \mathcal{F} \tilde{\varepsilon} - \dot{\beta}(\tau)\mathcal{K}\mathcal{F}\tilde{\varepsilon} \\
    &\phantom{=}- \beta(\tau)^2\mathcal{K}\mathcal{F}\tilde{\varepsilon} + \frac{(d-1)(d-3+d\nu - \nu)}{4\nu} \dot{\beta}(\tau) \mathcal{F}\tilde{\varepsilon}  +  \mathcal{F}\mathcal{N}\tilde{\varepsilon}.
\end{aligned}
\end{equation}
Since
$$S^{-1}\partial_{\tau}S = \mathcal{D}_{\tau}, \quad S^{-1}\lambda(\tau)^{-2}\xi S = \xi, \quad (Sg)(\tau,\xi):= g(\tau, \lambda(\tau)^{-2}\xi),$$
the operator on the left-hand side of (\ref{fourier-variable equation epsilon}) can be inverted as was shown in \cite[Section 3]{Krieger_2012}. If
$$\underline{\mathbf{x}}(\tau,\xi) = \begin{pmatrix}
    x_d(\tau) \\
    x_0(\tau) \\
    x(\tau,\xi)
\end{pmatrix} \quad \underline{\mathbf{f}}(\tau,\xi) = \begin{pmatrix}
    f_d(\tau) \\
    f_0(\tau) \\
    f(\tau,\xi)
\end{pmatrix} \quad \underline{\mathbf{\xi}} = \xi \cdot \begin{pmatrix}
    1_{\xi = \xi_d} & 0 & 0 \\
    0 & 1_{\xi = 0} & 0 \\
    0 & 0 & 1_{\xi > 0}
\end{pmatrix}$$
in dimension 5 (and similarly in dimension 4), then the inhomogeneous problem
$$(\mathcal{D}^2_{\tau} + \beta(\tau) \mathcal{D}_{\tau} + \underline{\mathbf{\xi}} ) \underline{\mathbf{x}}(\tau,\xi) = \underline{\mathbf{f}}(\tau,\xi), \quad  \tau > 0, \xi \in \{\xi_d\} \cup [0,+\infty) \label{eq:inhomogeneous problem with D_tau}$$ 
is solved as
\begin{equation} \label{fourier-variable equation inversion}
\begin{aligned}
    x(\tau, \xi) &= \int_{\tau}^{+\infty} H(\sigma,\tau, \lambda(\tau)^2 \xi), f\left(\sigma,\frac{\lambda(\tau)^2}{\lambda(\sigma)^2} \xi\right)d\sigma, \\
    H(\sigma,\tau,\xi) &= \xi^{-\frac{1}{2}}  \sin \left[ \xi^{\frac{1}{2}} \int_{\tau}^{\sigma} \lambda(u)^{-1}du \right], \\
    \mathcal{D}_{\tau} x(\tau, \xi) &= \int_{\tau}^{+\infty} (\partial_{\tau}H)(\sigma,\tau, \lambda(\tau)^2 \xi) f\left(\sigma,\frac{\lambda(\tau)^2}{\lambda(\sigma)^2} \xi\right)d\sigma, \\
    x_0(\tau) &= \int_{\tau}^{+\infty}H_0(\tau,\sigma) f_0(\sigma)d\sigma, \quad H_0(\tau,\sigma) = \nu \sigma^{\frac{1+\nu}{\nu}} \left(\tau^{-\frac{1}{\nu}} - \sigma^{-\frac{1}{\nu}} \right), \\
    x_d(\tau) &= \int_{\tau}^{+\infty}H_d(\tau,\sigma) \tilde{f}_d(\sigma)d\sigma, \quad H_{d}(\tau,\sigma) = -\frac{1}{2}|\xi_d|^{-\frac{1}{2}} \exp \left( -\frac{1}{2}|\xi_d|^{\frac{1}{2}} |\tau-\sigma| \right), \\
    \tilde{f}_d(\tau) &= f_d(\tau)-\beta(\tau)x_d(\tau).
\end{aligned}
\end{equation}

\begin{definition}\label{definition banach space}
        For $\alpha \in \mathbb R$, $N \in \mathbb N$, $\tau_0 \geq 1$, let $L^{\infty,N} L^{2,\alpha}_{\rho}$ be the set of measurable functions $f(\tau, \xi)$ for which the following norm is finite
        $$||f||_{L^{\infty,N} L^{2,\alpha}_{\rho}} = \sup_{\tau \geq \tau_0} \tau^{N} ||f(\tau, \cdot)||_{L^{2,\alpha}_{\rho}} < +\infty$$
        % and $L^{\infty,N}$ be the be the set of measurable functions $f(\tau)$ for which 
        % $$||f||_{L^{\infty,N}} = \sup_{\tau \geq \tau_0} \tau^{N} |f(\tau)| < +\infty$$
        and $L^{2,\alpha}_{\rho} = \mathbb R^{d-3} \times L^2((0,+\infty), \langle \xi \rangle^{2\alpha}\rho(\xi)d\xi)$ as in Definition \ref{def:l^2,alpha_rho}.
        \index{L infty N, 2 alpha rho@$L^{\infty,N} L^{2,\alpha}_{\rho}$, a mixed-normed space on Fourier side}
\end{definition}

Our last goal in this section is to prove that the fixed-point iteration $(\underline{\mathbf{x}}_{n}, \mathcal{D}_{\tau}\underline{\mathbf{x}}_{n})$ defined via 
\begin{align} \label{contraction operator}
    \underline{\mathbf{x}}_{n+1} &= (\mathcal{D}^2_{\tau} + \beta(\tau) \mathcal{D}_{\tau} + \xi )^{-1} \left[ (d-1) \beta(\tau) \mathcal{D}_{\tau} \underline{\mathbf{x}}_{n} + (d-2) \beta(\tau)^2\mathcal{K}\underline{\mathbf{x}}_{n} -2\beta(\tau) \mathcal{K} D_{\tau}\underline{\mathbf{x}}_{n} \right. \notag \\
    &\phantom{=}\left. + 2\beta(\tau)^2[\xi\partial_{\xi},\mathcal{K}]\underline{\mathbf{x}}_{n} - \dot{\beta}(\tau)\mathcal{K}\underline{\mathbf{x}}_{n} - \beta(\tau)^2\mathcal{K}\underline{\mathbf{x}}_{n} \right. \notag \\
    &\phantom{=}\left.+ \frac{(d-1)(d-3+d\nu - \nu)}{4\nu} \dot{\beta}(\tau) \mathcal{F}\tilde{\varepsilon}  +  \mathcal{F}\mathcal{N} (\mathcal{F}^{-1}\underline{\mathbf{x}}_{n}) \right] \\
    &= (\mathcal{D}^2_{\tau} + \beta(\tau) \mathcal{D}_{\tau} + \xi )^{-1} \left[ T (\underline{\mathbf{x}}_{n}, \mathcal{D}_{\tau}\underline{\mathbf{x}}_{n}) + \tilde{\mathcal{N}} \underline{\mathbf{x}}_{n} + \lambda(\tau)^{-2}\mathcal{F}R^{\frac{d-1}{2}}e_{k-1} \right], \notag \\
    \underline{\mathbf{x}}_0 &= (\mathcal{D}^2_{\tau} + \beta(\tau) \mathcal{D}_{\tau} + \xi )^{-1} \left[ \lambda(\tau)^{-2}\mathcal{F}R^{\frac{d-1}{2}}e_{k-1}\right], \notag
\end{align}
with $T$ linear and $\tilde{\mathcal{N}}$ nonlinear, is a Cauchy sequence in the Banach space
$$
\underline{\mathbf{x}} = (x_d, x_0, x) \in  L^{\infty,N-2}L^{2,\alpha+\frac{1}{2}}_{\rho}
$$
for an appropriate choice of $\alpha, N, \tau_0$. To this end, we proceed as follows. First, we have a look at the boundedness of the inverse operator
$$
\begin{pmatrix}
    (\mathcal{D}^2_{\tau} + \beta(\tau) \mathcal{D}_{\tau} + \xi )^{-1} \\
    D_{\tau} (\mathcal{D}^2_{\tau} + \beta(\tau) \mathcal{D}_{\tau} + \xi )^{-1}
\end{pmatrix} :  L^{\infty,N}L^{2,\alpha}_{\rho} \rightarrow  L^{\infty,N-2}L^{2,\alpha+\frac{1}{2}}_{\rho} \times  L^{\infty,N-1}L^{2,\alpha}_{\rho}.
$$
We prove that for arbitrarily small $\kappa > 0$, this map is bounded with norm $\leq \kappa$ if $N$ is large enough (depending on $\kappa$, $\nu$, $\alpha$) and $\tau_0$ is small enough (depending on $\kappa, \nu, \alpha, N$) in Definition \ref{definition banach space}. This is the content of Theorem \ref{thm:D_tau boundedness}. Then, we observe that the linear part $T$ of (\ref{contraction operator}) is a bounded operator inbetween
$$
T : L^{\infty,N-2}L^{2,\alpha+\frac{1}{2}}_{\rho} \rightarrow L^{\infty,N}L^{2,\alpha}_{\rho},
$$
thanks to Theorem \ref{thm:boundedness of K operator} and the gain of smallness coming from $\beta(\tau)$. Moreover, the operator norm of $T$ depends only on $\nu$, $\alpha$ and is independent of $N$ or how $\tau_0$ is chosen in Definition \ref{def:l^2,alpha_rho}. Similarly, we prove in Proposition \ref{prop:local lipschitz} that the nonlinear part $\mathcal{N}$ is locally Lipschitz (with local constant independent of $N$ and $\tau_0$) from $L^{\infty,N-2}L^{2,\alpha+\frac{1}{2}}_{\rho}$ to $L^{\infty,N}L^{2,\alpha}_{\rho}$, and that the (non-inverted) forcing term $\lambda(\tau)^{-2}\mathcal{F}R^{\frac{d-1}{2}}e_{k-1}$ belongs  $L^{\infty,N}L^{2,\alpha}_{\rho}$. This is enough to prove that the fixed-point iteration (\ref{contraction operator}) converges for an appropriate choice of $N$ and $\tau_0$.

\begin{theorem}\label{thm:D_tau boundedness}
    Let $\underline{\mathbf{f}} = (f,f_0,f_d)  \in L^{\infty,N} L^{2,\alpha}_{\rho}$. If
    $$
    (x,x_0,x_d) =  (\mathcal{D}^2_{\tau} + \beta(\tau) \mathcal{D}_{\tau} + \xi )^{-1}(f,f_0,f_d)
    $$
    is inverted as in (\ref{fourier-variable equation inversion}), then
    \begin{align*}
        ||x||_{L^{\infty,N-2} L^{2,\alpha+\frac{1}{2}}_{\rho}} + ||\mathcal{D}_{\tau} x||_{L^{\infty,N-1} L^{2,\alpha}_{\rho}} &\leq C(\nu,\alpha) \frac{1}{N} ||f||_{L^{\infty,N} L^{2,\alpha}_{\rho}}, \\
        ||x_0||_{L^{\infty,N-2}} + ||\partial_{\tau} x_0||_{L^{\infty,N-1}} &\leq C(\nu) \frac{1}{N} ||f_0||_{L^{\infty,N}}, \\ 
        ||x_d||_{L^{\infty,N}} + ||\partial_{\tau}x_d||_{L^{\infty,N}} &\leq C(\nu,N)||f_d||_{L^{\infty,N}}
    \end{align*}
    for all $N \geq N_0(\nu,\alpha)$, independently of how $\tau_0 \geq 1$ is chosen in Definition \ref{def:l^2,alpha_rho}.
\end{theorem}

\begin{proof}
    This is a consequence of the following bounds
    \begin{align*}
        |H(\sigma,\tau,\xi)| &\leq C(\nu) \cdot \tau \langle \xi \rangle^{-\frac{1}{2}}, \\
        |\partial_{\tau}H(\sigma,\tau,\xi)| &\leq C(\nu) \cdot 1, \\ 
        \Bigl| \Bigl| f \left( \sigma, \frac{\lambda(\tau)^2}{\lambda(\sigma)^2} \cdot \right) \langle \cdot \rangle^{\beta}  \Bigr| \Bigr|_{L^{2,\alpha}_{\rho}} &\leq \left( \frac{\sigma}{\tau} \right)^{C(\nu,\alpha,\beta)} ||f(\sigma, \cdot)||_{L^{2,\alpha+\beta}_{\rho}}, \\
        |H_0(\tau,\sigma)| &\leq C(\nu) \cdot \tau  \left( \frac{\sigma}{\tau} \right)^{2}, \\
        |\partial_{\tau}H_0(\tau,\sigma)|  &\leq C(\nu) \cdot \left( \frac{\sigma}{\tau} \right)^{2},
    \end{align*}
    for some large constants $C > 0$ and the exponential decay of $H_d$, $\partial_{\tau}H_d$. For example, one deduces
    \begin{align*}
        ||x(\tau,\cdot)||_{L^{2,\alpha+1/2}_{\rho}} &\leq \int_{\tau}^{+\infty} \tau \Bigl| \Bigl| f \left( \sigma, \frac{\lambda(\tau)^2}{\lambda(\sigma)^2} \cdot \right) \langle \cdot \rangle^{-\frac{1}{2}}  \Bigr| \Bigr|_{L^{2,\alpha+1/2}_{\rho}}d\sigma \\
        &\leq \int_{\tau}^{+\infty} \tau \left( \frac{\sigma}{\tau} \right)^{C(\nu,\alpha)} ||f(\sigma, \cdot)||_{L^{2,\alpha+1/2}_{\rho}} d\sigma \\
        &\leq \int_{\tau}^{+\infty} \tau^{1-C(\nu,\alpha)} \left( \frac{\sigma^{C(\nu,\alpha)}}{\sigma^{N}} \right) \sigma^N ||f(\sigma, \cdot)||_{L^{2,\alpha+1/2}_{\rho}} d\sigma \\
        &\leq \tau^{1-C(\nu,\alpha)} \left( \int_{\tau}^{+\infty} \sigma^{C(\nu,\alpha)-N} d\sigma \right) ||f||_{L^{\infty,N} L^{2,\alpha}_{\rho}} \\
        &\leq \frac{1}{N-C(\nu,\alpha)-1}\tau^{2-N} ||f||_{L^{\infty,N} L^{2,\alpha}_{\rho}} 
    \end{align*}
    if $N > C(\nu,\alpha) + 1$.
\end{proof}

As a corollary, when $\nu$ and $\alpha$ are fixed, for any arbitrarily small constant $\kappa > 0$, one can fix $N$ large enough (depending on $\kappa,\nu,\alpha$) and then $\tau_0 \geq 1$ large enough (depending on $\kappa,\nu, N$) in the definition of $L^{\infty,N} L^{2,\alpha}_{\rho}$ so that 
    \begin{align}\label{contraction constant kappa}
        ||x||_{L^{\infty,N-2} L^{2,\alpha+\frac{1}{2}}_{\rho}} + ||\mathcal{D}_{\tau} x||_{L^{\infty,N-1} L^{2,\alpha}_{\rho}} &\leq \kappa ||f||_{L^{\infty,N} L^{2,\alpha}_{\rho}}, \\
        ||x_0||_{L^{\infty,N-2}} + ||\partial_{\tau} x_0||_{L^{\infty,N-1}} &\leq \kappa ||f_0||_{L^{\infty,N}}, \notag \\ 
        ||x_d||_{L^{\infty,N-2}} + ||\partial_{\tau}x_d||_{L^{\infty,N-1}} &\leq  \kappa ||f_d||_{L^{\infty,N}}. \notag
    \end{align}

\begin{lemma} \label{lemma:isometry L, H}
    For $\alpha \geq 0$ fixed, we have the following equivalence of norms 
    $$
    ||\underline{\mathbf{x}}||_{L^{2,\alpha}_{\rho}} \asymp ||R^{-\frac{d-1}{2}}\mathcal{F}^{-1}\underline{\mathbf{x}}||_{H^{2\alpha}_{\mathrm{rad}}(\mathbb R^d)}.
    $$
\end{lemma}

\begin{proof}
    See \cite[Lemma 6.6]{Krieger_2007_waveEq}.
\end{proof}

% \begin{remark}
%     In particular, $L^{2,\alpha}_{\rho}$ is isomorphic to the subspace of $H^{2\alpha}_{\mathrm{rad}}(\mathbb R^d)$ which is orthogonal to $\phi_d(R), \phi_0(R)$. 
    
%     We note that if $x_0, x_d \in \mathbb R$, 
%     \begin{align*}
%         |x_0| &\asymp ||R^{-\frac{d-1}{2}}\mathcal{F}^{-1}x_0||_{H^{2\alpha}_{\mathrm{rad}}(\mathbb R^d)} =  |x_0| \cdot ||R^{-\frac{d-1}{2}}\phi_0(R)||_{H^{2\alpha}_{\mathrm{rad}}(\mathbb R^d)} \\
%         |x_d| &\asymp ||R^{-\frac{d-1}{2}}\mathcal{F}^{-1}x_d||_{H^{2\alpha}_{\mathrm{rad}}(\mathbb R^d)} =  |x_d| \cdot ||R^{-\frac{d-1}{2}}\phi_d(R)||_{H^{2\alpha}_{\mathrm{rad}}(\mathbb R^d)} 
%     \end{align*}
%     by linearity, which leads to
%     $$
%     ||x||_{L^{2,\alpha}_{\rho} \times \mathbb R^{d-3}} \asymp ||R^{-\frac{d-1}{2}}\mathcal{F}^{-1}x||_{H^{2\alpha}_{\mathrm{rad}}(\mathbb R^d)}
%     $$
% \end{remark}

\begin{theorem}[Strauss estimates] \label{Strauss Estimates Thm}
    Let $u(x) \in H_{\mathrm{rad}}^s(\mathbb R^d)$, $d \geq 2$, $1 < 2s < d$. Then $u(x)$ is continuous a.e. on $x \neq 0$. Moreover, there exists some universal constant $C(d,s) > 0$ for which
    \begin{align*}
        |u(x)| &\leq C |x|^{\frac{1}{2}-\frac{d}{2}}||u||_{H^s(\mathbb R^d)} \quad |x| \geq 1, \\
        |u(x)| &\leq C |x|^{s-\frac{d}{2}}||u||_{H^s(\mathbb R^d)} \quad |x| \leq 1.
    \end{align*}
\end{theorem}

\begin{proof}
    See \cite[Theorem 10 and Theorem 13]{strauss_estimates}.
\end{proof}

\begin{corollary} \label{Strauss Estimates Corollary}
     Let $u(x) \in H_{\mathrm{rad}}^s(\mathbb R^d)$, $d \geq 2$, $s > d/2$. Consider
     $$v = \prod_{l=1}^n  \partial_r^{i_l}u$$
     and assume that $i_l \in \mathbb N_{\geq 0}$, $s-i_l > 1/2$ and $s-i_1 - ... - i_n \geq 0$. Then $v(x) \in L^2(\mathbb R^d)$ with
     $$||v||_{L^2(\mathbb R^d)} \lesssim ||u||_{H^s(\mathbb R^d)}^n.$$
     Similarly, it holds that
     $$w = \partial_r^i u \cdot u^{n-1} \in L^2(\mathbb R^d), \quad ||w||_{L^2(\mathbb R^d)} \lesssim ||u||_{H^s(\mathbb R^d)}^n,$$
     for any $n \geq 1$ and $s \geq i \geq 0$.
\end{corollary}

\begin{proof}
We start with $w$ which holds true because $\partial_r^i u \in L^2$ and $u^{n-1} \in L^{\infty}$. As for $v$, the case $n = 1$ is trivial. Assume $n \geq 2$. Using Strauss Estimates, each term in the product is $L^{\infty} \cap L^2$ away from the origin and at the origin, the worse singularity that can happen for $\partial_r^{i_l}u$ is $|r|^{\min\{s-i_l-d/2-\delta,0\}}$ where we fix $$0 < \delta < \min\{n^{-1}(n-1)(s-d/2), s-i_l-d/2: l \in \{1,...,n\}, s-i_l-d/2 > 0\}.$$ 
Hence, the product (\ref{nonlinear sobolev estimate, square term}) is in $L^2(\mathbb R^d)$ away from the origin and has a singularity at worse 
$$|r|^{\min\{s-\sum i_l - \frac{d}{2} - n\delta + (n-1) (s-d/2),0\}} \leq |r|^{-\frac{d}{2}+}$$
at the origin which is also square-integrable.
\end{proof}

\begin{proposition}\label{prop:local lipschitz}
    Let $d \in \{4,5\}$ and $d/2 < 1 + 2\alpha < 1 + (6-d)\nu/2$. In particular, $\nu > 3$ if $d = 5$ and $\nu > 1$ if $d = 4$. Let also $N, \tau_0 \gg 1 + \nu$ and consider a pair $(u_{k-1},e_{k-1})$ with $e_{k-1}$ having smallness of order $\tau^{-2N}$, obtained from Theorem \ref{Thm:Approximate Solution} or Theorem \ref{Thm:Approximate Solution d = 4}, extended outside the cone as functions having support in $0 \leq R < 2\tau$, as well as the same regularity and smallness (see Remark \ref{extension of v_k, e_k}).

    Then the forcing term in (\ref{contraction operator}) satisfies $\lambda(\tau)^{-2}\mathcal{F}R^{\frac{d-1}{2}}e_{k-1} \in L^{\infty,N} L^{2,\alpha}_{\rho}$ and the non-linear map $\tilde{\mathcal{N}}$ given by
    \begin{align*}
        \underline{\mathbf{x}} \mapsto &\lambda(\tau)^{-2}\mathcal{F} R^{\frac{d-1}{2}}\left[F\left( u_{k-1}+\chi\left(R\tau^{-1}\right) R^{-\frac{d-1}{2}}\mathcal{F}^{-1}\underline{\mathbf{x}} \right) \right.
    \\
    &\phantom{=}\left. -F(u_{k-1})-F'(u_0) \chi 
  \left(R\tau^{-1}\right)R^{-\frac{d-1}{2}}\mathcal{F}^{-1}\underline{\mathbf{x}}\right]
    \end{align*}
    is locally Lipschitz from $L^{\infty,N-2} L^{2,\alpha+1/2}_{\rho}$ to $L^{\infty,N} L^{2,\alpha}_{\rho}$, with Lipschitz constants independent of $N$ or $\tau_0$.
\end{proposition}

\begin{proof}
The forcing term $e_{k-1}$ has regularity $e_{k-1} \in L^{\infty,2N} H^{2\alpha}_{\mathrm{rad}}$ by construction, hence \sloppy $\lambda(\tau)^{-2}\mathcal{F}R^{\frac{d-1}{2}}e_{k-1} \in L^{\infty,N} L^{2,\alpha}_{\rho}$ by Proposition \ref{lemma:isometry L, H}. Using the same Proposition \ref{lemma:isometry L, H}, it suffices to prove that
    $$y \mapsto \lambda(\tau)^{-2}\left[F(u_{k-1}+\chi \left(R\tau^{-1}\right) y)-F(u_{k-1})-F'(u_0) \chi \left(R\tau^{-1}\right) y\right]$$
    is locally Lipschitz from $L^{\infty,N-2} H^{2\alpha+1}_{\mathrm{rad}}$ to $L^{\infty,N} H^{2\alpha}_{\mathrm{rad}}$. We note that the Lipschitz constants do not depend on $N$ or $\tau_0$ because if the map is locally Lipschitz from $L^{\infty,N^*-2} L^{2,\alpha+1/2}_{\rho}$ to $L^{\infty,N^*} L^{2,\alpha}_{\rho}$ for some specific $N^* = N^*(\nu)$, then it is also locally Lipschitz with the same constant for any $N \geq N^*$. The same holds if we take a bigger $\tau_0$ in the definition of these spaces. 

    If $1+2\alpha > d/2$, then $H^{2\alpha+1}(\mathbb R^d)$ is an algebra and $H^{2\alpha+1} \hookrightarrow L^{\infty}(\mathbb R^d) \cap C^0(\mathbb R^d)$. Hence, $y$ decays as $\tau^{-N} \lesssim \tau^{-N/2}R^{-N/2}$ and is negligible compared to $u_{k-1}$, whose dominant component is $u_{0} \cdot \chi(R\tau^{-1})$ (see the extension from Remark \ref{extension of v_k, e_k}). In particular, $u_{k-1} + sy$, $s \in [0,1]$, stays non-negative on $0 \leq R < +\infty, \tau \geq \tau_0$. From now on, we ignore the cutoff and simply assume that $y$ is supported on $0 \leq R < 2\tau$ and negligible compared to $u_{k-1}$. 
    
    First, we prove that the mapping has the correct range. Write
    \begin{equation}\label{nonlinear sobolev estimate, integral form}
    \begin{aligned}
        \lambda^{-2}\left[ F(u_{k-1}+ y)-F(u_{k-1})-F'(u_0)y \right] &= \lambda^{-2}\left[ F(u_{k-1}+y)-F(u_{k-1})-F'(u_{k-1})y \right. \\
        &+ \left.F'(u_{k-1})y - F'(u_0)y \right]\\
        &= \lambda^{-2} y^2 \int_0^1\int_0^1 s_1F''(u_{k-1}+s_1s_2y)ds_2ds_1 \\
        &+ \lambda^{-2} y(u_{k-1}-u_0)\int_0^1 F''(u_0 + s(u_{k-1}-u_0)) ds.
    \end{aligned}
    \end{equation}
    
    Then we need to estimate the $L^{\infty,N}L^2$ and $L^{\infty,N}\dot{H}^{2\alpha}$ norms of these products.    Combining $H^{2\alpha+1} \hookrightarrow L^2(\mathbb R^d) \cap L^{\infty}(\mathbb R^d)$ with
    $$||u_{k-1}||_{L^{\infty}(\mathbb R^d)} \asymp \lambda^{\frac{d-2}{2}}, \quad ||u_{k-1}-u_0||_{L^{\infty}(\mathbb R^d)} \asymp \lambda^{\frac{d-2}{2}}\tau^{-2},$$
    the $L^{\infty,N}L^2$ bound will follow from the inequality $|a+b|^{p-2} \leq |a|^{p-2} + |b|^{p-2}$ and a simple application of Hölder's inequality.

    As for the $L^{\infty,N}\dot{H}^{2\alpha}$ bound, for $\tau \geq \tau_0$ fixed, we can (weakly) differentiate $\ceil{2\alpha} \in [2,1+2\alpha)$ times both products in (\ref{nonlinear sobolev estimate, integral form}) and estimate them. We need to be careful with the $\ceil{2\alpha}$-th derivative which can introduce a singularity for $\partial_R^{\ceil{2\alpha}} u_{k-1}$ at $R = \tau$ or make $\partial_R^{\ceil{2\alpha}}y$ of low regularity $L^2 \cap H^{\frac{1}{2}}$ where Strauss estimates are unavailable.
    
    First, we treat the product which is nonlinear with respect to $y$. Using Faà di Bruno formula, one has
    $$\left| \partial_R^n   F''\left( u_{k-1}+s_1s_2y \right) \right| \lesssim  \sum_{1 \cdot m_1 + ... + n \cdot m_n = n} \left| F^{(2 + m_1 + ... + m_n)} \left( u_{k-1}+s_1s_2y \right)  \right| \prod_{l=1}^n \left| \partial_R^l (u_{k-1} + s_1s_2y) \right|^{m_l}.$$
    Since
    \begin{align}
        |u_{k-1}+s_1s_2y| &\asymp |u_{k-1}|,  \\
        | \partial_R^{l} (u_{k-1}+s_1s_2y)| &\lesssim \frac{|u_{k-1}|}{1+R^l} + s_1s_2|\partial_R^l y|, \quad \text{ if } l < \ceil{2\alpha}, \\
         | \partial_R^{\ceil{2\alpha}} (u_{k-1}+s_1s_2y)| &\lesssim \frac{|u_{k-1}|}{1+R^{\ceil{2\alpha}}} \cdot \left(1 + \chi(R\tau^{-1}) \cdot \left|1-\frac{R}{\tau} \right|^{-\frac{1}{2}+}\right) +  s_1s_2|\partial_R^{\ceil{2\alpha}} y|,
    \end{align}
    (see Theorem \ref{Thm:Approximate Solution} and Theorem \ref{Thm:Approximate Solution d = 4}) in order to estimate the $L^2(\mathbb R^d)$ norm of
    $$\lambda^{-2} \partial_{R}^{\ceil{2\alpha}} \left( y^2 \int_0^1\int_0^1 s_1F''(u_{k-1}+s_1s_2y)ds_2ds_1 \right),$$
    it is enough to estimate the norm of
    \begin{equation} \label{nonlinear sobolev estimate, square term}
         \lambda^{-2} \lvert \partial_R^i y \rvert \cdot \lvert \partial_R^j y \rvert \cdot |u_{k-1}|^{p-2} \left(1 + \prod_{l=1}^n \left| \frac{\partial_R^l y}{u_{k-1}} \right|^{m_l} \right) 
    \end{equation}
    when $i + j + n = \ceil{2\alpha}$, $1 \cdot m_1 + ... + n \cdot m_n = n$, and to treat separately the term
    \begin{equation} \label{nonlinear sobolev estimate, square term, special case}
        \lambda^{-2} y^2 \cdot |u_{k-1}|^{p-2} \left(1 + \chi(R\tau^{-1}) \cdot \left|1-\frac{R}{\tau} \right|^{-\frac{1}{2}+}  \right).
    \end{equation}
    In both cases, we can ignore the $\lambda^{-2}$ and $u_{k-1}$ because
     $$||u_{k-1}||_{L^{\infty}(\mathbb R^d)} \asymp  \lambda^{\frac{d-2}{2}}, \quad ||u_{k-1}^{-1}||_{L^{\infty}(R \sim \tau)} \lesssim \frac{\tau^2}{\lambda^{\frac{d-2}{2}}},$$
    which cause no issue since we will get a $\tau^{-2N}$, $N \gg \ceil{2\alpha} \geq 2$, factor from the $y$ products. Moreover, 
    \begin{align*}
        ||y^2||_{L^2(\mathbb R^d)} +  \left| \left| y^2 \cdot \partial_R^{\ceil{2\alpha}} y \right| \right|_{L^2(\mathbb R^d)} + \left| \left|  \partial_R^i y \cdot \partial_R^j y \cdot \left( \prod_{l=1}^n \partial_R^l y \right)^{m_l} \right| \right|_{L^2(\mathbb R^d)} \\
        \lesssim ||y||_{H^{1+2\alpha}}^2 + ||y||_{H^{1+2\alpha}}^3 + ||y||_{H^{1+2\alpha}}^{\ceil{2\alpha}}
    \end{align*}
    using Corollary \ref{Strauss Estimates Corollary}, which allows treating (\ref{nonlinear sobolev estimate, square term}). It remains to estimate (\ref{nonlinear sobolev estimate, square term, special case}). In that case,
    \begin{align*}
    \left| \left|  y^2 \cdot \chi(R\tau^{-1}) \cdot \left(1-\frac{R}{\tau} \right)^{-\frac{1}{2}+}  \right| \right|_{L^2(R \nsim \tau)} &\lesssim ||y^2||_{L^2} \lesssim \tau^{-2(N-2)
} ||y||_{L^{\infty,N-2}H^{1+2\alpha}}^2,  \\
\left| \left|  y^2 \cdot \chi(R\tau^{-1}) \cdot \left(1-\frac{R}{\tau} \right)^{-\frac{1}{2}+}  \right| \right|_{L^2(R \sim \tau)} 
     &\lesssim ||y||_{L^{\infty}}^2  \left| \left| \left( 1-\frac{R}{\tau} \right)^{-\frac{1}{2}+} \right| \right|_{L^2(R \sim \tau)}  \\
        &\lesssim \tau^{-2(N-2) + d/2
} ||y||_{L^{\infty,N-2}H^{1+2\alpha}}^2.
    \end{align*}

    Now, we deal with the term in (\ref{nonlinear sobolev estimate, integral form}) which is linear in $y$. Similarly to the nonlinear case, it is enough to estimate
    \begin{equation} \label{linear sobolev estimate, square term}
         \lambda^{-2} \lvert \partial_R^i y \rvert \cdot \lvert \partial_R^j (u_{k-1}-u_0) \rvert \cdot |u_{k-1}|^{p-2}
    \end{equation}
    when $i + j \leq \ceil{2\alpha}$ and to treat separately the case
    \begin{equation} \label{linear sobolev estimate, square term, special case}
                 \lambda^{-2} \lvert y \rvert \cdot \lvert u_{k-1}-u_0 \rvert \cdot |u_{k-1}|^{p-2} \cdot \left( 1 + \chi(R\tau^{-1}) \cdot R^{-\ceil{2\alpha}} \cdot \left|1-\frac{R}{\tau} \right|^{-\frac{1}{2}+} \right).
    \end{equation}
    Using 
$$||u_{k-1}||_{L^{\infty}(\mathbb R^d)}^{p-2} \asymp  \lambda^{\frac{6-d}{2}}, \quad ||u_{k-1}-u_0||_{L^{\infty}(\mathbb R^d)} \lesssim \frac{\lambda^{\frac{d-2}{2}}}{\tau^2},$$
    and noticing that the $L^2$-contribution of the singularity on $R \sim \tau$ is cancelled by the $R^{-\frac{d+3}{2}}$ pointwise decay, which comes from the $\ceil{2\alpha} \geq 2$ derivative of $u_{k-1}-u_0, u_{k-1}$ and Strauss estimates applied to $y$, we observe a smallness gain of $\tau^{-2}$ and conclude as before.

    As for the local Lipschitz bound, one needs to estimate
    \begin{align*}
       \lambda^{-2}&\left[F(u_{k-1}+ y_1)-F(u_{k-1} + y_2)-F'(u_{k-1})(y_1-y_2) \right] 
       \\ +  \lambda^{-2}&\left[F'(u_{k-1})(y_1-y_2)-F'(u_0)(y_1-y_2) \right] \\
        = \lambda^{-2} &(y_1-y_2)^2 \int_0^1\int_0^1 s_1F''(u_{k-1} + s_1s_2(y_1-y_2)]) ds_2ds_1 \\
        +  \lambda^{-2} &(y_1-y_2)(u_{k-1}-u_0)\int_0^1F''(u_0 + s(u_{k-1}-u_0))ds
    \end{align*}
    for $y = y_1-y_2 \in L^{\infty,N-2}H^{2\alpha+1}_{\mathrm{rad}}$ supported on $0 \leq R < 2\tau$. The proof is exactly the same as before.
\end{proof}

For $\nu, \alpha$ as in Proposition \ref{prop:local lipschitz}, the strategy is to choose a threshold $N^*(\nu), \tau^*(\nu)$ for which the proposition applies when $N \geq N^*(\nu), \tau_0 \geq \tau^*(\nu)$ with some Lipschitz constant $C^*$ near $0$. Now, we can choose $\kappa$ small enough in (\ref{contraction constant kappa}) (depending on $C^*$, $\mathcal{K}$, $\nu$, $\alpha$) and then $N$, $\tau_0$ large enough so that the right-hand side operator of (\ref{contraction operator}) becomes a contraction on a small closed ball centered at 
\begin{align*}
\underline{\mathbf{x}}_0 &= (\mathcal{D}^2_{\tau} + \beta(\tau) \mathcal{D}_{\tau} + \xi )^{-1} \left[ \lambda(\tau)^{-2}\mathcal{F}R^{\frac{d-1}{2}}e_{k-1}\right] \in L^{\infty,N-2}L^{2,\alpha+\frac{1}{2}}_{\rho} ,
\end{align*}
which proves that the fixed-point iteration (\ref{contraction operator}) converges.

 \section{End of the proof} \label{section:end of proof}
 This final section concludes the proof of the main theorem by showing that the exact solution $u$, which has been rigorously constructed inside the light cone, extends as an exact solution on $\mathbb R^d$. The argument proceeds in three steps. First, we use the constructed solution $u$ to define initial data at a small time $t_0$ and invoke local well-posedness theory to guarantee the existence of an exact solution $v$ evolving backward from this data. Second, we apply the principle of finite speed of propagation to the difference $w = u - v$ to prove that $w$ must be zero inside the cone. Finally, we rely on the small-data global well-posedness theory to show that $v$ does not blow up before time zero.
 
 Let $d \in \{4,5\}$ and $\nu > (d-2)/(6-d)$. Theorems \ref{Thm:Approximate Solution} and \ref{Thm:Approximate Solution d = 4} alongside the fixed point argument from Section \ref{section:exact solution fourier method} show that there exists a radial function $u(x,t)$ on $\mathbb R^d \times [0,t_0]$, $t_0 \ll 1$, which solves a nonlinear equation
\begin{align*}
    \square u &= \square u_{k-1} + e_{k-1} + F\left[ \left( 1-\chi\left(R \tau^{-1}\right) \right) u_{k-1} + \chi \left(R \tau^{-1}\right) u \right] \\
    &\phantom{=}-F(u_{k-1}) + F'(u_0)\left( 1-\chi\left(R \tau^{-1}\right) \right) (u-u_{k-1}),
\end{align*}
 where 
 \begin{enumerate}
 \item $F(x) = |x|^{p-1}x$.
     \item $0 \leq \chi \leq 1$ is a smooth cutoff which is $1$ on $|x| \leq 1$ and $0$ on $|x| \geq 2$.
     \item $u_{k-1}, e_{k-1}$ were extended outside the cone (Remark \ref{extension of v_k, e_k}).
     \item Inside the cone $0 < |x| < t, 0 < t \leq t_0$, the relation $e_{k-1} = F(u_{k-1})-\square u_{k-1}$ holds and $\chi(R\tau^{-1}) = 1$, so that $\square u = F(u)$.
 \end{enumerate}
 This solution is of the form
$$
u(x,t) = \lambda(t)^{\frac{d-2}{2}}W(\lambda(t)x) \chi(|x|/t) + \eta(x,t), \quad \eta(x,t) = u^e(x,t) + \varepsilon(x,t), \quad \lambda(t) = t^{-1-\nu},
$$
where $u^e \in C^{\frac{1}{2}+\frac{6-d}{2}\nu-}(\mathbb R^d)$ has support in $0 \leq R < 2\tau$ and
\begin{align*}
\sup_{0 < t < t_0} t^{-\frac{6-d}{2}\nu-1}||u^e||_{H^{1+\frac{6-d}{2}\nu-}(\mathbb R^d)} +  t^{-\frac{6-d}{2}\nu} ||\partial_t u^e||_{H^{\frac{6-d}{2}\nu-}(\mathbb R^d)} &< +\infty, \\
\sup_{0 < t < t_0} t^{-N_0}||\varepsilon||_{H^{1+\frac{6-d}{2}\nu-}(\mathbb R^d)}+ t^{-N_0+1}||\partial_t \varepsilon||_{H^{\frac{6-d}{2}\nu-}(\mathbb R^d)} &< +\infty,
\end{align*}
for an arbitrarily large $N_0 \gg 1+\nu$, as well as
\begin{align*}
 \lim \limits_{t \to 0} \int_{|x| < c t} \left| \partial_{x_i} \left[ \lambda(t)^{\frac{d-2}{2}}W(\lambda(t)x) \right] \right|^2 dx &= \lim \limits_{t \to 0} \int_{|x| < c \lambda t} \left| \partial_{x_i} W(x) \right|^2 dx = ||\partial_{x_i} W||_{L^2}^2, \\
    \lim \limits_{t \to 0} \int_{|x| > c t} \left| \partial_{x_i} \left[ \lambda(t)^{\frac{d-2}{2}}W(\lambda(t)x) \right] \right|^2 dx &= \lim \limits_{t \to 0} \int_{|x| > c \lambda t} \left| \partial_{x_i} W(x) \right|^2 dx = 0, \\
    \int_{|x| < c t} \left| \partial_{t} \left[ \lambda(t)^{\frac{d-2}{2}}W(\lambda(t)x) \right] \right|^2 dx &= \mathcal{O}\left((t\lambda)^{-\frac{1}{2}} \right), \\
    \int_{|x| < c t} \left| \lambda(t)^{\frac{d-2}{2}}W(\lambda(t)x) \right|^{p+1} dx &= \mathcal{O}\left((t\lambda)^{-\frac{d-2}{2}}\log(t\lambda) \right),
\end{align*}
for any constant $c > 0$. These estimates hold true with $u_0 = \lambda(t)^{\frac{d-2}{2}}W(\lambda(t)x) \chi(|x|/t)$ instead of $\lambda(t)^{\frac{d-2}{2}}W(\lambda(t)x)$. It suffices to note that the cutoff $\chi(|x|/t)$ creates harmless additional terms
$$
\lambda(t)^{\frac{d-2}{2}}W(\lambda(t)x) \chi'(|x|/t) \frac{x_i}{t^2}, \quad \lambda(t)^{\frac{d-2}{2}}W(\lambda(t)x) \chi'(|x|/t) \frac{1}{t},
$$
when taking derivatives. On the region of support $|x| \sim t$, one has 
$$
||\lambda(t)^{\frac{d-2}{2}}W(\lambda(t)x) \chi'(|x|/t) t^{-1} ||_{L^2} \lesssim t^{\frac{1}{2}} \left| \left|\nabla_x \left[ \lambda(t)^{\frac{d-2}{2}}W(\lambda(t)x) \right] \right| \right|_{L^2} \lesssim t^{\frac{1}{2}} ||W||_{\dot{H}^1}
$$
thanks to Hölder's inequality and Hardy's inequality. 
% In particular, 
% \begin{align*}
%     \lim \limits_{t \to 0} \int_{|x| > t} |\nabla_{t,x} u(t)|^2  + |u(t)|^{p+1} dx &= 0 \\
%     \lim \limits_{t \to 0} \int_{|x| < t} |\nabla_{t,x} u(t)|^2  + |u(t)|^{p+1} dx &= ||W||_{\dot{H}^1} \\
%     \lim \limits_{t \to 0} \int_{x \in \mathbb R^d} |\nabla_{t,x} (u-u_0)(t)|^2  + |(u-u_0)(t)|^{p+1} dx &= 0
% \end{align*}

Finally, we remark that the regularity of $(\eta, \partial_t \eta)$ is at least $H^2 \times H^1$. Our final goal is to construct a solution to (NLW) on $\mathbb R^d \times (0,t_0]$ that coincides with $u$ inside the cone $0 < |x| \leq t, 0 < t \leq t_0$. Recall now the following local well-posedness theorem for (NLW).
\begin{theorem}[Local well-posedness for (NLW)]\label{thm:local well-posedness for nlw}
    Let $S(t)((u_0, u_1))$ denote the solution operator for the linear wave equation with initial conditions $(u_0, u_1) \in \dot{H}^1 \times L^2$ at $t = 0$. Let $0 \in I$ be any interval. If
    $$
    ||(u_0, u_1)||_{\dot{H}^1 \times L^2} \leq A,
    $$
    there exists $\delta(A) > 0$ for which 
    $$
    ||S(t)(u_0, u_1)||_{L_{t,x}^{\frac{2(d+1)}{(d-2)}}(\mathbb R^d \times I)} \leq \delta
    $$
    implies the existence of a unique solution $(u, \partial_t u) \in C^0(I, \dot{H}^1 \times L^2)$ of (NLW) with initial data $(u_0, u_1)$ at $t = 0$.
\end{theorem}

\begin{proof}
   See \cite[Theorem 2.7]{Kenig_Merle_local_wellposed_low_dim}.
\end{proof}

\begin{remark}[Additional properties]\label{rmk:strichartz,global well-posedness,regularity for nlw}
    Strichartz estimates (\cite[Lemma 2.1]{Kenig_Merle_local_wellposed_low_dim}) show that
$$
||S(t)(u_0, u_1)||_{L_{t,x}^{\frac{2(d+1)}{(d-2)}}(\mathbb R^d \times I)} \leq C ||(u_0, u_1)||_{\dot{H}^1 \times L^2}
$$
for some constant independent of $I$, meaning that Theorem \ref{thm:local well-posedness for nlw} applies if we choose $I$ sufficiently small.

 Moreover, if $A$ is small enough, the conclusion of the theorem always holds and we obtain existence of a global solution (see the proof of Theorem 2.7 in \cite{Kenig_Merle_local_wellposed_low_dim}, as well as Remark 2.10 in the same paper).

 One also has persistence of regularity: if $(u_0,u_1) \in \dot{H}^1 \cap \dot{H}^{1+\mu} \times H^{\mu} = \mathcal{H}$ for some $0 \leq \mu \leq 1$, then $(u,\partial_t u) \in C^0(I,\mathcal{H})$ (\cite[Remark 2.9]{Kenig_Merle_local_wellposed_low_dim}). In the $(u_0,u_1) \in H^2 \times H^1$ case, it follows from Duhamel formula that $\partial_{tt} u \in C^0(I,L^2)$. 
%  Finally, we have finite speed of propagation (\cite{Kenig_Merle_local_wellposed_low_dim}, Remark 2.12): if $(u_0,u_1)$, $(v_0,v_1) \in \dot{H}^1 \times L^2$ are two initial data for which the smallness assumptions from Theorem \ref{thm:local well-posedness for nlw} are verified on $I$ and if the initial data coincide on some ball $B(x_0,t_0)$, then $u = v$ on the backwards light cones
% \begin{align*}
%  \{(x,t): 0 \leq |x-x_0| \leq t_0-t, 0 \leq t \leq t_0\} &\cap (\mathbb R^d \times I) \\
%  \{(x,t): 0 \leq |x-x_0| \leq t_0+t, -t_0 \leq t \leq 0\} &\cap (\mathbb R^d \times I)
% \end{align*}
\end{remark}

% we will also need a finite propagation speed result for a nonlinearity other than $F(x) = |x|^{p-1}x$.

\begin{theorem}[Finite Speed of Propagation] \label{thm:finite propagation speed}
    Let $x_0 \in \mathbb R^d$, $t_0 \geq 0$ and
    $$
    K = \{(x,t): 0 \leq |x-x_0| \leq t_0-t, 0 \leq t \leq t_0\}.
    $$ 
    Let $u(x,t)$ be a ``strong'' solution of a nonlinear wave equation
    $$
    \square u = f(u), \quad t \in I = [0,T]
    $$
    for which 
    $$
    u(x,0) = \partial_t u(x,0) = 0, \quad x \in B(x_0,t_0), \quad x_0 \in \mathbb R^d, t_0 > 0.
    $$
    By “strong” solution, we mean that $u(x,t)$ has the following smoothness
    $$
    u \in C^0([0,T], H^2), \quad \partial_t u \in C^0( [0,T], H^1), \quad \partial_{tt} u \in C^0( [0,T], L^2),
    $$
    and $u(x,t) \in L^{\infty}(K \cap \mathbb (\mathbb R^d \times [0,t_1]))$ for any $0 \leq t_1 < \min\{t_0,T\}$. Moreover, assume that, given this regularity, $f(u)$ is a measurable function of $(x,t)$ for which
    $$
|f(u)| \leq C(||u||_{L^{\infty}(K \cap (\mathbb R^d \times [0,t_1]))})|u|
    $$
    almost everywhere on $K \cap (\mathbb R^d \times [0,T])$. Then $u = 0$ on $K \cap (\mathbb R^d \times [0,T])$.
\end{theorem}

% \begin{remark}
%     The same result holds backwards in time by replacing $[0,T]$ and $K$ with $[-T,0]$ and 
%     $$
%     \{(x,t): 0 \leq |x-x_0| \leq t_0+t, -t_0 \leq t \leq 0 \}
%     $$
% \end{remark}

\begin{proof}
    We follow the usual energy argument (\cite[Section 12.1.2]{evans10}). For any $0 \leq t_1 < \min\{t_0,T\}$, we prove that $u = 0$ on $\{(x,t): 0 \leq |x-x_0| \leq t_0-t, 0 \leq t \leq t_1 \}$. Let
    $$
    E(t) = \frac{1}{2} \int_{B(x_0,t_0-t)} u_t^2 + |\nabla_x u|^2 + u^2 dx, \quad 0 \leq t_1.
    $$
    Differentiation yields
    $$
    E'(t) =  \int_{B(x_0,t_0-t)} u_t u_{tt} + \nabla_x u \cdot \nabla_x u_t  + u u_t dx -  \frac{1}{2} \int_{|x-x_0| = t_0-t} u_t^2 + |\nabla_x u|^2 + u^2 dS.
    $$
    The differentiation formula 
    $$
    \frac{d}{dt} \int_{B(x_0,t_0-t)} g(t,x) dx = \int_{B(x_0,t_0-t)} g_t(t,x) dx - \int_{|x-x_0| = t_0-t} g(t,x) dS
    $$
    is justified for classical $g(t,x) \in C^1(I \times \mathbb R^d)$ functions. It also holds in the vector-valued setting $g \in C^1(I, L^2(\mathbb R^d))$ by approximating 
    $$
    g(t,x) = \sum_{n=1}^{\infty}\langle g(t,\cdot), v_k(\cdot) \rangle_{L^2} \cdot v_k(x)
    $$
    using a complete orthonormal basis of smooth functions $\{v_1,v_2,...\}$ of $L^2(\mathbb R^d)$. Uniform convergence in the $C^1([0,T],L^2(\mathbb R^d))$-norm follows from Dini's theorem: the decreasing sequence of continuous functions
    $$
    t \mapsto \left| \left| \sum_{k=n}^{\infty}\langle g(t,\cdot), v_k(\cdot) \rangle_{L^2} \cdot v_k(x) \right| \right|_{L^2(\mathbb R^d)}^2 = \sum_{k=n}^{\infty}\langle g(t,\cdot), v_k(\cdot) \rangle_{L^2}^2, \quad t \in [0,T],
    $$
    converges pointwisely to $0$ as $n \to +\infty$, hence uniformly by Dini's theorem. Uniform convergence of the sequence 
    $$
    g_t(t,x) = \sum_{n=1}^{\infty}\langle g_t(t,\cdot), v_k(\cdot) \rangle_{L^2} \cdot v_k(x)
    $$
    holds in the same fashion. An integration by parts, which is valid for Sobolev functions (\cite{evans_measure_theory}, Section 4.6), yields 
    \begin{align*}
    E'(t) =  \int_{B(x_0,t_0-t)} u_t (u_{tt} + \Delta_x u + u ) dx &-  \frac{1}{2}\int_{|x-x_0| = t_0-t} u_t^2 + |\nabla_x u|^2 + u^2 dS \\
    &+ \int_{|x-x_0| = t_0-t} (\nabla_x u \cdot \nu) u_t dS, 
    \end{align*}
    where $\nu$ is the normal outward-pointing vector of the surface $|x-x_0| = t_0-t$. Hence,
    $$
    E'(t) \leq \int_{B(x_0,t_0-t)} u_t (-F(u) + u ) dx \leq C   \int_{B(x_0,t_0-t)} |u_tu| dx \leq C E(t)
    $$ 
    since $ab \leq a^2/2 + b^2/2$ for $a,b \geq 0$. Since $E(0) = 0$, Grönwall's Lemma implies that $E(t) = 0$.
\end{proof}

Let $0 < A \ll 1$ be small enough so that the global well-posedness result from Remark \ref{rmk:strichartz,global well-posedness,regularity for nlw} holds for initial data $||(u_0,u_1)||_{\dot{H}^1 \times L^2} \leq 3A$. Let $0 < t_0 \ll 1$ be small enough (depending on $A$ and $||W||_{\dot{H^1}}$) so that for all $0 < t \leq t_0$,
\begin{align} 
\left| \int_{|x| < \frac{1}{2} t} |\nabla_{x} u(t)|^2 dx  - ||W||_{\dot{H^1}}^2 \right| &\ll A, \label{estimate on nabla_x u inside cone}  \\
\int_{|x| > t} |\nabla_{x} u(t)|^2dx + \int_{x \in \mathbb R^d} |\partial_t u(t)|^2  + |u(t)|^{p+1} dx &\ll A, \label{negligble estimate on u}  
    % \int_{x \in \mathbb R^d} |\nabla_{t,x} (u-u_0)(t)|^2  + |(u-u_0)(t)|^{p+1} dx &\ll A \label{estimate on u - u_0}
\end{align}
where $u_0 = \lambda(t)^{\frac{d-2}{2}}W(\lambda(t)x) \chi(|x|/t)$.

Let $v$ be the local solution of (NLW) constructed at $t_0$ with initial data $(u(t_0), \partial_tu(t_0))$ (we solve (NLW) backwards in time). Assume that $v$ exists on $I = [T, t_0]$ with $0 < T < t_0$. Then the difference $w = u-v$ solves a nonlinear equation
$$
\square w = f(w)
$$
with initial conditions
    $$
    u(x,0) = \partial_t u(x,0) = 0, \quad x \in B(0,t_0),
    $$
and where $f(w) = F(u)-F(u-w) = w \int_{-1}^0 F'(u + s(u-v))ds$ on the cone $0 \leq |x| \leq t, 0 < t \leq t_0$. By Theorem \ref{thm:finite propagation speed} ($w$ has the desired regularity and the local boundedness properties hold thanks to radiality of $u$, $v$), $w = 0$ and $v = u$ on the section of the cone $0 \leq |x| \leq t, 0 < T \leq t \leq t_0$.

Next, we prove that the $\dot{H}^1 \times L^2$-norm of $v$ stays small outside the cone via conservation of energy. Let 
\begin{align*}
    E(v(t),\partial_t v(t)) = E(v(t_0), \partial_t v(t_0)) &= 
    E(u(t_0),\partial_t u(t_0)) \\
    &= \int_{\mathbb R^d} \frac{1}{2} |\nabla_{t,x}u(t_0)|^2 - \frac{1}{p+1}|u(t_0)|^{p+1}dx \\
    &= \frac{1}{2}||W||_{\dot{H}^1}^2  \pm A/8, \quad t \in I
\end{align*}
using (\ref{estimate on nabla_x u inside cone}) and (\ref{negligble estimate on u}. Inside the cone, $v = u$ so we also have 
$$
\int_{|x| < \frac{1}{2} t} \frac{1}{2} |\nabla_{t,x}v(t)|^2 + \frac{1}{p+1}|v(t)|^{p+1}dx = \frac{1}{2}||W||_{\dot{H}^1}^2 \pm A/8, \quad t \in I.
$$
Hence,
\begin{align*}
    % \int_{|x| \geq t} \frac{1}{2} |\nabla_{t,x}v(t)|^2 + \frac{1}{p+1}|v(t)|^{p+1}dx &= E(v(t),\partial_t v(t)) -  \int_{|x| < t} \frac{1}{2} |\nabla_{t,x}v(t)|^2 + \frac{1}{p+1}|v(t)|^{p+1}dx +  2 \int_{|x| \geq t} \frac{1}{p+1}|v(t)|^{p+1}dx\\
        \int_{|x| \geq \frac{1}{2} t} \frac{1}{2} |\nabla_{t,x}v(t)|^2 + \frac{1}{p+1}|v(t)|^{p+1}dx &\leq A/4 + 2 \int_{|x| \geq \frac{1}{2} t} \frac{1}{p+1}|v(t)|^{p+1}dx \\
    &\leq A/4 + C \left( \int_{|x| \geq \frac{1}{2} t} |\nabla_x v(t)|^{2}dx  \right)^{\frac{p+1}{2}}, \quad t \in I,
\end{align*}
where $C$ is the norm of the Sobolev embedding $\dot{H}^1(\mathbb R^d) \hookrightarrow L^{p+1}(\mathbb R^d)$. In other words, the continuous function
$$
t \mapsto h(t) =  \int_{|x| \geq \frac{1}{2} t} |\nabla_{t,x}v(t)|^2, \quad t \in I,
$$
satisfies
\begin{align*}
    h(t) &\leq A/2 + 2Ch(t)^{\frac{p+1}{2}}, \quad t \in I, \\
    h(t_0) &\ll A.
\end{align*}
Assume for a contradiction that there is some $t \in I$ for which $h(t) > A$. By continuity, there is a $t^* \in I$  for which $h(t^*) = A$, meaning that
$$
A \leq A/2 + 2C A^{\frac{p+1}{2}} \iff A \geq \left( \frac{1}{4C} \right)^{\frac{2}{p-1}}.
$$
The quantity on the right-hand side depends only on $d$ and $p$, therefore, we can choose $A$ to be sufficiently small at the outset to prevent this from occurring. Thus, $h(t) < A$ for all $t \in I$.

Finally, we prove that $v$ exists up to time $t = 0$ (hence $v$ extends $u$ outside the cone). Assume that $v$ exists for time $t \in [T_{-},t_0)$ with $0 < T_{-} \leq t_0$. Let $0 \leq \psi \leq 1$ be a smooth cutoff which is $1$ on $|x| \leq 1/2$ and $0$ on $|x| \geq 3/4$. Consider the solution $w$ obtained by solving (NLW) with initial data
$$
(w_0, w_1) = \left(1-\psi\left( \frac{|x|}{T_{-}} \right) \right) v(T), \partial_t v(T) \in H^2 \times H^1.
$$
For $|x| \geq 3T_{-}/4$, this coincides with $v(T_{-}), \partial_t v(T_{-})$ and we have small energy
$$
\int_{|x| \geq 3T_{-}/4} |\nabla_x w_0|^2 + |w_1|^2 dx \leq A.
$$
If $t_0$ was chosen sufficiently small in the first place (depending on $\psi$, $d$, $p$, which are independent of $v$ and its interval of existence), then
$$
\int_{\frac{1}{2}T_{-} \leq |x| \leq 3T_{-}/4} |\nabla_x w_0|^2 + |w_1|^2 dx \leq 2A,
$$
as well using Hardy's inequality to estimate the components where the derivative falls in the cutoff. The small-energy global well-posedness theory implies that $w$ is a global solution and finite speed of propagation implies that $v = w$ on some neighbourhood of $\{x \in \mathbb R^d: |x| \geq 3T_{-}/4\} \times \{T_{-}\} \subset \mathbb R^d \times [T_{-},t_0]$. Hence, $v$ can be extended as
$$
v(x,t) = \begin{cases}
    u(x,t), \quad (x,t) \in \{(x,t): 0 < |x| < t, 0 < t \leq T_{-} ]\} \\
    w(x,t), \quad (x,t) \in \{(x,t): |x| \geq t, 0 < t \leq T_{-} ]
\end{cases}
$$
which concludes the proof.

\appendix
\section{Some results about regular singular ODEs} \label{section:appendix, ode}

In this appendix, we consider a linear ordinary differential equation
\begin{equation} \label{regular singular fuchs ode}
    u''(z) + p(z)u'(z) + q(z)u(z) = 0, \quad z \in \mathbb C,
\end{equation}
around a regular singular point $0 \in \mathbb C$, meaning that
$$p(z) = \frac{1}{z}\sum_{n=0}^{+\infty}p_n z^n, \quad q(z) = \frac{1}{z^2}\sum_{n=0}^{+\infty}q_n z^n, \quad |z| < R.$$
The standard method for finding solutions for this equation (with zero or analytical forcing term) is to make a power series Ansatz. This is called the Frobenius method. The goal of this appendix is to generalize the Frobenius method to solve the equation with power and logarithmic forcing terms.

We recall (see \cite[Chapter 4]{Teschl_ode}) that if $\{r_1, r_2\}$, $\Re(r_1) \geq \Re(r_2)$, are the roots of the indicial equation $\alpha^2 + (p_0-1)\alpha+q_0 = 0$, then one can find a fundamental system of the form
\begin{equation} \label{fuchs fundamental system}
    u_1(z) = z^{r_1} h_1(z), \quad u_2(z) = \underbrace{z^{r_2}h_2(z)}_{=: \tilde{u}_2(z)} + c \cdot \log(z) u_1(z),
\end{equation}
where $h_i(z)$ is analytic at $0$ with $h_i(0) = 1$ and radius of convergence at least equal to the distance between $0$ and the next singularity of $p(z)$ and $q(z)$. Moreover, if $r_1 - r_2 \notin \mathbb N_{\geq 0}$, then the constant $c \in \mathbb C$ is necessarily $0$ and if $r_1 = r_2$, then $c \in \mathbb C$ is necessarily non-zero.

Finally, one observes that $r_1 + r_2 = 1 - p_0$, $r_1r_2 = q_0$ and the Wronskian is of the form
\begin{equation} \label{fuchs wronskian}
    W(z) = Cz^{-p_0}\exp\left(-\sum_{n=1}^{+\infty}\frac{p_n}{n} z^{n} \right) = z^{-p_0}h_3(z),
\end{equation}
where $h_3(z)$ is analytic and non-zero on $|z| < R$.

In the following, we are interested in solving inhomogeneous regular singular problems where the forcing term can be a combination of powers and logarithms.

% \begin{theorem}[Hadamard Multiplication Theorem] \label{hadamard mult thm}
%     Let $f(z), g(z)$ be holomorphic at $z = 0$
%     $$f(z) = \sum_{n=0}^{+\infty}f_nz^n, \quad g(z) = \sum_{n=0}^{+\infty}g_nz^n$$
%     with respective radius of convergence $R_f$ and $R_g$. Then the Hadamard product
%     $$h(z) = \sum_{n=0}^{\infty}f_ng_nz^n$$
%     is holomorphic on $|z| < R_f \cdot R_g$ and
%     $$h(x \cdot y) = \frac{1}{2\pi}\int_0^{2\pi} f(xe^{i\theta})g(ye^{-i\theta})d\theta$$
%     for any $|x| < R_f$ and $|y| < R_g$.
% \end{theorem}

% \begin{proof}
%     See \cite{Hadamard}.
% \end{proof}

\begin{proposition}[Parseval Identity]
    Let $f(z)$ be holomorphic on $B(0,R)$ with
    $$f(z) = \sum_{n=0}^{+\infty}f_nz^n, \quad |z| < R.$$
    Then for any $0 < r < R$,
    $$\sum_{n=0}^{\infty}|f_n|^2r^{2n} = \frac{1}{2\pi}\int_0^{2\pi} |f(re^{i\theta})|^2 d\theta.$$
\end{proposition}

\begin{proof}
    Write 
    $$ \frac{1}{2\pi}\int_0^{2\pi} |f(re^{i\theta})|^2 d\theta =  \frac{1}{2\pi}\int_0^{2\pi} f(re^{i\theta}) \overline{f(re^{i\theta})} d\theta$$
    and expand both $f(re^{i\theta})$, $\overline{f(re^{i\theta})}$ around zero.
\end{proof}

\begin{definition}[Wiener Space]\label{wiener space}
    The Wiener Space $A(|z| < R)$ is the normed vector space of holomorphic functions 
    $$f(z) = \sum_{n=0}^{\infty} f_nz^n, \quad  |z| < R,$$
    with coefficients $(f_n R^n)_{n = 0}^{\infty} \in \ell^1(\mathbb N)$. The norm on this space is defined as
    $$||f||_{A(|z| < R)}:= ||(f_n R^n)||_{\ell^1(\mathbb N)}$$
    and we observe that
    $$||\sum_{n=m_1}^{m_2}f_nz^n ||_{L^{\infty}(|z| < R)} \leq ||f||_{A(|z| < R)} \quad \forall 0 \leq m_1 \leq m_2 \leq +\infty,$$
    as well as
    $$
||f^{(k)}||_{A(|z| < \delta R)} = \left( \delta R \right)^{-k}  \sum_{n=0}^{+\infty} \frac{n!}{(n-k)!} \delta^{n} |f_n| R^n \lesssim_{\delta,R} ||f||_{A(|z| < R)} \quad \forall k \in \mathbb N_{\geq 0}
    $$
    whenever $0 < \delta < 1$.
    
    This space is an algebra and for $g(z) \in A(|z| < R)$ fixed, the multiplication operator $T_g: f \mapsto f \cdot g$ is a bounded operator with norm $||T_f|| = ||g||_{A(|z| < R)}$.
    \index{A(z < R)@$A(\lvert z \rvert < R)$, the Wiener space}
\end{definition}
% \begin{theorem}[Hadamard Multiplication Theorem]
%     Let
%     $$f(z) = \sum_{n=0}^{+\infty}f_nz^n, \quad \phi(z) = \sum_{n=0}^{+\infty}\phi_nz^n$$
%     be analytic at $z = 0$ with respective radius of convergence $R_f, R_{\phi}$. Then the Hadamard product
%     $$\psi(z) = \sum_{n=0}^{+\infty}f_n \phi_n z^n$$
%     has radius of convergence at least $R_f \cdot R_{\phi}$ and
%     $$\psi(x \cdot y) = \frac{1}{2\pi}\int_0^{2\pi} f(x e^{i\theta}) \phi(ye^{-i\theta}) d\theta$$
%     for any $x,y \in \mathbb C$ with $|x| < R_f$, $|y| < R_{\phi}$.
% \end{theorem}

% \begin{proof}
%     \cite{Hadamard}.
% \end{proof}

% \begin{corollary}
%     Let
%     $$f(z) = \sum_{n=0}^{+\infty}f_nz^n, \quad |z| < R$$
%     be analytic at $z = 0$ and let $(a_n)_{n \geq 0}$ be any sequence in $\ell^{\infty}(\mathbb N)$. Then for any $r < R$,
%     $$\left| \left| \sum_{n=0}^{+\infty}a_n f_nz^n \right| \right|_{L^{\infty}(|z| < r)} \leq \sum_{n=0}^{+\infty} |a_nf_n|r^n \leq \frac{||a_n||_{\ell^{\infty}(\mathbb N)}}{} \cdot ||f||_{L^{\infty}(|z| < r)}$$ 
% \end{corollary}

% \begin{proof}
%     Let
%     $$\phi(z) = \sum_{n=0}^{+\infty} |a_n| \text{sgn}(f_n) z^n, \quad \psi(z) = \sum_{n=0}^{+\infty} |a_n f_n| z^n$$
%     Apply Hadamard Multiplication Theorem to obtain
%     $$ \sum_{n=0}^{+\infty} |a_nf_n|r^n = ||\psi||_{L^{\infty}(|z|<r)} \leq ||\phi||_{L^{\infty}(|z|<1)} \cdot ||f||_{L^{\infty}(|z|<r)}$$
% \end{proof}

\begin{theorem}[Nonhomogeneous ODE with singular forcing term] \label{thm:inhomogeneous fuchs ode}
Consider a regular singular ODE (\ref{regular singular fuchs ode}) around $0 \in \mathbb C$ with indicial roots $\{r_1,r_2\}$, $\Re(r_1) \geq \Re(r_2)$, and assume that $p(z), q(z)$ have radius convergence $R + \varepsilon$. Let $u_1(z) = z^{r_1}h_1(z), u_2(z) = z^{r_2}h_2(z) + c \cdot u_1(z)\log(z)$ be the fundamental system from (\ref{fuchs fundamental system}). Let $\beta \in \mathbb C$, $j \in \mathbb N_{\geq 0}$ and 
$$g(z) = \sum_{n=0}^{\infty}g_nz^n, \quad |z| < R+\varepsilon.$$ 
In the following, we define $g_{r_i-\beta}$ to be the $(r_i-\beta)$-th order term of $g(z)$ series expansion at $z = 0$ if $r_i - \beta \in \mathbb N_{\geq 0}$ and $g_{r_i-\beta} = 0$ otherwise.

The inhomogeneous equation
\begin{equation} \label{eq:inhomogeneous fuchs ode}
    w''(z) + p(z)w'(z) + q(z)w(z) = z^{\beta-2} g(z) \log(z)^j, \quad |z| < R+\varepsilon, z \notin \mathbb R_{\leq 0},
\end{equation} 
has a particular solution $w(z)$ given by 
\begin{equation} \label{eq:inhomogeneous fuchs power series solution}
    w(z) = z^{\beta} \sum_{k=0}^{j+2} w_k(z)\log(z)^k,
    % w(z) = z^{\beta} \sum_{\substack{n = 0 \\ n \neq r_i-\beta}}^{+\infty}w_nz^n + w_{-1}u_1(z)\log(z) + w_{-2}u_2(z)\log(z) + w_{-3}u_1(z)\log(z)^2, \quad |z| < R
\end{equation} 
where 
\begin{enumerate}
    \item Each $w_k(z)$ is holomorphic on $|z| < R+\varepsilon$.
    \item $w_{j+2}(z) = 0$ if $r_1-r_2 \notin \mathbb N_{\geq 0}$ or $r_2-\beta \notin \mathbb N_{\geq 0}$. In other words, a non-trivial $\log(z)^{j+2}$ factor can only occur when both $r_1-\beta, r_2-\beta \in \mathbb N_{\geq 0}$. 
    \item $w_{j+1}(z) = 0$ if both $r_1-\beta, r_2-\beta \notin \mathbb N_{\geq 0}$. In other words, a non-trivial $\log(z)^{j+1}$ factor can only occur when $r_1-\beta \in \mathbb N_{\geq 0}$ or $r_2-\beta \in \mathbb N_{\geq 0}$.
    \item There exists $C(u_1(z),u_2(z),r_1,r_2,R)$ such that
    \begin{align*}
    ||w_{j+2}(z)||_{A(|z| < R)} &\leq C \cdot (j+1)^{-2} \cdot ||g(z)||_{L^{\infty}(|z| < R)},  \\   
    ||w_{j+1}(z)||_{A(|z| < R)}  &\leq C \cdot (j+1)^{-1} \cdot ||g(z)||_{L^{\infty}(|z| < R)}.
    \end{align*}    
    If both $r_1-\beta < 0, r_2-\beta < 0$, then one also has 
   \begin{equation}
          ||w_{k}(z)||_{A(|z| < R)}  \leq C \cdot \left( \frac{j}{\min_{i\in\{1,2\}} \{|\beta-r_i| \} } \right)^j \cdot ||g(z)||_{L^{\infty}(|z| < R)} \label{exponential upper bound on w_k}
   \end{equation}
    for all $k \in \{0,1,...,j\}$. Otherwise, 
    $$||w_{k}(z)||_{A(|z| < R)} \leq C \cdot j^j \cdot \min_{i \in \{1,2\}} \text{dist}(r_i-\beta,\mathbb Z \setminus \{r_i-\beta\} )^{-(j+1)} \cdot ||g(z)||_{L^{\infty}(|z| < R)}.$$
    \item If we write
   \begin{align*}
   w'(z) = z^{\beta-1} \sum_{k=0}^{j+2} w_{1,k}(z)\log(z)^k, \\
   w''(z) = z^{\beta-2} \sum_{k=0}^{j+2} w_{2,k}(z)\log(z)^k,
   \end{align*}
    then the estimates in (4) also hold with $w_{i,k}(z)$ instead of $w_k(z)$. For the second derivative, the constant $C$ will also depend on $p(z), q(z)$. 
    \item For $n > 2$, if we write
   $$w^{(n)}(z) = z^{\beta-n} \sum_{k=0}^{j+2} w_{n,k}(z)\log(z)^k,$$
   then the estimates in (4) also hold with $w_{n,k}(z)$ instead of $w_k(z)$ if we replace 
   \begin{align*}
       C(u_1,u_2,r_1,r_2,R) &\mapsto C(n,p,q,u_1,u_2,r_1,r_2,R) \cdot j^{n-2} \cdot \prod_{i=1}^{n-2}|\beta-i-1| \\
       ||g(z)||_{L^{\infty}(|z| < R)} &\mapsto \max_{0 \leq i \leq n-2} ||g^{(i)}(z)||_{L^{\infty}(|z| < R)}
   \end{align*}
\end{enumerate}
\end{theorem}

\begin{proof}
We use the fundamental system $\{u_1,u_2\}$ given by (\ref{fuchs fundamental system}). By variation of parameters, a particular solution is given by any 
\begin{align*}
    \tilde{w}(z) = \int_{[R/2,z]} [u_2(z)u_1(y)-u_1(z)u_2(y)]W(u_1,u_2)(y)^{-1}y^{\beta-2} g(y) \log(y)^j dy, \\
    \tilde{w}'(z) = \int_{[R/2,z]} [u_2'(z)u_1(y)-u_1'(z)u_2(y)]W(u_1,u_2)(y)^{-1}y^{\beta-2} g(y) \log(y)^j dy,
\end{align*}
when $|z| < R+\varepsilon, z \notin \mathbb R_{\leq 0}$, modulo some linear combination of $\{u_1,u_2\}$. Write
\begin{align*}
    u_1(y) \cdot W(u_1,u_2)(y)^{-1} \cdot g(y) &= y^{r_1 + p_0 }\underbrace{\sum_{n=0}^{\infty}a_n y^n}_{=: a(y)}, \quad |y| < R+\varepsilon, \\
    u_2(y) \cdot W(u_1,u_2)(y)^{-1} \cdot g(y) &= y^{r_2 + p_0} \underbrace{\sum_{n=0}^{\infty}b_n y^n}_{=: b(y)} + c \cdot y^{r_1 + p_0} \log(y) a(y) , \quad |y| < R+\varepsilon.
\end{align*}
Expanding everything in the variation of parameters, we find
\begin{align*}
    \tilde{w}(z) &= z^{r_2} h_2(z) \sum_{n=0}^{+\infty} a_{n} \int_{\frac{R}{2}}^z y^{r_1 + p_0 + \beta - 2+ n} \log(y)^j dy \\
    &+ c \cdot z^{r_1}h_1(z) \log(z) \sum_{n=0}^{+\infty} a_{n} \int_{\frac{R}{2}}^z y^{r_1 + p_0 + \beta - 2 + n} \log(y)^j dy \\
    &- z^{r_1}h_1(z) \sum_{n=0}^{+\infty} b_{n} \int_{\frac{R}{2}}^z y^{r_2 + p_0 + \beta - 2 + n} \log(y)^jdy \\
    &-c \cdot z^{r_1}h_1(z) \sum_{n=0}^{+\infty} a_{n} \int_{\frac{R}{2}}^z y^{r_1 + p_0 + \beta -2 + n} \log(y)^{j+1} dy.
\end{align*}
Using  $r_1 + r_2 = 1-p_0$, this rewrites as
\begin{align*}
\tilde{w}(z) &=  z^{r_2} h_2(z) \sum_{n=0}^{+\infty} a_{n} \int_{\frac{R}{2}}^z y^{\beta-r_2 + n - 1} \log(y)^j dy \\
&+ c \cdot z^{r_1}h_1(z) \log(z) \sum_{n=0}^{+\infty} a_{n} \int_{\frac{R}{2}}^z y^{\beta-r_2 + n - 1} \log(y)^j dy \\
    &- z^{r_1}h_1(z) \sum_{n=0}^{+\infty} b_{n} \int_{\frac{R}{2}}^z y^{\beta-r_1 + n - 1} \log(y)^j dy  \\
    &- c \cdot z^{r_1}h_1(z) \sum_{n=0}^{+\infty} a_{n} \int_{\frac{R}{2}}^z y^{\beta-r_2 + n - 1} \log(y)^{j+1}dy.
\end{align*}
Replacing $z^{r_i}h_i(z)$ above by 
$$\frac{d}{dz}(z^{r_i}h_i(z)) = z^{r_i-1}(r_ih_i(z) + zh_i'(z)) = z^{r_i-1}\tilde{h}_i(z),$$ 
we get a similar formula for $\tilde{w}'(z)$. If $\delta \in \mathbb R \setminus \{-1\}$, $j \in \mathbb N_{\geq 0}$, then a primitive of $y^{\delta} \log(y)^j$ is given by 
$$z^{\delta+1} \sum_{k=0}^j \frac{(-1)^k j!}{(j-k)!} \cdot \frac{\log(z)^{j-k}}{(\delta+1)^{k+1}}$$
and if $\delta = -1$, then a primitive is given by $(j+1)^{-1}\log(y)^{j+1}$. For $k \in \{0,...,j\}$, let
\begin{align*}
    A_{j,k}(z) =  \frac{(-1)^k j!}{(j-k)!} \sum_{\substack{n=0 \\ n \neq r_2-\beta}}^{+\infty} \frac{a_{n}}{(\beta-r_2+n)^{k+1} }z^n, \quad B_{j,k}(z) =  \frac{(-1)^k j!}{(j-k)!} \sum_{\substack{n=0 \\ n \neq r_1-\beta}}^{+\infty} \frac{b_{n}}{(\beta-r_1+n)^{k+1} }z^n.
\end{align*}
Ignoring the integration constants (which are zero modulo the fundamental system), we can find a solution of the form
\begin{align*}
    w(z) &= z^{\beta} \left[ h_2(z) \sum_{k=0}^j A_{j,k}(z)\log(z)^{j-k} - h_1(z)\sum_{k=0}^j B_{j,k}(z)\log(z)^{j-k} \right. \\
    &\phantom{=}+   c \cdot z^{r_1-r_2}h_1(z) \left( \sum_{k=1}^{j} A_{j,k}(z)\log(z)^{j+1-k}- \sum_{k=1}^{j+1} A_{j+1,k}(z)\log(z)^{j+1-k} \right)  \\
    &\phantom{=} + \left( h_2(z) a_{r_2-\beta} z^{r_2-\beta} - h_1(z)b_{r_1-\beta}z^{r_1-\beta}\right) \frac{\log(z)^{j+1}}{(j+1)} \\
     &\phantom{=} +  \left. c \cdot z^{r_1-r_2}h_1(z)a_{r_2-\beta}z^{r_2-\beta} \frac{\log(z)^{j+2}}{(j+1)(j+2)}  \right],
\end{align*}
where we recall that $c = 0$ if $r_1-r_2 \notin \mathbb N_{\geq 0}$ and we used the convention that $a_{r_i-\beta} = b_{r_i-\beta} = 0$ if $r_i - \beta \notin \mathbb N_{\geq 0}$. A similar expression for $w'(z)$ can be obtained by replacing $\beta$ by $\beta-1$ and $h_i(z)$ by $r_ih_i(z) + zh_i'(z)$.

\item[\textbf{Estimates on the solution: }] It is clear from the definition that
$$||a(z)||_{L^{\infty}(|z| < R)} + ||b(z)||_{L^{\infty}(|z| < R)} \lesssim_{u_1(z),u_2(z),R} ||g||_{L^{\infty}(|z| < R)}.$$
Moreover, it follows from Cauchy-Schwarz and Parseval identity that
$$||A_{j,k}(z)||_{A(|z| < R)} + ||B_{j,k}(z)||_{A(|z| < R)} \lesssim_{u_1(z),u_2(z),R} S_{r_1,r_2,j,k,\beta} \cdot ||g||_{L^{\infty}(|z| < R)},$$
where
$$S_{r_1,r_2,j,k,\beta} = \frac{j!}{(j-k)!} \sum_{i=1}^2 \left( \sum_{\substack{n=0 \\ n \neq r_i-\beta}}^{+\infty} \frac{1}{|\beta-r_i+n|^{2(k+1)} }  \right)^{\frac{1}{2}}.$$
If both $r_1-\beta, r_2-\beta \notin \mathbb R_{\geq 0}$, one has 
$$S_{r_1,r_2,j,k,\beta} \lesssim \left( \frac{j}{\min_{i\in\{1,2\}} \{|\beta-r_i| \} } \right)^j.$$
Otherwise, 
$$S_{r_1,r_2,j,k,\beta} \lesssim  j^j \cdot \min\{|\beta-r_i+n|: i \in \{1,2\}, n \in \mathbb Z, n \neq r_i-\beta \}^{-(j+1)}.$$
Finally, one has Cauchy's inequality
$$|a_{r_i-\beta}R^{r_i-\beta}| \leq ||a(z)||_{L^{\infty}(|z| < R)}, \quad |b_{r_i-\beta}R^{r_i-\beta}| \leq ||b(z)||_{L^{\infty}(|z| < R)}.$$
Combining everything, together with the boundedness of the multiplication operator on $A(|z| < R)$, we get the desired estimate for $w(z)$ and $w'(z)$. The estimates for $w^{(2+n)}(z)$, $n \geq 0$, follow by differentiating the ODE $n$ times.
\end{proof}

\begin{remark}[On analytic solutions to the inhomogeneous equation]\label{remark:analytic sol inhomogeneous fuchs equation}
    If $\beta \in \mathbb N_{\geq 2}$ and $j = 0$, i.e., the inhomogeneous equation (\ref{eq:inhomogeneous fuchs ode}) has an analytic forcing term, and both $r_1-\beta, r_2-\beta \notin \mathbb N_{\geq 0}$, we observe that the solution (\ref{eq:inhomogeneous fuchs power series solution}) is analytic. In this case, the solution takes the form
    \begin{align*}
        w(z) &= z^{\beta} \left( h_2(z) A_{0,0}(z) - h_1(z) B_{0,0}(z) \right) \\
        &= z^{\beta-r_2} u_2(z) A_{0,0}(z) - z^{\beta-r_1}u_1(z) B_{0,0}(z).
    \end{align*}
    If $\beta = 2$, $j = 0$, $r_2-\beta \notin \mathbb N_{\geq 0}$ but $r_2 \in \mathbb N_{\geq 0}$ and $r_1-\beta \in \mathbb N_{\geq 0}$, then the solution  (\ref{eq:inhomogeneous fuchs power series solution}) has a logarithmic component proportional to $u_1(z)\log(z)$ that can be removed by adding an appropriate multiple of $u_2(z)$, thus yielding an analytic solution.
\end{remark}

\section{Renormalization Step in dimension 4} \label{section:appendix, dimension 4}
We perform the main inductive argument of the renormalization procedure in dimension $d = 4$, explaining how to construct the even correction terms $v_{2k}$ from the error $e_{2k-1}$ by solving a wave-like equation in self-similar coordinates, and the odd correction terms $v_{2k+1}$ from the error $e_{2k}$ using an elliptic-like equation. We prove that at each step, there is a systematic decrease in the error, finishing the proof of Theorem \ref{Thm:Approximate Solution}. The formalism in dimension $d = 4$ is slightly different. In this situation, we can use the simpler framework from Krieger-Schlag-Tataru (\cite{Krieger_2007_waveEq}) to construct our approximate solution as the exponent $p = 3$ in the nonlinearity is an integer. Let $\nu > 0$ in this appendix. The restriction $\nu > 1$ in Theorem \ref{thm:blow-up, main thm} arises only in Proposition \ref{prop:local lipschitz}.

\begin{definition}
Let $\tilde{\mathcal{Q}}_{\beta}, \mathcal{Q}_{\beta}$ be defined as in Definition \ref{definition of tilde Q} but with no logarithmic singularity $\log(a)^j$ in the expansion at $a = 0$. In other words, we restrict to functions which are holomorphic at $a = 0$. We define $\mathcal{Q} = \mathcal{Q}_{\nu + 1/2}$ and $\mathcal{Q}' = \mathcal{Q}_{\nu - 1/2}$. This new family $\mathcal{Q}'$ is obtained from $\mathcal{Q}$ by applying $a \partial_a$, $a^{-1} \partial_a$ or $(1-a^2) \partial_{aa}$ and $\mathcal{Q}$ is obtained from $\mathcal{Q'}$ by applying $(1-a^2)$. Moreover, $\mathcal{Q} \subset \mathcal{Q}'$.
\index{Q@$\tilde{\mathcal{Q}}$, $\mathcal{Q}$, families of holomorphic functions with respect to the self-similar variable $a$}
\end{definition}

\begin{definition}[Space $S^m(R^k \log (R)^l, \mathcal{Q})$]
$S^m(R^k \log (R)^l)$ is the class of real-analytic functions $w(R): [0,\infty)  \rightarrow \mathbb R$ for which
\index{S-space@$S^{2n}(R^I,\log(R)^J)$, a vector space of smooth functions with respect to the radial variable $R$}
\begin{enumerate}
    \item $w$ has a zero of order $m$ at $R = 0$ and $R^{-m}w(R)$ has an even Taylor expansion at $R = 0$.
    \item $w(R)$ has the following expansion at $R = +\infty$
    $$w(R,a,b) = R^k \sum_{j=0}^l w_j(R^{-1}) \log(R)^{j},$$
    where $w_j$ has an even Taylor expansion at $R = 0$.
\end{enumerate}

$IS^m(R^k \log (R)^l, \mathcal{Q})$ will denote the space of analytic functions $u(r,t)$ on the cone $C_0 = \{(r,t): 0 \leq r < t, 0 < t < t_0\}$ given by a finite sum
$$u(r,t) = \fsum_{i} P_i((t\lambda)^{-2}) Q_i(r/t) w(r \lambda) = \fsum_{i} P_i(b) Q_i(a) w_i(r \lambda)$$
on the cone, for some polynomials $P_i(b)$, $Q_i \in \mathcal{Q}$ and $w_i \in S^m(R^k \log (R)^l)$. We have a similar definition with $\mathcal{Q}'$ instead of $\mathcal{Q}$.
\index{b@$b = (t\lambda)^{-2}$, a useful variable in dimension $d = 4$}
\index{IS-space@$IS^m(R^k \log (R)^l, \mathcal{Q})$, a vector space of smooth functions with respect to the variables $a$,$b$,$R$}
\end{definition}

\begin{proposition}
The following simple rules of calculations will be used throughout the proof:

\begin{enumerate}
    \item $(t\lambda)^{-2} = b = a^2R^{-2}$
    \item $IS^{m_1}(R^{k_1} \log (R)^{l_1}, \mathcal{Q}) IS^{m_2}(R^{k_2} \log (R)^{l_2}, \mathcal{Q}) \subset IS^{m_1+m_2}(R^{k_1+k_2} \log (R)^{l_1+l_2}, \mathcal{Q})$
    \item $P(b,a^2) IS^m(R^k \log (R)^l, \mathcal{Q}) \subset IS^m(R^k \log (R)^l, \mathcal{Q})$ for any bivariate polynomial $P(x,y)$
    \item $IS^m(R^k \log (R)^l, \mathcal{Q}) = R^i IS^{m-i}(R^{k-i} \log(R)^l, \mathcal{Q})$ for any $i \in \mathbb Z_{\leq m}$
    \item $IS^m(R^k \log (R)^l, \mathcal{Q}) = (1+R^2)^{i/2} IS^{m}(R^{k-i} \log(R)^l, \mathcal{Q})$ for any $i \in \mathbb Z$
    \item $b^i(1+R^2)^i IS^m(R^k \log (R)^l, \mathcal{Q}) = (b+a^2)^i  IS^m(R^k \log (R)^l, \mathcal{Q}) \subset IS^m(R^k \log (R)^l, \mathcal{Q})$ for any $i \in \mathbb N$
\end{enumerate}
The same rules hold with $\mathcal{Q}'$ instead of $\mathcal{Q}$. Moreover, any differential operator mapping $\mathcal{Q}$ to $\mathcal{Q}'$ (such as $a \partial_a$) maps $IS^m(R^k \log (R)^l, \mathcal{Q})$ to  $IS^m(R^k \log (R)^l, \mathcal{Q}')$. The same statement holds when exchanging the roles of $\mathcal{Q}$ and $\mathcal{Q}'$.
\end{proposition}

\begin{proposition} \label{b-decomposition}
Let $w(R,a,b) \in S^m(R^k \log (R)^l, \mathcal{Q})$. Then $w(R,a,b) - w(R,a,0) \in bS^m(R^k \log (R)^l, \mathcal{Q})$. 
\end{proposition}

\begin{proof}
Write
$$w(R,a,b) - w(R,a,0) = b \int_0^1 \partial_b w(R,a,tb) dt.$$
\end{proof}

\begin{theorem}\label{Thm:Approximate Solution d = 4}
In dimension $d = 4$, we prove that
\begin{align}
    v_{2k-1} &\in \frac{\lambda}{(t\lambda)^{2k}} IS^2(R^0 \log(R)^{m_k}, Q) ,\label{v_2k-1}\\
    t^2e_{2k-1} &\in \frac{\lambda}{(t\lambda)^{2k}} IS^{2}(R^0 \log(R)^{p_k},Q'), \label{e_2k-1}\\
    v_{2k} &\in \frac{\lambda}{(t\lambda)^{2k}} a^2IS^0(R^0 \log(R)^{p_k}, Q) \subset \frac{\lambda}{(t\lambda)^{2k+2}} IS^2(R^2 \log(R)^{p_k}, Q), \label{v_2k}\\
    t^2e_{2k} &\in \frac{\lambda}{(t\lambda)^{2k}} ( IS^0(R^{-2} \log(R)^{q_k},Q) + b IS^{2}(R^0 \log(R)^{q_k},Q') ). \label{e_2k}
\end{align}

for some increasing sequences of non-negative integers $m_k$, $p_k$, $q_k$, where $p_0 = q_0 = 0$, $m_1 = 1$. Moreover, for the $IS(\cdot, \cdot)$ part of $v_{2k-1}$ and $v_{2k}$, one can find representatives which do not depend on $b$. Additionally, $v_1$ and $t^2e_1$ have no $\mathcal{Q}$ element in their definition and the dominant components of $v_1, t^2e_1, v_2$ have no logarithm.
\end{theorem}

\subsubsection{Initialization}

One checks that $u_0, t^2 e_0 \in \lambda IS^0(R^{-2})$. We also define
\begin{align*}
M_k(v) &= v (3u_k^2 + 3 u_k v + v^2), \\
N_{2k-1}(v) &= M_{2k-2}(v) - pu_0^{p-1}v, \quad N_{2k}(v) = M_{2k-1}(v).
\end{align*}

\subsubsection{Construction of $v_{2k-1}$ from $e_{2k-2}$}

We write $$t^2 e_{2k-2} = t^2 \hat{e}_{2k-2}^0 + t^2 \hat{e}_{2k-2}^1 \in \frac{\lambda}{(t\lambda)^{2k-2}}( IS^0(R^{-2} \log(R)^{q_{k-1}},Q) + b IS^{2}(R^0 \log(R)^{q_{k-1}},Q') )$$
and further split $t^2 \hat{e}_{2k-2}^0$ into $\lambda (t\lambda)^{-(2k-2)}(w^0 +bw^1)$, where $w^0$ does not depend on $b$, as in Proposition \ref{b-decomposition}. We then set $$t^2 e^0_{2k-1} = \frac{\lambda}{(t\lambda)^{2k-2}}w^0(R,a).$$
In radial coordinates, (\ref{odd v}) reads as 
$$t^2 \mathcal{L}_R v_{2k-1}(r,t) = t^2 e^0_{2k-2}(r,t),$$
where 
$\mathcal{L}_r = -\partial_r^2 - \frac{3}{r} \partial_r - 3u_0^2$ and $t$ is a parameter. We do the change of variables $R = r \lambda(t)$ and get
$$(t \lambda)^2 \mathcal{L} v_{2k-1}(R,t) = t^2 e^0_{2k-2}(r,t),$$
where 
$\mathcal{L} = -\partial_R^2 - \frac{3}{R} \partial_R - 3W(R)^2$. We write $t^2 e^0_{2k-2}(r,t) = \lambda (t\lambda)^{-(2k-2)}w^0(R,a)$ and look for a solution $\lambda (t\lambda)^{-2k}v(R,a)$ by treating $a, t$ as parameters. This is the same as solving
$$\mathcal{L} v(R,a) = w^0(R,a), \ a = r/t,$$
where we ignore terms of $\mathcal{L} v$ which involve $\partial_a$ or $\partial_{aa}$. The initial conditions required are $v(0,a) = v'(0,a) = 0$. Then we prove that 
$$v_{2k-1} \in \frac{\lambda}{(t\lambda)^{2k}} IS^2(R^0 \log(R)^{q_{k-1}+2}, \mathcal{Q})$$
as in dimension $5$. We note that for $v_1$, there is no logarithm in the dominant component: the equation at infinity is a regular singular ODE given by (\ref{step 1 at infinity}) with $-2z^{-1}V_1$ replaced by $-z^{-1}V_1$. Hence, applying Theorem \ref{thm:inhomogeneous fuchs ode} ($r_1 = 2, r_2 = 0, \beta = 1$) yields an analytic solution and a logarithmic component $u_1(z)\log(z)$ where $u_1(z) = o(z^2)$, meaning that $v_{1} \approx R^0 + R^{-1} + R^{-2}\log(R)$ as $R \to +\infty$.

\subsubsection{Construction of $e_{2k-1}$ from $v_{2k-1}$}

We have
$$t^2e_{2k-1} = t^2N_{2k-1}(v_{2k-1}) + t^2e^1_{2k-2} - t^2E^tv_{2k-1} - t^2E^av_{2k-1},$$
where $E^tv_{2k-1} = \partial_{tt} [\lambda (t\lambda)^{-2k} v(r \lambda ,r/t)]$ but we ignore the $a$-derivatives and $E^av_{2k-1} = \square [\lambda (t\lambda)^{-2k} v(r \lambda,r/t)]$ but we keep only the terms where at least one $a$-derivative appears. The proof that $t^2 N_{2k-1}(v_{2k-1})$ belongs to the right space is an algebraic computation using the fact that 
\begin{align}
\begin{split}
    u_{k}-u_0 = v_1 + \sum_{i=2}^{2k-2}v_j &\in \frac{\lambda}{(t \lambda)^2}IS^2(R^0 \log(R)^{n_k},\mathcal{Q}) + \frac{\lambda}{(t \lambda)^4}IS^2(R^2 \log(R)^{n_k},\mathcal{Q}) \\
    &\subset \frac{\lambda}{(t \lambda)^2}(IS^2(R^0 \log(R)^{n_k},\mathcal{Q}) + a^2R^{-2} IS^2(R^2 \log(R)^{n_k},\mathcal{Q})) \\
    &\subset \frac{\lambda}{(t \lambda)^2} IS^0(R^0 \log(R)^{n_k},\mathcal{Q}) \\
    &\subset \lambda b (1+R^2) IS^0(R^{-2} \log(R)^{n_k},\mathcal{Q}) \\
    &\subset \lambda IS^0(R^{-2} \log(R)^{n_k},\mathcal{Q})
\end{split}
\end{align}
and $u_0 \in \lambda IS^0(R^{-2})$, so that $u_{k} \in \lambda IS^0(R^{-2} \log(R)^{n_k},\mathcal{Q})$ as well.

For $t^2 E^tv_{2k-1}$, we observe that
$$t^2 \partial_{tt} \left( \frac{\lambda}{(t\lambda)^{2k}} S^2(R^0 \log(R)^{m_k}) \right) \subset \frac{\lambda}{(t\lambda)^{2k}} S^2(R^0 \log(R)^{m_k}).$$
Finally, for $t^2 E^av_{2k-1}$, we write $v_{2k-1}(r,t) = \lambda (t\lambda)^{-2k}v(R,a)$ and observe that
\begin{align*}
    t^2 E^av_{2k-1} &= 2 t^2 \partial_{t} \left( \frac{\lambda}{(t \lambda)^{2k}} \right) v_{a}(R,a) \frac{-r}{t^2} + t^2 \frac{\lambda}{(t \lambda)^{2k}} \left( 2v_{aR}(R,a) \partial_{t}(\lambda)  \frac{-r^2}{t^2}  \right. \\
    &+ \left. v_{a}(R,a) \frac{2r}{t^3} +  v_{aa}(R,a) \frac{r^2}{t^4} - v_{a}(R,a) \frac{3}{rt} - v_{aa}(R,a) \frac{1}{t^2} - v_{aR}(R,a) \frac{2 \lambda }{t} \right) \\
    &= -2 t \partial_{t} \left( \frac{\lambda}{(t \lambda)^{2k}} \right) a v_{a}(R,a) + \frac{\lambda}{(t \lambda)^{2k}} \left( 2(\nu + 1)aR v_{aR}(R,a) \right. \\
    &+ \left. 2av_{a}(R,a) - (1 - a^2)v_{aa}(R,a) - 3a^{-1}v_{a}(R,a) - 2a^{-1}Rv_{aR}(R,a)\right) \\
    &\in \frac{\lambda}{(t\lambda)^{2k}} IS^2(R^0 \log(R)^{q_{k-1}}, \mathcal{Q}')
\end{align*}
and we note that, since the dominant component of $v_1$ contains no logarithm, the same is true for the dominant component of $t^2e_1$.

\subsubsection{Construction of $v_{2k}$ from $e_{2k-1}$}

As in Step 1, we keep from $t^2 e_{2k-1}$ the part $t^2 e_{2k-1}^0$ whose $IS(\cdot, \cdot)$ component is independent of $b$. We consider the main asymptotic component of $t^2e_{2k-1}^0$, i.e.,
$$t^2 \tilde{e}^0_{2k-1} = \frac{\lambda}{(t\lambda)^{2k}} \sum_{j=0}^{p_k} q_j(a) \log(R)^j, \ q_j \in \mathcal{Q}',$$
and solve the equation
$$t^2(-\partial_t^2 + \partial_r^2 + \frac{3}{r}\partial_r) \tilde{v}_{2k} = -t^2 \tilde{e}^0_{2k-1}.$$
Using Theorem \ref{hypergeometric ode with error forcing term}, we find a solution $\tilde{v}_{2k}$ of the following form
$$\tilde{v}_{2k} = -\frac{\lambda}{(t\lambda)^{2k}} \sum_{j=0}^{p_k} W_{j}(a) \log(R)^j, \quad W_{j}(a) \in a^2\mathcal{Q},$$
where we note that no logarithmic singularity has been created at $a = 0$. Then, we define
$$v_{2k} =  -\frac{\lambda}{(t\lambda)^{2k}} \sum_{j=0}^{p_k} W_{2k,j}(a) \frac{1}{2^j} \log(1+R^2)^j \in \frac{\lambda}{(t\lambda)^{2k}} a^2IS^2(R^0 \log(R)^{p_k}, Q),$$
where we use that
\begin{align}
\begin{split}
    \log(1+R^2)^j &= \left(\log\left(1 + \frac{1}{R^2}\right) + 2 \log(R)\right)^j \\
    &= \left( \sum_{n=1}^{\infty} \frac{(-1)^{n+1}}{n} R^{-2n} + 2\log(R) \right)^j, \quad R > 1 \label{log decomposition}
\end{split}
\end{align}
to get the expansion at infinity. As for $v_1$ and $t^2e_1$, the dominant component of $v_2$ contains no logarithm.

\subsubsection{Construction of $e_{2k}$ from $v_{2k}$}

We define 
$$t^2 e^0_{2k-1} = \frac{\lambda}{(t\lambda)^{2k}} \sum_{j=0}^{p_k} q_j(a) \left(\frac{1}{2} \log(1+R^2)\right)^j$$
and write
\begin{align*}
    t^2e_{2k} &= t^2(e_{2k-1} - \square v_{2k} + N_{2k}(v_{2k})) \\
    &= t^2(e_{2k-1} - e^0_{2k-1}) + t^2(e^0_{2k-1} - \square v_{2k}) + t^2 N_{2k}(v_{2k}) 
\end{align*}
and we prove that each part belongs to the right space. The proof of $$t^2 N_{2k}(v_{2k}) \in \frac{\lambda}{(t\lambda)^{2k}}  IS^2(R^{-2} \log(R)^{p_k}, \mathcal{Q})$$ is just simple algebra. By construction, one has 
$$t^2(e_{2k-1} - e^0_{2k-1}) \in \frac{\lambda}{(t\lambda)^{2k}} IS^0(R^{-2} \log(R)^{p_k}, \mathcal{Q}'),$$
where everything is straightforward except the expansion at infinity which follows from formula (\ref{log decomposition}) and the fact that the main asymptotic component of $t^2(e_{2k-1} - e^0_{2k-1})$ is of the form
$$\frac{\lambda}{(t\lambda)^{2k}} \sum_{j=0}^{p_k} q_j(a) \left( \log(R)^j - \frac{1}{2^j} \log(1+R^2)^j \right).$$
This finishes the study of the term $t^2(e_{2k-1} - e^0_{2k-1})$ since one has the inclusion
$$IS^0(R^{-2} \log(R)^{p_k}, \mathcal{Q}') \subset IS^0(R^{-2} \log(R)^{p_k}, \mathcal{Q}) + bIS^2(R^0 \log(R)^{p_k}, \mathcal{Q}')$$
by writing $w(R,a,b) = (1-a^2)w(R,a,b) + bR^2w(R,a,b)$.

It remains to prove that 
$$f = t^2(e^0_{2k-1} - \square v_{2k}) \in \frac{\lambda}{(t\lambda)^{2k}} IS^0(R^{-2} \log(R)^{p_k}, \mathcal{Q}').$$
Again, the only part which is not straightforward is the behaviour at infinity. For this, we write
\begin{align*}
    f &= t^2(\tilde{e}^0_{2k-1} - \square \tilde{v}_{2k}) + t^2(e^0_{2k-1} - \tilde{e}^0_{2k-1}) - t^2 \square(v_{2k} - \tilde{v}_{2k}) \\
    &= 0 + t^2(e^0_{2k-1} - \tilde{e}^0_{2k-1}) + t^2\square(v_{2k} - \tilde{v}_{2k}).
\end{align*}
Note that $t^2(e^0_{2k-1} - \tilde{e}^0_{2k-1})$ is the main asymptotic component of $-t^2(e_{2k-1} - e^0_{2k-1})$ which was already studied earlier. So it remains to deal with $t^2\square(v_{2k} - \tilde{v}_{2k})$, where $$(v_{2k} - \tilde{v}_{2k}) = \frac{\lambda}{(t\lambda)^{2k}}\sum_{j=1}^{p_k}a^2W_j(a) [\log(R)^j - \log(1+R^2)^j], \quad W_j \in \mathcal{Q}.$$
One computes explicitly
\begin{align*}
t^2 \square \left( \frac{\lambda}{(t\lambda)^{2k}}w(R,a) \right) = \frac{\lambda}{(t\lambda)^{2k}} &\left[ 1 + a^{-1}\partial_a + a \partial_a + (a^2-1)\partial_{aa} + a^{-2}R\partial_R + R\partial_R  \right. \\
&\left. \phantom{[1} + a^{-1} \partial_a R \partial_R + a \partial_a R \partial_R + R^2 \partial_{RR} + a^{-2} R^2 \partial_{RR} \right] w(R,a)
\end{align*}
up to some omitted multiplicative constants depending on $k,\nu$. Hence, for terms of the form
$$w(R,a) = \frac{\lambda}{(t\lambda)^{2k}}a^2W_j(a) [\log(R)^j - \log(1+R^2)^j], \quad W_j \in \mathcal{Q},$$
we get the desired results. 
% Note that from (\ref{log decomposition}),
% $$R^k \partial_R^k \left[ \log(R)^j - \frac{1}{2^j}\log(1+R^2)^j\right] = a_k R^{-2}\log(R)^{j-1} + \text{ lower order terms}, \quad R > 1$$

\printindex

\section*{Declaration}
The authors declare that they have no conflict of interest.

\vfill

\nocite{*}
\bibliographystyle{alpha}
\bibliography{biblio}

\end{document}